\documentclass[10pt,psamsfonts]{article}

\usepackage{amsmath, amsfonts, amsthm, amssymb, amscd, url, slashed, stmaryrd, upgreek, thmtools, mathdots, cite}
\usepackage[all,arc,color]{xy}
\usepackage{enumerate}
\usepackage{mathrsfs}
\usepackage{mathtools}
\usepackage[utf8]{inputenc}
\usepackage{tikz-cd}

 \usepackage[colorlinks = true,
            linkcolor = blue,
            urlcolor  = cyan,
            citecolor = red,
            anchorcolor = blue,
            pagebackref]{hyperref}
            
\usepackage[alphabetic,backrefs]{amsrefs}
 
 \usepackage[parfill]{parskip}
 \usepackage{anysize}
 \marginsize{1in}{1in}{1in}{1in}
 
\begingroup
 \makeatletter
 \@for\theoremstyle:=definition,remark,plain\do{%
 \expandafter\g@addto@macro\csname th@\theoremstyle\endcsname{%
 \addtolength\thm@preskip\parskip
 }%
 }
 \endgroup

\declaretheorem[name=Theorem,numberwithin=section]{thm}
\declaretheorem[name=Proposition,numberlike=thm]{prop}
\declaretheorem[name=Lemma,numberlike=thm]{lemma} 
\declaretheorem[name=Corollary,numberlike=thm]{cor}

\declaretheorem[name=Definition,style=definition,qed=$\blacktriangle$,numberlike=thm]{defn}

\declaretheorem[name=Remark,style=definition,qed=$\blacktriangle$,numberlike=thm]{rmk}

\newcounter{noteCounter}
\newcommand{\dd}{\mathrm{d}}

\DeclareMathOperator\Div{div}
\DeclareMathOperator\curl{curl}
\DeclareMathOperator\vol{vol}

\DeclareMathOperator\tr{tr}

\DeclareMathOperator\Id{Id}

\newcommand{\wt}{\widetilde}
\newcommand{\wh}{\widehat}

\newcommand{\hT}{\wh{T}}
\newcommand\tRc{\mathrm{Rc}}
\newcommand{\tR}{\mathrm{R}}

\newcommand{\tRm}{\mathrm{Rm}}
\newcommand{\tW}{\mathrm{W}}

\newcommand{\operp}{\overset{\scriptscriptstyle{\perp}}{\oplus}}
\newcommand{\mb}[1]{\mathbf{#1}}
\newcommand{\sA}{\mathsf A}
\newcommand{\sB}{\mathsf B}
\newcommand{\sT}{\mathsf T}
\newcommand{\sI}{\mathsf I}
\newcommand{\sR}{\mathsf{R}}
\newcommand{\sW}{\mathsf{W}}
\newcommand{\cL}{\mathcal L}

\newcommand{\cS}{\mathcal S}

\newcommand{\G}{\mathrm{G}_2}

\newcommand{\Spin}[1]{\mathrm{Spin}(#1)}
\newcommand{\GL}[1]{\mathrm{GL}(#1)}

\newcommand{\U}[1]{\mathrm{U}(#1)}
\newcommand{\SU}[1]{\mathrm{SU}(#1)}
\newcommand{\SO}[1]{\mathrm{SO}(#1)}

\newcommand{\ol}{\overline}

\newcommand{\R}{\mathbb R}

\newcommand{\PR}{\mathbb P}

\newcommand{\Lam}{\Lambda}
\newcommand{\Sym}{\mathrm{S}}
\newcommand{\K}{\mathcal{K}}
\newcommand{\W}{\mathcal{W}}

\newcommand{\Span}{\mathrm{span}}

\newcommand{\ph}{\varphi}
\newcommand{\ps}{\psi}
\newcommand{\st}{\star}
\newcommand{\hk}{\mathbin{\! \hbox{\vrule height0.3pt width5pt depth 0.2pt \vrule height5pt width0.4pt depth 0.2pt}}}

\newcommand{\lieg}{\mathfrak {g}_2}
\newcommand{\w}{\wedge}

\newcommand{\nab}[1]{\nabla_{#1}}
\newcommand{\hnab}[1]{\wh{\nabla}_{#1}}
\newcommand{\tnab}[1]{\wt{\nabla}_{#1}}
\newcommand{\del}{\partial}
\newcommand{\delbar}{\overline{\partial}}

\newcommand{\dds}{\frac{d}{d s}}
\newcommand{\delt}{\partial_t}

\newcommand{\dx}[1]{d x^{#1}}

\newcommand{\eps}{\varepsilon}
\newcommand{\im}{\operatorname{im}}

\newcommand{\vf}{\mathfrak{X}}

\newcommand{\rest}[2]{ { \left. {#1} \right|}_{{#2} }}

\newcommand{\sym}[1]{\Sym^{#1} (T^* M)}

\newcommand{\cT}{\mathcal T}
\newcommand{\oct}{\circledcirc}
\newcommand{\Pop}{\mathsf P}
\newcommand{\Vop}{\mathsf V}
\newcommand{\symm}{\mathrm{sym}}
\renewcommand{\skew}{\mathrm{skew}}
\newcommand{\KK}[1]{{}_{#1} \! K}
\newcommand{\lot}{\ell ot}

\makeatletter
\let\c@equation\c@thm
\makeatother
\numberwithin{equation}{section}

\begin{document}

\title{Flows of $\G$-structures, II: \\ Curvature, torsion, symbols, and functionals}

\author{Shubham Dwivedi \\ \emph{Institut f\"ur Mathematik, Humboldt Universit\"at zu Berlin} \\ \tt{dwivedis@hu-berlin.de} \and Panagiotis Gianniotis \\ \emph{National and Kapodistrian University of Athens} \\\tt{pgianniotis@math.uoa.gr} \and Spiro Karigiannis \\ \emph{Department of Pure Mathematics, University of Waterloo} \\ \tt{karigiannis@uwaterloo.ca}}

\maketitle

\begin{abstract}
We continue the investigation of general geometric flows of $\G$-structures initiated by the third author in ``Flows of $\G$-structures, I.'' Specifically, we determine the possible geometric flows (up to lower order terms) of $\G$-structures which are second-order quasilinear, by explicitly computing all independent second-order differential invariants of $\G$-structures which are $3$-forms. There are four symmetric $2$-tensors and two vector fields. We do this by deriving explicit computational descriptions of the decompositions of the curvature and the covariant derivative of the torsion into irreducible $\G$-representations, as well as the decomposition of the $\G$-Bianchi identity into independent relations. We also show that these six tensors arise as leading order contributions to the Euler-Lagrange equations for the energy functionals of the four independent torsion components, and we establish a $\G$-analogue of the classical block decomposition of the Riemann curvature operator on oriented $4$-dimensional Riemannian manifolds.

Finally, we present a large class of geometric flows of $\G$-structures which are directly amenable to a DeTurck type trick to establish short-time existence and uniqueness, with no initial assumption on the torsion, vastly generalizing an earlier result of Weiss--Witt for the negative gradient flow of the Dirichlet energy. This result is proved through a careful analysis of the principal symbols of the linearizations of these operators, establishing particular linear combinations for which one can prove that the failure of strict parabolicity is due precisely to the diffeomorphism invariance.

A detailed introductory section on various foundational results on $\G$-structures, several of which are not readily available in the literature, should be of wider interest and applicability.
\end{abstract}

\tableofcontents

\section{Introduction} \label{introsec}

The present paper is a direct sequel to the paper~\cite{K-flows} by the third author, which initiated the study of general flows of $\G$-structures. More precisely, given a smoothly varying family $\ph_t$ of $\G$-structures, in~\cite{K-flows} we determined the induced variations of various tensors induced from $\ph_t$, including the metric $g_t$, the dual $4$-form $\ps_t$, and the torsion $T_t$. One corollary of this analysis was the $\G$-Bianchi identity, the fundamental relation between torsion and curvature for $\G$-structures, which plays a crucial role throughout $\G$-geometry. It is reviewed in Section~\ref{sec:g2bianchi} of the present paper.

While prior familiarity with~\cite{K-flows} is helpful, it is not strictly necessary. The most important results of~\cite{K-flows} are reproved here using improved (streamlined) arguments. The reader willing to accept the fundamental contraction identities between $\ph$, $\ps$, and $g$ in equations~\eqref{eq:phph},~\eqref{eq:phps}, and~\eqref{eq:psps} will find the paper mostly self-contained. It is somewhat remarkable how essentially all the interesting features of $\G$-geometry can be traced back to the fundamental identity $\ph_{ijp} \ph_{klp} = g_{ik} g_{jl} - g_{il} g_{jk} - \ps_{ijkl}$ of~\eqref{eq:phph}, which is itself simply a manifestation of the non-associativity of octonion multiplication. See~\cite{K-intro} for more details.

In this section we first review some of the historical developments of flows of $\G$-structures, then motivate and summarize our main results, and finally discuss notation and preliminaries. Readers who are not already familiar with $\G$-structures should probably first read Sections~\ref{sec:preliminaries},~\ref{sec:review}, and~\ref{sec:evolution-g2} before coming back to Sections~\ref{sec:history} and~\ref{sec:summary-results}.

\subsection{A brief history of flows of $\G$-structures} \label{sec:history}

We give a brief and incomplete review of some of the important developments in the study of flows of $\G$-structures. An excellent and highly recommended survey of the state of the art as of 2018 is Lotay~\cite{Minischool-Lotay}. Throughout this section we assume $M$ is a \emph{compact} $7$-manifold admitting $\G$-structures.

\textbf{The Laplacian flow.} The first proposal for a natural flow of $\G$-structures was introduced by Bryant~\cite{Bryant}, and is now called the \emph{Laplacian flow}. Its formulation was motivated by a fundamental existence theorem of Joyce, which (somewhat imprecisely) is the following.

\begin{thm}[Joyce {\cite[Th.~11.6.1]{Joyce}}] \label{thm:Joyce}
Let $\ph$ be a closed $\G$-structure on $M$. Suppose $\ph$ is ``almost'' coclosed, in the sense of precise estimates on $C^0, L^2, L^{14}_1$ norms of the torsion. Then there exists a \emph{torsion-free} $\G$-structure $\tilde \ph$ in the cohomology class $[\ph] \in H^3 (M, \R)$. Moreover, $\tilde \ph$ is $C^0$-close to $\ph$, and it is (modulo diffeomorphisms isotopic to the identity) the unique torsion-free $\G$-structure in the class $[\ph]$ \emph{sufficiently close} to $\ph$.
\end{thm}

Joyce's theorem says that it is fruitful to start with a closed $\G$-structure $\ph$ and to deform it within its cohomology class $[\ph] \in H^3 (M, \R)$. This is similar in spirit to Yau's proof of the Calabi conjecture, which says that (on a compact K\"ahler manifold with vanishing first Chern class), there is a unique Ricci-flat K\"ahler metric in each K\"ahler class. Note however that in the $\G$ case, we still do not know if torsion-free $\G$-structures (if they even exist) are unique in a given cohomology class. We are very far from having a Calabi--Yau type theorem in $\G$-geometry.

Bryant's Laplacian flow is defined to be
\begin{equation} \label{eq:Lap-flow}
\tfrac{\partial}{\partial t} \ph_t = \Delta_{\dd} \ph_t, \qquad \dd \ph_0 = 0.
\end{equation}
That is, we start with a \emph{closed} $\G$-structure $\ph_0$ and evolve it in the direction of its Hodge Laplacian. Bryant--Xu~\cite{Bryant-Xu} established short-time existence and uniqueness for this flow, and showed that the closed condition is preserved along the flow.

\begin{rmk} \label{rmk:Lap-flow-STE}
Strictly speaking, Bryant--Xu actually showed the following. Suppose $\ph_t = \ph_0 + \dd \sigma_t$, with $\sigma_0 = 0$. They considered the flow~\eqref{eq:Lap-flow} for such $\ph_t$, which becomes
\begin{equation} \label{eq:Lap-flow-alt}
\tfrac{\partial}{\partial t} (d \sigma_t) = \dd \dd^* (\ph_0 + \dd \sigma_t).
\end{equation}
They showed that the flow~\eqref{eq:Lap-flow-alt} has short-time existence and uniqueness. In theory (although it is probably unlikely), this does not preclude the existence of other short-time solutions of~\eqref{eq:Lap-flow} which immediately fail to remain closed. It just proves uniqueness amongst solutions to~\eqref{eq:Lap-flow} which stay in the given cohomology class. This is somewhat different in spirit from other geometric flows which preserve some initial condition, such as Ricci flow preserving positive scalar curvature~\cite{Chow-Knopf}, or mean curvature flow preserving the Lagrangian condition~\cite{Smoczyk}, where short-time existence is established in general and then preservation of a given condition along the flow is established using the maximum principle. Bryant--Xu's approach is instead similar to Cao's approach to K\"ahler--Ricci flow~\cite{Cao}, where he similarly from the outset forces the solution to lie in a fixed K\"ahler class and proves short-time existence and uniqueness that way. In that case, we do indeed have general uniqueness, because Shi~\cite{Shi} proved in general that the K\"ahler condition is preserved under Ricci flow using the maximum principle.
\end{rmk}

Further motivation for the Laplacian flow comes from work of Hitchin~\cite{Hitchin}. Consider the functional $\ph \mapsto \int_M \vol_{\ph}$ on the space of $\G$-structures on $M$. Hitchin showed that, when restricted to a fixed cohomology class in $H^3 (M, \R)$, the critical points of this functional are the torsion-free $\G$-structures, and they are strict local maxima (transverse to the action of diffeomorphisms). In fact, the (positive) gradient flow of this functional, restricted to a cohomology class, is precisely the Laplacian flow~\eqref{eq:Lap-flow-alt}.

Several important foundational analytic results for the Laplacian flow were established by Lotay--Wei in a series of papers~\cite{LW1, LW2, LW3}. These include: characterization of the blow-up time, dynamical stability, and real analyticity. The Laplacian flow has also been studied with symmetry (various dimensionally reduced situations) in~\cite{FY, LL, PS}.

\textbf{The coflow and its modification.} One could equally take the Hodge dual $4$-form $\ps = \st_{\ph} \ph$ as the fundamental object, so it makes sense to consider the Hodge Laplacian flow of the $4$-form, namely:
\begin{equation} \label{eq:coflow}
\tfrac{\partial}{\partial t} \ps_t = \Delta_{\dd} \ps_t, \qquad \dd \ps_0 = 0.
\end{equation}
This flow was introduced (with the opposite sign) in Karigiannis--McKay--Tsui~\cite{KMT}, where it was called the \emph{coflow}. The approach of Bryant--Xu for short-time existence and uniqueness of~\eqref{eq:Lap-flow-alt} does not work for the coflow~\eqref{eq:coflow}. This issue was clarified by Grigorian~\cite{Grigorian-MCF}, who introduced the \emph{modified coflow}
$$ \tfrac{\partial}{\partial t} \ps_t = \Delta_{\dd} \ps_t + 2 \dd (C - (\tr_t T_t) \st_t \ps_t), \qquad \dd \ps_0 = 0, $$
where $T_t$ is the torsion of $\ps_t$ and $C$ is a constant. Grigorian proved short-time existence and uniqueness for the modified coflow, and that the coclosed condition is preserved, in the same particular sense as described in Remark~\ref{rmk:Lap-flow-STE}. However, the fixed points of the modified coflow are not well understood, and in particular include more than just torsion-free $\G$-structures. It remains an open question whether or not the coflow, whose fixed points are precisely the torsion-free $\G$-structures, has short-time existence and uniqueness. Additionally, it is worth remarking that the coflow has a variational interpretation as the gradient flow of the Hitchin volume functional restricted to the cohomology class $[\ps_0] \in H^4 (M, \R)$.

\begin{rmk} \label{rmk:coflow}
The fact that a $4$-form version of the Laplacian flow is less well-behaved seems to be closely related to the difficulty in proving a $4$-form version of Joyce's Existence Theorem~\ref{thm:Joyce}. That is, if we start with a \emph{closed positive $4$-form $\ps$}, which is sufficiently close to torsion-free in some precise sense, can we always deform to a nearby torsion-free $4$-form in the same cohomology class? This question is currently being investigated by Dwivedi--Karigiannis.
\end{rmk}

\textbf{The Dirichlet energy flow.} The Dirichlet energy functional for $\G$-structures is the map
$$ \ph \mapsto \int_M |T|^2 \vol, $$
which is the (square of the) $L^2$ norm of the torsion. Weiss--Witt~\cite{WW1, WW2} used a DeTurck trick to show that the negative gradient flow of the Dirichlet energy has short-time existence and uniqueness. Their result is now a special case of our main Theorem~\ref{thm:main-short-time}. (See Remark~\ref{rmk:Weiss-Witt-special-case}.)

On a fixed oriented spin Riemannian $7$-manifold, a $\G$-structure is equivalent (up to sign) to a unit spinor field. Using a spinorial approach, Ammann--Weis--Witt~\cite{AWW} also studied the negative gradient flow of the Dirichlet energy, thought of as a function of a unit spinor field. They proved general short-time existence and uniqueness. A characterization of the blow-up time for this flow was obtained by He--Wang~\cite{HW}.

\textbf{The isometric flow.} Another flow of $\G$-structures that has received some attention is the \emph{isometric flow} or $\Div T$ flow. The map that associates to a $\G$-structure $\ph$ its induced Riemannian metric $g_{\ph}$ is not injective. In fact, at a point in $M$, the space of $\G$-structures inducing a given metric is an $\R \PR^7$. The isometric flow is the negative gradient flow of the functional $\ph \mapsto \int_M |T|^2 \vol$, \emph{restricted} to the set of $\G$-structures inducing a fixed metric, and it takes the form
\begin{equation} \label{eq:isometric-flow}
\tfrac{\partial}{\partial t} \ph_t = (\Div T_t) \hk \ps_t.
\end{equation}
This flow was introduced by Grigorian~\cite{Grigorian-Oct}. Many analytic properties of this flow were established by Grigorian~\cite{Grigorian-IF} and Dwivedi--Gianniotis--Karigiannis~\cite{DGK}, including: derivative estimates, characterization of blow-up time, compactness of solution space, almost monotonicity of localized energy, $\eps$-regularity, long-time existence and convergence given small initial entropy, and structure of the singular set.

The isometric flow is easier to study because it is strictly parabolic. That is, one does not need a DeTurck trick to establish short-time existence and uniqueness. The reason this flow is parabolic is because the metric $g$ on $M$ is fixed, so there is no diffeomorphism invariance. However, as the metric does not change along the isometric flow, if one wanted to use this flow to evolve to a torsion-free $\G$-structure, one would have to start with a Ricc-flat metric to begin with, which is not practical. Nevertheless, this flow plays an important role in the present paper, as we prove in Theorem~\ref{thm:main-short-time} that a ``coupling'' of Ricci flow and the isometric flow has good short-time existence and uniqueness. (See also Remark~\ref{rmk:Ricci-isometric-coupled}.) A survey of results about the isometric flow was given by Grigorian~\cite{Grigorian-IF-survey}. The isometric flow of $\Spin{7}$-structures was also studied by Dwivedi--Loubeau--S\'a Earp~\cite{DLS}.

The isometric flow was extended to $n$-manifolds with general $G$-structures satisfying $G \subseteq \SO{n}$ by Loubeau--S\'a Earp in~\cite{LS}. This is the negative gradient flow of the $L^2$ norm of the intrinsic torsion of the $G$-structure, restricted to $G$-structures inducing the same metric. Critical points for such flows are called \emph{harmonic} $G$-structures. This work was continued by Fadel--Loubeau--Moreno--S\'a Earp in~\cite{FLMS}.

\subsection{Motivation and brief summary of main results} \label{sec:summary-results}

While the Laplacian flow of $\G$-structures has certainly had the most success amongst all geometric flows of $\G$-structures considered thus far, it remains unclear if this is the ``best'' way to evolve $\G$-structures towards torsion-free $\G$-structures. Here are two reasons for this:
\begin{itemize} \setlength\itemsep{-1mm}
\item The proof of short-time existence (STE) is somewhat unsatisfying, given that the preservation of the closed condition is built-in from the outset. An argument similar to preservation of conditions in other natural geometric flows, as discussed in Remark~\ref{rmk:Lap-flow-STE}, using a maximum principle, would be desirable. But such a result requires general STE and uniqueness for the flow $\frac{\partial}{\partial t} \ph_t = \Delta_d \ph_t$ without any assumption on the initial torsion. Such a result is not known and seems out of reach.
\item It is not clear if \emph{starting closed and preserving the cohomology class} is the right thing to do. We do not know if there is global uniqueness of torsion-free $\G$-structures in a given cohomology class. Moreover, at present it is unknown what are necessary and sufficient conditions for existence of a closed $\G$-structure on a compact $7$-manifold which admits $\G$-structures. [For example, it is an important open problem whether the standard smooth $S^7$ admits closed $\G$-structures.] In this sense the coflow is perhaps better, because Crowley--Nordstr\"om~\cite{CN} proved using Gromov's $h$-principle that any manifold which admits $\G$-structures admits \emph{coclosed} $\G$-structures. Recall, however, that short-time existence of the coflow starting from a coclosed $\G$-structure is still open.
\end{itemize}

Indeed, it is instructive to compare $\G$-geometry with K\"ahler/Calabi--Yau geometry. Yau's solution~\cite{Yau} of the Calabi conjecture says that (provided the first Chern class vanishes) we can start with a K\"ahler form $\omega$ and find a (\emph{globally unique}) Ricci-flat K\"ahler form $\tilde \omega$ in the same cohomology class $[\omega]$. The parabolic version, proved by Cao~\cite{Cao}, uses K\"ahler--Ricci flow. Both the elliptic and parabolic approaches rely heavily on the $\del \delbar$ lemma of K\"ahler geometry. There is no such analogous result in $\G$-geometry. The problem is that while the complex and symplectic geometry of a K\"ahler manifold are essentially independent, there is no such decoupling of $\G$-geometry into two independent geometries.

It thus makes sense to consider other ``reasonable'' geometric flows of $\G$-structures. A reasonable flow of $\G$-structures should of course have short-time existence and uniqueness. Ideally (in analogy with the Ricci flow), it should be amenable to a DeTurck trick yielding equivalence with a strictly parabolic (heat-like) flow. In particular, it should be of the form
$$ \frac{\partial}{\partial t} \ph_t = P(\ph_t), $$
where $\ph \mapsto P(\ph)$ is some second-order differential invariant of $\G$-structures. Thus we need to determine all the (independent) second-order differential invariants of $\G$-structures \emph{which are $3$-forms}.

\begin{rmk} \label{rmk:RB}
One could ask the same question for flows of Riemannian metrics. The only second-order differential invariant of a metric $g$ is the Riemann curvature tensor, which in general decomposes (see equation~\eqref{eq:curvature-n} for details) intro three independent pieces as irreducible $\SO{n}$-representations. Exactly two of these can be identified with symmetric $2$-tensors, namely $Rg$ and $\tRc^0$, where $R$ is the scalar curvature and $\tRc^0$ is the traceless Ricci tensor. Thus the most general geometric flow of Riemannian metrics is $\frac{\partial}{\partial t} g_t = a \tRc_t + b R_t g_t$ for some constants $a, b$. In fact, if we take $a = - 2$ (in which case $b=0$ corresponds to the Ricci flow), then we obtain the \emph{Ricci--Bourguignon flow}, which has good short-time existence for $|b|$ small. (See~\cite{CCDMM} for details.)
\end{rmk}

In this paper we determine all of the \emph{independent} second-order differential invariants of a $\G$-structure which correspond to $3$-forms, and thus classify (up to lower order terms) all possible geometric flows of $\G$-structures which could be heat-like. In the process, we derive several explicit formulas for the decompositions into irreducible $\G$-representations of certain $\G$-representations. We apply these results to the Riemann curvature tensor $\tRm$ of a $\G$-structure, reproducing some results of Cleyton--Ivanov~\cite{CI}, and to the covariant derivative $\nab{} T$ of the torsion. We also explicitly decompose the $\G$-Bianchi identity into independent relations. We make several applications of these identities.

First, in Section~\ref{sec:functionals-revisited} we derive explicit formulas for the first variation of the $L^2$ norms of the various components of the torsion of a $\G$-structure, giving (to leading order) a geometric interpretation of these six second-order differential invariants. See Corollary~\ref{cor:evolution-torsion-quantities-4} and Corollary~\ref{cor:evolution-torsion-functionals} for precise statements.

Next, in Section~\ref{sec:curvature-decomp}, we establish a $\G$-analogue of the classical block decomposition of the Riemann curvature operator on oriented $4$-dimensional Riemannian manifolds. In particular, our analysis provides a geometric interpretation for two classes of ``generalized Einstein'' $\G$-structures in the sense of Cleyton--Ivanov~\cite[Equation (4.23)]{CI}. (See Corollary~\ref{cor:main-curv-decomp} and Remark~\ref{rmk:CI-Einstein}.)

Then, in Theorem~\ref{thm:invariants} we determine that there are \emph{six independent second-order differential invariants of a $\G$-structure which are $3$-forms}. Four of these correspond to symmetric $2$-tensors, yielding elements of $\Omega^3_{1+27}$, and two correspond to vector fields, yielding elements of $\Omega^3_7$. Explicitly, these are:
$$ \boxed{\, \, \, \underbrace{Rg, \, \tRc, \, F, \, \mathcal{L}_{\Vop T} g,}_{\text{\small{symmetric $2$-tensors}}} \qquad \underbrace{\Div T, \, \Div T^t.}_{\text{\small{vector fields}}} \, \,} $$
Here the torsion $T$ of the $\G$-structure is a $2$-tensor with transpose $T^t$, where $\Vop T$ is the ``vector torsion'' corresponding to the $\Omega^2_7 \cong \Omega^1$ component, and $F$ is a symmetric $2$-tensor obtained from the Riemann curvature tensor $R_{ijkl}$ and the $\G$-structure by $F_{pq} = R_{ijkl} \ph_{ijp} \ph_{klq}$.

Thus, up to lower order terms which we denote by $\lot$ (and which by scaling arguments from Section~\ref{sec:scaling} must be quadratic in the torsion), we determine that all possible geometric flows of $\G$-structures which are second-order quasilinear must be of the form
$$ \frac{\partial}{\partial t} \ph(t) = (\mu \tRc + \nu R g + a \cL_{\Vop T} g + \lambda F) \diamond \ph + (b_1 \Div T+ b_2 \Div T^t) \hk \ps + \lot $$
for some constants $\mu, \nu, a, \lambda, b_1, b_2$. (The $\diamond$ operation on forms is defined in~\eqref{eq:diamond-coords}. It is essentially the induced Lie algebra action of $\GL{7, \R}$ on forms.) Our principal application of this result is the following theorem, where we take $\mu = -1$ and $\nu = 0$.

\textbf{Theorem~\ref{thm:main-short-time}.} \emph{Let $(M,\ph_0)$ be a compact $7$-manifold with a $\G$-structure $\ph_0$. Consider the flow}
\begin{equation} \label{eq:g2flow_general-intro}
\begin{aligned}
\frac{\partial}{\partial t} \ph(t) & = (- \tRc + a \cL_{\Vop T} g + \lambda F) \diamond \ph + (b_1 \Div T+ b_2 \Div T^t) \hk \ps, \\
\ph(0) & = \ph_0,
\end{aligned}
\end{equation}
\emph{and suppose that $0 \leq b_1 - a-1 < 4$, $b_1 + b_2 \geq 1$ and $|\lambda| < \frac{1}{4} c$, where $c = 1 - \frac{1}{4} (b_1 - a-1) > 0$.}

\emph{Then there exists $\eps > 0$ and a unique smooth one-parameter family of $\G$-structures $\ph(t)$ for $t \in [0,\eps)$, solving~\eqref{eq:g2flow_general-intro}.}

Theorem~\ref{thm:main-short-time} is proved precisely by showing that the flows of the form~\eqref{eq:g2flow_general-intro} are amenable to a DeTurck trick (with some modifications, see Remark~\ref{rmk:STE}). We do this in Section~\ref{sec:symbols-short-time} via a careful analysis of the principal symbols of the linearizations of the various operators involved,  including an important estimate on the symbol of the linearization of the curvature-type operator $F$ in~\eqref{eq:symbDFineq}.

We present in Section~\ref{sec:review} an extensive and detailed discussion of foundational results on $\G$-structures, including many explicit computations deriving formulas that are needed throughout the paper. We also include several applications of this material to topics in $\G$-geometry not directly related to the main results of this paper. We hope that collecting such material here will prove to be useful to early career researchers in the field.

Several of the ``representation-theoretic'' results, especially about the decomposition of curvature and torsion, have appeared before in some form, for example in Bryant~\cite{Bryant} and Cleyton--Ivanov~\cite{CI}. Those treatments tend to use the full machinery of abstract representation theory (such as roots and weights). Our approach throughout the paper, and especially in Section~\ref{sec:more-rep-theory}, is very concrete, relying entirely on explicit computation in a local orthonormal frame using the fundamental contraction identities of a $\G$-structure. While this point of view is less elegant, it is more accessible, and it gives the reader a good sense of how the local geometry of $\G$-structures essentially comes from the contraction identities.

\textbf{Acknowledgements.} The authors thank Albert Chau, Jingyi Chen, Jason Lotay, Sebastien Picard, Henrique S\'a Earp, and Caleb Suan for useful discussions. Some of this research was conducted in 2017 during the Special Thematic Program on Geometric Analysis at the Fields Institute. The research of the third author is supported by an NSERC Discovery Grant.

\subsection{Notation and preliminaries} \label{sec:preliminaries}

Throughout the paper, $M$ is a smooth manifold equipped with a Riemannian metric $g$ that is usually but not always induced from a $\G$-structure $\ph$. We use the metric to identify vector fields with $1$-forms. We express tensors with respect to a local frame $\{ e_1, \ldots, e_n \}$ that is orthonormal with respect to $g$, and therefore all our indices are \emph{subscripts}, and any repeated indices are summed over all possible values from $1$ to $\dim M$. One needs to be especially careful when doing this when we \emph{differentiate a contraction}. For example, if we are differentiating then we need to recall that $A_{ii}$ really means $A_{ij } g^{ij}$. This is always made clear when it is an issue. We write $\nab{k}$ for $\nab{e_k}$ and $\partial_k$ for $\frac{\partial}{\partial x^k}$, so $\delt$ denotes $\frac{\partial}{\partial t}$. Whenever an operator like $\nab{p}$ or $\delt$ appears, it acts \emph{only} on the term immediately following, unless there are parentheses. Thus, for example, $\nab{p} \ph_{ijk} \ps_{qijk}$ means $(\nab{p} \ph_{ijk}) \ps_{qijk}$, and \emph{not} $\nab{p} (\ph_{ijk} \ps_{qijk})$.

Given a tensor bundle $E$ over $M$, we use $\Gamma(E)$ to denote the space of smooth sections of $E$. These spaces are denoted in other ways in some particular cases:
\begin{itemize} \setlength\itemsep{-1mm}
\item $\Omega^k = \Gamma( \Lambda^k (T^* M))$ is the space of smooth $k$-forms on $M$
\item $\Omega^{\bullet} = \oplus_{k=0}^n \Omega^k$ is the space of all smooth forms on $M$, where $n = \dim M$
\item $\vf = \Gamma(T M)$ is the space of smooth vector fields on $M$
\item $\cT^k = \Gamma( \otimes^k (T^* M))$ is the space of smooth covariant $k$-tensors on $M$
\item $\cS^k = \Gamma (\sym{k})$ is the space of smooth \emph{symmetric} $k$-tensors on $M$
\item $\cS^2 (\Lambda^2) = \Gamma( \Sym^2 (\Lambda^2 T^* M) )$ is the space of smooth $4$-tensors which satisfy all the symmetries of the Riemann curvature tensor except possibly the first Bianchi identity, namely $U \in \cS^2 (\Lambda^2)$ iff $U_{ijkl} = - U_{jikl} = - U_{ijlk} = U_{klij}$.
\end{itemize}
We regard a $k$-form on $M$ as a totally skew-symmetric $k$-tensor on $M$. Thus all inner products of tensors (even those of $k$-forms) are inner products as tensors. That is, if $\alpha = \tfrac{1}{k!} \alpha_{i_1 \cdots i_k} \dx{i_1} \w \cdots \w \dx{i_k}$ and $\beta = \tfrac{1}{k!} \beta_{j_1 \cdots j_k} \dx{j_1} \w \cdots \w \dx{j_k}$ are two $k$-forms, their pointwise inner product \emph{as tensors} is given by
\begin{equation} \label{eq:tensorsinnerproduct}
\langle \alpha , \beta \rangle = \alpha_{i_1 \cdots i_k} \beta_{i_1 \cdots i_k}.
\end{equation}
[In particular, there is no factor of $\frac{1}{k!}$ in~\eqref{eq:tensorsinnerproduct} as there is in~\cite{K-flows} which considers their pointwise inner product \emph{as $k$-forms}.]

Let $A = A_{ij} \dx{i} \otimes \dx{j} \in \cT^2$. We define $A^t \in \cT^2$ by $(A^t)_{ij} = A_{ji}$, the \emph{transpose} of $A$. Then we set
\begin{equation*}
A_{\symm} = \tfrac{1}{2} (A + A^t), \qquad A_{\skew} = \tfrac{1}{2} (A - A^t).
\end{equation*}
Thus we have $T^* M \otimes T^* M = \sym{2} \oplus \Lambda^2 (T^* M)$, and we can write
\begin{equation*}
A = A_{\symm} + A_{\skew}
\end{equation*}
uniquely in terms of a symmetric tensor $A_{\symm}$ and a $2$-form $A_{\skew}$. The \emph{trace} of $A$ with respect to $g$ is $\tr A = A_{ii} = \tr A_{\symm}$. We can hence further decompose
\begin{equation} \label{eq:Asym}
A_{\symm} = \tfrac{1}{n} (\tr A) g + A_0
\end{equation}
where $A_0 = A_{\symm} - \tfrac{1}{n} (\tr A) g$ is the \emph{traceless part} of $A_{\symm}$ and $n = \dim M$.

On $\cT^2$ we have a composition operation, which we denote by juxtaposition. Specifically, if $A, B \in \cT^2$, we define
\begin{equation} \label{eq:product}
(AB)_{ij} = A_{ip} B_{pj}.
\end{equation}
Let $A, B \in \cT^2$, and recall that
\begin{equation} \label{eq:tensinnerproduct}
\langle A, B \rangle = A_{ij} B_{ij}.
\end{equation}
It is easy to see that in the decomposition
\begin{equation} \label{eq:Adecomp}
A = \tfrac{1}{n} (\tr A) g + A_0 + A_{\skew}
\end{equation}
all three summands are mutually (pointwise) orthogonal.

The following identities are trivial to check:
\begin{equation} \label{eq:tensinnerproductidentities}
\begin{aligned}
\langle A , g \rangle & = \tr A, \\
\langle A, B \rangle & = \tr(A^t B) = \tr(B A^t) = \tr(B^t A) = \tr(A B^t), \\
\langle A^t, B^t \rangle & = \langle A, B \rangle.
\end{aligned}
\end{equation}
Note that the last equation above says that the transpose operation is an \emph{isometry}. 

We use $\cS^2_0$ to denote those sections in $\cS^2$ that are traceless. Thus we have
\begin{equation} \label{eq:tens-decomp}
\cT^2 = \{ f g \mid f \in \Omega^0 \} \oplus \cS^2_0 \oplus \Omega^2 \cong \Omega^0 \oplus \cS^0 \oplus \Omega^2
\end{equation}
and the above splitting is pointwise orthogonal with respect to the inner product~\eqref{eq:tensinnerproduct}.

Let $S \in \Gamma(E)$ where $E = \otimes^k T^* M$, and let $W \in \vf$. We have
\begin{align*}
( \cL_W S ) (X_1, \ldots, X_k) & = W \big( S(X_1, \ldots, X_k) \big) - S(\cL_W X_1, \ldots, X_k) - \cdots - S(X_1, \ldots, \cL_W X_k) \\
& = (\nab{W} S) (X_1, \ldots, X_k) + S (\nab{W} X_1, \ldots, X_k) + \cdots + S(X_1, \ldots, \nab{W} X_k) \\
& \qquad {} - S(\cL_W X_1, \ldots, X_k) - \cdots - S(X_1, \ldots, \cL_W X_k).
\end{align*}
From $\nab{W} X - \cL_W X = \nab{W} X - (\nab{W} X - \nab{X} W) = \nab{X} W$, we obtain
$$ ( \cL_W S ) (X_1, \ldots, X_k) = (\nab{W} S) (X_1, \ldots, X_k) + S(\nab{X_1} W, \ldots, X_k) + \cdots + S(X_1, \ldots \nab{X_k} W). $$
In terms of a local orthonormal frame, in the cases where $k=2$ or $k=3$, the above becomes
\begin{equation} \label{eq:Lie-derivative-frame}
\begin{aligned}
(\cL_W S)_{ij} & = W_p \nab{p} S_{ij} + \nab{i} W_p S_{pj} + \nab{j} W_p S_{ip}, \\
(\cL_W S)_{ijk} & = W_p \nab{p} S_{ijk} + \nab{i} W_p S_{pjk} + \nab{j} W_p S_{ipk} + \nab{k} W_p S_{ijp}.
\end{aligned}
\end{equation}

If $X$ is a vector field, its \emph{divergence} $\Div X$ is the function $\nab{i} X_i$, and it equals $- \dd^* X$, where $X$ is identified with its metric dual $1$-form. In terms of a local orthonormal frame we have $\Div X = \nab{p} X_p$. The divergence theorem says that $\int_M (\Div X) \vol = 0$ if $X$ is compactly supported.

Given a linear map $P \colon \Gamma(E) \to \Gamma(F)$, where $E$ and $F$ are tensor bundles over $M$, its \emph{formal adjoint} $P^* \colon \Gamma(F) \to \Gamma(E)$ is the unique linear map such that $\int_M \langle P K, L \rangle \vol = \int_M \langle K, P^* L \rangle \vol$, whenever $K \in \Gamma(E)$ and $L \in \Gamma(F)$ are compactly supported. By the divergence theorem, this means that the difference $\langle P K, L \rangle - \langle K, P^* L \rangle$ is the divergence of a compactly supported vector field on $M$.

Let $S \in \Gamma(E)$ where $E = \otimes^k T^* M$. Then $\nabla S \in \Gamma(T^* M \otimes E) = \Gamma(\otimes^{k+1} T^* M)$, where
\begin{equation} \label{eq:nablaS}
(\nabla S) (X, \cdot) = (\nab{X} S) (\cdot) \in \Gamma(E).
\end{equation}
If $k \geq 1$, the \emph{divergence} of $S$, denoted $\Div S$, is the element of $\Gamma( \otimes^{k-1} T^* M)$ obtained by contracting $\nabla S$ on the first two indices using the metric. That is,
$$ (\Div S)_{i_1 \cdots i_{k-1}} = \nab{p} S_{p i_1 \cdots i_{k-1}}. $$
Note that, under the identification of $1$-forms with vector fields, when $k=1$ this agrees with the usual divergence of vector fields described above.

We write $\nabla^2_{X,Y} S$ for $( \nabla(\nabla S) ) (X, Y, \cdot) \in \Gamma(E)$. Using~\eqref{eq:nablaS}, we have
\begin{align*}
(\nabla^2_{X,Y} S) (Z_1, \ldots, Z_k) & = ( \nabla(\nabla S) ) (X, Y, Z_1, \ldots, Z_k) = ( \nab{X} (\nabla S) ) (Y, Z_1, \ldots, Z_k) \\
& = X \big( (\nabla S) (Y, Z_1, \ldots, Z_k) \big) - (\nabla S) (\nab{X} Y, Z_1, \ldots, Z_k) \\
& {} \qquad - \textstyle{\sum_{i=1}^k} (\nabla S) (Y, Z_1, \ldots, \nab{X} Z_i, \ldots, Z_k) \\
& = X \big( (\nab{Y} S) (Z_1, \ldots, Z_k) \big) - (\nab{\nab{X} Y} S) (Z_1, \ldots, Z_k) \\
& {} \qquad - \textstyle{\sum_{i=1}^k} (\nab{Y} S) (Z_1, \ldots, \nab{X} Z_i, \ldots, Z_k) \\
& = (\nab{X} (\nab{Y} S)) (Z_1, \ldots, Z_k) - (\nab{\nab{X}Y} S) (Z_1, \ldots, Z_k),
\end{align*}
and hence
\begin{equation*}
\nabla^2_{X,Y} S = \nab{X} (\nab{Y} S) - \nab{\nab{X}Y} S.
\end{equation*}
Thus in a local frame we write $\nab{i} \nab{j} S$ for $\nabla^2_{ij} S = \nab{i}(\nab{j} S) - \Gamma_{ij}^k \nab{k} S$. Consequently, using the symbol $\Delta$ to denote the \emph{analyst's Laplacian}, we have $\Delta = \nab{k} \nab{k}$. Note that $\Delta$ is the negative of the \emph{rough Laplacian} $\nabla^* \nabla$, where $\nabla$ is the Levi-Civita connection of $g$. The Hodge Laplacian on forms is $\Delta_{\dd} = \dd \dd^* + \dd^* \dd$.

Our convention for labelling the Riemann curvature tensor is
\begin{equation} \label{eq:Riemann-convention}
R_{ijkl} = g ( \nab{e_i} (\nab{e_j} e_k) - \nab{e_j} (\nab{e_i} e_k) - \nab{[e_i, e_j]} e_k, e_l )
\end{equation}
in terms of a local orthonormal frame. With this convention, the Ricci tensor is $R_{jk} = R_{ljkl}$, and the Ricci identity for a $k$-tensor is
\begin{equation} \label{eq:ricci-identity}
\nab{p} \nab{q} S_{i_1 \cdots i_k} - \nab{q} \nab{p} S_{i_1 \cdots i_k} = - \sum_{l=1}^k R_{pqi_l m} S_{i_1 \cdots i_{l-1} m i_{l+1} \cdots i_k}.
\end{equation}
We also have the Riemannian second Bianchi identity
\begin{equation} \label{eq:riem2B1}
\nab{i} R_{jkab} + \nab{j} R_{kiab} + \nab{k} R_{ijab} = 0,
\end{equation}
which when contracted on $i,a$ gives
\begin{equation} \label{eq:riem2B2}
\nab{i} R_{ibjk} = \nab{k} R_{jb} - \nab{j} R_{kb}.
\end{equation}
A further contraction on $j,b$ gives the contracted second Bianchi identity
\begin{equation} \label{eq:riem2B3}
\nab{i} R_{ik} = \tfrac{1}{2} \nab{k} R.
\end{equation}

\section{Foundational results on $\G$-structures} \label{sec:review}

Here we review and establish various general results about a $\G$-structure $\ph$ that are needed later in the paper. The include the decomposition of forms and other algebraic relations; discussion of the torsion and its covariant derivative; infinitesimal $\G$-symmetries; the $\G$-Bianchi identity; the rough and Hodge Laplacians of $\ph$; and scaling of $\G$-structures. We also present three applications of this material: Taylor expansion of $\ph$; the optimal $\ph$-compatible connection; and conformal change of $\G$-structures. These applications are not needed in the paper but are included for the convenience of the reader, making this section a fairly complete introduction to computational aspects of the geometry of $\G$-structures.

\textbf{Note.} There are two common conventions for the orientation induced by a $\G$-structure. See~\cite{K-notes} for a detailed explanation of orientations and sign conventions in $\G$-geometry. If the reader prefers the opposite orientation to ours, they can probably safely change the sign of $\ps$, $\vol$, and $\st$ throughout.

\subsection{$\G$-structures and contraction identities} \label{sec:contractions}

To fix notation, we begin with a brief review of $\G$-structures. Then we discuss various fundamental contraction identities that are used frequently throughout the paper. Good references for $\G$-structures are~\cite{Bryant, Joyce, K-intro}.

Let $M^7$ be a smooth orientable $7$-manifold. A \emph{$\G$-structure} is a smooth $3$-form $\ph$ that is ``nondegenerate'' or ``positive'' in the sense that it determines a Riemannian metric $g_{\ph}$ and a volume form $\vol_{\ph}$ in a nonlinear way, via the following identity:
\begin{equation} \label{eq:nondegenerate}
(X \hk \ph) \w (Y \hk \ph) \w \ph = - 6 g_{\ph}(X, Y) \vol_{\ph},
\end{equation}
for any vector fields $X$ and $Y$ on $M$, where $ \hk $ denotes the interior product. From~\eqref{eq:nondegenerate} it is possible to extract the metric $g_{\ph}$ and the volume form $\vol_{\ph}$ separately. Such a structure is called a $\G$-structure because the stabilizer in $\mathrm{GL}(7, \R)$ of $\ph$ at a point $p \in M$ is the exceptional Lie group $\G$.

Since $M$ is assumed to be orientable, one can show that such $\G$-structures exist if and only if $M$ is \emph{spinnable}, which is equivalent to the vanishing of the second Stiefel-Whitney class of $M$. When $M$ admits $\G$-structures, the space $\Omega^3_+$ of nondegenerate $3$-forms on $M$ is an \emph{open} subset of $\Omega^3$.

Let $\st_{\ph}$ denote the Hodge star operator associated to the metric $g_{\ph}$ and volume form $\vol_{\ph}$. We denote by $\ps$ the dual $4$-form $\ps = \st_{\ph} \ph$. We note also that for any $\G$-structure we have $g_{\ph} (\ph, \ph) = 7$, so the Riemannian volume form of $g_{\ph}$ is $\vol_{\ph} = \frac{1}{7} \ph \w \st \ph$.

We collect here several important contraction identities involving the $3$-form $\ph$ and the $4$-form $\ps$ of a $\G$-structure, and some of their useful consequences. The proofs of~\eqref{eq:phph},~\eqref{eq:phps}, and~\eqref{eq:psps} can be found in~\cite{K-flows}.

Contractions of $\ph$ with $\ph$:
\begin{equation} \label{eq:phph}
\begin{aligned}
\ph_{ijk} \ph_{abk} & = g_{ia} g_{jb} - g_{ib} g_{ja} - \ps_{ijab}, \\
\ph_{ijk} \ph_{ajk} & = 6 g_{ia}, \\
\ph_{ijk} \ph_{ijk} & = 42.
\end{aligned}
\end{equation}

Contractions of $\ph$ with $\ps$:
\begin{equation} \label{eq:phps}
\begin{aligned}
\ph_{ijk} \ps_{abck} & = g_{ia} \ph_{jbc} + g_{ib} \ph_{ajc} + g_{ic} \ph_{abj} - g_{ja} \ph_{ibc} - g_{jb} \ph_{aic} - g_{jc} \ph_{abi}, \\
\ph_{ijk} \ps_{abjk} & = - 4 \ph_{iab}, \\ 
\ph_{ijk} \ps_{aijk} & = 0.
\end{aligned}
\end{equation}

Contractions of $\ps$ with $\ps$:
\begin{equation} \label{eq:psps}
\begin{aligned}
\ps_{ijkl} \ps_{abcl} & = - \ph_{ajk} \ph_{ibc} - \ph_{iak} \ph_{jbc} - \ph_{ija} \ph_{kbc} \\
& \qquad {} + g_{ia} g_{jb} g_{kc} + g_{ib} g_{jc} g_{ka} + g_{ic} g_{ja} g_{kb} - g_{ia} g_{jc} g_{kb} - g_{ib} g_{ja} g_{kc} - g_{ic} g_{jb} g_{ka} \\
& \qquad {} - g_{ia} \ps_{jkbc} - g_{ja} \ps_{kibc} - g_{ka} \ps_{ijbc} + g_{ab} \ps_{ijkc} - g_{ac} \ps_{ijkb}, \\
\ps_{ijkl} \ps_{abkl} & = 4 g_{ia} g_{jb} - 4 g_{ib} g_{ja} - 2 \ps_{ijab}, \\ 
\ps_{ijkl} \ps_{ajkl} & = 24 g_{ia}, \\
\ps_{ijkl} \ps_{ijkl} & = 168.
\end{aligned}
\end{equation}

We use the contraction identities~\eqref{eq:phph},~\eqref{eq:phps}, and~\eqref{eq:psps} throughout without specific mention.

\subsection{The decomposition of forms} \label{sec:forms}

In this section, we review the decompositions of the spaces of differential forms on a manifold with $\G$-structure. Translating back and forth between various isomorphic representations of $\G$ is essential throughout this paper. Much of this material is well known. We present it here in a more elegant and computationally efficient manner than it appears in~\cite{K-flows}. Moreover, we give explicit formulas, in terms of local orthonormal frames, for the decompositions of $3$-forms and $4$-forms into their orthogonal components. These formulas are harder to find in the existing literature. Later, in Section~\ref{sec:more-rep-theory}, we consider other important decompositions of spaces of tensors into $\G$ representations, which are then used in Section~\ref{sec:curvature-torsion} to decompose the second-order differential invariants of a $\G$-structure into irreducible $\G$-representations, and to derive the independent relations between these components.

On a manifold $(M, \ph)$ with $\G$-structure, the space $\Omega^{k}$ decomposes into subspaces, when $k=2,3,4,5$. Explicitly, we have
\begin{align*}
\Omega^2 & = \Omega^2_7 \oplus \Omega^2_{14}, & \qquad \Omega^5 & = \Omega^5_7 \oplus \Omega^5_{14}, \\
\Omega^3 & = \Omega^3_1 \oplus \Omega^3_7 \oplus \Omega^3_{27}, & \qquad \Omega^4 & = \Omega^4_1 \oplus \Omega^4_7 \oplus \Omega^4_{27},
\end{align*}
where $\Omega^k_l$ has pointwise dimension $l$ and the decomposition is orthogonal with respect to $g$. Note that the Hodge star $\st$ is an isometry and $\Omega^k_l = \st (\Omega^{7-k}_l)$.

The spaces $\Omega^k_7$ are all isomorphic to $\Omega^1$ and to $\vf$. The spaces $\Omega^k_{27}$ are isomorphic to $\cS_0$, the \emph{traceless} (with respect to $g$) symmetric $2$-tensors $\cS^2$ on $M$. These isomorphisms are crucial, and are described explicitly in the rest of this section. The reader is directed to~\cite[Section 2.2]{K-flows} and~\cite[Section 4.3]{K-intro} for any details that we omit here.

\textbf{The space $\Omega^2$ of $2$-forms.} Consider the following linear operator on $\Omega^2$:
\begin{align*}
\Pop & \colon \Omega^2 \to \Omega^2, \\
& \quad \beta \mapsto \Pop \beta = 2 \st (\ph \w \beta).
\end{align*}
[We have put a factor of $2$ in the definition of $\Pop$ to avoid a factor of $\tfrac{1}{2}$ in equation~\eqref{eq:Pdefn} below.]

In terms of local coordinates, let $\beta = \frac{1}{2} \beta_{ij} \dx{i} \w \dx{j}$. Then we have $\Pop \beta = \frac{1}{2} (\Pop \beta)_{ij} \dx{i} \w \dx{j}$, where
\begin{equation} \label{eq:Pdefn}
(\Pop \beta)_{ab} = \beta_{ij} \ps_{ijab} = \ps_{abij} \beta_{ij}.
\end{equation}
It is easy to check that
\begin{equation} \label{eq:Pop-coords}
\langle \Pop \beta, \mu \rangle = \langle \beta, \Pop \mu \rangle = \ps_{ijab} \beta_{ij} \mu_{ab},
\end{equation}
and hence $\Pop$ is (pointwise) self-adjoint and thus orthogonally diagonalizable with real eigenvalues.

We compute
\begin{align*}
(\Pop^2 \beta)_{ab} & = \ps_{abij} (\Pop \beta)_{ij} = \ps_{abij} \ps_{ijpq} \beta_{pq} \\
& = ( 4 g_{ap} g_{bq} - 4 g_{aq} g_{bp} - 2 \ps_{abpq} ) \beta_{pq} \\
& = 4 \beta_{ab} - 4 \beta_{ba} - 2 \ps_{abpq} \beta_{pq} = 8 \beta_{ab} - 2 (\Pop \beta)_{ab}.
\end{align*}
Thus we deduce that
\begin{equation} \label{eq:Pop-squared}
\Pop^2 = 8 \sI - 2\Pop, \quad \text{where $\sI \colon \Omega^2 \to \Omega^2$ is the identity operator.}
\end{equation}
so $(\Pop + 4 \sI) (\Pop - 2 \sI) = 0$. Therefore the eigenvalues of $\Pop$ are $-4$ and $+2$. We can thus describe the decomposition of $\Omega^2$ as follows:
\begin{align*}
\Omega^2_7 & = \{ \beta \in \Omega^2 \mid \Pop \beta = -4 \beta \}, \\
\Omega^2_{14} & = \{ \beta \in \Omega^2 \mid \Pop \beta = 2 \beta \},
\end{align*}
and we have
\begin{equation} \label{eq:omega2}
\Omega^2 = \Omega^2_7 \oplus \Omega^2_{14}.
\end{equation}

There are alternate descriptions of $\Omega^2_7$ and $\Omega^2_{14}$ that are also very important. First, suppose that $\beta_{ij} = X_k \ph_{kij} \in \Omega^2$ for some vector field $X$. Then we have
\begin{equation*}
(\Pop \beta)_{ab} = \ps_{abij} X_k \ph_{kij} = - 4 X_k \ph_{kab} = - 4 \beta_{ab}.
\end{equation*}
Thus by dimension count we conclude that
\begin{equation} \label{eq:omega27desc}
\begin{aligned}
\beta \in \Omega^2_7 & \, \Longleftrightarrow \, \beta_{ij} \ps_{ijab} = -4 \beta_{ab} \\
& \, \Longleftrightarrow \, \beta_{ij} = X_k \ph_{kij} \quad \text{for some $X \in \vf$},
\end{aligned}
\end{equation}
Suppose that $\beta \in \Omega^2_{14}$. Then we have $\beta_{ab} = \tfrac{1}{2} \ps_{abij} \beta_{ij}$. Hence, we obtain
\begin{equation*}
\beta_{ab} \ph_{abk} = \tfrac{1}{2} \beta_{ij} \ps_{abij} \ph_{abk} = - 2 \beta_{ij} \ph_{ijk},
\end{equation*}
so $\beta_{ab} \ph_{abk} = 0$. Again by dimension count we conclude that
\begin{equation} \label{eq:omega214desc}
\begin{aligned}
\beta \in \Omega^2_{14} & \, \Longleftrightarrow \, \beta_{ij} \ps_{ijab} = 2 \beta_{ab} \\
& \, \Longleftrightarrow \, \beta_{ij} \ph_{ijk} = 0.
\end{aligned}
\end{equation}

The invariant way of writing~\eqref{eq:omega27desc} and~\eqref{eq:omega214desc} is
\begin{align*}
\Omega^2_7 & = \{ X \hk \ph \mid X \in \vf \} = \{ \beta \in \Omega^2 \mid \Pop \beta = -4 \beta \}, \\
\Omega^2_{14} & = \{ \beta \in \Omega^2 \mid \beta \w \ps = 0 \} = \{ \beta \in \Omega^2 \mid \Pop \beta = 2 \beta \}.
\end{align*}

Consider~\eqref{eq:omega27desc} with $X$ replaced by $\tfrac{1}{6} X$, so that $\beta = \tfrac{1}{6} X \hk \ph$. Then $X$ can be reconstructed from $\beta$ as follows. If $\beta_{ij} = \frac{1}{6} X_k \ph_{kij}$, then from~\eqref{eq:phph} we find that $\beta_{ij} \ph_{ijp} = X_p$. Thus we have
\begin{equation} \label{eq:omega27vf}
\beta_{ab} = \tfrac{1}{6} X_l \ph_{lab} \, \Longleftrightarrow \, X_k = \beta_{ab} \ph_{abk}.
\end{equation}
Note that we have
\begin{equation} \label{eq:1727metric}
(\tfrac{1}{6} X_k \ph_{kij}) (\tfrac{1}{6} Y_l \ph_{lij}) = \tfrac{1}{6} X_k Y_k,
\end{equation}
which can also be written invariantly as
\begin{equation} \label{eq:1727metric-b}
\langle X \hk \ph, Y \hk \ph \rangle = 6 \langle X, Y \rangle.
\end{equation}
Define a map $\Vop \colon \cT^2 \to \Omega^1$ as follows. For $A \in \cT^2$, we set
\begin{equation} \label{eq:vecA}
(\Vop A)_k = A_{ij} \ph_{ijk}.
\end{equation}

It is clear from~\eqref{eq:omega214desc} that $\ker \Vop = \cS^2 \oplus \Omega^2_{14}$, so only the $\Omega^2_7$ part of $A$ contributes to $\Vop A$, and we call it the \emph{vector part} of $A$. Then~\eqref{eq:omega27vf} can be rewritten as
\begin{equation} \label{eq:vec-transform}
A_7 = \tfrac{1}{6} (\Vop A) \hk \ph, \qquad \Vop (X \hk \ph) = 6 X,
\end{equation}
and~\eqref{eq:1727metric-b} becomes
\begin{equation} \label{eq:vec-metric}
\langle \Vop A, \Vop B \rangle = 6 \langle A_7, B_7 \rangle \quad \text{for $A, B \in \cT^2$}.
\end{equation}

Let $\pi_{7}$ and $\pi_{14}$ denote the orthogonal projections from $\Omega^2$ to $\Omega^2_7$ and $\Omega^2_{14}$, respectively. We write $\beta_7 = \pi_7 \beta$ and $\beta_{14} = \pi_{14} \beta$ for any $\beta \in \Omega^2$. Then we have
\begin{equation} \label{eq:Ponbeta}
\Pop \beta = -4 \beta_7 + 2 \beta_{14},
\end{equation}
from which it follows that
\begin{equation} \label{eq:omega2-proj}
\beta_7 = \tfrac{1}{6} (2 \beta - \Pop \beta), \qquad \beta_{14} = \tfrac{1}{6} (4 \beta + \Pop \beta).
\end{equation}

\textbf{The spaces $\Omega^3$ of $3$-forms and $\Omega^4$ of $4$-forms.} To describe the decomposition of the spaces $\Omega^3$ and $\Omega^4$, we follow the approach of~\cite[Section 2.2]{K-flows}, but with improved notation and simplified arguments that apply to both symmetric and to skew-symmetric tensors. 

Let $\sigma \in \Omega^k$. Given $A = A_{ij} \dx{i} \otimes \dx{j} \in \cT^2$, we define
\begin{equation} \label{eq:diamond-coords}
(A \diamond \sigma)_{i_1 i_2 \cdots i_k} = A_{i_1 p} \sigma_{p i_2 \cdots i_k} + A_{i_2 p} \sigma_{i_1 p i_3 \cdots i_k} + \cdots + A_{i_k p} \sigma_{i_1 i_2 \cdots i_{k-1} p} .
\end{equation}
Note from~\eqref{eq:diamond-coords} that if $A = g$ is the metric, we get
\begin{equation} \label{eq:gdiamond}
g \diamond \sigma = k \sigma, \qquad \text{for $\sigma \in \Omega^k$}.
\end{equation}

By the orthogonal decomposition~\eqref{eq:omega2} of $\Omega^2$, we can further decompose~\eqref{eq:tens-decomp} as
\begin{equation} \label{eq:tens-decomp2}
\cT^2 \cong \Omega^0 \oplus \cS^2_0 \oplus \Omega^2_7 \oplus \Omega^2_{14}.
\end{equation}
With respect to this splitting, we can write
\begin{equation} \label{eq:A-decomp2}
A = \tfrac{1}{7} (\tr A) g + A_{27} + A_7 + A_{14},
\end{equation}
where $A_{27}$ is the traceless symmetric part of $A$. We can extend the action of $\Pop$ in~\eqref{eq:Pop-coords} to all of $\cT^2$, by defining
\begin{equation} \label{eq:Pop-defn2}
(\Pop A)_{ab} = A_{ij} \ps_{ijab}.
\end{equation}
Then it is easy to see that $\ker \Pop = \cS$ and
\begin{equation} \label{eq:Pop-action}
\Pop A = \Pop( \tfrac{1}{7} (\tr A) g + A_{27} + A_7 + A_{14} ) = - 4 A_7 + 2 A_{14}.
\end{equation}

By~\eqref{eq:diamond-coords}, we have two \emph{linear maps} $\cT^2 \to \Omega^k$ for $k = 3, 4$ given by
\begin{align*}
A & \mapsto A \diamond \ph, \\
A & \mapsto A \diamond \ps,
\end{align*}
where explicitly
\begin{align} \label{eq:diamond3}
(A \diamond \ph)_{ijk} & = A_{ip} \ph_{pjk} + A_{jp} \ph_{ipk} + A_{kp} \ph_{ijp}, \\ \label{eq:diamond4}
(A \diamond \ps)_{ijkl} & = A_{ip} \ps_{pjkl} + A_{jp} \ps_{ipkl} + A_{kp} \ps_{ijpl} + A_{lp} \ps_{ijkp}.
\end{align}

\begin{prop} \label{prop:ABinnerproduct}
Let $A$ and $B$ be sections of $\cT^2$. Then with respect to the decompositions~\eqref{eq:A-decomp2} for $A$ and $B$, we have
\begin{align} \label{eq:ABph}
\langle A \diamond \ph, B \diamond \ph \rangle & = \tfrac{54}{7} (\tr A) (\tr B) + 12 \langle A_{27}, B_{27} \rangle + 36 \langle A_7, B_7 \rangle, \\ \label{eq:ABps}
\langle A \diamond \ps, B \diamond \ps \rangle & = \tfrac{384}{7} (\tr A) (\tr B) + 48 \langle A_{27}, B_{27} \rangle + 144 \langle A_7, B_7 \rangle.
\end{align}
\end{prop}
\begin{proof}
We use~\eqref{eq:tensorsinnerproduct} and~\eqref{eq:diamond3} to compute
\begin{align*}
\langle A \diamond \ph, B \diamond \ph \rangle & = (A \diamond \ph)_{ijk} (B \diamond \ph)_{ijk} \\
& = ( A_{ip} \ph_{pjk} + A_{jp} \ph_{ipk} + A_{kp} \ph_{ijp} ) (B \diamond \ph)_{ijk} \\
& = 3 A_{ip} \ph_{pjk} (B \diamond \ph)_{ijk}
\end{align*}
using the fact that $(B \diamond \ph)_{ijk}$ is skew-symmetric in its indices. Continuing in the same fashion we find
\begin{align*}
\langle A \diamond \ph, B \diamond \ph \rangle & = 3 A_{ip} \ph_{pjk} ( B_{iq} \ph_{qjk} + B_{jq} \ph_{iqk} + B_{kq} \ph_{ijq} ) \\
& = 3 A_{ip} \ph_{pjk} ( B_{iq} \ph_{qjk} + 2 B_{jq} \ph_{iqk}).
\end{align*}
Now we expand the contractions of $\ph$ with itself, to obtain
\begin{align*}
\langle A \diamond \ph, B \diamond \ph \rangle & = 3 A_{ip} B_{iq} (6 g_{pq}) + 6 A_{ip} B_{jq} ( g_{pi} g_{jq} - g_{pq} g_{ji} - \ps_{pjiq}) \\
& = 18 A_{ip} B_{ip} + 6 A_{ii} B_{jj} - 6 A_{iq} B_{iq} - 6 A_{ip} B_{jq} \ps_{ipjq} \\
& = 12 \langle A, B \rangle + 6 (\tr A) (\tr B) - 6 \langle \Pop A, B \rangle
\end{align*}
using the linear map $\Pop$ from~\eqref{eq:Pop-defn2}. Applying~\eqref{eq:Pop-action} and the orthogonality of the decompositions~\eqref{eq:A-decomp2} for $A$ and $B$, we conclude that
\begin{align*}
\langle A \diamond \ph, B \diamond \ph \rangle & = 12 \big( \tfrac{1}{49} (\tr A) (\tr B) \langle g, g \rangle + \langle A_{27}, B_{27} \rangle + \langle A_7, B_7 \rangle + \langle A_{14}, B_{14} \rangle \big) \\
& \qquad {} + 6 (\tr A) (\tr B) - 6 \langle -4 A_7 + 2 A_{14}, B_7 + B_{14} \rangle \\
& = \tfrac{12}{7} (\tr A) (\tr B) + 12 \langle A_{27}, B_{27} \rangle + 12 \langle A_7, B_7 \rangle + 12 \langle A_{14}, B_{14} \rangle \\
& \qquad {} + 6 (\tr A) (\tr B) + 24 \langle A_7, B_7 \rangle - 12 \langle A_{14}, B_{14} \rangle \\
& = \tfrac{54}{7} (\tr A) (\tr B) + 12 \langle A_{27}, B_{27} \rangle + 36 \langle A_7, B_7 \rangle,
\end{align*}
which establishes~\eqref{eq:ABph}. Equation~\eqref{eq:ABps} is proved in an identical manner using the identities in~\eqref{eq:psps}.
\end{proof}

\begin{cor} \label{cor:ABinnerproduct}
The $2$-tensor $A$ lies in $\Omega^2_{14}$ if and only if $A \diamond \ph = 0$ or equivalently $A \diamond \ps = 0$. Moreover, when restricted to the subspace $\cS^2 \oplus \Omega^2_7$ of $\cT^2$, that is to the pointwise orthogonal complement of $\Omega^2_{14}$, the maps $A \mapsto A \diamond \ph$ and $A \mapsto A \diamond \ps$ are linear isomorphisms onto $\Omega^3$ and $\Omega^4$, respectively.
\end{cor}
\begin{proof}
Equation~\eqref{eq:ABph} with $A = B$ gives
$$ | A \diamond \ph |^2 = \tfrac{54}{7} (\tr A)^2 + 12 | A_{27} |^2 + 36 | A_7 |^2, $$
from which we get $A \diamond \ph = 0$ if and only if $A = A_{14}$, establishing the first claim. Moreover, if $A_{14} = 0$, then $A \diamond \ph = 0$ if and only if $A = 0$,
Hence, the map $A \mapsto A \diamond \ph$ is injective on the orthogonal complement of $\Omega^2_{14}$. By dimension count, both sides are (pointwise) $35$-dimensional, so the map is a linear isomorphism. The argument for $A \mapsto A \diamond \ps$ is identical, because all that matters is that the coefficients in~\eqref{eq:ABph} and~\eqref{eq:ABps} are all positive.
\end{proof}

Note that a consequence of Corollary~\ref{cor:ABinnerproduct} is that
\begin{equation} \label{eq:omega214-diamond}
\begin{aligned}
A_{ij} \in \Omega^2_{14} & \, \Longleftrightarrow \, A_{ip} \ph_{pjk} + A_{jp} \ph_{ipk} + A_{kp} \ph_{ijp} = 0 \\
& \, \Longleftrightarrow \, A_{ip} \ps_{pjkl} + A_{jp} \ps_{ipkl} + A_{kp} \ps_{ijpl} + A_{lp} \ps_{ijkp} = 0.
\end{aligned}
\end{equation}

We have thus established the following decompositions:
\begin{equation*}
\Omega^3 = \Omega^3_1 \oplus \Omega^3_7 \oplus \Omega^3_{27}, \qquad
\Omega^4 = \Omega^4_1 \oplus \Omega^4_7 \oplus \Omega^4_{27},
\end{equation*}
where the decompositions are orthogonal with respect to the pointwise inner product on forms induced from $g$. Explicitly, using~\eqref{eq:gdiamond}, we have
\begin{equation} \label{eq:omega34decomp}
\begin{aligned}
\Omega^3_1 & = \{ f \ph \mid f \in \Omega^0 \}, & 
\Omega^4_1 & = \{ f \ps \mid f \in \Omega^0 \}, \\
\Omega^3_7 & = \{ A \diamond \ph \mid A \in \Omega^2_7 \}, &
\Omega^4_7 & = \{ A \diamond \ps \mid A \in \Omega^2_7 \}, \\
\Omega^3_{27} & = \{ A \diamond \ph \mid A \in \cS^2_0 \}, &
\Omega^4_{27} & = \{ A \diamond \ps \mid A \in \cS^2_0 \}.
\end{aligned}
\end{equation}

Next we compute the inverse of the isomorphisms $\cS^2 \oplus \Omega^2_7 \overset{\cong}{\to} \Omega^k$ where $k = 3$ or $k=4$.

\begin{cor} \label{cor:diamondinverse}
Let $\gamma \in \Omega^3$ and let $\eta \in \Omega^4$. We know that $\gamma = A \diamond \ph$ and $\eta = B \diamond \ps$ for some unique smooth sections $A = \tfrac{1}{7} (\tr A) g + A_{27} + A_7$ and $B = \tfrac{1}{7} (\tr B) g + B_{27} + B_7$ in $\cS^2 \oplus \Omega^2_7$. Define $\gamma^{\ph}$ and $\eta^{\ps}$ in $\cT^2$ by
\begin{equation*}
\gamma^{\ph}_{ia} = \gamma_{ijk} \ph_{ajk}, \qquad \eta^{\ps}_{ia} = \eta_{ijkl} \ps_{ajkl}.
\end{equation*}
Then we have
\begin{equation} \label{eq:diamondinverse-3}
\tr A = \tfrac{1}{18} \tr \gamma^{\ph}, \qquad A_{27} = \tfrac{1}{4} \gamma^{\ph}_{27}, \qquad A_7 = \tfrac{1}{12} \gamma^{\ph}_7,
\end{equation}
and
\begin{equation} \label{eq:diamondinverse-4}
\tr B = \tfrac{1}{96} \tr \eta^{\ps}, \qquad B_{27} = \tfrac{1}{12} \eta^{\ps}_{27}, \qquad B_7 = \tfrac{1}{36} \eta^{\ps}_7.
\end{equation}
\end{cor}
\begin{proof}
Let $C = \frac{1}{7} (\tr C) g + C_{27} + C_7 \in \cS^2 \oplus \Omega^2_7$ be arbitrary. From~\eqref{eq:ABph} we have
\begin{equation} \label{eq:diamond-inverse-temp}
\langle A \diamond \ph, C \diamond \ph \rangle = \tfrac{54}{7} (\tr A) (\tr C) + 12 \langle A_{27}, C_{27} \rangle + 36 \langle A_7, C_7 \rangle.
\end{equation}
We compute
\begin{align*}
\langle \gamma, C \diamond \ph \rangle & = \gamma_{ijk} (C_{ip} \ph_{pjk} + C_{jp} \ph_{ipk} + C_{kp} \ph_{ijp}) \\
& = 3 \gamma_{ijk} C_{ip} \ph_{pjk} = 3 \gamma^{\ph}_{ip} C_{ip} = 3 \langle \gamma^{\ph}, C \rangle \\
& = 3 \langle \tfrac{1}{7} (\tr \gamma^{\ph}) g + \gamma^{\ph}_{27} + \gamma^{\ph}_7, \tfrac{1}{7} (\tr C) g + C_{27} + C_7 \rangle \\
& = \tfrac{3}{7} (\tr \gamma^{\ph}) (\tr C) + 3 \langle \gamma^{\ph}_{27}, C_{27} \rangle + 3 \langle \gamma^{\ph}_7, C_7 \rangle.
\end{align*}
Comparing the above expression with~\eqref{eq:diamond-inverse-temp}, which also holds for all $C$, we deduce from nondegeneracy that
$$ 54 \tr A = 3 \tr \gamma^{\ph}, \qquad 12 A_{27} = 3 \gamma^{\ph}_{27}, \qquad 36 A_7 = 3 \gamma^{\ph}_7, $$ 
which is precisely~\eqref{eq:diamondinverse-3}. Equation~\eqref{eq:diamondinverse-4} is established in the same way using~\eqref{eq:ABps}.
\end{proof}

\begin{rmk} \label{rmk:omega34}
Corollary~\ref{cor:diamondinverse} essentially says the following. The components in $\Omega^0 \oplus \cS^2_0 \oplus \Omega^2_7$ of the element $A \in \cT^2$ such that $A \diamond \ph = \gamma \in \Omega^3$ correspond (up to some explicit constant factors) to the components of $\gamma^{\ph} \in \cT^2$. Similarly for $\eta \in \Omega^4$ with $\eta^{\ps} \in \cT^2$. It is not obvious but one can check using~\eqref{eq:omega27desc} that the elements $\gamma^{\ph}$ and $\eta^{\ps}$ of $\cT^2$ have no $\Omega^2_{14}$ component.
\end{rmk}

\begin{cor} \label{cor:sevenreps}
Let $X \in \vf$ be a smooth vector field on $M$. The $3$-form $\gamma = X \hk \ps$ can be written as $A \diamond \ph$ for $A = - \frac{1}{3} X \hk \ph \in \Omega^2_7$. This can also be written in the useful form $(X \hk \ph) \diamond \ph = - 3 X \hk \ps$.
\end{cor}
\begin{proof}
We have $\gamma_{ijk} = X_m \ps_{mijk}$, and thus
\begin{equation*}
\gamma^{\ph}_{ia} = \gamma_{ijk} \ph_{ajk} = X_m \ps_{mijk} \ph_{ajk} = -4 X_m \ph_{mia}.
\end{equation*}
Hence by~\eqref{eq:omega27desc} we find that $\gamma^{\ph}_{ia} \in \Omega^2_7$, so Corollary~\ref{cor:diamondinverse} gives $\gamma = A \diamond \ph$ for $A = A_7 \in \Omega^2_7$ given by
\begin{equation*}
(A_7)_{ia} = \tfrac{1}{12} \gamma^{\ph}_{ia} = - \tfrac{1}{3} X_m \ph_{mia}
\end{equation*}
as claimed.
\end{proof}

We can now define an important first-order linear differential operator on $(M, \ph)$, called the \emph{curl}, which takes vector fields to vector fields.
\begin{defn} \label{defn:curl}
Let $W \in \vf$. The \emph{curl} of $W$, denoted $\curl W$, is the vector field given by
$$ (\curl W) \hk \ph = 6 ( \nabla W )_7. $$
By~\eqref{eq:omega27vf} we can write this as
\begin{equation} \label{eq:defn-curl}
(\curl W)_k = (\nab{i} W_j) \ph_{ijk} \quad \text{or equivalently as} \quad \big( (\nabla W)_7 \big){}_{ij} = \tfrac{1}{6} (\curl W)_p \ph_{pij}.
\end{equation}
Using Corollary~\ref{cor:sevenreps}, we have the useful relation
\begin{equation} \label{eq:curl-hk-ps}
(\curl W) \hk \ps = - \tfrac{1}{3} (\curl W \hk \ph) \diamond \ph = - 2 (\nabla W)_7 \diamond \ph.
\end{equation}
The curl operation plays an important role throughout the present paper. For example, it is needed to describe infinitesimal symmetries of $\ph$ in Corollary~\ref{cor:inf-symmetry}.
\end{defn}

\subsection{Further algebraic relations induced by a $\G$-structure} \label{sec:further-algebraic}

In this section, we discuss some further algebraic relations on a manifold $(M, \ph)$ with $\G$-structure, including the interaction of the linear operators $\Vop$ and $\Pop$, and an operation $A \oct A$ on a $2$-tensor $A$. These relations are important for understanding the decomposition of various quadratic expressions in the torsion of a $\G$-structure, in Section~\ref{sec:curvature-torsion}.

\begin{lemma} \label{lemma:14-part}
Let $A \in \cT^2 \cong \Omega^0 \oplus \cS^2_0 \oplus \Omega^2_7 \oplus \Omega^2_{14}$ decompose as
$$ A = A_{\symm} + A_{\skew} = \tfrac{1}{7} (\tr A) g + A_{27} + A_7 + A_{14}. $$
Then we have
\begin{equation} \label{eq:14-part}
\Vop (\Pop A) = - 4 \Vop A, \qquad A_7 = \tfrac{1}{6} (\Vop A) \hk \ph, \qquad A_{14} = \tfrac{1}{3} (\Vop A) \hk \ph + \tfrac{1}{2} \Pop A.
\end{equation}
\end{lemma}
\begin{proof}
The first equation is immediate from~\eqref{eq:Pop-action} and $\Vop B = \Vop B_7$ for any $B$. The same equation also gives $A_{14} = 2 A_7 + \frac{1}{2} \Pop A$. From~\eqref{eq:vec-transform} we have $A_7 = \frac{1}{6} (\Vop A) \hk \ph$. Combining these two expressions yields the remaining results.
\end{proof}

\begin{lemma} \label{lemma:V-of-PAA}
Let $A \in \cT^2$, so $\Pop A \in \Omega^2 \subseteq \cT^2$ and $(\Pop A)A \in \cT^2$. Then we have
\begin{equation} \label{eq:V-of-PAA}
\Vop ( (\Pop A) A ) = \Vop (A^2) - (\tr A) \Vop A + 2 A (\Vop A) - A^t (\Vop A).
\end{equation}
\end{lemma}
\begin{proof}
Using~\eqref{eq:vecA} and~\eqref{eq:Pdefn}, we compute
\begin{align*}
\Vop \big( (\Pop A) A \big){}_k & = ( (\Pop A)A )_{ij} \ph_{ijk} = (\Pop A)_{im} A_{mj} \ph_{ijk} \\
& = A_{pq} \ps_{pqim} A_{mj} \ph_{ijk} = - A_{pq} A_{mj} (\ph_{jki} \ps_{pqmi}) \\
& = - A_{pq} A_{mj} ( g_{jp} \ph_{kqm} + g_{jq} \ph_{pkm} + g_{jm} \ph_{pqk} - g_{kp} \ph_{jqm} - g_{kq} \ph_{pjm} - g_{km} \ph_{pqj} ) \\
& = - A^2_{mq} \ph_{kqm} + 0 - (\tr A) A_{pq} \ph_{pqk} + A_{kq} (\Vop A)_q - A_{pk} (\Vop A)_p + A_{kj} (\Vop A)_j,
\end{align*}
which simplifies to~\eqref{eq:V-of-PAA}.
\end{proof}

There is a particular $2$-tensor $A \oct A$ that arises frequently which is a kind of ``square'' of a $2$-tensor $A$, which is \emph{not} the same as the usual square $(A^2)_{ij} = A_{ip} A_{pj}$ obtained from the identification of bilinear forms with operators given by the metric $g$, as it explicitly depends on the $\G$-structure $\ph$. This $2$-tensor $A \oct A$ is defined to be
\begin{equation} \label{eq:Aoct-defn}
(A \oct A)_{pq} = A_{im} A_{jn} \ph_{ijp} \ph_{mnq}.
\end{equation}

\begin{rmk} \label{rmk:octoproduct}
One way to think about $A \oct A$ is as follows. A $\G$-structure $\ph$ induces a \emph{cross product} $\times$ on sections of $TM$ by $\langle X \times Y, Z \rangle = \ph(X, Y, Z)$. This gives $(X \times Y)_k = X_p Y_q \ph_{pqk}$. Let $A \in \cT^2$, and write $A = A_{im} e_i \otimes e_m = A_{jn} e_j \otimes e_n$. Then
$$ A \oct A = (A \oct A)_{pq} e_p \otimes e_q = A_{im} A_{jn} (\ph_{ijp} e_p) \otimes (\ph_{mnq} e_q) = A_{im} A_{jn} (e_i \times e_j) \otimes (e_m \times e_n). $$
Thus $A \oct A$ can be thought of as the \emph{cross product} of $A$ with itself where the cross product $\times$ on sections of $TM$ induces a cross product $\oct = \times \otimes \times$ on the tensor product $TM \otimes TM$. In fact we can consider $A \oct B$ for any $2$-tensors $A, B$ on $M$. (Note that $\oct$ is \emph{not} skew-symmetric in general.) This operation $\oct$ plays a role in the study of the curvature of the moduli space of compact torsion-free $\G$-structures. See Karigiannis--Loftin~\cite{KLL} for more details.
\end{rmk}

\begin{prop} \label{prop:Aoct}
Let $A \oct A$ be as in~\eqref{eq:Aoct-defn}. The following identities hold:
\begin{equation} \label{eq:Aoct-identities}
\begin{aligned}
\tr (A \oct A) & = (\tr A)^2 - \langle A, A^t \rangle + \langle A, \Pop A \rangle, \\
\Pop (A \oct A) & = 4 (\tr A) A_{\skew} - 4 (A^2)_{\skew} - 4 ( (\Pop A) A )_{\skew} - 2 ( A^t (\Vop A) ) \hk \ph, \\
\Vop(A \oct A) & = 2 A (\Vop A) + 2 A^t (\Vop A) - 2 (\tr A) \Vop A + 2 \Vop (A^2).
\end{aligned}
\end{equation}
\end{prop}
\begin{proof}
From~\eqref{eq:Aoct-defn} we have
\begin{align*}
\tr (A \oct A) & = (A \oct A)_{pp} = A_{im} A_{jn} \ph_{ijp} \ph_{mnp} = A_{im} A_{jn} (g_{im} g_{jn} - g_{in} g_{jm} - \ps_{ijmn}) \\
& = (\tr A)^2 - \langle A, A^t \rangle + \langle A, \Pop A \rangle.
\end{align*}
We also have
\begin{align*}
( \Pop (A \oct A) )_{kl} & = (A \oct A)_{pq} \ps_{klpq} = ( A_{im} A_{jn} \ph_{ijp} \ph_{mnq} ) \ps_{klpq} \\
& = A_{im} A_{jn} \ph_{ijp} ( g_{mk} \ph_{nlp} + g_{ml} \ph_{knp} + g_{mp} \ph_{kln} - (m \leftrightarrow n) ).
\end{align*}
Since the factor $A_{im} A_{jn} \ph_{ijp}$ above is skew in $m, n$, we obtain
\begin{align*}
( \Pop (A \oct A) )_{kl} & = 2 A_{im} A_{jn} \ph_{ijp} ( g_{mk} \ph_{nlp} + g_{ml} \ph_{knp} + g_{mp} \ph_{kln} ) \\
& = 2 A_{ik} A_{jn} (g_{in} g_{jl} - g_{il} g_{jn} - \ps_{ijnl}) \\
& \qquad {} + 2 A_{il} A_{jn} (g_{ik} g_{jn} - g_{in} g_{jk} - \ps_{ijkn}) + 2 A_{ip} A_{jn} \ph_{ijp} \ph_{kln}
\end{align*}
which simplifies further to
\begin{align*}
( \Pop (A \oct A) )_{kl} & = 2 A_{ik} A_{li} -2 (\tr A) A_{lk} - 2 A_{ik} (\Pop A)_{il} \\
& \qquad + 2 (\tr A) A_{kl} - 2 A_{il} A_{ki} + 2 A_{il} (\Pop A)_{ik} - 2 (\Vop A)_j A_{jn} \ph_{nkl},
\end{align*}
which is equivalent to the second equation in~\eqref{eq:Aoct-identities}.

Similarly, we compute
\begin{align*}
( \Vop(A \oct A) )_l & = (A \oct A)_{pq} \ph_{pql} = ( A_{im} A_{jn} \ph_{ijp} \ph_{mnq} ) \ph_{pql} \\
& = A_{im} A_{jn} \ph_{mnq} ( g_{iq} g_{jl} - g_{il} g_{jq} - \ps_{ijql} ) \\
& = A_{im} A_{ln} \ph_{mni} - A_{lm} A_{qn} \ph_{mnq} + A_{im} A_{jn} \ph_{mnq} \ps_{ijlq} \\
& = 2 ( A(\Vop A) )_l + A_{im} A_{jn} \ph_{mnq} \ps_{ijlq}.
\end{align*}
This becomes
\begin{align*}
( \Vop(A \oct A) )_l & = 2 ( A(\Vop A) )_l + A_{im} A_{jn} ( g_{mi} \ph_{njl} + g_{mj} \ph_{inl} + g_{ml} \ph_{ijn} ) \\
& \qquad {} - A_{im} A_{jn} ( g_{ni} \ph_{mjl} + g_{nj} \ph_{iml} + g_{nl} \ph_{ijm} ) \\
& = 2 ( A(\Vop A) )_l - (\tr A) (\Vop A)_l + ( \Vop (A^2) )_l + ( A^t (\Vop A) )_l \\
& \qquad {} + ( \Vop (A^2) )_l - (\tr A) (\Vop A)_l + ( A^t (\Vop A) )_l
\end{align*}
which is equivalent to the third equation in~\eqref{eq:Aoct-identities}.
\end{proof}

\begin{cor} \label{cor:Aoct-7-14}
Let $A \oct A$ be as in~\eqref{eq:Aoct-defn}. Then we have
\begin{equation} \label{eq:Aoct-7-14}
\begin{aligned}
(A \oct A)_7 & = \big( \tfrac{1}{3} A (\Vop A) + \tfrac{1}{3} A^t (\Vop A) - \tfrac{1}{3} (\tr A) \Vop A + \tfrac{1}{3} \Vop(A^2) \big) \hk \ph, \\
(A \oct A)_{14} & = 2 (\tr A) A_{14} - 2 (A^2)_{14} - 2 ( (\Pop A) A )_{14}.
\end{aligned}
\end{equation}
\end{cor}
\begin{proof}
We use the identities in~\eqref{eq:Aoct-identities} and~\eqref{eq:14-part}. The expression for $(A \oct A)_7$ is immediate. For $(A \oct A)_{14}$, we compute
\begin{align*}
(A \oct A)_{14} & = \tfrac{1}{3} (\Vop (A \oct A)) \hk \ph + \tfrac{1}{2} \Pop (A \oct A) \\
& = \tfrac{1}{3} (2 A (\Vop A) + 2 A^t (\Vop A) - 2 (\tr A) \Vop A + 2 \Vop (A^2)) \hk \ph \\
& \qquad {} + \tfrac{1}{2} (4 (\tr A) A_{\skew} - 4 (A^2)_{\skew} - 4 ( (\Pop A) A )_{\skew} - 2 ( A^t (\Vop A) ) \hk \ph) \\
& = ( \tfrac{2}{3} A(\Vop A) - \tfrac{1}{3} A^t (\Vop A) - \tfrac{2}{3} (\tr A) \Vop A + \tfrac{2}{3} \Vop (A^2) ) \hk \ph \\
& \qquad {} + 2 (\tr A) A_{\skew} - 2 (A^2)_{\skew} - 2 ( (\Pop A) A )_{\skew}.
\end{align*}
Using~\eqref{eq:14-part} again to write $A_{\skew} = A_7 + A_{14} = \frac{1}{6} (\Vop A) \hk \ph + A_{14}$, the above becomes
\begin{align*}
(A \oct A)_{14} & = ( \tfrac{2}{3} A(\Vop A) - \tfrac{1}{3} A^t (\Vop A) - \tfrac{2}{3} (\tr A) \Vop A + \tfrac{2}{3} \Vop (A^2) ) \hk \ph \\
& \qquad {} + \tfrac{1}{6} \big( 2 (\tr A) \Vop A - 2 \Vop (A^2) - 2 \Vop( (\Pop A) A ) \big) \hk \ph \\
& \qquad {} + 2 (\tr A) A_{14} - 2 (A^2)_{14} - 2 ( (\Pop A) A )_{14} \\
& = \big( \tfrac{2}{3} A(\Vop A) - \tfrac{1}{3} A^t (\Vop A) - \tfrac{1}{3} (\tr A) \Vop A + \tfrac{1}{3} \Vop (A^2) - \tfrac{1}{3} \Vop( (\Pop A) A ) \big) \hk \ph \\
& \qquad {} + 2 (\tr A) A_{14} - 2 (A^2)_{14} - 2 ( (\Pop A) A )_{14}.
\end{align*}
The first line above vanishes, as expected, by~\eqref{eq:V-of-PAA}, yielding the result.
\end{proof}

\begin{rmk} \label{rmk:Aoct-parts}
The expressions for $\tr (A \oct A)$ in~\eqref{eq:Aoct-identities} and for $(A \oct A)_7$ and $(A \oct A)_{14}$ in~\eqref{eq:Aoct-7-14} show that the $\Omega^0 \oplus \Omega^2_7 \oplus \Omega^2_{14}$ components of $A \oct A \in \cT^2$ can all be expressed in terms of the simpler operations associated to a $\G$-structure, namely the operators $\Vop$ and $\Pop$, and the usual operations on $\cT^2$ available on any Riemannian manifold. Only the component $(A \oct A)_{27} \in \cS^2_0$ cannot be so expressed.
\end{rmk}

\subsection{The torsion of a $\G$-structure} \label{sec:torsion}

The \emph{torsion} of a $\G$-structure $\ph$ is a tensor that measures the failure of the metric $g_{\ph}$ to have holonomy contained in $\G$. By the holonomy principle, the torsion should be $\nabla \ph$. However, it is more convenient to ``package'' the torsion in a couple of alternative forms, which we now describe.

\begin{lemma} \label{lemma:first-torsion}
For any vector field $X$ on $M$, the $3$-form $\nab{X} \ph$ lies in $\Omega^3_7$.
\end{lemma}
\begin{proof}
A proof was given in~\cite[Lemma 2.24]{K-flows}. Nevertheless, we give a quick demonstration here using Corollary~\ref{cor:diamondinverse}. To establish the claim, for $\gamma = \nab{m} \ph$, we need to show that $\gamma^{\ph}_{ia} = \gamma_{ijk} \ph_{ajk}$ is skew-symmetric. But using~\eqref{eq:phph} we have
\begin{align*}
\gamma^{\ph}_{ia} & = \nab{m} \ph_{ijk} \ph_{ajk} = \nab{m} (\ph_{ijk} \ph_{ajk}) - \ph_{ijk} \nab{m} \ph_{ajk} \\
& = \nab{m} (6 g_{ia}) - \ph_{ijk} \gamma_{ajk} = 0 - \gamma^{\ph}_{ai}. \qedhere
\end{align*}
\end{proof}

It follows from Corollary~\ref{cor:sevenreps} that there exists a $2$-tensor $T$ such that
\begin{equation} \label{eq:nablaph}
\nab{m} \ph_{ijk} = T_{mp} \ps_{pijk}.
\end{equation}
We call $T$ the \emph{torsion} of the $\G$-structure. It follows immediately from~\eqref{eq:psps} that
\begin{equation} \label{eq:Tfromph}
T_{pq} = \tfrac{1}{24} \nab{p} \ph_{jkl} \ps_{qjkl},
\end{equation}
confirming that $T = 0$ if and only if $\nab{} \ph = 0$. If we differentiate the first identity in~\eqref{eq:phph}, and use~\eqref{eq:nablaph} and the first identity in~\eqref{eq:phps}, one obtains
\begin{equation} \label{eq:nablaps}
\nab{p} \ps_{ijkl} = - T_{pi} \ph_{jkl} + T_{pj} \ph_{ikl} - T_{pk} \ph_{ijl} + T_{pl} \ph_{ijk}
\end{equation}
expressing $\nab{} \ps$ in terms of $T$.

\begin{rmk} \label{rmk:curl-adjoint}
As an application of the definition of $T$, which is useful for Section~\ref{sec:Lie-derivative}, we compute the formal adjoint $\curl^* \colon \vf \to \vf$ of the curl operator introduced in Definition~\ref{defn:curl}, as follows. Let $W, V \in \vf$. Then using~\eqref{eq:Pdefn} we have
\begin{align*}
\langle \curl W, V \rangle & = (\curl W)_k V_k = (\nab{i} W_j \ph_{ijk}) V_k \\
& = \nab{i} (W_j \ph_{ijk} V_k) - W_j \nab{i} \ph_{ijk} V_k - W_j \ph_{ijk} \nab{i} V_k \\
& = \Div(\cdot) - W_j V_k T_{ip} \ps_{pijk} + W_j (\nab{i} V_k \ph_{ikj}) \\
& = \Div(\cdot) + W_j V_k (\Pop T)_{jk} + W_j (\curl V)_j \\
& = \Div(\cdot) + \langle W, (\Pop T) (V) \rangle + \langle W, \curl V \rangle.
\end{align*}
Thus, integrating both sides over $M$ and using the divergence theorem, we find that $\curl^* \colon \vf \to \vf$ is given by
\begin{equation} \label{eq:curl-adjoint}
\curl^* = \curl + \Pop T.
\end{equation}
Note that the second term $\Pop T$ in~\eqref{eq:curl-adjoint} is a $2$-form, and so is a (pointwise) skew-adjoint endomorphism. Moreover, if $T_{\skew} = T_7 + T_{14} = 0$, then $\curl$ is formally self-adjoint.
\end{rmk}

Because the torsion lies in $\cT^2$, we can use the decomposition~\eqref{eq:tens-decomp2} to write
\begin{equation} \label{eq:torsion-decomp}
T = T_1 + T_{27} + T_7 + T_{14} \qquad \text{where $T_1 = \tfrac{1}{7} (\tr T) g$},
\end{equation}
as in~\eqref{eq:A-decomp2}. We also have
\begin{equation} \label{eq:Tt}
T^t = T_1 + T_{27} - T_7 - T_{14}
\end{equation}
 and from~\eqref{eq:Ponbeta} we get
 \begin{equation} \label{eq:PT}
\Pop T = - 4 T_7 + 2 T_{14}.
\end{equation}
From these we obtain
\begin{equation} \label{eq:torsion-formulas}
\begin{aligned}
|T|^2 & = |T_1|^2 + |T_{27}|^2 + |T_7|^2 + |T_{14}|^2, \\
\langle T, T^t \rangle & = |T_1|^2 + |T_{27}|^2 - |T_7|^2 - |T_{14}|^2, \\
\langle T, \Pop T \rangle & = - 4 |T_7|^2 + 2 |T_{14}|^2, \\
(\tr T)^2 & = 7 |T_1|^2,
\end{aligned}
\end{equation}
which are equivalent to
\begin{equation} \label{eq:torsion-formulas-b}
\begin{aligned}
|T_1|^2 & = \tfrac{1}{7} (\tr T)^2, \\
|T_{27}|^2 & = \tfrac{1}{2} |T|^2 + \tfrac{1}{2} \langle T, T^t \rangle - \tfrac{1}{7} (\tr T)^2, \\
|T_7|^2 & = \tfrac{1}{6} |T|^2 - \tfrac{1}{6} \langle T, T^t \rangle - \tfrac{1}{6} \langle T, \Pop T \rangle, \\
|T_{14}|^2 & = \tfrac{1}{3} |T|^2 - \tfrac{1}{3} \langle T, T^t \rangle + \tfrac{1}{6} \langle T, \Pop T \rangle.
\end{aligned}
\end{equation}
The relations~\eqref{eq:torsion-formulas} and~\eqref{eq:torsion-formulas-b} express the four pointwise torsion energies $|T_k|^2$ for $k=1,27,7,14$ in terms of the four functions $|T|^2$, $\langle T, T^t \rangle$, $\langle T, \Pop T \rangle$, and $(\tr T)^2$ and conversely. These relations are used often in the sequel, particularly in Section~\ref{sec:functionals-revisited} to compute the Euler--Lagrange equations for various torsion functionals. We also remark that from~\eqref{eq:vecA} and~\eqref{eq:vec-transform} we can write
\begin{equation} \label{eq:VT}
T_7 = \tfrac{1}{6} (\Vop T) \hk \ph, \qquad \text{where $(\Vop T)_k = T_{ij} \ph_{ijk}$}.
\end{equation}

\begin{prop} \label{prop:FG}
The forms $\dd \ph \in \Omega^3$ and $\dd^* \ph \in \Omega^2$ are related to the components of the torsion via
$$ \dd \ph = ( T_1 + T_{27} + T_7 ) \diamond \ps, \qquad \dd^* \ph = - 4 T_7 + 2 T_{14}. $$
Consequently, we recover the classical theorem of Fern\'andez--Gray~\cite{FG}, which says that $\ph$ is torsion-free if and only if $\dd \ph = 0$ and $\dd^* \ph = 0$. 
\end{prop}
\begin{proof}
Using~\eqref{eq:nablaph}, we compute
\begin{align*}
(d \ph)_{ijkl} & = \nab{i} \ph_{jkl} - \nab{j} \ph_{ikl} + \nab{k} \ph_{ijl} - \nab{l} \ph_{ijk} \\
& = T_{ip} \ps_{pjkl} - T_{jp} \ps_{pikl} + T_{kp} \ps_{pijl} - T_{lp} \ps_{pijk} \\
& = T_{ip} \ps_{pjkl} + T_{jp} \ps_{ipkl} + T_{kp} \ps_{ijpl} + T_{lp} \ps_{ijkp} \\
& = (T \diamond \ps)_{ijkl}.
\end{align*}
The first equation now follows from~\eqref{eq:torsion-decomp} and the fact that $T_{14} \diamond \ps = 0$ from Corollary~\ref{cor:ABinnerproduct}.

Similarly, using~\eqref{eq:nablaph} and~\eqref{eq:PT} we compute
$$ (d^* \ph)_{jk} = - \nab{i} \ph_{ijk} = - T_{im} \ps_{mijk} = T_{im} \ps_{imjk} = (\Pop T)_{jk} = - 4 T_7 + 2 T_{14}, $$
as claimed.
\end{proof}

\textbf{Alternative description of torsion.} There is another way of packaging the torsion of a $\G$-structure, using the isomorphism $\Omega^1 \cong \Omega^2_7$ encapsulated in~\eqref{eq:omega27vf}. Explicitly, define $\hT \in \Gamma(T^* M \otimes \Lambda^2_7 (T^* M))$ by
\begin{equation} \label{eq:hatTdefn}
\hT_{pij} = T_{pq} \ph_{qij}, \qquad T_{pq} = \tfrac{1}{6} \hT_{pij} \ph_{qij}.
\end{equation}
For fixed $p$, we have $\hT_{pij}$ lies in $\Omega^2_7$ in $i,j$. Thus by~\eqref{eq:omega27desc} we have
\begin{equation} \label{eq:PonhT}
\hT_{pij} \ps_{ijkl} = - 4 \hT_{pkl}.
\end{equation}
We can think of $\hT$ as a $1$-form on $M$ with values in $\Lambda^2_7 (T^* M)$, via the pairing $( \hT(X) )_{ij} = X_p \hT_{pij}$.

\begin{rmk} \label{rmk:intrinsic-torsion}
This description of the torsion of a $G$-structure on a Riemannian manifold $(M^n, g)$ as a $1$-form taking values at each point in the orthogonal complement $\mathfrak{g}^{\perp}$ of the Lie algebra $\mathfrak{g} \subset \mathfrak{so}(n) \cong \Lambda^2$ of $G$ is usually called the \emph{intrinsic torsion} of the $G$-structure.
\end{rmk}

\begin{lemma} \label{lemma:nablaph-hT}
Fix $p \in \{ 1, \ldots, 7 \}$. At the point $x \in M$, we can write $\hT_p = \hT_{pij} e_i \otimes e_j$ as an element of $\Lambda^2_7(T_x^* M)$. Then we have
\begin{equation} \label{eq:nablaph-hT}
\nab{p} \ph_{abc} = - \tfrac{1}{3} (\hT_p \diamond \ph)_{abc}, \qquad \nab{p} \ps_{abcd} = - \tfrac{1}{3} (\hT_p \diamond \ps)_{abcd}.
\end{equation}
\end{lemma}
\begin{proof}
Using~\eqref{eq:nablaph} and~\eqref{eq:hatTdefn}, we compute
\begin{align*}
\nab{p} \ph_{abc} & = T_{pq} \ps_{qabc} = - \tfrac{1}{6} \hT_{pij} \ph_{ijq} \ps_{abcq} \\
& = - \tfrac{1}{6} \hT_{pij} ( g_{ia} \ph_{jbc} + g_{ib} \ph_{ajc} + g_{ic} \ph_{abj} - g_{ja} \ph_{ibc} - g_{jb} \ph_{aic} - g_{jc} \ph_{abi} ).
\end{align*}
Since $\hT_{pij}$ is skew in $i,j$ this becomes
\begin{align*}
\nab{p} \ph_{abc} & = - \tfrac{1}{3} \hT_{pij} ( g_{ia} \ph_{jbc} + g_{ib} \ph_{ajc} + g_{ic} \ph_{abj} ) \\
& = - \tfrac{1}{3} ( \hT_{paj} \ph_{jbc} + \hT_{pbj} \ph_{ajc} + \hT_{pcj} \ph_{abj} ) = - \tfrac{1}{3} (\hT_p \diamond \ph)_{abc}
\end{align*}
as claimed. The formula for $\nab{p} \ps_{abcd}$ in~\eqref{eq:nablaph-hT} can be derived by differentiating the first identity in~\eqref{eq:phph} and then using the formula for $\nab{p} \ph_{abc}$ from~\eqref{eq:nablaph-hT} and the first identity in~\eqref{eq:phph}.
\end{proof}

\subsection{The covariant derivative of the torsion} \label{sec:nabT}

Let $\nab{} T$ denote the covariant derivative of the torsion, which is a $3$-tensor with components $\nab{i} T_{jk}$. Various tensors constructed from $\nab{}T$ play an important role.

Recall that $\tr T = T_{kk}$ is a function. Thus, its gradient $\nab{} (\tr T)$ is the vector field
$$ \nab{i} (\tr T) = \nab{i} T_{kk}. $$
There are three different kinds of divergences of the torsion that arise often. There are the two \emph{vector fields} $\Div T$ and $\Div T^t$, which are given by
$$ (\Div T)_k = \nab{i} T_{ik}, \qquad (\Div T^t)_k = \nab{i} T_{ki}. $$
Recall also that $\Vop T$ is the vector field $(\Vop T)_k = T_{pq} \ph_{pqk}$. Hence, its divergence $\Div (\Vop T)$ is the \emph{function} given by
$$ \Div (\Vop T) = \nab{k} (\Vop T)_k. $$

There are three distinct $2$-tensors we can extract from the covariant derivative $\nab{} T$ of the torsion. These appear in the evolution of various torsion functionals in Sections~\ref{sec:functionals} and~\ref{sec:functionals-revisited}, and play a crucial role in Section~\ref{sec:curvature-torsion} to understand the decomposition of $\nab{}T$ into independent components and their relations to the Riemann curvature $\tRm$.

\begin{defn} \label{defn:K}
We can contract the $3$-tensor $\nab{} T$ with the $3$-form $\ph$ on two of the three corresponding pairs of indices to obtain a $2$-tensor. We denote these by $\KK{a}$ for $a=1,2,3$ where $a$ refers to the index of $\nab{}T$ that is \emph{not} contracted. That is,
\begin{equation} \label{eq:K-defn}
\KK{1}_{ab} = \nab{a} T_{pq} \ph_{bpq}, \qquad \KK{2}_{ab} = \nab{p} T_{aq} \ph_{pbq}, \qquad \KK{3}_{ab} = \nab{p} T_{qa} \ph_{pqb}.
\end{equation}
Note that $\tr \KK{1} = \tr \KK{2} = \tr \KK{3} = \nab{i} T_{jk} \ph_{ijk}$.
\end{defn}

For $a=1$, we can simplify $\KK{1}_{ab}$ as follows. Using~\eqref{eq:nablaph},~\eqref{eq:Pdefn}, and~\eqref{eq:vecA}, we have
\begin{align*}
\KK{1}_{ab} & = \nab{a} T_{pq} \ph_{bpq} = \nab{a} (T_{pq} \ph_{bpq}) - T_{pq} \nab{a} \ph_{bpq} \\
& = \nab{a} (\Vop T)_b - T_{pq} T_{am} \ps_{mbpq} = \nab{a} (\Vop T)_b - (T (\Pop T))_{ab}.
\end{align*}
Thus we obtain the useful relations
\begin{equation} \label{eq:KK1-simp}
\KK{1} + T (\Pop T) = \nab{} (\Vop T), \qquad \KK{1}^t - (\Pop T) T^t = (\nab{} (\Vop T))^t.
\end{equation}
Moreover, using~\eqref{eq:vecA} and~\eqref{eq:Pop-defn2}, we obtain the useful relation
\begin{align} \nonumber
\tr \KK{a} = \nab{i} T_{jk} \ph_{ijk} & = \nab{i} (T_{jk} \ph_{ijk}) - T_{jk} \nab{i} \ph_{ijk} \\ \nonumber
& = \nab{i} (\Vop T)_i - T_{jk} T_{ip} \ps_{pijk} \\ \label{eq:divVT}
& = \Div (\Vop T) + \langle T, \Pop T \rangle.
\end{align}

\begin{rmk} \label{rmk:KK-curv}
The symmetric parts of $\KK{2}$ and $\KK{3}$ are identified later in Section~\ref{sec:g2bianchi-revisited} with simpler expressions obtained from the Riemann curvature tensor, quadratic expressions in the torsion, and $\cL_{\Vop T} g$. Specifically, these identifications are given in equations~\eqref{eq:KK3symm} and~\eqref{eq:KK2symm}. Note that~\eqref{eq:KK1-simp} shows that the symmetric part of $\KK{1}$ is $\frac{1}{2} \cL_{\Vop T} g$, up to lower order terms.
\end{rmk}

We can also define a vector field $\langle \nab{}T, \ps \rangle$ by
\begin{equation} \label{eq:nabTps}
\langle \nab{}T, \ps \rangle_m = \nab{i} T_{jk} \ps_{ijkm}.
\end{equation}

Finally, we can consider the curl of $\Vop T$, which is another vector field.

\begin{lemma} \label{lemma:curlVT}
The vector field $\curl (\Vop T)$ is related to the vector fields $\Div T$, $\Div T^t$, and $\langle \nab{}T, \ps \rangle$ by
\begin{equation} \label{eq:curlVT}
\curl (\Vop T) = \Div T^t - \Div T + \langle \nab{} T, \ps \rangle + 2 T^t (\Vop T) - T (\Vop T) - (\tr T) \Vop T + \Vop (T^2).
\end{equation}
\end{lemma}
\begin{proof}
Using~\eqref{eq:defn-curl} and~\eqref{eq:nablaph}, we compute
\begin{align*}
(\curl (\Vop T))_k & = \nab{a} (\Vop T)_b \ph_{abk} = \nab{a} (T_{pq} \ph_{pqb}) \ph_{abk} \\
& = \nab{a} T_{pq} \ph_{pqb} \ph_{kab} + T_{pq} \nab{a} \ph_{pqb} \ph_{kab} \\
& = \nab{a} T_{pq} (g_{pk} g_{qa} - g_{pa} g_{qk} - \ps_{pqka}) + T_{pq} T_{am} \ps_{mpqb} \ph_{kab} \\
& = \nab{q} T_{kq} - \nab{p} T_{pk} + \nab{a} T_{pq} \ps_{apqk} \\
& \qquad {} + T_{pq} T_{am} (g_{km} \ph_{apq} + g_{kp} \ph_{maq} + g_{kq} \ph_{mpa} - g_{am} \ph_{kpq} - g_{ap} \ph_{mkq} - g_{aq} \ph_{mpk}),
\end{align*}
which simplifies further to
\begin{align*}
(\curl (\Vop T))_k & = (\Div T^t)_k - (\Div T)_k + \langle \nab{} T, \ps \rangle_k + (T^t (\Vop T))_k - (T (\Vop T))_k + (T^t (\Vop T))_k \\
& \qquad {} - (\tr T) (\Vop T)_k - (\Vop (T^t T))_k + (\Vop (T^2))_k.
\end{align*}
Since $T^t T$ is symmetric, $\Vop (T^t T) = 0$, and we obtain~\eqref{eq:curlVT}.
\end{proof}

\begin{rmk} \label{rmk:curlVT}
We simplify the expression~\eqref{eq:curlVT} for $\curl (\Vop T)$ considerably in Corollary~\ref{cor:curlVT-revisited} after we obtain an identity for $\langle \nab{} T, \ps \rangle$ in Section~\ref{sec:g2bianchi-revisited}.
\end{rmk}

We require the following identities for $\Vop(\KK{a})$ to simplify both the decomposition of the $\G$-Bianchi identity in Section~\ref{sec:g2bianchi-revisited} and the evolution equations for certain torsion functionals in Section~\ref{sec:functionals-revisited}.
\begin{lemma} \label{lemma:Vop-KK}
The expressions $\Vop( \KK{a} ) \in \Omega^1_7$ for each $a = 1, 2, 3$ are given by
\begin{equation} \label{eq:Vop-KK}
\begin{aligned}
\Vop( \KK{1} ) & = \Div T^t - \Div T + \langle \nab{}T, \ps \rangle, \\
\Vop( \KK{2} ) & = \Div T - \nab{} (\tr T) + \langle \nab{} T, \ps \rangle, \\
\Vop( \KK{3} ) & = \nab{} (\tr T) - \Div T^t + \langle \nab{} T, \ps \rangle.
\end{aligned}
\end{equation}
\end{lemma}
\begin{proof}
Using~\eqref{eq:vecA}, we find
\begin{align*}
(\Vop (\KK{1}))_k & = (\nab{a} T_{pq} \ph_{bpq}) \ph_{kab} = \nab{a} T_{pq} (\ph_{pqb} \ph_{kab}) \\
& = \nab{a} T_{pq} (g_{pk} g_{qa} - g_{pa} g_{qk} - \ps_{pqka}) \\
& = \nab{q} T_{kq} - \nab{p} T_{pk} + \nab{a} T_{pq} \ps_{apqk}.
\end{align*}
Similarly we have
\begin{align*}
(\Vop (\KK{2}))_k & = (\nab{p} T_{aq} \ph_{pbq}) \ph_{kab} = - \nab{p} T_{aq} (\ph_{pqb} \ph_{kab}) \\
& = - \nab{p} T_{aq} (g_{pk} g_{qa} - g_{pa} g_{qk} - \ps_{pqka}) \\
& = - \nab{k} (\tr T) + \nab{p} T_{pk} + \nab{p} T_{aq} \ps_{paqk},
\end{align*}
and
\begin{align*}
(\Vop (\KK{3}))_k & = (\nab{p} T_{qa} \ph_{pqb}) \ph_{kab} = \nab{p} T_{qa} (\ph_{pqb} \ph_{kab}) \\
& = \nab{p} T_{qa} (g_{pk} g_{qa} - g_{pa} g_{qk} - \ps_{pqka}) \\
& = \nab{k} (\tr T) - \nab{p} T_{kp} + \nab{p} T_{qa} \ps_{pqak},
\end{align*}
yielding the three expressions in~\eqref{eq:Vop-KK}.
\end{proof}

\subsection{Application: Taylor series expansion of $\ph$} \label{sec:taylor-series}

[In this section only we break from our convention of using a local orthonormal frame, and instead use Riemannian normal coordinates, suitably adapted to $\G$-structures.]
 
Recall the usual Taylor series expansion for the Riemannian metric $g$ with respect to Riemannian normal coordinates, which demonstrates that the Riemann curvature is the second-order obstruction to $(M^n, g)$ being locally isomorphic to the canonical flat model $(\R^n, g_0)$.

In this section we establish a Taylor series expansion of a $\G$-structure $\ph$ with respect to \emph{$\G$-adapted Riemannian normal coordinates}. This yields an explicit demonstration that the torsion $T$ is the first-order obstruction to $(M, \ph)$ being locally isomorphic to the canonical flat model $(\R^7, \ph_0)$, and gives a geometric interpretation for a particular combination of curvature and $\nabla T$, as the second-order obstruction.

We begin by briefly reviewing the well-known classical story for general Riemannian metrics, in order to both fix notation and obtain formulas we need for the $\G$ case. Let $(M^n,g)$ be Riemannian manifold and fix $x \in M$. The \emph{exponential map} $\exp_x \colon U \to M$ at $x$ is defined on some open neighbourhood $U$ of the origin in $T_x M$, and is given by $\exp_x (v) = \gamma_{v}(1)$ where $\gamma_{v}$ is the unique Riemannian geodesic with $\gamma_{v}(0) = x$ and $\gamma'_{v} (0) = v \in T_x M$. The map $\exp_x$ is a diffeomorphism from $U$ onto some open neighbourhood $\exp_x (U)$ of $x$ in $M$. Choosing an orthonormal basis $\{ e_1, \ldots, e_n \}$ of $T_x M$ gives a linear isomorphism $T_x M \cong \R^n$, and combining this with $\exp_x^{-1}$ yields a coordinate chart for $M$ centred at $x$, in which the geodesics emanating from $x$ correspond to rays through the origin in $\R^n$. That is, in such coordinates, $\gamma_v (t) = (c^1 t, \ldots, c^n t)$ where $v = c^i e_i \in T_x M$. Substituting $\gamma^i_v (t) = c^i t$ into the geodesic equation
\begin{equation*}
\frac{d^2 \gamma_v^l}{d t^2} + \Gamma^l_{ij} (\gamma_v(t)) \frac{d \gamma^i_v}{dt} \frac{d \gamma^j_v}{dt} = 0
\end{equation*}
we get
\begin{equation} \label{eq:geodesic}
\Gamma^l_{ij} (\gamma_v (t)) c^i c^j = 0.
\end{equation}
Evaluating~\eqref{eq:geodesic} at $t=0$, we obtain $\Gamma^l_{ij}(x) c^i c^j = 0$. Since $\Gamma^l_{ij}$ is symmetric in $i,j$ in a coordinate frame, we deduce that $\Gamma^l_{ij}$ vanishes at $x$. Moreover, differentiating~\eqref{eq:geodesic} with respect to $t$ and using the chain rule gives
\begin{equation*}
(\partial_k \Gamma^l_{ij}) (\gamma_v(t)) c^i c^j c^k = 0.
\end{equation*}
Evaluating the above at $t=0$, we obtain $(\partial_k \Gamma^l_{ij})(x) c^i c^j c^k = 0$. We deduce that the symmetrization of $\partial_k \Gamma^l_{ij}$ in $i,j,k$ vanishes at $x$. In summary we have
\begin{equation} \label{eq:normal-coords}
g_{ij} = \delta_{ij}, \qquad \Gamma^l_{ij} = 0, \qquad \partial_k \Gamma^l_{ij} + \partial_i \Gamma^l_{jk} + \partial_j \Gamma^l_{ik} = 0, \quad \text{at the point $x$}.
\end{equation}
It follows from the formula for $\Gamma^l_{ij}$ in local coordinates that
\begin{equation} \label{eq:partial-gij}
\partial_p g_{ij} = g_{il} \Gamma^l_{jp} + g_{jl} \Gamma^l_{ip}.
\end{equation}
Since the Christoffel symbols vanish at $x$, we deduce from~\eqref{eq:partial-gij} that
\begin{equation} \label{eq:partial-gij-2}
\partial_p g_{ij} = 0, \quad \text{at the point $x$}.
\end{equation}
Moreover, the formula for the Riemann curvature tensor in local coordinates gives
\begin{equation*}
R^l_{ijk} = \partial_i \Gamma^l_{jk} - \partial_j \Gamma^l_{ik} \quad \text{at the point $x$}.
\end{equation*}
Combining the above with the third equation in~\eqref{eq:normal-coords}, we compute
\begin{equation*}
R^l_{ijk} + R^l_{ikj} = \partial_i \Gamma^l_{jk} - \partial_j \Gamma^l_{ik} + \partial_i \Gamma^l_{kj} - \partial_k \Gamma^l_{ij} = 3 \, \partial_i \Gamma^l_{jk}.
\end{equation*}
We deduce that
\begin{equation} \label{eq:normal-coords-2}
\partial_i \Gamma^l_{jk} = \tfrac{1}{3} (R^l_{ijk} + R^l_{ikj}), \quad \text{at the point $x$}.
\end{equation}
Taking the partial derivative of~\eqref{eq:partial-gij} we obtain
\begin{equation*}
\partial_q \partial_p g_{ij} = (\partial_q g_{il}) \Gamma^l_{jp} + (\partial_q g_{jl}) \Gamma^l_{ip} + g_{il} (\partial_q \Gamma^l_{jp}) + g_{jl} (\partial_q \Gamma^l_{ip}).
\end{equation*}
Evaluating at $x$, the first two terms vanish by~\eqref{eq:normal-coords}, and by~\eqref{eq:normal-coords-2} and the symmetries of the curvature tensor, we get
\begin{align} \nonumber
\partial_q \partial_p g_{ij} & = \tfrac{1}{3} g_{il} (R^l_{qjp} + R^l_{qpj}) + \tfrac{1}{3} g_{jl} (R^l_{qip} + R^l_{qpi}) \\ \nonumber
& = \tfrac{1}{3} (R_{qjpi} + R_{qpji} + R_{qipj} + R_{qpij}) \\ \label{eq:partial2-gij}
& = \tfrac{1}{3} (R_{qjpi} + R_{qipj}), \quad \text{at the point $x$}.
\end{align}

\begin{prop} \label{prop:taylor-metric}
Let $(x^1, \ldots, x^n)$ be Riemannian normal coordinates centred at $x \in M$. The components $g_{ij}$ of the metric tensor have a Taylor expansion about $0$, which is the point in $\R^n$ corresponding to $x \in M$, given by
\begin{equation*}
g_{ij} (x^1, \ldots, x^n) = \delta_{ij} + \prescript{g}{}{\! \mathcal Q_{pq \, ij}} x^p x^q + O(\|x\|^3),
\end{equation*}
where
\begin{equation} \label{eq:taylor-metric-Q}
\prescript{g}{}{\! \mathcal Q_{pq \, ij}} = \tfrac{1}{6} (R_{piqj} + R_{pjqi})
\end{equation}
and $R_{piqj}$ and $R_{pjqi}$ are both evaluated at $0$.
\end{prop}
\begin{proof}
The constant term is $\delta_{ij}$ by the first equation in~\eqref{eq:normal-coords} and the linear term vanishes by~\eqref{eq:partial-gij-2}. Using~\eqref{eq:partial2-gij} and the symmetries of the Riemann tensor, the quadratic term can be written as
\begin{equation*}
\tfrac{1}{2} (\partial_q \partial_p g_{ij}) (0) x^p x^q = \tfrac{1}{6} (R_{qjpi} + R_{qipj})x^p x^q = \tfrac{1}{6} (R_{piqj} + R_{pjqi}) x^p x^q. \qedhere
\end{equation*}
\end{proof}

\begin{rmk} \label{rmk:taylor-metric}
Proposition~\ref{prop:taylor-metric} shows that, in Riemannian normal coordinates $x^1, \ldots, x^n$ centred at $x \in M$, the metric $g$ agrees with the Euclidean metric $g_{ij} = \delta_{ij}$ up to second-order, if and only if the Riemann curvature tensor $\tRm$ vanishes at $x$. Sufficiency is obvious. To see necessity, let $i=p$ and $j=q$ in~\eqref{eq:taylor-metric-Q}. We get the vanishing of $R_{ppqq} + R_{pqqp} = 0 + R_{pqqp}$, which says that all sectional curvatures vanish, and thus $\tRm$ vanishes as is well-known.
\end{rmk}

Now let $(M^7, \ph)$ be a manifold with $\G$-structure. We have assembled all we need to establish the analogous second-order Taylor expansion of $\ph$. We can choose our local orthonormal frame $\{ e_1, \ldots, e_7 \}$ of $T_x M$ to be \emph{$\G$-adapted}, meaning that \emph{at the point $x$}, the components $\ph_{ijk}$ of $\ph$ agree with those of the standard flat model on $\R^7$.

\begin{thm} \label{thm:taylor-ph}
Let $(x^1, \ldots, x^7)$ be $\G$-adapted Riemannian normal coordinates centred at $x \in M$. The components $\ph_{ijk}$ of $\ph$ have Taylor expansions about $0$, which is the point in $\R^7$ corresponding to $x \in M$, given by
\begin{equation} \label{eq:taylor-ph}
\ph_{ijk} (x^1, \ldots, x^7) = \ph_{ijk} + (T_{qm} \ps_{mijk}) x^q + \prescript{\ph}{}{\! \mathcal Q_{pq \, ijk}} x^p x^q + O(\|x\|^3),
\end{equation}
where
\begin{align} \nonumber
\prescript{\ph}{}{\! \mathcal Q_{pq \, ijk}} & = \tfrac{1}{2} \nab{p} T_{qm} \ps_{mijk} - \tfrac{1}{2} (TT^t)_{pq} \ph_{ijk} \\ \nonumber
& \qquad {} + \tfrac{1}{2} T_{pm} (T_{qi} \ph_{mjk} + T_{qj} \ph_{mki} + T_{qk} \ph_{mij}) \\ \label{eq:taylor-ph-B}
& \qquad {} + \tfrac{1}{6} (R_{piqm} \ph_{mjk} + R_{pjqm} \ph_{mki} + R_{pkqm} \ph_{mij}).
\end{align}
Here all coefficient tensors on the right-hand side are evaluated at $0$.
\end{thm}
\begin{proof}
The constant term in~\eqref{eq:taylor-ph} is due to our choice of a $\G$-adapted orthonormal frame at $x$. The local coordinate formula for the covariant derivative gives
\begin{equation} \label{eq:taylor-temp}
\partial_q \ph_{ijk} = \nab{q} \ph_{ijk} + \Gamma^m_{qi} \ph_{mjk} + \Gamma^m_{qj} \ph_{imk} + \Gamma^m_{qk} \ph_{ijm}.
\end{equation}
Evaluating~\eqref{eq:taylor-temp} at $x$ and using~\eqref{eq:normal-coords} and~\eqref{eq:nablaph} gives
\begin{equation*}
\partial_q \ph_{ijk} = T_{qm} \ps_{mijk}, \quad \text{at the point $x$},
\end{equation*}
yielding the linear term in~\eqref{eq:taylor-ph}. Taking the partial derivative of~\eqref{eq:taylor-temp} we obtain
\begin{equation*}
\partial_p \partial_q \ph_{ijk} = \partial_p \nab{q} \ph_{ijk} + (\partial_p \Gamma^m_{qi}) \ph_{mjk} + (\partial_p \Gamma^m_{qj}) \ph_{imk} + (\partial_p \Gamma^m_{qk}) \ph_{ijm} + (\text{terms with $\Gamma$'s}).
\end{equation*}
As in~\eqref{eq:taylor-temp}, the first term on the right-hand side above can be written as
\begin{equation*}
\partial_p \nab{q} \ph_{ijk} = \nab{p} \nab{q} \ph_{ijk} + (\text{terms with $\Gamma$'s}).
\end{equation*}
Evaluating both of the above expressions at $x$, the terms with $\Gamma$'s vanish, and we are left with
\begin{equation*}
(\partial_p \partial_q \ph_{ijk}) = (\nab{p} \nab{q} \ph_{ijk}) + \sum_{i,j,k \text{ cyclic}} (\partial_p \Gamma^m_{qi}) \ph_{mjk}, \quad \text{at the point $x$}.
\end{equation*}
Using~\eqref{eq:normal-coords-2} and equations~\eqref{eq:nablaph} and~\eqref{eq:nablaps}, at the point $x$ the above expression is
\begin{align*}
(\partial_p \partial_q \ph_{ijk}) |_0 & = \nab{p} (T_{qm} \ps_{mijk}) + \sum_{i,j,k \text{ cyclic}} \tfrac{1}{3} (R^m_{pqi} + R^m_{piq}) \ph_{mjk} \\
& = \nab{p} T_{qm} \ps_{mijk} - T_{qm} T_{pm} \ph_{ijk} + T_{qm} \sum_{i,j,k \text{ cyclic}} T_{pi} \ph_{mjk} + \sum_{i,j,k \text{ cyclic}} \tfrac{1}{3} (R^m_{pqi} + R^m_{piq}) \ph_{mjk} \\
& = \nab{p} T_{qm} \ps_{mijk} - (T T^t)_{pq} \ph_{ijk} + \sum_{i,j,k \text{ cyclic}} (T_{qm} T_{pi} + \tfrac{1}{3} R_{pqim} + \tfrac{1}{3} R_{piqm}) \ph_{mjk}.
\end{align*}
We multiply by $\frac{1}{2} x^p x^q$, and sum over $p,q$. The first curvature term drops out, leaving us with
\begin{equation*}
\tfrac{1}{2} (\partial_p \partial_q \ph_{ijk}) |_0 x^p x^q = \prescript{\ph}{}{\! \mathcal Q_{pq \, ijk}} x^p x^q,
\end{equation*}
where $\prescript{\ph}{}{\! \mathcal Q_{pq \, ijk}}$ is given by~\eqref{eq:taylor-ph-B}.
\end{proof}

\begin{rmk} \label{rmk:taylor-ph}
Theorem~\ref{thm:taylor-ph} shows that, in $\G$-adapted Riemannian normal coordinates $x^1, \ldots, x^7$ centred at $x \in M$, the $3$-form $\ph$ agrees with the standard $3$-form $\ph_0$ on $\R^7$ up to first-order, if and only if the torsion $T$ vanishes at $x$. In this case, agreement up to second-order is then given by the additional vanishing of the symmetrization $\prescript{\ph}{}{\! \mathcal Q_{pq \, ijk}} + \prescript{\ph}{}{\! \mathcal Q_{qp \, ijk}}$. Since $T = 0$ already, the vanishing of the quadratic terms involves only the curvature, and it is easy to show by contracting with $\ph_{njk}$ that this is equivalent to flatness at $x$. This is not surprising, as we show in Section~\ref{sec:curvature-torsion} that when $T = 0$, the Riemann curvature tensor has only one potentially nonzero component in terms of irreducible $\G$-representations.
\end{rmk}

\begin{rmk} \label{rmk:taylor-others}
Using similar methods, one could also establish such Taylor series expansions for other geometric structures such as $\U{m}$, $\SU{m}$, or $\Spin{7}$-structures.
\end{rmk}

\subsection{Infinitesimal $\G$-symmetries} \label{sec:Lie-derivative}

In this section, we use Lemma~\ref{lemma:nablaph-hT} to derive a general formula for the \emph{Lie derivative} $\cL_W \ph$ of a $\G$-structure $\ph$ with respect to a vector field $W$. We also determine the formal adjoint $\delta \colon \Omega^3 \to \vf$ of the map $\delta^* \colon \vf \to \Omega^3$ given by $\delta^* W = \cL_W \ph$. These results are used crucially throughout Section~\ref{sec:symbols-short-time} to analyze a large class of flows of $\G$-structures and to prove that their failure to be strictly parabolic is due precisely to diffeomorphism invariance, thus admitting a DeTurck trick argument.

Applying~\eqref{eq:Lie-derivative-frame} to $S = \ph$, we have
\begin{equation} \label{eq:Lie-derivative-framea}
(\cL_W \ph)_{ijk} = W_p \nab{p} \ph_{ijk} + \nab{i} W_p \ph_{pjk} + \nab{j} W_p \ph_{ipk} + \nab{k} W_p \ph_{ijp}.
\end{equation}
Using~\eqref{eq:nablaph-hT} and~\eqref{eq:diamond3}, we can rewrite equation~\eqref{eq:Lie-derivative-framea} as
\begin{equation} \label{eq:Lie-derivative-frameb}
\cL_W \ph = - \tfrac{1}{3} W_p \hT_p \diamond \ph + (\nabla W) \diamond \ph = (\nabla W - \tfrac{1}{3} \hT(W)) \diamond \ph.
\end{equation}
Write $\nabla W = (\nabla W)_{\symm} + (\nabla W)_7 + (\nabla W)_{14}$, where
$$ \big( (\nabla W)_{\symm} \big){}_{ij} = \tfrac{1}{2} (\nab{i} W_j + \nab{j} W_i) = \tfrac{1}{2} (\cL_W g)_{ij}. $$
By Corollary~\ref{cor:ABinnerproduct}, we have $(\nabla W)_{14} \diamond \ph = 0$. Hence equation~\eqref{eq:Lie-derivative-frameb} becomes
\begin{equation} \label{eq:Lie-derivative-framec}
\cL_W \ph = \big( \tfrac{1}{2} \cL_W g - \tfrac{1}{3} \hT(W) + (\nabla W)_7 \big) \diamond \ph.
\end{equation}
From Definition~\ref{defn:curl} we have $(\nabla W)_7 = \tfrac{1}{6} \curl W \hk \ph$. Moreover, by~\eqref{eq:hatTdefn} we have
$$ \hT(W)_{ij} = W_p \hT_{pij} = W_p T_{pq} \ph_{qij} = (T^t_{qp} W_p) \ph_{qij} = \big( (T^t W) \hk \ph \big){}_{ij}. $$
Using these two observations, equation~\eqref{eq:Lie-derivative-framec} finally becomes
\begin{equation} \label{eq:Lie-derivative-framed}
\cL_W \ph = \big( \tfrac{1}{2} \cL_W g + (- \tfrac{1}{3} T^t W + \tfrac{1}{6} \curl W ) \hk \ph \big) \diamond \ph.
\end{equation}
Using Corollary~\ref{cor:sevenreps}, the above can also be written in the useful form
\begin{equation} \label{eq:Lie-derivative-alternate}
\cL_W \ph = \tfrac{1}{2} (\cL_W g) \diamond \ph + (T^t W - \tfrac{1}{2} \curl W ) \hk \ps.
\end{equation}

\begin{defn} \label{defn:inf-symmetry}
A vector field $W$ on $(M, \ph)$ is called an \emph{infinitesimal $\G$-symmetry} if $\cL_W \ph = 0$. Note that this means that the \emph{flow} of $W$ preserves $\ph$.
\end{defn}

\begin{cor} \label{cor:inf-symmetry}
Let $W \in \vf$. Then $W$ is an infinitesimal $\G$-symmetry if and only if $W$ is a Killing field of $g$ and the \emph{curl} of $W$ is $\tfrac{1}{3} T^t W$. That is,
$$ \cL_W \ph = 0 \quad \iff \quad \cL_W g = 0 \, \, \, \text{and} \, \, \, \curl W = 2 \, T^t W. $$
\end{cor}
\begin{proof}
The proof is immediate from~\eqref{eq:Lie-derivative-framed} and Corollary~\ref{cor:ABinnerproduct}.
\end{proof}

\begin{rmk} \label{rmk:inf-symmetry}
Since $\ph$ determines the metric $g$, we expect that $\cL_W \ph = 0$ implies $\cL_W g = 0$, as we have seen above. The content of Corollary~\ref{cor:inf-symmetry} is that the infinitesimal $\G$-symmetries $W$ are precisely those Killing fields which satisfy the additional condition that $\curl W = 2 \, T^t W$. Note in the torsion-free case, this says that $W$ must be \emph{curl-free}. Corollary~\ref{cor:inf-symmetry} has appeared before in various guises. For example, it is implicit in~\cite[equation (2.28)]{DGK}. The torsion-free case appears in~\cite[Proposition 2.15]{KL}. The closed case appears in~\cite[Lemma 9.3]{LW1}. In the nearly parallel case, it is implicit in~\cite[Section 4.1]{DS}.
\end{rmk}

\subsection{The $\G$-Bianchi identity} \label{sec:g2bianchi}

In this section we discuss the \emph{$\G$-Bianchi identity}, and derive its simplest consequence. The $\G$-Bianchi identity is an identity for any $\G$-structure $\ph$, relating the Riemann curvature $\tRm$ of $g_{\ph}$ with the torsion $T$ of $\ph$ and its covariant derivative $\nab{}T$. It was originally derived in~\cite[Theorem 4.2]{K-flows} by analyzing the diffeomorphism invariance of the torsion tensor $T$. A much simpler proof can be obtained using the Ricci identity~\eqref{eq:ricci-identity} and the fundamental contraction identities in Section~\ref{sec:contractions}. Such a proof appeared in~\cite[Lemma 2.1]{LW1}. We review it here for completeness.

\begin{prop} \label{lemma:g2bianchi}
For any $\G$-structure $\ph$, the following identity holds:
\begin{equation} \label{eq:g2bianchi}
\nab{i} T_{jk} - \nab{j} T_{ik} = T_{ip} T_{jq} \ph_{pqk} + \tfrac{1}{2} R_{ijpq} \ph_{pqk}.
\end{equation}
The identity~\eqref{eq:g2bianchi} is often referred to as the \emph{$\G$-Bianchi identity}.
\end{prop}
\begin{proof}
We take the covariant derivative of~\eqref{eq:nablaph} and substitute~\eqref{eq:nablaps}. This gives
\begin{align*}
\nab{m} \nab{p} \ph_{ijk} & = \nab{m} T_{pq} \ps_{qijk} + T_{pq} \nab{m} \ps_{qijk} \\
& = \nab{m} T_{pq} \ps_{qijk} +T_{pq} ( -T_{mq} \ph_{ijk} + T_{mi} \ph_{qjk} - T_{mj} \ph_{qik} + T_{mk} \ph_{qij}).
\end{align*}
Interchange the roles of $p$ and $m$ and take the difference, and use the fact that $T_{pq} T_{mq}$ is symmetric in $p,m$. We get
\begin{align*}
\nab{m} \nab{p} \ph_{ijk} - \nab{p} \nab{m} \ph_{ijk} & = (\nab{m} T_{pq} - \nab{p} T_{mq}) \ps_{qijk} \\
& \qquad {} + T_{pq} (T_{mi} \ph_{qjk} - T_{mj} \ph_{qik} + T_{mk} \ph_{qij}) \\
& \qquad {} - T_{mq} (T_{pi} \ph_{qjk} - T_{pj} \ph_{qik} + T_{pk} \ph_{qij}).
\end{align*}
Apply the Ricci identity~\eqref{eq:ricci-identity} to the left-hand sides gives
\begin{align*}
-R_{mpiq} \ph_{qjk} - R_{mpjq} \ph_{iqk} - R_{mpkq} \ph_{ijq} & = (\nab{m} T_{pq} - \nab{p} T_{mq}) \ps_{qijk} \\
& \qquad {} + T_{pq} (T_{mi} \ph_{qjk} - T_{mj} \ph_{qik} + T_{mk} \ph_{qij}) \\
& \qquad {} - T_{mq} (T_{pi} \ph_{qjk} - T_{pj} \ph_{qik} + T_{pk} \ph_{qij}).
\end{align*}
Contract both sides of the above expression with $\ps_{lijk}$ and use the fact that the left-hand side and each of the three terms on the right-hand side are totally skew in $i,j,k$. We obtain
\begin{equation*}
-3 R_{mpiq} \ph_{qjk} \ps_{lijk} = (\nab{m} T_{pq} - \nab{p} T_{mq}) \ps_{qijk} \ps_{lijk} + 3 T_{pq} T_{mi} \ph_{qjk} \ps_{lijk} - 3 T_{mq} T_{pi} \ph_{qjk} \ps_{lijk}.
\end{equation*}
Apply the contraction identities to rewrite the above as
\begin{align*}
12 R_{mpiq} \ph_{qli} = 24 (\nab{m} T_{pl} - \nab{p} T_{ml}) - 12 T_{pq} T_{mi} \ph_{qli} + 12 T_{mq} T_{pi} \ph_{qli}.
\end{align*}
The above expression can be rearranged and reindexed to give precisely~\eqref{eq:g2bianchi}.
\end{proof}

The simplest consequence of the $\G$-Bianchi identity is an expression for the scalar curvature in terms of the torsion.

\begin{cor} \label{cor:scalar-curvature}
The scalar curvature $R = R_{ijji}$ of a $\G$-structure $\ph$ can be expressed entirely in terms of the torsion as
\begin{equation} \label{eq:scalar-curvature}
R = (\tr T)^2 - \langle T, T^t \rangle + \langle T, \Pop T \rangle - 2 \nab{i} T_{jk} \ph_{ijk}.
\end{equation}
It can equivalently be expressed as
\begin{equation} \label{eq:scalar-curvature-b}
R = (\tr T)^2 - \langle T, T^t \rangle - \langle T, \Pop T \rangle - 2 \Div (\Vop T)
\end{equation}
where $(\Vop T)_k = T_{ij} \ph_{ihk}$ as in~\eqref{eq:vecA}, or alternatively as
\begin{equation} \label{eq:scalar-curvature-c}
R = 6|T_1|^2 - |T_{27}|^2 + 5 |T_7|^2 - |T_{14}|^2 - 2 \Div (\Vop T).
\end{equation}
\end{cor}
\begin{proof}
Contracting both sides of~\eqref{eq:g2bianchi} with $\ph_{ijk}$ gives
\begin{align*}
2 \nab{i} T_{jk} \ph_{ijk} & = (T_{ip} T_{jq} + \tfrac{1}{2} R_{ijpq}) (g_{ip} g_{jq} - g_{iq} g_{jp} - \ps_{ijpq}) \\
& = T_{ii} T_{jj} - T_{ij} T_{ji} +T_{ip} (T_{jq} \ps_{jqip}) + \tfrac{1}{2} R_{ijij} - \tfrac{1}{2} R_{ijji} - \tfrac{1}{2} R_{ijpq} \ps_{ijpq}.
\end{align*}
The last term above vanishes by the skew-symmetry of $\ps_{ijpq}$ and the first Riemannian Bianchi identity $R_{ijpq} + R_{jpiq} + R_{pijq} = 0$. (See also Corollary~\ref{cor:U-IP-diamond-ps}.) Using~\eqref{eq:Pdefn}, the above becomes
\begin{equation} \label{eq:scalar-curv-temp}
2 \nab{i} T_{jk} \ph_{ijk} = (\tr T)^2 - \langle T, T^t \rangle + \langle T, \Pop T \rangle - R,
\end{equation}
which is equation~\eqref{eq:scalar-curvature}. Using~\eqref{eq:divVT}, the left-hand side of~\eqref{eq:scalar-curv-temp} is $2 \Div (\Vop T) + 2 \langle T, \Pop T \rangle$, so~\eqref{eq:scalar-curv-temp} becomes
$$ R = (\tr T)^2 - \langle T, T^t \rangle - \langle T, \Pop T \rangle - 2 \Div (\Vop T), $$
which is~\eqref{eq:scalar-curvature-b}. Finally, the expression~\eqref{eq:scalar-curvature-c} follows from substitution of~\eqref{eq:torsion-formulas}.
\end{proof}

In Section~\ref{sec:g2bianchi-revisited} we derive several independent relations from the $\G$-Bianchi identity. The relation in Corollary~\ref{cor:scalar-curvature} is one of these, and it is the only scalar relation.

\subsection{The rough and Hodge Laplacians of $\G$-structures} \label{sec:Laplacians}

In this section we derive formulas for the rough Laplacian $\nabla^* \nabla \ph$ and the Hodge Laplacian $\Delta_{\dd} \ph$ of a $\G$-structure $\ph$. In the process we also introduce the ``$\ph$-Ricci tensor'' $F_{pq}$ of $\ph$, which is a symmetric $2$-tensor that plays an important role throughout the paper. The next result also appeared independently in~\cite[Example 1.23]{FLMS}.

\begin{prop} \label{prop:rough-Lap}
Let $\ph$ be a $\G$-structure on $M$. Its rough Laplacian $\nabla^* \nabla \ph$ is
\begin{equation} \label{eq:rough-Lap-ph}
\nabla^* \nabla \ph = \big( \tfrac{1}{3} (\Div T) \hk \ph + \tfrac{1}{3} |T|^2 g - T^t T \big) \diamond \ph.
\end{equation}
More precisely, we have $\nabla^* \nabla \ph = A \diamond \ph$, where $A = A_1 + A_{27} + A_7 \in \cS^2 \oplus \Omega^2_7$, such that
\begin{equation} \label{eq:rough-Lap-ph-b}
A_1 = \tfrac{4}{21} |T|^2 g, \qquad A_{27} = \tfrac{1}{7} |T|^2 g - T^t T, \qquad A_7 = \tfrac{1}{3} (\Div T) \hk \ph.
\end{equation}
\end{prop}
\begin{proof}
We compute using~\eqref{eq:nablaph} and~\eqref{eq:nablaps} that
\begin{align*}
(\nabla^* \nabla \ph)_{ijk} & = - \nab{p} \nab{p} \ph_{ijk} = - \nab{p} ( T_{pq} \ps_{qijk} ) \\
& = - \nab{p} T_{pq} \ps_{qijk} - T_{pq} \nab{p} \ps_{qijk} \\
& = - (\Div T)_q \ps_{qijk} - T_{pq} ( - T_{pq} \ph_{ijk} + T_{pi} \ph_{qjk} - T_{pj} \ph_{qik} + T_{pk} \ph_{qij} ) \\
& = - (\Div T)_q \ps_{qijk} + |T|^2 \ph_{ijk} - (T^t T)_{iq} \ph_{qjk} - (T^t T)_{jq} \ph_{iqk} - (T^t T)_{kq} \ph_{ijq}.
\end{align*}
Using~\eqref{eq:diamond-coords},~\eqref{eq:gdiamond}, and Corollary~\ref{cor:sevenreps}, the above can be written
$$ (\nabla^* \nabla \ph) = \tfrac{1}{3} \big( (\Div T) \hk \ph \big) \diamond \ph + \big( \tfrac{1}{3} |T|^2 g - T^t T \big) \diamond \ph, $$
establishing~\eqref{eq:rough-Lap-ph}. Since $\tr (T^t T ) = |T|^2$, we have $\tr A = \frac{7}{3} |T|^2 - |T|^2 = \frac{4}{3} |T|^2$. From $A_1 = \frac{1}{7} (\tr A) g$ and $A_{27} = A_{1+27} - A_1$, we get $A_1 = \frac{4}{21} |T|^2 g$ and $A_{27} = \frac{1}{7} |T|^2 g - T^t T$, establishing~\eqref{eq:rough-Lap-ph-b}.
\end{proof}

Proposition~\ref{prop:rough-Lap} has the following interesting corollary, which does not seem to be well-known.

\begin{cor} \label{cor:rough-Lap}
Let $\ph$ be a $\G$-structure on $M$. Then $\nabla \ph = 0$ if and only if $\nabla^* \nabla \ph = 0$. That is, the torsion-free $\G$-structures are precisely those whose rough Laplacian vanishes. (Note that this is immediate from integration by parts if $M$ is compact, but we \emph{do not} assume that $M$ is compact here.)
\end{cor}
\begin{proof}
One direction is trivial. Conversely, $\nabla^* \nabla \ph = A \diamond \ph = 0$ if and only if $A_1 = A_{27} = A_7 = 0$. But Proposition~\ref{prop:rough-Lap} shows that $A_1 = \frac{4}{21} |T|^2 g = 0$ already forces $T = 0$.
\end{proof}

Next we consider the Hodge Laplacian $\Delta_{\dd} \ph$. For completeness, and to avoid uncertainty on the part of the reader regarding notation and conventions, we quickly derive the classical Weitzenb\"ock formula for the particular case of $3$-forms. If $\gamma \in \Omega^3$, then the Weitzenb\"ock formula says $\Delta_{\dd} \gamma = \nabla^* \nabla \gamma + \tRc \cdot \gamma + \tRm \cdot \gamma$ where the final two terms are some particular contractions of the Ricci curvature $\tRc$ and the Riemann curvature $\tRm$ with $\gamma$, respectively.

We have
\begin{align*}
(\dd \dd^* \gamma)_{ijk} & = \nab{i} (\dd^* \gamma)_{jk} + \nab{j} (\dd^* \gamma)_{ki} + \nab{k} (\dd^* \gamma)_{ij} \\
& = - \nab{i} \nab{p} \gamma_{pjk} - \nab{j} \nab{p} \gamma_{pki} - \nab{k} \nab{p} \gamma_{pij},
\end{align*}
and
\begin{align*}
(\dd^* \dd \gamma)_{ijk} & = - \nab{p} (\dd \gamma)_{pijk} = - \nab{p} (\nab{p} \gamma_{ijk} - \nab{i} \gamma_{pjk} + \nab{j} \gamma_{pik} - \nab{k} \gamma_{pij}) \\
& = (\nabla^* \nabla \gamma)_{ijk} + \nab{p} \nab{i} \gamma_{pjk} + \nab{p} \nab{j} \gamma_{pki} + \nab{p} \nab{k} \gamma_{pij}.
\end{align*}
Adding the above expressions, we obtain
\begin{align*}
(\Delta_{\dd} \gamma)_{ijk} & = (\dd \dd^* \gamma)_{ijk} + (\dd^* \dd \gamma)_{ijk} \\
& = (\nabla^* \nabla \gamma)_{ijk} + (\nab{p} \nab{i} - \nab{i} \nab{p}) \gamma_{pjk} + (\nab{p} \nab{j} - \nab{j} \nab{p}) \gamma_{pki} + (\nab{p} \nab{k} - \nab{k} \nab{p}) \gamma_{pij}. 
\end{align*}
Using the Ricci identity~\eqref{eq:ricci-identity} and the symmetries of the Riemann curvature tensor, we can write
\begin{align*}
(\nab{p} \nab{i} - \nab{i} \nab{p}) \gamma_{pjk} & = - R_{pipm} \gamma_{mjk} - R_{pijm} \gamma_{pmk} - R_{pikm} \gamma_{pjm} \\
& = R_{im} \gamma_{mjk} + R_{pimj} \gamma_{pmk} + R_{mkpi} \gamma_{mpj}.
\end{align*}
Interchanging the roles of $m$ and $p$ in the last term, and cyclically permuting the above expression in $i,j,k$, the expression for $\Delta_{\dd} \gamma$ becomes
\begin{align*}
(\Delta_{\dd} \gamma)_{ijk} & = (\nabla^* \nabla \gamma)_{ijk} + R_{im} \gamma_{mjk} + R_{jm} \gamma_{mki} + R_{km} \gamma_{mij} \\
& \qquad {} + 2 R_{pimj} \gamma_{pmk} + 2 R_{pjmk} \gamma_{pmi} + 2 R_{pkmi} \gamma_{pmj}.
\end{align*}
Note that by the Riemannian first Bianchi identity, we have
$$ R_{pimj} \gamma_{pmk} = - R_{pmji} \gamma_{pmk} - R_{pjim} \gamma_{pmk} = R_{pmij} \gamma_{pmk} - R_{mipj} \gamma_{mpk}, $$
which can be rearranged to yield
$$ 2 R_{pimj} \gamma_{pmk} = R_{pmij} \gamma_{pmk}. $$
We conclude that the Weitzenb\"ock formula on $3$-forms is
\begin{equation} \label{eq:Weitzenbock-3}
\begin{aligned}
(\Delta_{\dd} \gamma)_{ijk} & = (\nabla^* \nabla \gamma)_{ijk} + R_{im} \gamma_{mjk} + R_{jm} \gamma_{mki} + R_{km} \gamma_{mij} \\
& \qquad {} + R_{pmjk} \gamma_{pmi} + R_{pmki} \gamma_{pmj} + R_{pmij} \gamma_{pmk}.
\end{aligned}
\end{equation}

Before we can describe $\Delta_{\dd} \ph$, we need to introduce a symmetric $2$-tensor $F_{pq}$, called the $\ph$-Ricci tensor, which seems to have first appeared in Cleyton--Ivanov~\cite[Definition 3.1]{CI}.

\begin{defn} \label{defn:F}
The $\ph$-Ricci tensor $F_{pq}$ is the smooth $2$-tensor given in terms of a local frame by
\begin{equation} \label{eq:F-defn}
F_{pq} = R_{ijkl} \ph_{ijp} \ph_{klq},
\end{equation}
where $R_{ijkl}$ is the Riemann curvature tensor of $g$. It is clear that $F_{pq}$ is \emph{symmetric}. Because the curvature tensor lies in $\cS^2(\Lam^2) = \Gamma( \Sym^2(\Lam^2_7 \oplus \Lambda^2_{14}))$, we see from~\eqref{eq:omega27vf} that $F_{pq}$ is essentially the part of the curvature tensor which lies in $\Gamma(\Sym^2(\Lambda^2_7)) \cong \Gamma(\Sym^2 (T^* M)) = \cS^2$.

[Cleyton--Ivanov~\cite{CI} write $\rho^*$ for the $\ph$-Ricci tensor. We use $F$ to avoid the proliferation of too much notation. We chose $F$ as it often denotes a curvature, and because $\ph$ is the Greek version of $F$.]
\end{defn}

\begin{lemma} \label{lemma:trace-F}
The trace of $F_{ij}$ is $F_{pp} = - 2 R$, where $R = R_{pp} = R_{qppq}$ is the scalar curvature of $g$.
\end{lemma}
\begin{proof}
We compute
\begin{align*}
F_{pp} & = R_{ijkl} \ph_{ijp} \ph_{klp} = R_{ijkl} (g_{ik} g_{jl} - g_{il} g_{jk} - \ps_{ijkl}) \\
& = R_{ijij} - R_{ijji} - R_{ijkl} \ps_{ijkl} = - 2 R - R_{ijkl} \ps_{ijkl}.
\end{align*}
The last term vanishes by the skew-symmetry of $\ps$ and the first Bianchi identity, yielding the result. (See also Corollary~\ref{cor:U-IP-diamond-ps}.)
\end{proof}

\begin{rmk} \label{rmk:F}
In Section~\ref{sec:curvature-torsion} we examine in detail how the $\ph$-Ricci tensor $F_{pq}$ is related to the usual decomposition of Riemann curvature into scalar curvature $R$, traceless Ricci curvature $R^0_{ij}$, and Weyl curvature $W_{ijkl}$. We show that $F_{pq}$ is a particular linear combination $F_{pq} = a R g_{pq} + b R^0_{pq} + c (W_{27})_{pq}$, where $W_{27}$ is a traceless symmetric $2$-tensor $W_{27}$ extracted from the Weyl tensor. Lemma~\ref{lemma:trace-F} says that $a = - \frac{2}{7}$. These results appeared first in Cleyton--Ivanov~\cite{CI}.
\end{rmk}

\begin{lemma} \label{lemma:Hodge-F}
Let $\gamma \in \Omega^3$ be given by
$$ \gamma_{ijk} = R_{pmjk} \ph_{pmi} + R_{pmki} \ph_{pmj} + R_{pmij} \ph_{pmk}. $$
Then $\gamma = A \diamond \ph$, where $A_{ij} = \frac{1}{6} R g_{ij} + \frac{1}{4} F_{ij} - R_{ij}$.
\end{lemma}
\begin{proof}
In the notation of Corollary~\ref{cor:diamondinverse}, we have $\gamma^{\ph}_{ia} = \gamma_{ijk} \ph_{ajk}$. Using~\eqref{eq:F-defn} we compute
\begin{align*}
\gamma^{\ph}_{ia} & = (R_{pmjk} \ph_{pmi} + R_{pmki} \ph_{pmj} + R_{pmij} \ph_{pmk}) \ph_{ajk} \\
& = R_{pmjk} \ph_{pmi} \ph_{jka} + 2 R_{pmki} \ph_{pmj} \ph_{kaj} \\
& = F_{ia} + 2 R_{pmki} (g_{pk} g_{ma} - g_{pa} g_{mk} - \ps_{pmka}) \\
& = F_{ia} + 2 R_{kaki} - 2 R_{akki} - 0 \\
& = F_{ia} - 2 R_{ia} - 2 R_{ia},
\end{align*}
so $\gamma^{\ph}_{ia} = F_{ia} - 4 R_{ia}$ is symmetric. Lemma~\ref{lemma:trace-F} gives $\tr \gamma^{\ph} = -2 R - 4 R = - 6 R$, so the trace-free part is $(\gamma^{\ph}_{27})_{ij} = \gamma^{\ph}_{ij} - \frac{1}{7} (-6 R) g_{ij} = F_{ij} - 4 R_{ij} + \frac{6}{7} R g_{ij}$. Then Corollary~\ref{cor:diamondinverse} says that $\gamma = A \diamond \ph$ where
$$ \tr A = \tfrac{1}{18} \tr \gamma^{\ph} = - \tfrac{1}{3} R, \qquad (A_{27}) = \tfrac{1}{4} (\gamma^{\ph}_{27})_{ij} = \tfrac{1}{4} F_{ij} - R_{ij} + \tfrac{3}{14} R g_{ij}. $$
Thus we deduce that $A_{ij} = \frac{1}{7} (\tr A) g_{ij} + (A_{27})_{ij} = \frac{1}{6} R g_{ij} + \frac{1}{4} F_{ij} - R_{ij}$.
\end{proof}

\begin{prop} \label{prop:Hodge-Lap}
Let $\ph$ be a $\G$-structure on $M$. Its Hodge Laplacian $\Delta_{\dd} \ph$ is
\begin{equation} \label{eq:Hodge-Lap-ph}
\Delta_{\dd} \ph = \big( \tfrac{1}{3} (\Div T) \hk \ph + \tfrac{1}{3} |T|^2 g - T^t T + \tfrac{1}{6} R g + \tfrac{1}{4} F \big) \diamond \ph.
\end{equation}
More precisely, we have $\Delta_{\dd} \ph = A \diamond \ph$, where $A = A_1 + A_{27} + A_7 \in \cS^2 \oplus \Omega^2_7$, such that
\begin{equation} \label{eq:Hodge-Lap-ph-b}
A_1 = (\tfrac{4}{21} |T|^2 + \tfrac{2}{21} R) g, \qquad A_{27} = \tfrac{1}{7} |T|^2 g - T^t T + \tfrac{1}{4} F_{27}, \qquad A_7 = \tfrac{1}{3} (\Div T) \hk \ph,
\end{equation}
where $F_{27}$ is the trace-free part of $F$.
\end{prop}
\begin{proof}
Letting $\gamma = \ph$ in the Weitzenb\"ock formula~\eqref{eq:Weitzenbock-3}, we have
\begin{equation} \label{eq:Hodge-Lap-temp}
\begin{aligned}
(\Delta_{\dd} \ph)_{ijk} & = (\nabla^* \nabla \ph)_{ijk} + R_{im} \ph_{mjk} + R_{jm} \ph_{mki} + R_{km} \ph_{mij} \\
& \qquad {} + R_{pmjk} \ph_{pmi} + R_{pmki} \ph_{pmj} + R_{pmij} \ph_{pmk}.
\end{aligned}
\end{equation}
The Ricci curvature terms in~\eqref{eq:Hodge-Lap-temp} are precisely $(\tRc \diamond \ph)_{ijk}$, and Lemma~\ref{lemma:Hodge-F} shows that the Riemann curvature terms are $((\frac{1}{6} R g + \frac{1}{4} F - \tRc) \diamond \ph)_{ijk}$. Adding these together, the Ricci terms cancel, leaving us with
$$ (\Delta_{\dd} \ph)_{ijk} = (\nabla^* \nabla \ph)_{ijk} + ((\tfrac{1}{6} R g + \tfrac{1}{4} F) \diamond \ph)_{ijk}. $$
Comparing the above with~\eqref{eq:rough-Lap-ph} gives~\eqref{eq:Hodge-Lap-ph}. From Lemma~\ref{lemma:trace-F}, we have
$$ \tfrac{1}{4} F = \tfrac{1}{4} (\tfrac{1}{7} (\tr F) g + F_{27}) = - \tfrac{1}{14} R g + \tfrac{1}{4} F_{27}, $$
which, since $\frac{1}{6} - \frac{1}{14} = \frac{2}{21}$, yields~\eqref{eq:Hodge-Lap-ph-b}.
\end{proof}
(See Remark~\ref{rmk:Hodge-Lap-revisited} for a different form of~\eqref{eq:Hodge-Lap-ph} once we have shown that $F$ can be expressed in terms of scalar curvature, Ricci curvature, and another object $\varpi$ obtained from the Weyl curvature.)

Proposition~\ref{prop:Hodge-Lap} also has an interesting corollary which does not seem to be well-known.

\begin{cor} \label{cor:Hodge-Lap}
Let $\ph$ be a $\G$-structure on $M$. If $T = 0$, then $\Delta_{\dd} \ph = 0$. Conversely, suppose $\Delta_{\dd} \ph = 0$. If in addition we have $\Div (\Vop T) = 0$, then $\ph$ is necessarily torsion free. (Note that $\Delta_{\dd} \ph = 0$ implies $T = 0$ is immediate from integration by parts if $M$ is compact, but we \emph{do not} assume that $M$ is compact here.)
\end{cor}
\begin{proof}
Since $T = 0$ is equivalent to $\nabla \ph = 0$, any torsion-free $\G$-structure $\ph$ is always Hodge-harmonic. Conversely, suppose that $\Delta_{\dd} \ph = 0$. Then Proposition~\ref{prop:Hodge-Lap} gives $A_1 = 0$ where $\Delta_{\dd} \ph = A \diamond \ph$, so $\frac{4}{21} |T|^2 + \frac{2}{21} R = 0$. Substituting~\eqref{eq:scalar-curvature-c} we have
\begin{align*}
0 & = \tfrac{4}{21} \big( |T_1|^2 + |T_{27}|^2 + |T_7|^2 + |T_{14}|^2 \big) + \tfrac{2}{21} \big( 6|T_1|^2 - |T_{27}|^2 + 5 |T_7|^2 - |T_{14}|^2 - 2 \Div (\Vop T) \big) \\
& = \tfrac{1}{21} \big( 16 |T_1|^2 + 2 |T_{27}|^2 + 14 |T_7|^2 + 2 |T_{14}|^2 \big) - \tfrac{4}{21} \Div (\Vop T).
\end{align*}
Thus if $\Div (\Vop T) = 0$ then we must have $T = 0$.
\end{proof}

\begin{rmk} \label{rmk:Hodge-Lap}
Corollary~\ref{cor:Hodge-Lap} says the following. If $\Div (\Vop T) = 0$, then $\Delta_{\dd} \ph = 0$ if and only if $T = 0$. In particular, if $\ph$ is closed or coclosed, then $\Vop T = 0$, so closed or coclosed $\G$-structures are torsion-free if and only if they are Hodge-harmonic, \emph{irrespective of the compactness of $M$}. The closed case is implicit in~\cite[equation (2.20)]{LW1}. The coclosed case does not appear to have been observed before.
\end{rmk}

It is unknown whether $\Delta_{\dd} \ph = 0$ implies $T = 0$ in general (without assuming $M$ is compact). Note that in the proof of Corollary~\ref{cor:Hodge-Lap}, we only used that $\pi_1 (\Delta_{\dd} \ph) = 0$. It is in principle possible that also using $\pi_7 (\Delta_{\dd} \ph) = 0$ and $\pi_{27} (\Delta_{\dd} \ph) = 0$ may allow one to prove that $T = 0$, but if we do not allow integration by parts then this seems unlikely.

\subsection{Application: The optimal $\ph$-connection of a $\G$-structure $\ph$} \label{sec:connections}

In this section we use the torsion of $\ph$ to define the \emph{optimal $\ph$-connection} $\hnab{}$ of a $\G$-structure and then use this connection to give a new geometric interpretation of the $\G$-Bianchi identity. What we call the optimal $\ph$-connection is by many authors called the ``canonical connection'' and by other authors is sometimes called the ``natural connection''. However, there are other connection in $\G$-geometry that are sometimes called ``canonical''. (See Remark~\ref{rmk:skew-torsion-stuff}.) For this reason, and because of the inequality that appears in Definition~\ref{defn:optimal-connection} below, we prefer to use the term \emph{optimal $\ph$-connection}.

Some of the results in this section are well-known to experts working on metric-compatible connections with torsion. We include a detailed treatment here using the general computational machinery developed earlier in this section, for completeness.

We begin by recalling some basic facts, to fix notation. As usual, let $\nab{}$ denote the Levi-Civita connection of a metric $g$. Then any other connection $\tnab{}$ on the tangent bundle can be written as $\tnab{} = \nab{} + \sA$, where $\sA \in \cT^3$. Explicitly, in a local orthonormal frame we have
\begin{equation*}
\tnab{i} e_j = \nab{i} e_j + \sA_{ijk} e_k.
\end{equation*}
It follows that
\begin{equation} \label{eq:tnab-general}
\tnab{p} \alpha_{i_1 \cdots i_k} = \nab{p} \alpha_{i_1 \cdots i_k} - \sA_{pi_1m} \alpha_{mi_2 \cdots i_k} - \cdots - \sA_{pi_km} \alpha_{i_1 \cdots i_{k-1} m}
\end{equation}
for a $k$-tensor $\alpha_{i_1 \cdots i_k}$. Applying~\eqref{eq:tnab-general} to the metric, we get
\begin{equation*}
\tnab{p} g_{ij} = \nab{p} g_{ij} - \sA_{pim} g_{mj} - \sA_{pjm} g_{im} = 0 - \sA_{pij} - \sA_{pji},
\end{equation*}
and thus
\begin{equation} \label{eq:tnab-metric}
\text{$\tnab{}$ is metric compatible if and only if $\sA_{pij} = - \sA_{pji}$}.
\end{equation}
This of course just says that $\tnab{} = \nab{} + \sA$ is metric compatible if and only if the Lie algebra valued $1$-form $\sA = \sA_p e_p$ takes values in $\mathfrak{so}(n)$ where $n= \dim M$.

The \emph{torsion} of the connection $\tnab{}$ is a vector-valued $2$-form $\sT = \frac{1}{2} \sT_{pij} e_p \otimes (e_i \w e_j)$ given by
\begin{align*}
\sT_{pij} e_p & = \tnab{i} e_j - \tnab{j} e_i - [e_i, e_j] \\
& = \nab{i} e_j + \sA_{ijp} e_p - \nab{j} e_i - \sA_{jip} e_p - [e_i, e_j] \\
& = (A_{ijp} - A_{jip}) e_p.
\end{align*}
Thus we can express the torsion $\sT$ of $\tnab{} = \nab{} + \sA$ in terms of $\sA$ as
\begin{equation} \label{eq:sT}
\sT_{pij} = \sA_{ijp} - \sA_{jip}.
\end{equation}

\begin{lemma} \label{lemma:tnab-metric-comp}
Let $\tnab{} = \nab{} + \sA$ be a metric compatible connection. Then $\sA$ is completely determined by the torsion $\sT$ of $\tnab{}$ via
\begin{equation} \label{eq:sA-from-sT}
\sA_{ijk} = \tfrac{1}{2} ( \sT_{jki} + \sT_{kij} - \sT_{ijk} ).
\end{equation}
\end{lemma}
\begin{proof}
Using both~\eqref{eq:sT} and~\eqref{eq:tnab-metric}, we have
\begin{align*}
\sT_{ijk} & = \sA_{jki} - \sA_{kji} = \sA_{jki} + \sA_{kij}, \\
\sT_{kij} & = \sA_{ijk} - \sA_{jik} = \sA_{ijk} + \sA_{jki}, \\
\sT_{jki} & = \sA_{kij} - \sA_{ikj} = \sA_{kij} + \sA_{ijk}.
\end{align*}
Adding the last two equations and subtracting the first gives precisely~\eqref{eq:sA-from-sT}.
\end{proof}

\begin{rmk} \label{rmk:LC}
Equations~\eqref{eq:sA-from-sT} and~\eqref{eq:sT} show that, in the metric compatible case, $\sA = 0$ if and only if $\sT = 0$. That is, a metric compatible connection $\tnab{} = \nab{} + \sA$ is torsion-free if and only if $\sA = 0$, so $\tnab{} = \nab{}$ is the Levi-Civita connection. Of course this is just the well-known ``fundamental theorem of Riemannian geometry''. Compare with Definition~\ref{defn:optimal-connection}.
\end{rmk}

\begin{cor} \label{cor:skew-torsion}
A metric compatible connection $\tnab{} = \nab{} + \sA$ has totally skew-symmetric torsion if and only if $\sA = \frac{1}{2} \sT$.
\end{cor}
\begin{proof}
Suppose $\sT_{ijk}$ is totally skew-symmetric. Then~\eqref{eq:sA-from-sT} show that $\sA_{ijk} = \frac{1}{2} (\sT_{ijk} + \sT_{ijk} - \sT_{ijk}) = \frac{1}{2} \sT_{ijk}$. Conversely, since $\sA_{ijk}$ is skew in $j,k$ and $\sT_{ijk}$ is skew in $i,j$ it follows that if $\sA = \frac{1}{2} \sT$ then $\sT$ is totally skew-symmetric.
\end{proof}

\begin{defn} \label{defn:ph-connection}
Let $\ph$ be a $\G$-structure on $M$, and let $\nab{}$ be the Levi-Civita connection of $g = g_{\ph}$. We say that a connection $\tnab{} = \nab{} + \sA$ is compatible with $\ph$ if $\tnab{} \ph = 0$. This means that parallel transport with respect to $\tnab{}$ preserves $\ph$, or equivalently that the \emph{restricted holonomy} of $\tnab{}$ is contained in $\G$. We also say that such a connection is a \emph{$\ph$-connection}.
\end{defn}

\begin{prop} \label{prop:ph-connection}
The connection $\tnab{} = \nab{} + \sA$ is compatible with $\ph$ if and only if $\sA = - \frac{1}{3} \hT + \sB$ where $\hT$ is the alternative characterization of the torsion of the $\G$-structure $\ph$ given in~\eqref{eq:hatTdefn} and $\mathsf{B}$ is a smooth section of $T^* M \otimes \Lambda^2_{14} (T^* M)$. That is, by~\eqref{eq:omega214desc}, we have $\sB_{pij} \ph_{ijk} = 0$.
\end{prop}
\begin{proof}
Fix $p \in \{ 1, \ldots, 7 \}$. Writing $\sA_p = \sA_{pij} e_i \otimes e_j$ and $\hT_p = \hT_{pij} e_i \otimes e_j$, equations~\eqref{eq:nablaph-hT} and~\eqref{eq:tnab-general} give
\begin{equation*}
\tnab{p} \ph_{abc} = - \tfrac{1}{3} (\hT_p \diamond \ph)_{abc} - (\sA_p \diamond \ph)_{abc}.
\end{equation*}
Letting $\sB = \frac{1}{3} \hT + \sA$ and $\sB_p = \sB_{pij} e_i \otimes e_j$, the above expression shows that $\tnab{} \ph = 0$ if and only if $\sB_p \diamond \ph = 0$ for all $p$. The claim now follows by Corollary~\ref{cor:ABinnerproduct}.
\end{proof}

\begin{thm} \label{thm:optimal-connection}
Let $\ph$ be a $\G$-structure, with induced metric $g$ and Levi-Civita connection $\nab{}$. There exists a \emph{unique} $\ph$-connection $\hnab{} = \nab{} + \sA$ such that $\hnab{} \ph = 0$ and $\sA \in \Gamma( T^* M \otimes \Lambda^2_7 (T^* M) )$, given by $\sA = - \frac{1}{3} \hT$.
\end{thm}
\begin{proof}
This follows immediately from Proposition~\ref{prop:ph-connection}, since $\sB \in \Gamma( T^* M \otimes \Lambda^2_{14} (T^* M) )$.
\end{proof}

\begin{defn} \label{defn:optimal-connection}
We call the $\ph$-connection $\hnab{}$ given by Theorem~\ref{thm:optimal-connection} the \emph{optimal $\ph$-connection} of the $\G$-structure $\ph$. Note that, for $\tnab{} = \nab{} + \sA = - \tfrac{1}{3} \hT + \sB$ as in Proposition~\ref{prop:ph-connection}, we have
$$ |\tnab{} - \nab{}|^2 = |\sA|^2 = \tfrac{1}{9} |\hT|^2 + |\sB|^2, $$
since $\Lambda^2_7$ and $\Lambda^2_{14}$ are orthogonal summands in $\Lambda^2 (T^* M)$. Thus we have
$$ |\tnab{} - \nab{}|^2 \geq |\hnab{} - \nab{}|^2 \quad \text{for all $\ph$-connections $\tnab{}$, with equality if and only if $\tnab{} = \hnab{}$.} $$
This is why we call $\hnab{}$ the \emph{optimal} $\ph$-connection. From Remark~\ref{rmk:LC} and the fact that $\sA = - \tfrac{1}{3} \hT$ for $\hnab{}$, we deduce that
\begin{align*}
\hnab{} = \nab{} & \iff \text{the $\G$-structure $\ph$ is torsion-free} \\
& \iff \text{the connection $\hnab{}$ is torsion-free.}
\end{align*}
The above characterization justifies the use of the word ``torsion'' for two different things: the torsion $\hT$ of the $\G$-structure $\ph$, and the torsion of the optimal $\ph$-connection $\hnab{}$.
\end{defn}

\begin{rmk} \label{rmk:skew-torsion-stuff}
Friedrich--Ivanov~\cite[Theorem 4.7]{FI} prove that there exists a unique $\ph$-connection \emph{with skew torsion} if and only if $T_{14} = 0$. Comparing with our Proposition~\ref{prop:ph-connection} shows that in this case $\sB$ must be uniquely determined by $\hT$. There is an extensive literature on metric compatible tangent bundle connections with totally skew torsion. An excellent survey is Agricola~\cite{Agricola}. There is also a brief abstract discussion of ``canonical'' connections for $\G$-structures in Bryant~\cite[Remark 7]{Bryant}.
\end{rmk}

In Section~\ref{sec:future}, we discuss how the curvature of $\hnab{}$ can be used to define some flows of $\G$-structures.

\begin{rmk} \label{rmk:g2bianchi-optimal}
It is easy to compute that the curvature tensor $\wh{R}_{ijkl}$ of $\hnab{} = \nab{} + \sA$ is
$$ \wh{R}_{ijkl} = R_{ijkl} + \nab{i} \sA_{jkl} - \nab{j} \sA_{ikl} + \sA_{iml} \sA_{jkm} - \sA_{jml} \sA_{ikm}. $$
Substituting from Theorem~\ref{thm:optimal-connection} that $\sA_{ijk} = - \frac{1}{3} \hT_{ijk} = - \frac{1}{3} T_{ip} \ph_{pjk}$, a computation reveals that the $\G$-Bianchi identity~\eqref{eq:g2bianchi} is precisely equivalent to the fact that $\wh{R}_{ijkl}$ is in $\Omega^2_{14}$, thought of as a $2$-form in the skew-symmetric indices $k, l$. That is, the $\G$-Bianchi identity is precisely the statement that, as an $\mathfrak{so}(7)$-valued $2$-form on $M$, the curvature of $\hnab{}$ actually lies in the subalgebra of $\lieg$-valued $2$-forms. This is of course expected by the Ambrose--Singer holonomy theorem, since $\hnab{} \ph = 0$ and $\G$ is the group preserving $\ph$.
\end{rmk}

\subsection{Scaling of $\G$-structures} \label{sec:scaling}

In this section we carefully discuss the effect of \emph{scaling} on tensors induced from a $\G$-structure. The motivation for this is that we seek to understand which tensors scale the right way to be considered as the right-hand side for a geometric flow
\begin{equation} \label{eq:flow}
\delt \ph = \gamma_{\ph}
\end{equation}
of $\G$-structures. The right-hand side of~\eqref{eq:flow} is some $3$-form $\gamma_{\ph}$ which should depend on second-order derivatives of $\ph$. In this section we need to sometimes be careful about our subscript/superscript abuse of notation.

Let $\lambda \in \R$ be positive. Then $\wt{\ph} = \lambda^3 \ph$ is a $\G$-structure. From~\eqref{eq:nondegenerate}, it follows easily that
\begin{equation} \label{eq:scaling-general}
\wt{g} = \lambda^2 g, \qquad \wt{g}^{-1} = \lambda^{-2} g^{-1}, \qquad \wt{\ps} = \lambda^4 \ps, \qquad \wt{\vol} = \lambda^7 \vol.
\end{equation}
The fact that $\wt{g} = \lambda^2 g$ says that we are scaling each ``space'' coordinate by $\wt{x} = \lambda x$. Parabolic theory says that the ``time'' coordinate should scale by $\wt{t} = \lambda^2 t$. Thus a flow of $\G$-structures scales by
\begin{equation} \label{eq:time-scale}
\frac{\partial}{\partial \wt{t}} \, \wt{\ph} = \frac{\lambda^3}{\lambda^2} \delt \ph = \lambda \delt \ph.
\end{equation}
Thus, to agree with the scaling of~\eqref{eq:time-scale} we need
\begin{equation} \label{eq:gamma-ph-scaling}
\gamma_{\wt{\ph}} = \lambda \gamma_{\ph} \quad \text{when} \quad \wt{\ph} = \lambda^3 \ph.
\end{equation}
We know by Section~\ref{sec:forms} that $\gamma = A \diamond \ph$ for some $A \in \cT^2$. Suppose that $\wt{A} = \lambda^s A$. Thus we have
\begin{align*}
(\wt{A} \,\, \wt{\diamond} \,\, \wt{\ph})_{ijk} & = \wt{A}_{ip} \, \wt{g}^{pq} \, \wt{\ph}_{qjk} + \wt{A}_{jp} \, \wt{g}^{pq} \, \wt{\ph}_{iqk} + \wt{A}_{kp} \, \wt{g}^{pq} \, \wt{\ph}_{ijq} \\
& = \lambda^s \lambda^{-2} \lambda^3 (A_{ip} g^{pq} \ph_{qjk} + A_{jp} g^{pq} \ph_{iqk} + A_{kp} g^{pq} \ph_{ijq}) \\
& = \lambda^{1+s} (A \diamond \ph)_{ijk}.
\end{align*}
To agree with the scaling of~\eqref{eq:gamma-ph-scaling}, we need $1+s=1$, so $s=0$. That is, the elements $A \in \cT^2$ that can be taken to give the right-hand side $A \diamond \ph$ of a geometric flow of $\G$-structures must depend on $\ph$ in such a way that
\begin{equation} \label{eq:A-ph-scaling}
\wt{A} = A \quad \text{when} \quad \wt{\ph} = \lambda^3 \ph.
\end{equation}

We now study how the various tensors determined by $\ph$ transform under $\wt{\ph} = \lambda^3 \ph$. Because the Christoffel symbols $\Gamma^k_{ij}$ are invariant under constant scaling of the metric, we have $\wt{\nabla} = \nabla$. It is easy to check that the Riemann, Ricci, scalar curvature, and the tensor $F$ of~\eqref{eq:F-defn} for $\wt{\ph} = \lambda^3 \ph$ are
\begin{equation} \label{eq:scaling-curvature}
\wt{R}_{ijkl} = \lambda^2 R_{ijkl}, \qquad \wt{R}_{jk} = R_{jk}, \qquad \wt{R} = \lambda^{-2} R, \qquad \wt{F}_{jk} = F_{jk}.
\end{equation}
In particular, the tensors $\tRc$, $R g$, and $F$ satisfy~\eqref{eq:A-ph-scaling} so that they have the correct scaling to be the right-hand side $\gamma_{\ph} = A \diamond \ph$ for a geometric flow~\eqref{eq:flow} of $\G$-structures.

Next consider the torsion $T$. From~\eqref{eq:Tfromph} we have
\begin{align*}
\wt{T}_{pq} & = \tfrac{1}{24} \wt{\nabla}_p \wt{\ph}_{jkl} \, \wt{\ps}_{qabc} \, \wt{g}^{ja} \, \wt{g}^{kb} \, \wt{g}^{lc} \\
& = \tfrac{1}{24} \nab{p} (\lambda^3 \ph_{jkl}) (\lambda^4 \ps_{qabc}) (\lambda^{-2} g^{ja}) (\lambda^{-2} g^{kb}) (\lambda^{-2} g^{lc}) \\
& = \lambda \tfrac{1}{24} \nab{p} \ph_{jkl} \ps_{qabc} g^{ja} g^{kb} g^{lc}.
\end{align*}
Thus we find that
\begin{equation} \label{eq:scaling-torsion}
\wt{T}_{pq} = \lambda T_{pq}.
\end{equation}
Note that if we decompose $T = T_1 + T_{27} + T_7 + T_{14}$, then we have
\begin{equation} \label{eq:scaling-torsion-comps}
\wt{T}_1 = \lambda T_1, \qquad \wt{T}_{27} = \lambda T_{27}, \qquad \wt{T}_7 = \lambda T_7, \qquad \wt{T}_{14} = \lambda T_{14}.
\end{equation}

Let $A, B \in \cT^2$ and consider the composition product $AB \in \cT^2$ where $(AB)_{ij} = A_{ip} g^{pq} A_{qj}$. Then
\begin{equation*}
(\wt{AB})_{ij} = A_{ip} \wt{g}^{pq} B_{qj} = \lambda^{-2} \wt{A}_{ip} g^{pq} \wt{B}_{qj} = \lambda^{-2} (\wt{A} \wt{B})_{ij}.
\end{equation*}
Hence, for example,
\begin{equation} \label{eq:scaling-AB}
\text{if $\wt{A} = \lambda A$ and $\wt{B} = \lambda B$, then $\wt{AB} = AB$}.
\end{equation}
In particular, we see from~\eqref{eq:scaling-curvature},~\eqref{eq:scaling-torsion-comps}, and~\eqref{eq:scaling-AB} that the tensors $T_k T_{k'}$ for $k, k' \in \{ 1, 27, 7, 14 \}$ satisfy~\eqref{eq:A-ph-scaling} so that they have the correct scaling to be the right-hand side $\gamma_{\ph} = A \diamond \ph$ for a geometric flow~\eqref{eq:flow} of $\G$-structures.

Finally, consider $\nabla T$, the covariant derivative of the torsion. We have
\begin{equation} \label{eq:scaling-nabla-T}
\wt{\nabla}_i \wt{T}_{pq} = \lambda \nab{i} T_{pq}.
\end{equation}
It then follows from~\eqref{eq:scaling-nabla-T} and~\eqref{eq:K-defn} that
\begin{equation} \label{eq:scaling-KKa}
\wt{\KK{1}}_{ab} = \KK{1}_{ab}, \qquad \wt{\KK{2}}_{ab} = \KK{2}_{ab}, \qquad \wt{\KK{3}}_{ab} = \KK{3}_{ab}.
\end{equation}
In particular, we see from~\eqref{eq:scaling-KKa} that the tensors $\KK{1}$, $\KK{2}$, and $\KK{3}$ satisfy~\eqref{eq:A-ph-scaling} so that they have the correct scaling to be the right-hand side $\gamma_{\ph} = A \diamond \ph$ for a geometric flow~\eqref{eq:flow} of $\G$-structures.

Finally, we note the effect of scaling on the norms of various quantities. It follows from~\eqref{eq:scaling-curvature} that
\begin{equation} \label{eq:scaling-curvature-norms}
|\wt{\tRm}|^2_{\wt{g}} = \lambda^{-4} |\tRm|^2_{g}, \qquad |\wt{\mathrm{Rc}}|^2_{\wt{g}} = \lambda^{-4} |\mathrm{Rc}|^2_{g}, \qquad \wt{R}^2 = \lambda^{-4} R^2, \qquad |\wt{F}|^2_{\wt{g}} = \lambda^{-4} |F|^2_{g},
\end{equation}
and it follows from~\eqref{eq:scaling-torsion} and~\eqref{eq:scaling-torsion-comps} that
\begin{equation} \label{eq:scaling-torsion-norms}
|\wt{T}|^2_{\wt{g}} = \lambda^{-2} |T|^2_{g}, \qquad |\wt{T}_k|^2_{\wt{g}} = \lambda^{-2} |T_k|^2_{g}, \quad \text{ for $k \in \{1, 27, 7, 14 \}$}.
\end{equation}
Finally it follows from~\eqref{eq:scaling-nabla-T} and~\eqref{eq:scaling-KKa} that
\begin{equation} \label{eq:scaling-nabla-T-norm}
|\wt{\nabla} \wt{T}|^2_{\wt{g}} = \lambda^{-4} |\nabla T|^2_{g}, \qquad |\wt{\KK{a}}|^2_{\wt{g}} = \lambda^{-4} |\KK{a}|^2_g, \quad \text{for $a = 1,2,3$.}
\end{equation}
Comparing~\eqref{eq:scaling-curvature-norms},~\eqref{eq:scaling-torsion-norms}, and~\eqref{eq:scaling-nabla-T-norm} shows that the \emph{norms} of tensors derived from the curvature $\tRm$, and the norms of tensors derived from the covariant derivative of the torsion $\nabla T$, scale in the same way. And by~\eqref{eq:scaling-torsion-norms}, these also scale in the same way as the norms of ``squares'' of the torsion $T$. Thus it makes sense to consider all such quantities on an equal footing.

For example, Lotay--Wei~\cite{LW1} define a quantity $\Lambda = (| \tRm |^2 + | \nab{} T|^2)^{\frac{1}{2}}$ which controls the existence of the Laplacian flow, in the sense that $\Lambda$ blows up at the singular time. Control of $\Lambda$ gives control of all possible components of $\tRm$ and of $\nabla T$. For a general geometric flow of $\G$-structures, one might need to instead consider the more general expression $(| \tRm |^2 + | \nab{} T|^2 + |T|^4)^{\frac{1}{2}}$ due to the scaling
considerations described above. This is done by Chen in~\cite{Chen}.

\subsection{Application: Conformal change of $\G$-structures} \label{sec:conformal}

In this section we examine the effect of conformal change of $\G$-structures on torsion. This is treated in a coordinate-free manner in terms of the \emph{torsion forms} $\tau_1$, $\tau_{27}$, $\tau_7$, and $\tau_{14}$ in~\cite[Section 3.1]{K1}. (See also Fern\'andez--Gray~\cite[Section 6]{FG}.) Here we derive the same results in terms of a local orthonormal frame, as an application of the methods of computation introduced in Sections~\ref{sec:forms} and~\ref{sec:torsion}. In this section again we are careful about our subscripts and superscripts.

\begin{defn} \label{defn:conformal}
Let $\ph$ be a $\G$-structure on $M$. Let $f \in C^{\infty} (M)$ and define $\wt{\ph} = e^{3f} \ph$. Then $\wt{\ph}$ is a $\G$-structure \emph{conformal} to $\ph$.
\end{defn}

It follows from~\eqref{eq:scaling-general}, since $g$, $\ps$, and $\vol$ depend only pointwise on $\ph$, that
\begin{equation} \label{eq:conformal-general}
\wt{g} = e^{2f} g, \qquad \wt{g}^{-1} = e^{-2 f} g^{-1}, \qquad \wt{\ps} = e^{4f} \ps, \qquad \wt{\vol} = e^{7f} \vol.
\end{equation}
It is well-known (and can be verified easily) that if $\wt{g} = e^{2f} g$, then the Christoffel symbols $\wt{\Gamma}^k_{ij}$ of $\wt{g}$ are related to the Christoffel symbols $\Gamma^k_{ij}$ of $g$ by
\begin{equation} \label{eq:conformal-Christoffel}
\wt{\Gamma}^k_{ij} = \Gamma^k_{ij} + \delta^k_i \nab{j} f + \delta^k_j \nab{i} f - \nab{l} f g^{lk} g_{ij}.
\end{equation}

\begin{prop} \label{prop:conformal-torsion}
Let $\wt{T}$ be the torsion tensor of $\wt{\ph} = e^{3f} \ph$, and let $T$ be the torsion tensor of $\ph$. Then we have
\begin{equation} \label{eq:conformal-torsion}
\wt{T}_{pq} = e^f (T_{pq} + \nab{m} f \, \ph_{mpq} ).
\end{equation}
Consequently, the components $\wt{T}_k$ of $\wt{T}$ are
\begin{equation} \label{eq:conformal-torsion-comps}
\wt{T}_{1} = e^f T_1, \qquad \wt{T}_{27} = e^f T_{27}, \qquad \wt{T}_7 = e^f (T_7 + \nabla f \hk \ph), \qquad \wt{T}_{14} = e^f T_{14}.
\end{equation}
\end{prop}
\begin{proof}
Let $\wt{\nabla}$ be the Levi-Civita connection of $\wt{g}$. We first observe since $\wt{\nabla}_p f = \nab{p} f$ that
\begin{align*}
\wt{\nabla}_p \wt{\ph}_{ijk} & = \wt{\nabla}_p (e^{3f} \ph_{ijk}) = 3 e^{3f} \wt{\nabla}_p f \ph_{ijk} + e^{3f} \wt{\nabla}_p \ph_{ijk} \\
& = 3 \nab{p} f \wt{\ph}_{ijk} + e^{3f} \wt{\nabla}_p \ph_{ijk}.
\end{align*}
Using the above expression together with~\eqref{eq:Tfromph} for $\wt{\ph}$, we compute
\begin{align*}
24 \, \wt{T}_{pq} & = \wt{\nabla}_p \wt{\ph}_{ijk} \wt{\ps}_{qabc} \wt{g}^{ia} \wt{g}^{jb} \wt{g}^{kc} \\
& = (3 \nab{p} f \wt{\ph}_{ijk} + e^{3f} \wt{\nabla}_p \ph_{ijk}) \wt{\ps}_{qabc} \wt{g}^{ia} \wt{g}^{jb} \wt{g}^{kc}.
\end{align*}
The first term vanishes by~\eqref{eq:phps} applied to $\wt{\ph}$. Now we use~\eqref{eq:conformal-general} and revert to our usual abuse of notation to obtain
\begin{equation} \label{eq:conformal-temp}
\begin{aligned}
24 \, \wt{T}_{pq} & = e^{3f} \wt{\nabla}_p \ph_{ijk} e^{4f} \ps_{qabc} e^{-2f} g^{ia} e^{-2f} g^{jb} e^{-2f} g^{kc} \\
& = e^f \wt{\nabla}_p \ph_{ijk} \ps_{qijk}.
\end{aligned}
\end{equation}
We observe that
\begin{equation*}
\wt{\nabla}_p \ph_{ijk} = \nab{p} \ph_{ijk} - (\wt{\Gamma}^l_{pi} - \Gamma^l_{pi}) \ph_{ljk} - (\wt{\Gamma}^l_{pj} - \Gamma^l_{pj}) \ph_{ilk} - (\wt{\Gamma}^l_{pk} - \Gamma^l_{pk}) \ph_{ijl}.
\end{equation*}
The last three terms above combined are skew-symmetric in $i,j,k$. Therefore substituting into~\eqref{eq:conformal-temp} yields
\begin{equation*}
24 \, \wt{T}_{pq} = e^f ( \nab{p} \ph_{ijk} - 3 (\wt{\Gamma}^l_{pi} - \Gamma^l_{pi}) \ph_{ljk} ) \ps_{qijk}.
\end{equation*}
Using~\eqref{eq:Tfromph} for $\ph$, together with~\eqref{eq:conformal-Christoffel}, we compute
\begin{align*}
24 \, \wt{T}_{pq} & = e^f \big( 24 \, T_{pq} - 3 ( \delta^l_p \nab{i} f + \delta^l_i \nab{p} f - \nab{m} f g^{ml} g_{pi} ) (-4 \ph_{lqi}) \big) \\
& = 24 e^f T_{pq} + 12 e^f ( \nab{i} f \ph_{pqi} + \nab{p} f \ph_{i q i} - \nab{m} f \ph_{mqp}) \\
& = 24 e^f T_{pq} + 12 e^f \nab{m} f \ph_{mpq} + 0 + 12 e^f \nab{m} f \ph_{mpq},
\end{align*}
which establishes~\eqref{eq:conformal-torsion}. The equations in~\eqref{eq:conformal-torsion-comps} now follow immediately from~\eqref{eq:omega27desc}.
\end{proof}

\begin{rmk} \label{rmk:conformal}
Proposition~\ref{prop:conformal-torsion} shows that $T_1$, $T_{27}$, and $T_{14}$ change by $e^f$ when $\ph$ is changed by $e^{3f}$. Thus the vanishing or nonvanishing of these three components of the torsion depends only on the \emph{conformal class} of the $\G$-structure $\ph$. By contrast, the component $T_7$ of the torsion transforms under conformal change of $\ph$ in a slightly more complicated way. See~\cite[Theorem 2.32 and its succeeding paragraph]{K-flows} and~\cite[Section 3.1]{K1} for more discussion.
\end{rmk}

We can similarly compute the behaviour of $\nab{} T$ under a conformal change.
\begin{prop} \label{prop:conformal-nabT}
Let $\wt{T}$ be the torsion tensor of $\wt{\ph} = e^{3f} \ph$, and let $\tnab{}$ be the Levi-Civita connection of the induced metric $\wt{g} = e^{2f} g$, with $T$, $\nab{}$ the analogous objects from $\ph$. Then we have
\begin{equation} \label{eq:conformal-nabT}
\begin{aligned}
e^{-f} \tnab{i} \wt{T}_{pq} & = \nab{i} T_{pq} - \nab{i} f \, T_{pq} - \nab{p} f \, T_{iq} - \nab{q} f \, T_{pi} \\
& \qquad {} + g_{ip} \nab{m} f \, T_{mq} + g_{iq} \nab{m} f \, T_{pm} + \nab{m} f \, T_{ik} \ps_{kmpq} \\
& \qquad {} - \nab{i} f \nab{m} f \, \ph_{mpq} - \nab{p} f \, \nab{k} f \, \ph_{kiq} - \nab{q} f \nab{k} f \, \ph_{kpi} + \nab{i} \nab{m} f \, \ph_{mpq}.
\end{aligned}
\end{equation}
\end{prop}
\begin{proof}
Using~\eqref{eq:conformal-Christoffel}, we compute
\begin{align} \nonumber
\tnab{i} \wt{T}_{pq} & = \nab{i} \wt{T}_{pq} - (\wt{\Gamma}^m_{ip} - \Gamma^m_{ip}) \wt{T}_{mq} - (\wt{\Gamma}^m_{iq} - \Gamma^m_{iq}) \wt{T}_{pm} \\ \nonumber
& = \nab{i} \wt{T}_{pq} - (\delta^m_i \nab{p} f + \delta^m_p \nab{i} f - \nab{l} f g^{lm} g_{ip}) \wt{T}_{mq} - (\delta^m_i \nab{q} f + \delta^m_q \nab{i} f - \nab{l} f g^{lm} g_{iq}) \wt{T}_{pm} \\ \label{eq:conformal-nabT-temp}
& = \nab{i} \wt{T}_{pq} - \nab{p} f \, \wt{T}_{iq} - \nab{q} f \, \wt{T}_{pi} - 2 \nab{i} f \, \wt{T}_{pq} + g_{ip} \nab{m} f \, \wt{T}_{mq} + g_{iq} \nab{m} f \, \wt{T}_{pm}.
\end{align}
Differentiating~\eqref{eq:conformal-torsion} and using~\eqref{eq:nablaph}, we obtain
\begin{align*}
\nab{i} \wt{T}_{pq} & = e^f (\nab{i} f) (T_{pq} + \nab{m} f \, \ph_{mpq}) + e^f (\nab{i} T_{pq} + \nab{i} \nab{m} f \, \ph_{mpq} + \nab{m} f \nab{i} \ph_{mpq}) \\
& = e^f ( \nab{i} T_{pq} + \nab{i} f \, T_{pq} + \nab{m} f \, T_{ik} \ps_{kmpq} + \nab{i} f \nab{m} f \, \ph_{mpq} + \nab{i} \nab{m} f \, \ph_{mpq}).
\end{align*}
Substituting the above and~\eqref{eq:conformal-torsion} into~\eqref{eq:conformal-nabT-temp}, we have
\begin{align*}
e^{-f} \tnab{i} \wt{T}_{pq} & = \nab{i} T_{pq} + \nab{i} f \, T_{pq} + \nab{m} f \, T_{ik} \ps_{kmpq} + \nab{i} f \nab{m} f \, \ph_{mpq} + \nab{i} \nab{m} f \, \ph_{mpq} \\
& \qquad {} - \nab{p} f (T_{iq} + \nab{k} f \, \ph_{kiq}) - \nab{q} f (T_{pi} + \nab{k} f \, \ph_{kpi}) - 2 \nab{i} f (T_{pq} + \nab{k} f \, \ph_{kpq}) \\
& \qquad {} + g_{ip} \nab{m} f (T_{mq} + \nab{k} f \, \ph_{kmq}) + g_{iq} \nab{m} f (T_{pm} + \nab{k} f \, \ph_{kpm}).
\end{align*}
The second and fourth terms in the last line above vanish by skew-symmetry of $\ph$. Combining terms and rearranging, we obtain~\eqref{eq:conformal-nabT}.
\end{proof}

\section{Evolution of $\G$-structures} \label{sec:evolution-g2}

In this section, we consider the evolution of $\G$-structures. We also study the evolution of certain natural functionals which are quadratic in the torsion. This investigation serves to motivate the necessity of a detailed analysis of the decompositions of $\tRm$ and $\nab{} T$ into irreducible $\G$-representations, which we undertake in Sections~\ref{sec:more-rep-theory} and~\ref{sec:curvature-torsion}. We then revisit these quadratic torsion functionals in Section~\ref{sec:functionals-revisited}.

\subsection{Basic evolution equations for a flow of $\G$-structures} \label{sec:basic-flow}

In~\cite{K-flows}, the third author initiated the study of general flows of $\G$-structures. Explicitly, a general flow of $\G$-structures can be written in the form
\begin{equation} \label{eq:general-flow}
\delt \ph = h \diamond \ph + X \hk \ps
\end{equation}
for some time-dependent symmetric $2$-tensor $h$ and vector field $X$. (Note that the $\diamond$ operation defined in~\eqref{eq:diamond-coords} depends on the metric and hence on the $\G$-structure $\ph$.) Given $h$ and $X$, the evolutions of the metric $g$, the $4$-form $\ps$, the torsion $T$, and the independent components of the torsion, were computed in~\cite{K-flows}. In this section we give a much more efficient derivation of all these formulas. 

The key point is that it is more convenient to package the data of $h$ and $X$ together as follows. We can write
\begin{equation} \label{eq:general-flow-A}
\delt \ph = A \diamond \ph
\end{equation}
for a unique $A \in \cS \oplus \Omega^2_7$, where the symmetric part is $A_{1+27} = h$, and by~\eqref{eq:general-flow} and Corollary~\ref{cor:sevenreps}, the $\Omega^2_7$ part is $A_7 = - \frac{1}{3} X \hk \ph$.

\begin{rmk} \label{rmk:advantage-A-formulation}
There are two advantages of the $\delt \ph = A \diamond \ph$ formulation of a general flow as opposed to the original form~\eqref{eq:general-flow}. The first, and most direct advantage for our purposes, is that the derivation of the evolution equations of the metric, the $4$-form, the torsion, and the components of the torsion are much more efficient and the resulting formulas are significantly simpler. See Remarks~\ref{rmk:flows1-evolution-4-form} and~\ref{rmk:flows1-evolution-torsion}. Another advantage is that the $\delt \ph = A \diamond \ph$ approach carries over directly to flows of $\Spin{7}$-structures, while the original formulation~\eqref{eq:general-flow} does not. (See~\cite{K-flows-SP} for more about flows of $\Spin{7}$-structures.) In fact, this approach is amenable to the study of flows of a very broad class of geometric structures. (For example, see~\cite{D-Spin7, DLS, FLMS, LS}.)

There is, however, one advantage of the original formulation~\eqref{eq:general-flow}, in that a symmetric $2$-tensor $h$ and a vector field $X$ make sense independently of any $\G$-structure, whereas encoding the flow by a section $A \in \cS \oplus \Omega^2_7$ depends on a $\G$-structure $\ph$. However, we could consider $A$ to be a general $2$-tensor $A \in \cT^2$, because for any $\ph$, the component of $A$ in $\Omega^2_{14}$ does not contribute to $A \diamond \ph$, by Corollary~\ref{cor:ABinnerproduct}. It is only if we want $\delt \ph$ to determine $A$ uniquely that we need to project the skew-symmetric part onto $\Omega^2_7$.
\end{rmk}

\begin{lemma} \label{lemma:evolution-1}
Let $\ph$ be a time-dependent family of $\G$-structures evolving by the flow~\eqref{eq:general-flow-A}. Then the metric $g$, the $4$-form $\ps$, and the volume form $\vol$ evolve by
\begin{equation} \label{eq:evolution-1}
\delt g = (A + A^t) = 2 h, \qquad \delt \ps = A \diamond \ps, \qquad \delt \vol = (\tr A) \vol.
\end{equation}
\end{lemma}
\begin{proof}
Since $g$ and $\ps$ are nonlinear functions of $\ph$ and $\delt \ph = A \diamond \ph$, we have
$$ \delt g = g_* (A \diamond \ph) \qquad \text{and} \qquad \delt \ps = \ps_* (A \diamond \ph), $$
where $g_*$ and $\ps_*$ are the pushforwards (differentials) of the smooth maps that take a $\G$-structure to its metric and $4$-form, respectively. Thus we seek the \emph{first variation} of $g$ and $\ps$ as nonlinear functions of $\ph$. This is purely a pointwise calculation.

Fix a $\G$-structure $\ph$. Let $\ph_s$ be any $1$-parameter family of $\G$-structures such that $\ph_0 = \ph$ and $\rest{\dds}{s=0} \ph_s = A \diamond \ph$. Let $g_s$ and $\ps_s$ be the induced metric and $4$-form of $\ph_s$, respectively. We need to compute $\rest{\dds}{s=0} g_s$ and $\rest{\dds}{s=0} \ps_s$. We can choose $\ph_s = (e^{sA})^* \ph$. That is, if $\ph = \frac{1}{6} \ph_{ijk} e_i \w e_j \w e_k$, then for small $s$, the $3$-form
\begin{equation} \label{eq:evolution-temp}
\ph_s = (e^{sA})^* \ph = \frac{1}{6} \ph_{ijk} (e^{sA} e_i) \w (e^{sA} e_j) \w (e^{sA} e_k)
\end{equation}
is a $\G$-structure, with $\ph_0 = \ph$ and
\begin{align*}
\rest{\dds}{s=0} \ph_s & = \tfrac{1}{6} \ph_{ijk} ( A_{pi} e_p \w e_j \w e_k + A_{pj} e_i \w e_p \w e_k + A_{pk} e_i \w e_j \w e_p ) \\
& = \tfrac{1}{6} (A_{ip} \ph_{pjk} + A_{jp} \ph_{ipk} + A_{kp} \ph_{ijp}) e_i \w e_j \w e_k \\
& = A \diamond \ph.
\end{align*}
It then follows immediately from~\eqref{eq:evolution-temp} that
\begin{equation} \label{eq:evolution-temp2}
g_s = (e^{sA})^* g, \qquad \ps_s = (e^{sA})^* \ps.
\end{equation}
The second equation in~\eqref{eq:evolution-temp2} yields $\rest{\dds}{s=0} \ps_s = A \diamond \ps$ exactly as in the case of $\ph_s$. The first equation in~\eqref{eq:evolution-temp2} says
$$ g_s = g_{ij} (e^{sA} e_i) (e^{sA} e_j), $$
and thus
$$ \rest{\dds}{s=0} g_s = g_{ij} (A_{pi} e_p e_j + A_{pj} e_i e_p) = (A_{ip} g_{pj} + A_{jp} g_{ip}) e_i e_j. $$
Since we are using an orthonormal frame with respect to $g$, the right-hand side above is $A_{ij} + A_{ji} = 2 h_{ij}$, as claimed. [This argument for the evolution of the metric is greatly simplified from the original argument in~\cite{K-flows}. It is the same as the argument in~\cite[Proposition 3.1]{K-flows-SP} for flows of $\Spin{7}$-structures.]

The fact that $\delt g = 2 h$ implies $\delt \vol = (\tr h) \vol$ is standard, but $\tr h = \tr A$.
\end{proof}

\begin{rmk} \label{rmk:flows1-evolution-4-form}
In~\cite[Theorem 3.5]{K-flows}, the evolution of the $4$-form is given as
$$ \delt \ps = h \diamond \ps - X \w \ph. $$
In our notation, we have $X_k = \frac{1}{6} X_{ij} \ph_{ijk}$ where $X_{ij} = X_m \ph_{mij} = - 3 (A_7)_{ij}$. It is easy to verify using~\eqref{eq:phps} that the above is indeed equivalent to our equation $\delt \ps = A \diamond \ps$ from~\eqref{eq:evolution-1}.
\end{rmk}

From $\delt g = 2 h$, it is a standard result to compute the flow of the covariant derivative $\nabla$. As we always work with orthonormal frames, we have $\nab{i} e_j = \Gamma_{ijk} e_k$ for some $\Gamma_{ijk}$. By taking the time derivative of the Koszul formula one obtains
$$ g \big( ( \delt \nab{i} ) e_j, e_k \big) = \nab{i} h_{jk} + \nab{j} h_{ik} - \nab{k} h_{ij}. $$
The above says
\begin{equation} \label{eq:evolution-nabla}
\delt \Gamma_{ijk} = \nab{i} h_{jk} + \nab{j} h_{ik} - \nab{k} h_{ij}.
\end{equation}

\begin{prop} \label{prop:evolution-torsion}
Let $\ph$ be a time-dependent family of $\G$-structures evolving by the flow~\eqref{eq:general-flow-A}. Then the torsion $T$ evolves by
\begin{equation} \label{eq:evolution-torsion}
\delt T_{pq} = \tfrac{1}{2} ( \nab{i} A_{pj} + \nab{i} A_{jp} - \nab{p} A_{ij} ) \ph_{ijq} + T_{pk} A_{qk}.
\end{equation}
\end{prop}
\begin{proof}
Recall from~\eqref{eq:Tfromph} that $24 \, T_{pq} = \nab{p} \ph_{ijk} \ps_{qijk} = \nab{p} \ph_{ijk} \ps_{qabc} g^{ia} g^{jb} g^{kc}$. (Here we have to be careful to note that \emph{there are contractions with the inverse metric}, because we need to differentiate this equation.) From~\eqref{eq:evolution-1} we get $\delt g^{ij} = - 2 h^{ij}$, and thus
\begin{align*}
24 \delt T_{pq} & = \delt \nab{p} \ph_{ijk} \ps_{qabc} g^{ia} g^{jb} g^{kc} + \nab{p} \ph_{ijk} \delt \ps_{qabc} g^{ia} g^{jb} g^{kc} + \nab{p} \ph_{ijk} \ps_{qabc} \delt g^{ia} g^{jb} g^{kc} \\
& \qquad {} + \nab{p} \ph_{ijk} \ps_{qabc} g^{ia} \delt g^{jb} g^{kc} + \nab{p} \ph_{ijk} \ps_{qabc} g^{ia} g^{jb} \delt g^{kc} \\
& = \delt \nab{p} \ph_{ijk} \ps_{qijk} + \nab{p} \ph_{ijk} \delt \ps_{qijk} - 2 \nab{p} \ph_{ijk} \ps_{qajk} h_{ia} \\
& \qquad {} - 2 \nab{p} \ph_{ijk} \ps_{qibk} h_{jb} - 2 \nab{p} \ph_{ijk} \ps_{qijc} h_{kc}.
\end{align*}
The last three terms above are identical, so we have
\begin{equation} \label{eq:evolution-torsion-t0}
24 \delt T_{pq} = \delt \nab{p} \ph_{ijk} \ps_{qijk} + \nab{p} \ph_{ijk} \delt \ps_{qijk} - 6 \, \nab{p} \ph_{ijk} \ps_{qajk} h_{ia}.
\end{equation}
We thus need to compute the three terms on the right-hand side of~\eqref{eq:evolution-torsion-t0}.

Recall that for any $3$-tensor $\gamma_{ijk}$ we have
$$ \nab{p} \gamma_{ijk} = \partial_p \gamma_{ijk} - \Gamma_{pim} \gamma_{mjk} - \Gamma_{pjm} \gamma_{imk} - \Gamma_{pkm} \gamma_{ijm}. $$
Note that in this case, \emph{there are no contractions above with the inverse metric}, because the Christoffel symbols really should be written as $\Gamma^k_{ij}$. Hence, using this and~\eqref{eq:evolution-nabla}, we compute
\begin{align*}
\delt \nab{p} \ph_{ijk} & = \partial_p \delt \ph_{ijk} - \Gamma_{pim} \delt \ph_{mjk} - \Gamma_{pjm} \delt \ph_{imk} - \Gamma_{pkm} \delt \ph_{ijm} \\
& \qquad {} - \delt \Gamma_{pim} \, \ph_{mjk} - \delt \Gamma_{pjm} \, \ph_{imk} - \delt \Gamma_{pkm} \, \ph_{ijm} \\
& = \nab{p} \delt \ph_{ijk} - (\nab{p} h_{im} + \nab{i} h_{pm} - \nab{m} h_{pi}) \ph_{mjk} \\
& \qquad {} - (\nab{p} h_{jm} + \nab{j} h_{pm} - \nab{m} h_{pj}) \ph_{imk} - (\nab{p} h_{km} + \nab{k} h_{pm} - \nab{m} h_{pk}) \ph_{ijm}.
\end{align*}
Contracting the above with $\ps_{qijk}$ and using skew-symmetry of $\ph$, $\ps$, symmetry of $h$, and~\eqref{eq:phps}, we have
\begin{align*}
\delt \nab{p} \ph_{ijk} \ps_{qijk} & = \nab{p} ( A_{im} \ph_{mjk} + A_{jm} \ph_{imk} + A_{km} \ph_{ijm}) \ps_{qijk} \\
& \qquad {} - 3 (\nab{p} h_{im} + \nab{i} h_{pm} - \nab{m} h_{pi}) \ph_{mjk} \ps_{qijk} \\
& = 3 \nab{p} (A_{im} \ph_{mjk}) \ps_{qijk} - 3 (\nab{p} h_{im} + \nab{i} h_{pm} - \nab{m} h_{pi}) (-4 \ph_{mqi}) \\
& = 3\nab{p} A_{im} \ph_{mjk} \ps_{qijk} + 3 A_{im} \nab{p} \ph_{mjk} \ps_{qijk} + 12( 0 + 2 \nab{i} h_{pm}) \ph_{qim} \\
& = -12 \nab{p} A_{im} \ph_{qim} + 3 A_{im} \nab{p} \ph_{mjk} \ps_{qijk} + 24 \nab{i} h_{pm} \ph_{qim}.
\end{align*}
Recalling that $2 h_{ij} = A_{ij} + A_{ji}$, and using~\eqref{eq:nablaph} and~\eqref{eq:Pop-defn2} the above becomes
\begin{align} \nonumber
\delt \nab{p} \ph_{ijk} \ps_{qijk} & = 12( \nab{i} A_{pm} + \nab{i} A_{mp} - \nab{p} A_{im} ) \ph_{qim} + 3 A_{im} T_{pl} \ps_{lmjk} \ps_{qijk} \\ \nonumber
& = 12( \nab{i} A_{pm} + \nab{i} A_{mp} - \nab{p} A_{im} ) \ph_{qim} + 3 A_{im} T_{pl} (4 g_{lq} g_{mi} - 4 g_{li} g_{mq} - 2 \ps_{lmqi}) \\ \nonumber
& = 12( \nab{i} A_{pm} + \nab{i} A_{mp} - \nab{p} A_{im} ) \ph_{qim} \\ \label{eq:evolution-torsion-t1}
& \qquad {} + 12 (\tr A) T_{pq} - 12 T_{pl} A_{lq} - 6 T_{pl} (\Pop A)_{lq}.
\end{align}

Using~\eqref{eq:nablaph} and~\eqref{eq:evolution-1}, we have
\begin{align} \nonumber
\nab{p} \ph_{ijk} \delt \ps_{qijk} & = T_{pl} \ps_{lijk} (A \diamond \ps)_{qijk} \\ \nonumber
& = T_{pl} \ps_{lijk} (A_{qm} \ps_{mijk} + A_{im} \ps_{qmjk} + A_{jm} \ps_{qimk} + A_{km} \ps_{qijm}) \\ \nonumber
& = T_{pl} A_{qm} (\ps_{lijk} \ps_{mijk}) + 3 T_{pl} A_{im} (\ps_{lijk} \ps_{qmjk}) \\ \nonumber
& = T_{pl} A_{qm} (24 g_{lm}) + 3 T_{pl} A_{im} (4 g_{lq} g_{im} - 4 g_{lm} g_{iq} - 2 \ps_{liqm}) \\ \label{eq:evolution-torsion-t2}
& = 12 T_{pl} A_{ql} + 12 (\tr A) T_{pq} + 6 T_{pl} (\Pop A)_{lq}.
\end{align}
Similarly we compute
\begin{align} \nonumber
\nab{p} \ph_{ijk} \ps_{qajk} h_{ia} & = T_{pl} \ps_{lijk} \ps_{qajk} h_{ia} \\ \nonumber
& = T_{pl} h_{ia} (4 g_{lq} g_{ia} - 4 g_{la} g_{iq} - 2 \ps_{liqa}) \\ \nonumber
& = 4 (\tr h) T_{pq} - 4 T_{pl} h_{lq} - 0 \\ \label{eq:evolution-torsion-t3}
& = 4 (\tr A) T_{pq} - 2 T_{pl} A_{lq} - 2 T_{pl} A_{ql}.
\end{align}
Substituting~\eqref{eq:evolution-torsion-t1},~\eqref{eq:evolution-torsion-t2}, and~\eqref{eq:evolution-torsion-t3} into~\eqref{eq:evolution-torsion-t0} yields
\begin{align*}
24 \delt T_{pq} & = 12( \nab{i} A_{pm} + \nab{i} A_{mp} - \nab{p} A_{im} ) \ph_{qim} + 12 (\tr A) T_{pq} - 12 T_{pl} A_{lq} - 6 T_{pl} (\Pop A)_{lq} \\
& \qquad {} + 12 T_{pl} A_{ql} + 12 (\tr A) T_{pq} + 6 T_{pl} (\Pop A)_{lq} - 6 (4 (\tr A) T_{pq} - 2 T_{pl} A_{lq} - 2 T_{pl} A_{ql}) \\
& = 12( \nab{i} A_{pm} + \nab{i} A_{mp} - \nab{p} A_{im} ) \ph_{qim} + 24 T_{pl} A_{ql},
\end{align*}
which, after relabelling some indices, is equation~\eqref{eq:evolution-torsion}.
\end{proof}

\begin{rmk} \label{rmk:flows1-evolution-torsion}
In~\cite[Theorem 3.8]{K-flows}, the evolution of the torsion is given as
$$ \delt T_{pq} = T_{pl} h_{lq} + T_{pl} X_{lq} + (\nab{k} h_{ip}) \ph_{kiq} + \nab{p} X_q, $$
where $X_{ij} = X_m \ph_{mij} = - 3 (A_7)_{ij}$, so $X_k = \frac{1}{6} X_{ij} \ph_{ijk} = - \frac{1}{2} A_{ij} \ph_{ijk}$. It is easy to verify using~\eqref{eq:omega27desc} that the above is equivalent to our equation~\eqref{eq:evolution-torsion}.
\end{rmk}

\subsection{Diffeomorphism invariance and the $\G$-Bianchi identity} \label{sec:diffeo-inv}

In Section~\ref{sec:symbols-short-time} we discuss a class of geometric flows of $\G$-structures which are amenable to a DeTurck trick to establish short-time existence and uniqueness. This works if and only if the failure of the flow to be strictly parabolic is due solely do diffeomorphism invariance. It was shown in~\cite[Thm 4.2]{K-flows} that (infinitesimal) diffeomorphism invariance of the torsion tensor $T$ is equivalent to the $\G$-Bianchi identity~\eqref{eq:g2bianchi}. We give a more direct proof in this section for completeness.

\begin{prop} \label{prop:diffeo-inv}
The infinitesimal diffeomorphism invariance of the torsion tensor is equivalent to the $\G$-Bianchi identity.
\end{prop}
\begin{proof}
The torsion $T$ is a tensor determined entirely by the $\G$-structure $\ph$. Infinitesimal diffeomorphism invariance says that $T_{\Theta_s^* \ph} = \Theta_s^* (T_{\ph})$ for any $1$-parameter family of diffeomorphisms $\Theta_t$ generated by a vector field $W$. Taking $\rest{\dds}{s=0}$ of $T_{\Theta_s^* \ph} = \Theta_s^* (T_{\ph})$ gives
\begin{equation} \label{eq:diffeo-inv-1}
T_* (\cL_W \ph) = \cL_W T
\end{equation}
where $T = T_{\ph}$ is the torsion of $\ph$ and $T_*$ is the pushforward (differential) of the map $\ph \mapsto T_{\ph}$. We need to prove that~\eqref{eq:diffeo-inv-1} is equivalent to the $\G$-Bianchi identity~\eqref{eq:g2bianchi}.

By Proposition~\ref{prop:evolution-torsion}, for any $A \diamond \ph \in \Omega^3$ we have
\begin{equation} \label{eq:diffeo-inv-2}
(T_* (A \diamond \ph))_{pq} = \tfrac{1}{2} ( \nab{i} A_{pj} + \nab{i} A_{jp} - \nab{p} A_{ij} ) \ph_{ijq} + T_{pl} A_{ql}.
\end{equation}
Taking $A = \cL_W \ph$ as in~\eqref{eq:diffeo-inv-1}, we recall from~\eqref{eq:Lie-derivative-framed} that
$$ A = (- \tfrac{1}{3} T^t W + \tfrac{1}{6} \curl W ) \hk \ph + \tfrac{1}{2} \cL_W g, $$
In terms of a local frame, using~\eqref{eq:Lie-derivative-frame} to express $\cL_W g$, we can write the above as
\begin{align} \nonumber
6 A_{pq} & = - 2 T_{km} W_k \ph_{mpq} + (\nab{i} W_j) \ph_{ijm} \ph_{mpq} + 3 (\nab{p} W_q + \nab{q} W_p) \\ \nonumber
& = - 2 T_{km} W_k \ph_{mpq} + (\nab{i} W_j) (g_{ip} g_{jq} - g_{iq} g_{jp} - \ps_{ijpq}) + 3 (\nab{p} W_q + \nab{q} W_p) \\ \label{eq:diffeo-inv-A}
& = - 2 T_{km} W_k \ph_{mpq} + 4 \nab{p} W_q + 2 \nab{q} W_p - \nab{a} W_b \ps_{abpq}.
\end{align}
Covariantly differentiating~\eqref{eq:diffeo-inv-A} and using~\eqref{eq:nablaph} and~\eqref{eq:nablaps} gives
\begin{align*}
6 \, \nab{i} A_{pj} & = \nab{i} (- 2 T_{km} W_k \ph_{mpj} + 4 \nab{p} W_j + 2 \nab{j} W_p - \nab{a} W_b \ps_{abpj}) \\
& = - 2 \nab{i} T_{km} W_k \ph_{mpj} - 2 T_{km} \nab{i} W_k \ph_{mpj} - 2 T_{km} W_k T_{il} \ps_{lmpj} \\
& \qquad {} + 4 \nab{i} \nab{p} W_j + 2 \nab{i} \nab{j} W_p - \nab{i} \nab{a} W_b \ps_{abpj} \\
& \qquad {} - \nab{a} W_b (- T_{ia} \ph_{bpj} + T_{ib} \ph_{apj} - T_{ip} \ph_{abj} + T_{ij} \ph_{abp}).
\end{align*}
From the above and the skew-symmetry of $\ph$, $\ps$, we obtain
\begin{align*}
6 (\nab{i} A_{pj} + \nab{i} A_{jp} - \nab{p} A_{ij}) & = 2 \nab{p} T_{km} W_k \ph_{mij} + 2 T_{km} \nab{p} W_k \ph_{mij} + 2 T_{km} W_k T_{pl} \ps_{lmij} \\
& \qquad {} + 6 \nab{i} \nab{p} W_j + 6 \nab{i} \nab{j} W_p - 4 \nab{p} \nab{i} W_j - 2 \nab{p} \nab{j} W_i + \nab{p} \nab{a} W_b \ps_{abij} \\
& \qquad {} - (\nab{a} W_b) (T_{pa} \ph_{bij} - T_{pb} \ph_{aij} + T_{pi} \ph_{abj} - T_{pj} \ph_{abi}).
\end{align*}
Contracting the above with $\ph_{ijq}$, we have
\begin{align*}
6 (\nab{i} A_{pj} + \nab{i} A_{jp} - \nab{p} A_{ij}) \ph_{ijq} & = 12 \nab{p} T_{kq} W_k + 12 T_{kq} \nab{p} W_k - 8 T_{km} W_k T_{pl} \ph_{lmq} \\
& \qquad {} + (6 \nab{i} \nab{p} W_j + 6 \nab{i} \nab{j} W_p - 4 \nab{p} \nab{i} W_j - 2 \nab{p} \nab{j} W_i) \ph_{ijq} \\
& \qquad {} - 4 \nab{p} \nab{a} W_b \ph_{abq} - (\nab{a} W_b) ( 6 T_{pa} g_{bq} - 6 T_{pb} g_{aq} + 2 T_{pi} \ph_{abj} \ph_{qij}).
\end{align*}
The above further simplifies to
\begin{align*}
6 (\nab{i} A_{pj} + \nab{i} A_{jp} - \nab{p} A_{ij}) \ph_{ijq} & = 12 \nab{p} T_{kq} W_k + 12 T_{kq} \nab{p} W_k - 8 T_{km} W_k T_{pl} \ph_{lmq} \\
& \qquad {} + 6 \nab{i} \nab{p} W_j \ph_{ijq} + 3 (\nab{i} \nab{j} W_p - \nab{j} \nab{i} W_p) \ph_{ijq} - 2 \nab{p} \nab{i} W_j \ph_{ijq} \\
& \qquad {} - 4 \nab{p} \nab{a} W_b \ph_{abq} - 6 \nab{a} W_q T_{pa} + 6 \nab{q} W_b T_{pb} \\
& \qquad {} - 2 \nab{a} W_b T_{pi} (g_{aq} g_{bi} - g_{ai} g_{bq} - \ps_{abqi}).
\end{align*}
Several terms above combine, leaving us with
\begin{align*}
6 (\nab{i} A_{pj} + \nab{i} A_{jp} - \nab{p} A_{ij}) \ph_{ijq} & = 12 \nab{p} T_{kq} W_k + 12 T_{kq} \nab{p} W_k - 8 T_{km} W_k T_{pl} \ph_{lmq} \\
& \qquad {} + 6 \nab{i} \nab{p} W_j \ph_{ijq} + 3 (\nab{i} \nab{j} W_p - \nab{j} \nab{i} W_p) \ph_{ijq} - 6 \nab{p} \nab{i} W_j \ph_{ijq} \\
& \qquad {} - 4 \nab{a} W_q T_{pa} + 4 \nab{q} W_b T_{pb} - 2 \nab{a} W_b T_{pi} \ps_{abiq}.
\end{align*}
The Ricci identity~\eqref{eq:ricci-identity} gives $(\nab{i} \nab{j} W_p - \nab{j} \nab{i} W_p) = - R_{ijpm} W_m$, so the above expression can finally be written as
\begin{align*}
6 (\nab{i} A_{pj} + \nab{i} A_{jp} - \nab{p} A_{ij}) \ph_{ijq} & = 12 \nab{p} T_{kq} W_k + 12 T_{kq} \nab{p} W_k - 8 T_{km} W_k T_{pl} \ph_{lmq} \\
& \qquad {} + 6 \nab{i} \nab{p} W_j \ph_{ijq} - 3 R_{ijpm} W_m \ph_{ijq} - 6 \nab{p} \nab{i} W_j \ph_{ijq} \\
& \qquad {} - 4 \nab{a} W_q T_{pa} + 4 \nab{q} W_b T_{pb} - 2 \nab{a} W_b T_{pi} \ps_{abiq}.
\end{align*}
From~\eqref{eq:diffeo-inv-A} we also have
\begin{align*}
12 T_{pl} A_{ql} & = 2 T_{pl} (- 2 T_{km} W_k \ph_{mql} + 4 \nab{q} W_l + 2 \nab{l} W_q - \nab{a} W_b \ps_{abql}) \\
& = - 4 T_{pl} T_{km} W_k \ph_{lmq} + 8 \nab{q} W_l T_{pl} + 4 \nab{l} W_q T_{pl} + 2 \nab{a} W_b T_{pl} \ps_{ablq}.
\end{align*}
Adding the above two expressions, some cancellation, relabelling, and rearrangement yields
\begin{align} \nonumber
6 (\nab{i} A_{pj} + \nab{i} A_{jp} - \nab{p} A_{ij}) \ph_{ijq} + 12 T_{pl} A_{ql} & = 12 W_k \nab{p} T_{kq} - 3 W_k R_{ijpk} \ph_{ijq} - 12 W_k T_{pl} T_{km} \ph_{lmq} \\ \nonumber
& \qquad {} + 12 \nab{p} W_k T_{kq} + 12 \nab{q} W_k T_{pk} \\ \label{eq:diffeo-inv-3}
& \qquad {} + 6 \nab{i} \nab{p} W_j \ph_{ijq} - 6 \nab{p} \nab{i} W_j \ph_{ijq}.
\end{align}
We can now apply the Ricci identity again, to the last two terms above, as well as the Riemannian first Bianchi identity, giving
\begin{align*}
6 \nab{i} \nab{p} W_j \ph_{ijq} - 6 \nab{p} \nab{i} W_j \ph_{ijq} & = - 6 R_{ipjm} W_m \ph_{ijq} \\
& = - 3 R_{ipjm} W_m \ph_{ijq} - 3 R_{jpim} W_m \ph_{jiq} \\
& = 3 R_{pijm} W_m \ph_{ijq} + 3 R_{jpim} W_m \ph_{ijq} \\
& = - 3 R_{ijpm} W_m \ph_{ijq}.
\end{align*}
Substituting the above into~\eqref{eq:diffeo-inv-3}, the two curvature terms combine, leaving us with
\begin{align} \nonumber
6 (\nab{i} A_{pj} + \nab{i} A_{jp} - \nab{p} A_{ij}) \ph_{ijq} + 12 T_{pl} A_{ql} & = 12 W_k \nab{p} T_{kq} - 6 W_k R_{ijpk} \ph_{ijq} - 12 W_k T_{pl} T_{km} \ph_{lmq} \\ \label{eq:diffeo-inv-4}
& \qquad {} + 12 \nab{p} W_k T_{kq} + 12 \nab{q} W_k T_{pk}.
\end{align}

Applying~\eqref{eq:Lie-derivative-frame} to $S = T$, we have
\begin{equation} \label{eq:Lie-derivative-frameT}
12 (\cL_W T)_{pq} = 12 W_k \nab{k} T_{pq} + 12 \nab{p} W_k T_{kq} + 12 \nab{q} W_k T_{pk}.
\end{equation}
Taking the difference of equations~\eqref{eq:diffeo-inv-4} and~\eqref{eq:Lie-derivative-frameT} and using~\eqref{eq:diffeo-inv-1} and~\eqref{eq:diffeo-inv-2}, we deduce that
\begin{align*}
12( T_* (\cL_W \ph) - \cL_W T ) & = 12 W_k (\nab{p} T_{kq} - \nab{k} T_{pq} - T_{pa} T_{kb} \ph_{abq} - \tfrac{1}{2} R_{pkab} \ph_{abq}).
\end{align*}
Infinitesimal diffeomorphism invariance of the torsion is equivalent by~\eqref{eq:diffeo-inv-1} to the left-hand side above vanishing for all $W$. But the right-hand side above vanishing for all $W$ is equivalent to the $\G$-Bianchi identity~\eqref{eq:g2bianchi}.
\end{proof}

\subsection{Evolution of quadratic quantities associated to torsion} \label{sec:torsion-quantities}

In this section we compute the evolution of certain quadratic quantities associated to the torsion $T$ of an evolving $\G$-structure $\ph$. These are used in Section~\ref{sec:functionals} to compute the evolution equations for several natural torsion functionals.

\begin{prop} \label{prop:evolution-torsion-scalars}
Let $\ph$ be a time-dependent family of $\G$-structures evolving by the flow~\eqref{eq:general-flow-A}. We have the following evolution equations for various quadratic scalar quantities obtained from the torsion:
\begin{equation} \label{eq:evolution-torsion-scalars}
\begin{aligned}
\delt (\tr T)^2 & = - (\tr T) \nab{p} A_{ij} \ph_{pij} - 2 (\tr T) \langle T^t, A \rangle, \\
\delt |T|^2 & = (\nab{i} A_{pj} + \nab{i} A_{jp} - \nab{p} A_{ij}) \ph_{ijq} T_{pq} - 2 \langle T T^t, A \rangle, \\
\delt \langle T, T^t \rangle & = (\nab{i} A_{qj} + \nab{i} A_{jq} - \nab{q} A_{ij}) \ph_{ijp} T_{pq} - 2 \langle (T^t)^2, A \rangle, \\
\delt \langle T, \Pop T \rangle & = 2 \nab{i} A_{ji} \ph_{pqj} T_{pq} - 2 \nab{j} A_{ii} \ph_{pqj} T_{pq} + 2 \nab{i} A_{jp} \ph_{ijq} T_{pq} \\
& \qquad {} - 2 \nab{i} A_{jq} \ph_{ijp} T_{pq} - 2 \langle (\Pop T) T^t, A \rangle.
\end{aligned}
\end{equation}
\end{prop}
\begin{proof}
As in the proof of Proposition~\ref{prop:evolution-torsion}, we have to be careful to note that our quadratic scalar quantities involve \emph{contractions with the inverse metric}, which also need to be differentiated. Recall that from~\eqref{eq:evolution-1} we have $\delt g^{ij} = - 2 h^{ij}$, where $2 h = A + A^t$. We use the evolution equation~\eqref{eq:evolution-torsion} of the torsion throughout this proof. First we have
\begin{align*}
\delt (\tr T) & = \delt (T_{pq} g^{pq}) = \delt T_{pq} g^{pq} + T_{pq} \delt g^{pq} \\
& = \delt T_{pp} - 2 T_{pq} h_{pq} \\
& = \tfrac{1}{2} ( \nab{i} A_{pj} + \nab{i} A_{jp} - \nab{p} A_{ij} ) \ph_{ijp} + T_{pk} A_{pk} - T_{pq} (A_{pq} + A_{qp}) \\
& = - \tfrac{1}{2} \nab{p} A_{ij} \ph_{pij} - \langle T^t, A \rangle.
\end{align*}
The first equation in~\eqref{eq:evolution-torsion-scalars} now follows from $\delt (\tr T)^2 = 2 (\tr T) \delt (\tr T)$.

For the second equation, we compute
\begin{align*}
\delt |T|^2 & = \delt (T_{pq} T_{ab} g^{pa} g^{qb}) = 2 \delt T_{pq} T_{ab} g^{pa} g^{qb} + T_{pq} T_{ab} \delt g^{pa} g^{qb} + T_{pq} T_{ab} g^{pa} \delt g^{qb} \\
& = 2 \delt T_{pq} T_{pq} - 2 T_{pq} T_{aq} h_{pa} - 2 T_{pq} T_{pb} h_{qb} \\
& = 2 \big( \tfrac{1}{2} ( \nab{i} A_{pj} + \nab{i} A_{jp} - \nab{p} A_{ij} ) \ph_{ijq} + T_{pk} A_{qk} \big) T_{pq} \\
& \qquad {} - T_{pq} T_{aq} (A_{pa} + A_{ap}) - T_{pq} T_{pb} (A_{qb} + A_{bq}) \\
& = ( \nab{i} A_{pj} + \nab{i} A_{jp} - \nab{p} A_{ij} ) \ph_{ijq} T_{pq} + 2 \langle T^t T, A \rangle - 2 \langle T T^t, A \rangle - 2 \langle T^t T, A \rangle.
\end{align*}

For the third equation, we compute
\begin{align*}
\delt \langle T, T^t \rangle & = \delt (T_{pq} T_{ba} g^{pa} g^{qb}) = 2 \delt T_{pq} T_{ba} g^{pa} g^{qb} + T_{pq} T_{ba} \delt g^{pa} g^{qb} + T_{pq} T_{ba} g^{pa} \delt g^{qb} \\
& = 2 \delt T_{pq} T_{qp} - 2 T_{pq} T_{qa} h_{pa} - 2 T_{pq} T_{bp} h_{qb} \\
& = 2 \big( \tfrac{1}{2} ( \nab{i} A_{pj} + \nab{i} A_{jp} - \nab{p} A_{ij} ) \ph_{ijq} + T_{pk} A_{qk} \big) T_{qp} \\
& \qquad {} - T_{pq} T_{qa} (A_{pa} + A_{ap}) - T_{pq} T_{bp} (A_{qb} + A_{bq}) \\
& = ( \nab{i} A_{pj} + \nab{i} A_{jp} - \nab{p} A_{ij} ) \ph_{ijq} T_{qp} + 2 \langle T^2, A \rangle \\
& \qquad {} - \langle T^2, A \rangle - \langle (T^t)^2, A \rangle - \langle (T^t)^2, A \rangle - \langle T^2, A \rangle,
\end{align*}
and finally interchange the dummy indices $p,q$.

For the fourth equation, we proceed as before. Omitting some steps, we compute
\begin{align*}
\delt \langle T, \Pop T \rangle & = \delt (T_{pq} (\Pop T)_{ab} g^{pa} g^{qb} ) = \delt (T_{pq} \ps_{abij} T_{kl} g^{ik} g^{jl} g^{pa} g^{qb} ) \\
& = \delt T_{pq} \ps_{pqij} T_{ij} + T_{pq} \delt \ps_{pqij} T_{ij} + T_{pq} \ps_{pqij} \delt T_{ij} \\
& \qquad {} - 2 T_{pq} \ps_{pqij} T_{kj} h_{ik} - 2 T_{pq} \ps_{pqij} T_{il} h_{jl} - 2 T_{pq} \ps_{aqij} T_{ij} h_{pa} - 2 T_{pq} \ps_{pbij} T_{ij} h_{qb} \\
& = 2 \delt T_{pq} \ps_{pqij} T_{ij} + T_{pq} T_{ij} \delt \ps_{pqij} - 2 (\Pop T)_{ij} T_{kj} h_{ik} - 2 (\Pop T)_{ij} T_{il} h_{jl} \\
& \qquad {} - 2 T_{pq} (\Pop T)_{aq} h_{pa} - 2 T_{pq} (\Pop T)_{pb} h_{qb}.
\end{align*}
Using that $h$ is symmetric, $\Pop T$ is skew-symmetric, and equation~\eqref{eq:evolution-1}, the above becomes
\begin{align*}
\delt \langle T, \Pop T \rangle & = 2 \delt T_{pq} (\Pop T)_{pq} + T_{pq} T_{ij} (A \diamond \ps)_{pqij} + 2 \langle T (\Pop T), h \rangle + 2 \langle (\Pop T) T, h \rangle \\
& \qquad {} + 2 \langle T (\Pop T), h \rangle + 2 \langle (\Pop T) T, h \rangle \\
& = 2 \big( \tfrac{1}{2} ( \nab{i} A_{pj} + \nab{i} A_{jp} - \nab{p} A_{ij} ) \ph_{ijq} + T_{pk} A_{qk} \big) (\Pop T)_{pq} \\
& \qquad {} + T_{pq} T_{ij} (A_{pm} \ps_{mqij} + A_{qm} \ps_{pmij} + A_{im} \ps_{pqmj} + A_{jm} \ps_{pqim}) \\
& \qquad {} + 4 \langle T (\Pop T), h \rangle + 4 \langle (\Pop T) T, h \rangle,
\end{align*}
which, recalling that $2h = A + A^t$, then further simplifies to
\begin{align*}
\delt \langle T, \Pop T \rangle & = ( \nab{i} A_{pj} + \nab{i} A_{jp} - \nab{p} A_{ij} ) \ph_{ijq} (\Pop T)_{pq} - 2 \langle (\Pop T) T, A \rangle \\
& \qquad {} + T_{pq} (\Pop T)_{mq} A_{pm} + T_{pq} (\Pop T)_{pm} A_{qm} + (\Pop T)_{mj} T_{ij} A_{im} + (\Pop T)_{im} T_{ij} A_{jm} \\
& \qquad {} + 4 \langle T (\Pop T), h \rangle + 4 \langle (\Pop T) T, h \rangle \\
& = ( \nab{i} A_{pj} + \nab{i} A_{jp} - \nab{p} A_{ij} ) \ph_{ijq} (\Pop T)_{pq} - 2 \langle (\Pop T) T, A \rangle - \langle T (\Pop T), A \rangle + \langle T^t (\Pop T), A \rangle \\
& \qquad {} - \langle T (\Pop T), A \rangle + \langle T^t (\Pop T), A \rangle + 2 \langle T (\Pop T), A + A^t \rangle + 2 \langle (\Pop T) T, A + A^t \rangle.
\end{align*}
Collecting terms and applying the third equation in~\eqref{eq:tensinnerproductidentities} yields
\begin{align} \nonumber
\delt \langle T, \Pop T \rangle & = ( \nab{i} A_{pj} + \nab{i} A_{jp} - \nab{p} A_{ij} ) \ph_{ijq} (\Pop T)_{pq} + 2 \langle T^t (\Pop T), A \rangle \\ \nonumber
& \qquad {} + 2 \langle T (\Pop T), A^t \rangle + 2 \langle (\Pop T) T, A^t \rangle \\ \nonumber
& = ( \nab{i} A_{pj} + \nab{i} A_{jp} - \nab{p} A_{ij} ) \ph_{ijq} (\Pop T)_{pq} + 2 \langle T^t (\Pop T), A \rangle \\ \nonumber
& \qquad {} - 2 \langle (\Pop T) T^t, A \rangle - 2 \langle T^t (\Pop T), A \rangle \\ \label{eq:TPTev-temp}
& = ( \nab{i} A_{pj} + \nab{i} A_{jp} - \nab{p} A_{ij} ) \ph_{ijq} (\Pop T)_{pq} - 2 \langle (\Pop T) T^t, A \rangle.
\end{align}
Now consider the expression $\ph_{ijq} (\Pop T)_{pq}$. We rewrite this as
\begin{align*}
\ph_{ijq} (\Pop T)_{pq} & = \ph_{ijq} \ps_{abpq} T_{ab} \\
& = (g_{ia} \ph_{jbp} + g_{ib} \ph_{ajp} + g_{ip} \ph_{abj} - g_{ja} \ph_{ibp} - g_{jb} \ph_{aip} - g_{jp} \ph_{abi}) T_{ab}.
\end{align*}
The first two terms together of $( \nab{i} A_{pj} + \nab{i} A_{jp} - \nab{p} A_{ij} )$ are symmetric in $p, j$, and the first and third terms together are skew-symmetric in $i, p$. With these observations, from the above we obtain
\begin{align*}
& \qquad ( \nab{i} A_{pj} + \nab{i} A_{jp} - \nab{p} A_{ij} ) \ph_{ijq} (\Pop T)_{pq} \\
& = T_{ab} \big( g_{ia} \ph_{jbp} (- \nab{p} A_{ij}) + g_{ib} \ph_{ajp} (- \nab{p} A_{ij}) + g_{ip} \ph_{abj} (\nab{i} A_{jp}) \big) \\
& \qquad {} - T_{ab} \big( g_{ja} \ph_{ibp} (2 \nab{i} A_{pj} + \nab{i} A_{jp}) + g_{jb} \ph_{aip} (2 \nab{i} A_{pj} + \nab{i} A_{jp}) + g_{jp} \ph_{abi} (2 \nab{i} A_{pj} - \nab{p} A_{ij}) \big) \\
& = T_{ab} \big( - \nab{p} A_{aj} \ph_{jbp} - \nab{p} A_{bj} \ph_{ajp} + \nab{p} A_{jp} \ph_{abj} - 2 \nab{i} A_{pa} \ph_{ibp} - \nab{i} A_{ap} \ph_{ibp} \big) \\
& \qquad {} + T_{ab} \big( - 2 \nab{i} A_{pb} \ph_{aip} - \nab{i} A_{bp} \ph_{aip} - 2 \nab{i} A_{pp} \ph_{abi} + \nab{p} A_{ip} \ph_{abi} \big).
\end{align*}
Relabelling some indices and collecting terms, several terms cancel and the above becomes
\begin{align*}
& \qquad ( \nab{i} A_{pj} + \nab{i} A_{jp} - \nab{p} A_{ij} ) \ph_{ijq} (\Pop T)_{pq} \\
& = T_{ab} ( 2 \nab{p} A_{qp} \ph_{abq} + 2 \nab{p} A_{qa} \ph_{pqb} - 2 \nab{p} A_{qb} \ph_{pqa} - 2 \nab{q} A_{pp} \ph_{qab} ).
\end{align*}
Substituting the above into~\eqref{eq:TPTev-temp} and relabelling again yields the fourth equation in~\eqref{eq:evolution-torsion-scalars}.
\end{proof}

\subsection{Evolutions of torsion functionals} \label{sec:functionals}

In this section we consider several natural \emph{functionals} defined using the torsion of a $\G$-structure, and compute their associated Euler--Lagrange equations. These Euler--Lagrange equations yield a geometric interpretation for various irreducible components of $\tRm$ and $\nab{} T$, and motivates the detailed study of the explicit decompositions of $\tRm$ and $\nab{} T$ that is undertaken in Sections~\ref{sec:more-rep-theory} and~\ref{sec:curvature-torsion}. Whenever we integrate in this section, we assume that $M$ is \emph{compact} so that all integrals are defined.

\begin{lemma} \label{lemma:density-ev}
Let $Q$ be a scalar function evolving under a general flow~\eqref{eq:general-flow-A} of $\G$-structures. We have
\begin{equation} \label{eq:density-ev}
\delt \Big( \int_M Q \vol \Big) = \int_M \big( \delt Q + \langle Q g, A \rangle \big) \vol.
\end{equation}
\end{lemma}
\begin{proof}
Using~\eqref{eq:evolution-1} for the evolution of the volume form, we have
$$ \delt (Q \vol) = (\delt Q) \vol + Q (\tr A) \vol = (\delt Q + Q \langle g, A \rangle) \vol, $$
which yields~\eqref{eq:density-ev}.
\end{proof}

\begin{cor} \label{cor:evolution-torsion-quantities}
Let $\ph$ be a time-dependent family of $\G$-structures evolving by the flow~\eqref{eq:general-flow-A}. We have the following evolution equations for various quadratic integral quantities obtained from the torsion:
\begin{align*}
\delt \Big( \int_M (\tr T)^2 \vol \Big) & = \int_M \Big[ - (\tr T) \nab{p} A_{ij} \ph_{pij} + \big\langle (\tr T)^2 g - 2 (\tr T) T^t, A \big\rangle \Big] \vol, \\
\delt \Big( \int_M |T|^2 \vol \Big) & = \int_M \Big[ (\nab{i} A_{pj} + \nab{i} A_{jp} - \nab{p} A_{ij}) \ph_{ijq} T_{pq} + \big\langle |T|^2 g - 2 T T^t, A \big\rangle \Big] \vol, \\
\delt \Big( \int_M \langle T, T^t \rangle \vol \Big) & = \int_M \Big[ (\nab{i} A_{qj} + \nab{i} A_{jq} - \nab{q} A_{ij}) \ph_{ijp} T_{pq} + \big\langle \langle T, T^t \rangle g - 2 (T^t)^2, A \big\rangle \Big] \vol, \\
\delt \Big( \int_M \langle T, \Pop T \rangle \vol \Big) & = \int_M \Big[ 2 \nab{i} A_{ji} \ph_{pqj} T_{pq} - 2 \nab{j} A_{ii} \ph_{pqj} T_{pq} + 2 \nab{i} A_{jp} \ph_{ijq} T_{pq} \\
& \qquad \qquad {} - 2 \nab{i} A_{jq} \ph_{ijp} T_{pq} + \big\langle \langle T, \Pop T \rangle g - 2 (\Pop T) T^t, A \big\rangle \Big] \vol.
\end{align*}
\end{cor}
\begin{proof}
This is immediate from Proposition~\ref{prop:evolution-torsion-scalars} and Lemma~\ref{lemma:density-ev}.
\end{proof}

We want to integrate by parts on terms involving $\nabla A$ so that we can write the evolution equations~\eqref{eq:evolution-torsion-scalars} in the form $\int_M \big\langle \, \cdot \, , A \big\rangle \vol$. We need to make use of the notation $\KK{a}$ for various contractions of $\nab{} T$ with $\ph$ introduced in Definition~\ref{defn:K}.

\begin{lemma} \label{lemma:functionals-IBP}
There are nine distinct $\nabla A$ terms in~\eqref{eq:evolution-torsion-scalars}. Using integration by parts, they are:
\begin{align*}
(\tr T) \nab{p} A_{ij} \ph_{pij} & = \Div(\cdot) + \langle - (\nab{} \tr T) \hk \ph + (\tr T) \Pop T, A \rangle, \\
\nab{i} A_{pj} \ph_{ijq} T_{pq} & = \Div(\cdot) + \langle - T (\Pop T) - \KK{2}, A \rangle, \\
\nab{i} A_{jp} \ph_{ijq} T_{pq} & = \Div(\cdot) + \langle (\Pop T) T^t - \KK{2}^t, A \rangle, \\
\nab{p} A_{ij} \ph_{ijq} T_{pq} & = \Div(\cdot) + \langle - (\Div T) \hk \ph, A \rangle, \\
\nab{i} A_{qj} \ph_{ijp} T_{pq} & = \Div(\cdot) + \langle -T^t (\Pop T) + \KK{3}, A \rangle, \\
\nab{i} A_{jq} \ph_{ijp} T_{pq} & = \Div(\cdot) + \langle (\Pop T) T + \KK{3}^t, A \rangle, \\
\nab{q} A_{ij} \ph_{ijp} T_{pq} & = \Div(\cdot) + \langle \Pop (T^2) - (\Div T^t) \hk \ph, A \rangle, \\
\nab{i} A_{ji} \ph_{pqj} T_{pq} & = \Div(\cdot) + \langle (\Pop T) T^t - \KK{1}^t, A \rangle, \\
\nab{j} A_{ii} \ph_{pqj} T_{pq} & = \Div(\cdot) + \big\langle \langle T, \Pop T \rangle g - \langle \nab{} T, \ph \rangle g, A \big\rangle.
\end{align*}
\end{lemma}
\begin{proof}
Recall $\nab{i} \ph_{jkl} = T_{ip} \ps_{pjkl}$ from~\eqref{eq:nablaph}, which we use repeatedly. We compute
\begin{align*}
(\tr T) \nab{p} A_{ij} \ph_{pij} & = \nab{p} \big( (\tr T) A_{ij} \ph_{pij} \big) - (\nab{p} \tr T) A_{ij} \ph_{pij} - (\tr T) A_{ij} \nab{p} \ph_{pij} \\
& = \Div(\cdot) - \big( (\nabla \tr T) \hk \ph \big){}_{ij} A_{ij} - (\tr T) A_{ij} T_{pm} \ps_{mpij} \\
& = \Div(\cdot) - \langle (\nabla \tr T) \hk \ph, A \rangle + \langle (\tr T) \Pop T, A \rangle,
\end{align*}
yielding the first equation.

Similarly we have
\begin{align*}
\nab{i} A_{pj} \ph_{ijq} T_{pq} & = \Div(\cdot) - A_{pj} T_{im} \ps_{mijq} T_{pq} - A_{pj} \ph_{ijq} \nab{i} T_{pq} \\
& = \Div(\cdot) - \langle T (\Pop T), A \rangle - \KK{2}_{pj} A_{pj},
\end{align*}
which is the second equation. The third equation follows by replacing $A_{pj}$ with $A_{jp}$ in the above.

Observing that $T_{pm} T_{pq}$ is symmetric in $m,q$, we have
\begin{align*}
\nab{p} A_{ij} \ph_{ijq} T_{pq} & = \Div(\cdot) - A_{ij} T_{pm} \ps_{mijq} T_{pq} - A_{ij} \ph_{ijq} \nab{p} T_{pq} \\
& = \Div(\cdot) - 0 - (\Div T)_q \ph_{qij} A_{ij},
\end{align*}
yielding the fourth equation. The fifth and sixth equations follow from the second and third, respectively, by replacing $T_{pq}$ by $T_{qp}$ in the computations.

Continuing in the same way, we have
\begin{align*}
\nab{q} A_{ij} \ph_{ijp} T_{pq} & = \Div(\cdot) - A_{ij} T_{qm} \ps_{mijp} T_{pq} - A_{ij} \ph_{ijp} \nab{q} T_{pq} \\
& = \Div(\cdot) + (\Pop T^2)_{ij} A_{ij} - (\Div T^t)_p \ph_{pij} A_{ij},
\end{align*}
yielding the seventh equation. The eighth equation is obtained similarly.

Finally, we have
\begin{align*}
\nab{j} A_{ii} \ph_{pqj} T_{pq} & = \Div(\cdot) - A_{ii} T_{jm} \ps_{mpqj} T_{pq} - A_{ii} \ph_{pqj} \nab{j} T_{pq} \\
& = \Div(\cdot) + A_{ii} \langle T, \Pop T \rangle - \nab{j} T_{pq} \ph_{jpq} A_{ii},
\end{align*}
which simplifies to the ninth equation.
\end{proof}

\begin{prop} \label{prop:evolution-torsion-quantities-2}
Let $\ph$ be a time-dependent family of $\G$-structures evolving by the flow~\eqref{eq:general-flow-A}. We have the following evolution equations for various quadratic integral quantities obtained from the torsion:
\begin{align*}
\delt \Big( \int_M (\tr T)^2 \vol \Big) & = \int_M \big\langle (\nab{} \tr T) \hk \ph + (\tr T)^2 g - 2 (\tr T) T^t - (\tr T) \Pop T, A \big\rangle \vol, \\
\delt \Big( \int_M |T|^2 \vol \Big) & = \int_M \big\langle (\Div T) \hk \ph - 2 (\KK{2})_{\symm} + |T|^2 g - 2 T T^t - T (\Pop T) + (\Pop T) T^t, A \big\rangle \vol, \\
\delt \Big( \int_M \langle T, T^t \rangle \vol \Big) & = \int_M \Big[ \big\langle (\Div T^t) \hk \ph + 2 (\KK{3})_{\symm} + \langle T, T^t \rangle g - 2 (T^t)^2 \\
& \qquad \qquad {} + (\Pop T) T - T^t (\Pop T) - \Pop(T^2), A \big\rangle \Big] \vol, \\
\delt \Big( \int_M \langle T, \Pop T \rangle \vol \Big) & = \int_M \Big[ \big\langle - 2 \, \KK{1}^t - 2 \, \KK{2}^t - 2 \, \KK{3}^t + 2 (\tr \KK{a}) g - \langle T, \Pop T \rangle g \\
& \qquad \qquad {} + 2 (\Pop T) T^t - 2 (\Pop T) T, A \big\rangle \Big] \vol.
\end{align*}
\end{prop}
\begin{proof}
This follows from Corollary~\ref{cor:evolution-torsion-quantities}, Lemma~\ref{lemma:functionals-IBP}, and the divergence theorem. For the last equation we also use $\langle \nab{}T, \ph \rangle = \tr \KK{a}$ from~\eqref{eq:divVT}.
\end{proof}

The evolution equations in Proposition~\ref{prop:evolution-torsion-quantities-2} can be simplified further, because the $2$-tensors $\KK{2}$ and $\KK{3}$ obtained from $\nab{} T$ in Definition~\ref{defn:K} can actually be expressed in terms of curvature and lower order terms which are quadratic in torsion. We derive these relations in Section~\ref{sec:g2bianchi-revisited} by decomposing the $\G$-Bianchi identity into independent components, and revisit these torsion functionals in Section~\ref{sec:functionals-revisited}. Before we can do any of this, we first need a better understanding of the representation theory of $\G$, in a very concrete and computationally explicit way, which we do in the next section.

\section{More $\G$-representation theory} \label{sec:more-rep-theory}

In this section we investigate more deeply the representation theory of $\G$. In particular, we derive explicit formulas for the orthogonal projections onto the irreducible summands of various $\G$-representations. These results are used in Section~\ref{sec:curvature-torsion} to describe the decompositions of the Riemann curvature tensor $\tRm$ and the covariant derivative $\nab{} T$ of the torsion into irreducible components, to determine the Euler--Lagrange equations of certain quadratic torsion functionals, and to classify the independent second-order differential invariants of a $\G$-structure for the purposes of identifying all possible quasilinear second-order heat-like flows of $\G$-structures.

\subsection{The basic tool for describing tensor product decompositions} \label{sec:basic-tool}

The basic tool we employ repeatedly is the following elementary result.

\begin{lemma} \label{lemma:basic-tool}
Let $V$ and $W$ be finite-dimensional real vector spaces equipped with positive definite inner products, and suppose that
\begin{equation*}
V = V_1 \operp{} \cdots \operp V_m
\end{equation*}
is an \emph{orthogonal} direct sum of subspaces. Let $\iota \colon V \to W$ and $\rho \colon W \to V$ be linear maps. Suppose that for every $1 \leq k \leq m$, there exist $b_k, c_k$ \emph{both nonzero}, such that for all $v_k \in V_k$ and $w \in W$, we have
\begin{equation} \label{eq:basic-condition}
\mathrm{(i)} \,\, \rho \iota v_k = b_k v_k \quad \text{ and } \quad \mathrm{(ii)} \,\, \langle \rho w, v_k \rangle = c_k \langle w, \iota v_k \rangle.
\end{equation}
Then in fact we have an isomorphism of $W$ with an \emph{orthogonal} direct sum
\begin{equation} \label{eq:basic-tool-result}
W \cong (\ker \rho) \operp V = (\ker \rho) \operp V_1 \operp \cdots \operp V_k.
\end{equation}
\end{lemma}
\begin{proof}
Observe first that the two conditions in~\eqref{eq:basic-condition} can be expressed as
\begin{equation*}
\rho \iota = \oplus_{k=1}^m b_k \Id_{V_k} \quad \text{ and } \quad \rho^*|_{V_k} = c_k \iota|_{V_k}
\end{equation*}
where $\Id_{V_k}$ is the identity operator on $V_k$. It is clear that the first condition implies that $\iota$ is injective and $\rho$ is surjective. Let $w \in W$, and write
\begin{equation*}
w = \big( w - \iota \big( \textstyle{\sum_{k=1}^m} \tfrac{1}{b_k} (\rho w)_k \big) \big) + \iota \big( \textstyle{\sum_{k=1}^m} \tfrac{1}{b_k} (\rho w)_k \big)
\end{equation*}
where $(\rho w)_k$ denotes the component of $\rho w$ in $V_k$. The second term is in $\im \iota$ and, since $\rho \iota = \oplus_{k=1}^m b_k \Id_{V_k}$, the first term is in $\ker \rho$. Thus $W = (\ker \rho) + (\im \iota)$. If $w \in (\ker \rho) \cap (\im \iota)$, then $w = \iota v$ and $\rho w = \rho \iota v = \sum_{k=1}^m b_k v_k = 0$, so $v = 0$ and thus $w = 0$. Hence $(\ker \rho) \cap (\im \iota) = \{ 0 \}$, and $W = (\ker \rho) \oplus (\im \iota)$.

Consider the second condition in~\eqref{eq:basic-condition}. For $w \in \ker \rho$, it says $\langle w, \iota v_k \rangle = 0$ for all $v_k \in V_k$, and thus $w \in (\im \iota)^{\perp}$. Comparing dimensions, we have $(\ker \rho) = (\im \iota)^{\perp}$. For $w = \iota \tilde v_l$ with $\tilde v_l \in V_l$, we have
\begin{equation} \label{eq:iota-conf-isom}
\langle \iota \tilde v_l, \iota v_k \rangle = \tfrac{1}{c_k} \langle \rho \iota \tilde v_l, v_k \rangle = \tfrac{b_l}{c_k} \langle \tilde v_l, v_k \rangle.
\end{equation}
Thus $(\iota V_l) \perp (\iota V_k)$ for $l \neq k$, and hence
\begin{equation*}
\im \iota = (\iota V_1) \operp \cdots \operp (\iota V_n) \cong V_1 \operp \cdots \operp V_n. \qedhere
\end{equation*}
\end{proof}

\begin{rmk} \label{rmk:iota-conf-isom}
We note from~\eqref{eq:iota-conf-isom} with $k = l$ that $\rest{\iota}{V_k} \colon V_k \to \iota(V_k)$ is an isometry, up to a positive constant. In particular, it is always the case that $b_k$ and $c_k$ have the same sign.
\end{rmk}

Of course, Lemma~\ref{lemma:basic-tool} can be applied fibrewise to smooth tensors on a Riemannian manifold $(M, g)$. We use Lemma~\ref{lemma:basic-tool} several times in the rest of Section~\ref{sec:more-rep-theory} to describe the decompositions of various tensor products of $\G$ representations, which we then use to decompose the curvature, torsion, and the covariant derivative of torsion in Section~\ref{sec:curvature-torsion}.

As an example, we show here how to use Lemma~\ref{lemma:basic-tool} to quickly recover the well-known decomposition of Riemann curvature into $\mathrm{O}(n)$ representations. Assume that $n = \dim M \geq 3$. Recall that the space $\K$ of \emph{curvature tensors} on $(M, g)$ is the subspace $\K$ of $\cS^2 (\Lambda^2) = \Gamma(\mathrm{S}^2 (\Lambda^2 T^* M))$ of elements satisfying the first Bianchi identity. That is, if $U_{ijkl}$ is a curvature tensor, then
\begin{equation} \label{eq:curv-symmetries}
U_{ijkl} = - U_{jikl} = - U_{ijlk} = U_{klij}, \quad \text{and} \quad U_{ijkl} + U_{jkil} + U_{kijl} = 0.
\end{equation}
The space $\Omega^4$ of $4$-forms on $M$ is a subspace of $\cS^2 (\Lambda^2)$, and it is easy to see that the first Bianchi identity for $U$ is equivalent to saying that $U$ is (pointwise) orthogonal to $\Omega^4$. That is,
\begin{equation} \label{eq:curvature-tensors-decomp}
\cS^2 (\Lambda^2) = \Omega^4 \operp \K.
\end{equation}

Define a linear map $\iota_g \colon \cS^2 \to \K$ by
\begin{equation} \label{eq:iota-g}
(\iota_g h)_{ijkl} = g_{il} h_{jk} + g_{jk} h_{il} - g_{ik} h_{jl} - g_{jl} h_{ik}.
\end{equation}
It is easy to check that $U = \iota_g h$ satisfies the conditions~\eqref{eq:curv-symmetries}, so $\iota_g$ does indeed map into $\K$. [The tensor $\iota_g h$ is usually written $g \owedge h$, and is called the \emph{Kulkarni--Nomizu} product of $g$ with $h$.]

Define a linear map $\rho_g \colon \K \to \cS^2$ by
\begin{equation} \label{eq:rho-g}
(\rho_g U)_{jk} = U_{ijkl} g_{il} = U_{ljkl}.
\end{equation}
To verify that $\rho_g U$ is indeed symmetric, we use~\eqref{eq:curv-symmetries} to compute
$$ (\rho_g U)_{kj} = U_{lkjl} = U_{jllk} = U_{ljkl} = (\rho_g U)_{jk} $$
as claimed. We call $\rho_g U$ the \emph{Ricci contraction} of $U$ with respect to $g$, because it yields the Ricci curvature when applied to the Riemann curvature tensor of $g$.

Composing these two maps, we obtain
\begin{align} \nonumber
(\rho_g \iota_g h)_{jk} & = (\iota_g h)_{ljkl} \\ \nonumber
& = g_{ll} h_{jk} + g_{jk} h_{ll} - g_{lk} h_{jl} - g_{jl} h_{lk} \\ \nonumber
& = n h_{jk} + (\tr h) g_{jk} - h_{jk} - h_{jk} \\ \label{eq:curvature-ex-1}
& = (n - 2) h_{jk} + (\tr h) g_{jk}.
\end{align}
Recall that we have an orthogonal decomposition
$$ \underbrace{\cS^2}_V = \underbrace{\Omega^0 g}_{V_1} \operp \underbrace{\cS^2_0}_{V_2}. $$
It follows from~\eqref{eq:curvature-ex-1} that $\rho_g \iota_g g = (2n-2) g$, and that $\rho_g \iota_g h = (n-2) h$ for $h \in \cS^2_0$. Thus condition $\mathrm{(i)}$ of~\eqref{eq:basic-condition} is satisfied with $b_{1} = 2n-2$ and $b_{2} = n-2$.

Moreover, using the symmetries~\eqref{eq:curv-symmetries}, we have
\begin{align*}
\langle U, \iota_g h \rangle & = U_{ijkl} (\iota_g h)_{ijkl} \\
& = U_{ijkl} ( g_{il} h_{jk} + g_{jk} h_{il} - g_{ik} h_{jl} - g_{jl} h_{ik} ) \\
& = (U_{ijkl} g_{il}) h_{jk} + (U_{jilk} g_{jk}) h_{il} + (U_{ijlk} g_{ik}) h_{jl} + (U_{jikl} g_{jl}) h_{ik} \\
& = (\rho_g U)_{jk} h_{jk} + (\rho_g U)_{il} h_{il} + (\rho_g U)_{jl} h_{jl} + (\rho_g U)_{ik} h_{ik} \\
& = 4 \langle \rho_g U, h \rangle.
\end{align*}
Thus condition $\mathrm{(ii)}$ of~\eqref{eq:basic-condition} is satisfied with $c_{1} = c_{2} = \frac{1}{4}$.

We can therefore invoke Lemma~\ref{lemma:basic-tool} to conclude that we have a pointwise orthogonal decomposition
\begin{equation} \label{eq:curvature-classical-decomp}
\K = ( \ker \rho_g ) \operp \iota ( \Omega^0 g ) \operp \iota ( \cS^2_0 ).
\end{equation}
We also get from~\eqref{eq:iota-conf-isom} that if $h = \frac{1}{n} (\tr h) g + h_0$ and $f = \frac{1}{n} (\tr f) g + f_0$, with $h_0, f_0 \in \cS^2_0$, then we have
\begin{align*}
\langle \iota_g h, \iota_g f \rangle & = 4(2n-2) \langle \tfrac{1}{n} (\tr h) g, \tfrac{1}{n} (\tr f) g \rangle + 4(n-2) \langle h_0, f_0 \rangle \\
& = \tfrac{8(n-1)}{n} (\tr h) (\tr f) + 4(n-2) \langle h_0, f_0 \rangle.
\end{align*}
The space $\W = \ker \rho_g$ is called the space of \emph{Weyl tensors} on $(M, g)$. [These are the curvature tensors with vanishing Ricci curvature.]

Explicitly, for any $U \in \K$, we can write $U = U_{\W} + U_1 + U_0$ with
\begin{equation} \label{eq:curvature-ex-decomp}
\begin{aligned}
U_1 & = \tfrac{1}{2n-2} \iota_g (\rho_g U)_{1}, \\
U_0 & = \tfrac{1}{n-2} \iota_g (\rho_g U)_{0}, \\
U_{\W} & = U - U_1 - U_0,
\end{aligned}
\end{equation}
where $(\rho_g U)_1 = \frac{1}{n} (\tr \rho_g U) g \in \Omega^0 g$ is the pure trace part of $\rho_g U$ and $(\rho_g U)_0 = \rho_g U - (\rho_g U)_1 \in \cS^2_0$ is the trace-free part of $\rho_g U$.

Applied to the Riemann curvature tensor $U = \tRm$ of $g$, these components correspond, respectively, to the scalar curvature, the traceless Ricci curvature, and the Weyl curvature of $g$. In particular, writing $\tRc = \rho_g (\tRm)$ for the Ricci curvature, $R = \tr (\tRc)$ for the scalar curvature, and $\tW = \tRm_{\W}$ for the Weyl curvature, we have
\begin{equation} \label{eq:curvature-n}
\tRm = \tfrac{1}{2n(n-1)} R \, \iota_g g + \tfrac{1}{(n-2)} \, \iota_g (\tRc^0) + \tW.
\end{equation}

Some of these formulas (specific to $n=7$) are needed in Sections~\ref{sec:identities-Pop} and~\ref{sec:curvature-torsion}, specifically: 
\begin{equation} \label{eq:curvature-7}
\tRm = \tfrac{1}{84} R \, \iota_g g + \tfrac{1}{5} \, \iota_g (\tRc^0) + \tW,
\end{equation}
and
\begin{equation} \label{eq:iota-g-rho-g-7}
\rho_g \iota_g h = 5 h + (\tr h) g, \qquad \langle U, \iota_g h \rangle = 4 \langle \rho_g U, h \rangle.
\end{equation}

\subsection{Basic facts about representations of $\G$} \label{sec:rep-theory-basics}

In this section we review without proof some facts about finite-dimensional \emph{irreducible} representations of $\G$, and the decomposition of tensor products of such representations into irreducible summands. These can be verified using the \emph{LiE package}, available online~\cite{Lie}. (See Fulton--Harris~\cite{FH} for an introduction to representation theory.)

In the remaining parts of Section~\ref{sec:more-rep-theory} we give \emph{explicit concrete descriptions} of these decompositions. The only ingredient missing is the demonstration that the decompositions are not further reducible. The reader willing to accept this need only glance at equations~\eqref{eq:7-7},~\eqref{eq:7-14},~\eqref{eq:7-27}, and~\eqref{eq:sym14} in this section and move on to their explicit descriptions.

As the rank of $\G$ is $2$, the irreducible representations of $\G$ are indexed by their highest weight, which is an ordered pair $(p,q)$ with $p,q$ nonnegative integers. The first few such irreducible representations with their dimensions are given in Table~\ref{reps-table}.

\begin{table}[h!]
\begin{center}
\begin{tabular}{|c||c|c|c|c|c|c|c|c|c|}
\hline
Highest weight & (0,0) & (1,0) & (0,1) & (2,0) & (1,1) & (0,2) & (3,0) & (2,1) & $\cdots$ \\ \hline
Dimension/label & $\mb{1}$ & $\mb{7}$ & $\mb{14}$ & $\mb{27}$ & $\mb{64}$ & $\mb{77}$ & $\mb{77^*}$ & $\mb{189}$ & $\cdots$ \\
\hline
\end{tabular}
\end{center}
\caption{Dimensions of some irreducible representations of $\G$ by highest weight} \label{reps-table}
\end{table}

We make several remarks concerning Table~\ref{reps-table}:
\begin{itemize} \setlength\itemsep{-1mm}
\item The $1$-dimensional representation $\mb{1}$ is the \emph{trivial} representation.
\item The $7$-dimensional representation $\mb{7}$ is the \emph{standard} representation of $\G$ on $\R^7$ when $\G$ is viewed as a subgroup of $\SO{7} \subset \GL{7, \R}$.
\item The $14$-dimensional representation $\mb{14}$ is the \emph{adjoint} representation of $\G$ on its Lie algebra $\mathfrak{g}_2$.
\item The $27$-dimensional representation $\mb{27}$ is isomorphic to the \emph{traceless symmetric $2$-tensors} on $\R^7$ with the Euclidean inner product.
\item The $64$-dimensional representation $\mb{64}$ is described concretely in two different ways in Sections~\ref{sec:7-14} and~\ref{sec:7-27}, which are related in Section~\ref{sec:equiv-64}. It arises in the decompositions of both $\tRm$ and $\nabla T$.
\item There are two \emph{nonisomorphic} $77$-dimensional representations of $\G$, which we label by $\mb{77}$ for highest weight $(0,2)$ and $\mb{77^*}$ for highest weight $(3,0)$. These \emph{both arise} in practice, with $\mb{77}$ arising in the decomposition of $\tRm$ and $\mb{77^*}$ arising in the decomposition of $\nabla T$. The representations $\mb{77^*}$ and $\mb{77}$ are described concretely, in Sections~\ref{sec:7-27} and~\ref{sec:S2-14}, respectively. In particular, the space $\mb{77^*}$ is an irreducible summand in $\cS^3$, the space of fully symmetric cubics.
\end{itemize}

Using the \emph{LiE package}, we get the following decompositions of the tensor products of irreducible $\G$ representations into irreducible summands. We have
\begin{align} \label{eq:7-7}
\underbrace{\mb{7} \otimes \mb{7}}_{49} & = \mb{1} \oplus \mb{27} \oplus \mb{7} \oplus \mb{14}, \\ \label{eq:7-14}
\underbrace{\mb{7} \otimes \mb{14}}_{98} & = \mb{64} \oplus \mb{27} \oplus \mb{7}, \\ \label{eq:7-27}
\underbrace{\mb{7} \otimes \mb{27}}_{189} & = (\mb{77^*} \oplus \mb{7}) \oplus (\mb{64} \oplus \mb{27} \oplus \mb{14}), \\ \label{eq:sym14}
\underbrace{\Sym^2 (\mb{14})}_{105} & = \mb{77} \oplus \mb{1} \oplus \mb{27}.
\end{align}
The decomposition~\eqref{eq:7-7} was described concretely in Section~\ref{sec:forms}. We describe the three remaining decompositions~\eqref{eq:7-14},~\eqref{eq:7-27}, and~\eqref{eq:sym14} concretely in the rest of Section~\ref{sec:more-rep-theory}. The reason that the right-hand side of~\eqref{eq:7-27} is grouped the way it is becomes evident in Section~\ref{sec:7-27}.

\begin{rmk} \label{rmk:Lambda2-of-14}
The \emph{LiE package} also shows that
$$ \underbrace{\Lambda^2 (\mb{14})}_{91} = \mb{77^*} \oplus \mb{14}. $$
Similar methods can be applied to understand the splitting concretely, but we do not require this here.
\end{rmk}

\begin{rmk} \label{rmk:Weyl-decomp}
We also have a decomposition
$$ \Sym^2(\Lambda^2) = \Sym^2(\mb{7} \oplus \mb{14}) = \Sym^2(\mb{7}) \oplus (\mb{7} \otimes \mb{14}) \oplus \Sym^2(\mb{14}). $$
It follows from~\eqref{eq:7-7},~\eqref{eq:7-14}, and~\eqref{eq:sym14} that the above becomes
\begin{equation*}
(\mb{1} \oplus \mb{27}) \oplus (\mb{64} \oplus \mb{7} \oplus \mb{27}) \oplus (\mb{77} \oplus \mb{1} \oplus \mb{27}).
\end{equation*}
Recall from~\eqref{eq:curvature-tensors-decomp} that the Riemannian first Bianchi identity identifies the space $\K$ of curvature tensors as the orthogonal complement of $\Lambda^4 = \mb{1} \oplus \mb{7} \oplus \mb{27}$ in $\Sym^2 (\Lambda^2)$. Thus, we must have
\begin{equation*}
\K \cong \underbrace{\mb{1} \oplus \mb{27}}_{\text{Ricci}} \oplus \underbrace{\mb{27} \oplus \mb{64} \oplus \mb{77}}_{\text{Weyl}}.
\end{equation*}
That is, in the presence of a $\G$-structure, the Weyl tensor $W$ decomposes into three independent components $W_{27} + W_{64} + W_{77}$. We describe $W_{27}$ concretely at the end of Section~\ref{sec:identities-Pop}, and discuss $W_{64}$ and $W_{77}$ at the end of Section~\ref{sec:curvature-decomp}.
\end{rmk}

In the remainder of Section~\ref{sec:more-rep-theory} we derive explicit concrete descriptions of the decompositions~\eqref{eq:7-14},~\eqref{eq:7-27}, and~\eqref{eq:sym14}. We also establish an explicit isomorphism between two distinct concrete realizations of the 64-dimensional representation of $\G$ in Proposition~\ref{prop:two-64-desc}, and we develop many useful identities for elements of $\cS^2 (\Lambda^2)$ in Section~\ref{sec:identities-Pop} that are needed to understand the decomposition of curvature.

\subsection{The decomposition $\mb{7} \otimes \mb{14} = \mb{64} \oplus \mb{27} \oplus \mb{7}$} \label{sec:7-14}

Let $V = \Omega^3_7 \oplus \Omega^3_{27} \cong \mb{7} \oplus \mb{27}$, and let $W = \Omega^1_7 \otimes \Omega^2_{14}$. An element $\beta \in W$ can be expressed as $\beta_{ijk} e_i \otimes e_j \otimes e_k$, where
\begin{equation} \label{eq:7-14-conditions}
\beta_{ijk} = - \beta_{ikj}, \quad \text{ and } \quad \beta_{ijk} \ph_{ajk} = 0.
\end{equation}
Define a map $\rho \colon W \to \Omega^3$ by skew-symmetrization. That is,
\begin{equation} \label{eq:7-14-rho}
(\rho \beta)_{ijk} = \beta_{ijk} + \beta_{jki} + \beta_{kij}. 
\end{equation}
It is clear that $\rho \beta \in \Omega^3$, since $\beta_{ijk}$ is already skew in $j,k$. We claim that in fact $\rho \beta \in V = \Omega^3_7 \oplus \Omega^3_{27}$. To establish this claim, by Corollary~\ref{cor:diamondinverse} it suffices to show that $(\rho \beta)_{ijk} \ph_{ijk} = 0$. But we have
\begin{equation*}
(\rho \beta)_{ijk} \ph_{ijk} = 3 \beta_{ijk} \ph_{ijk} = 0
\end{equation*}
since $\beta_{ijk} \ph_{ajk} = 0$. Thus indeed $\rho$ maps $W$ into $V$.

Define a map $\iota \colon V \to W$ by
\begin{equation} \label{eq:7-14-iota}
(\iota \gamma)_{ijk} = 4 \gamma_{ijk} + \gamma_{ipq} \ps_{pqjk}.
\end{equation}
Note that by~\eqref{eq:omega2-proj}, up to a factor of $6$, the map $\iota$ is just the projection of the skew $j,k$ indices of $\gamma_{ijk}$ onto their $\Omega^2_{14}$ component. Thus by construction $\iota$ maps into $W$.

Now we consider $\rho \iota \colon V \to V$. First we note that $\gamma = A \diamond \ph$ for some unique $A = A_{27} + A_{7} \in \cS^2_0 \oplus \Omega^2_7$. This means $\gamma_{ijk} = A_{im} \ph_{mjk} + A_{jm} \ph_{imk} + A_{km} \ph_{ijm}$. Then using~\eqref{eq:vecA} we compute
\begin{align*}
\gamma_{ipq} \ps_{pqjk} & = (A_{im} \ph_{mpq} + A_{pm} \ph_{imq} + A_{qm} \ph_{ipm}) \ps_{pqjk} \\
& = - 4 A_{im} \ph_{mjk} + 2 A_{pm} ( g_{ip} \ph_{mjk} + g_{ij} \ph_{pmk} + g_{ik} \ph_{pjm} ) \\
& \qquad {} - 2 A_{pm} ( g_{mp} \ph_{ijk} + g_{mj} \ph_{pik} + g_{mk} \ph_{pji} ) \\
& = - 4 A_{im} \ph_{mjk} + 2 A_{im} \ph_{mjk} + 2 (\Vop A)_k g_{ij} - 2 (\Vop A)_j g_{ik} \\
& \qquad {} - 2 (\tr A) \ph_{ijk} - 2 A_{pj} \ph_{pik} - 2 A_{pk} \ph_{pji}.
\end{align*}
The first two terms combine, and $\tr A = 0$ in this case, so we have
\begin{equation} \label{eq:7-14-temp}
\gamma_{ipq} \ps_{pqjk} = - 2 A_{im} \ph_{mjk} + 2 A_{pj} \ph_{ipk} + 2 A_{pk} \ph_{ijp} + 2 (\Vop A)_k g_{ij} - 2 (\Vop A)_j g_{ik}.
\end{equation}
Cyclically permuting $i,j,k$ and adding, the terms with $\Vop A$ cancel in pairs, and we obtain
\begin{equation*}
\gamma_{ipq} \ps_{pqjk} + \gamma_{jpq} \ps_{pqki} + \gamma_{kpq} \ps_{pqij} = -2 (A \diamond \ph)_{ijk} + 4 (A^t \diamond \ph)_{ijk}.
\end{equation*}
Using the above expression, we have
\begin{align*}
(\rho \iota \gamma)_{ijk} & = (\iota \gamma)_{ijk} + (\iota \gamma)_{jki} + (\iota \gamma)_{kij} \\
& = 4 \gamma_{ijk} + \gamma_{ipq} \ps_{pqjk} + 4 \gamma_{jki} + \gamma_{jpq} \ps_{pqki} + 4 \gamma_{kij} + \gamma_{kpq} \ps_{pqij} \\
& = 12 \gamma_{ijk} - 2 (A \diamond \ph)_{ijk} + 4 (A^t \diamond \ph)_{ijk} \\
& = 14 (A_{27} \diamond \ph)_{ijk} + 6 (A_{7} \diamond \ph)_{ijk}.
\end{align*}
Thus condition $\mathrm{(i)}$ of~\eqref{eq:basic-condition} is satisfied with $b_{27} = 14$ and $b_{7} = 6$.

Similarly we compute
\begin{equation*}
\langle \rho \beta, \gamma \rangle = (\rho \beta)_{ijk} \gamma_{ijk} = (\beta_{ijk} + \beta_{jki} + \beta_{kij}) \gamma_{ijk} = 3 \beta_{ijk} \gamma_{ijk},
\end{equation*}
and using~\eqref{eq:omega214desc} we have
\begin{align*}
\langle \beta, \iota \gamma \rangle & = \beta_{ijk} (\iota \gamma)_{ijk} \\
& = \beta_{ijk} ( 4 \gamma_{ijk} + \gamma_{ipq} \ps_{pqjk} ) \\
& = 4 \beta_{ijk} \gamma_{ijk} + 2 \beta_{ipq} \gamma_{ipq} = 6 \beta_{ijk} \gamma_{ijk}.
\end{align*}
Thus we have $\langle \rho \beta, \gamma \rangle = \frac{1}{2} \langle \beta, \iota \gamma \rangle$, so condition $\mathrm{(ii)}$ of~\eqref{eq:basic-condition} is satisfied with $c_{27} = c_{7} = \frac{1}{2}$.

We can therefore invoke Lemma~\ref{lemma:basic-tool} to conclude that
\begin{equation*}
\mb{7} \otimes \mb{14} \cong \mb{64} \oplus (\mb{27} \oplus \mb{7}),
\end{equation*}
where explicitly we have $\beta = \beta_{64} + \beta_{27} + \beta_{7}$ with
\begin{equation} \label{eq:7-14-decomp}
\begin{aligned}
\beta_{27} & = \tfrac{1}{14} \iota (\rho \beta)_{27}, \\
\beta_7 & = \tfrac{1}{6} \iota (\rho \beta)_{7}, \\
\beta_{64} & = \beta - \beta_{27} - \beta_{7},
\end{aligned}
\end{equation}
where $(\rho \beta)_{27}$ and $(\rho \beta)_{7}$ are given by Corollary~\ref{cor:diamondinverse}. In particular, the $64$-dimensional representation corresponds to $\ker \rho$, and thus concretely we have
\begin{equation} \label{eq:64-desc-1}
\beta_{ijk} \in \mb{64} \, \iff \, \begin{cases} \beta_{ijk} = - \beta_{ikj}, \\ \beta_{ijk} \ph_{ajk} = 0, \\ \beta_{ijk} + \beta_{jki} + \beta_{kij} = 0. \end{cases}
\end{equation}

\subsection{The decomposition $\mb{7} \otimes \mb{27} = (\mb{77^*} \oplus \mb{7}) \oplus (\mb{64} \oplus \mb{27} \oplus \mb{14})$} \label{sec:7-27}

In order to understand concretely the decomposition $\mb{7} \otimes \mb{27} = (\mb{77^*} \oplus \mb{7}) \oplus (\mb{64} \oplus \mb{27} \oplus \mb{14})$ we apply Lemma~\ref{lemma:basic-tool} three times, so there are three different sets of $(V, W, \rho, \iota)$ in this section. This should not cause confusion.

\textbf{Step One.} We first consider the decomposition $\cS^3$, the space of fully symmetric cubic tensors. An element $h \in \cS^3$ can be written as $h = h_{ijk} e_i \otimes e_j \otimes e_k$ where $h_{ijk}$ is fully symmetric. Define a map $\rho \colon \cS^3 \to \Omega^1_7$ by $(\rho h)_k = h_{iik}$, which can be thought of as the ``trace'' of a symmetric cubic, yielding a $1$-form. Define a map $\iota \colon \Omega^1_7 \to \cS^3$ by $(\iota X)_{ijk} = X_i g_{jk} + X_j g_{ki} + X_k g_{ij}$. We compute
\begin{equation*}
(\rho \iota X)_k = (\iota X)_{iik} = X_i g_{ik} + X_i g_{ki} + X_k g_{ii} = 9 X_k.
\end{equation*}
We also have
\begin{equation*}
\langle \rho h, X \rangle = (\rho h)_k X_k = h_{iik} X_k
\end{equation*}
and
\begin{equation*}
\langle h, \iota X \rangle = h_{ijk} (\iota X)_{ijk} = h_{ijk} (X_i g_{jk} + X_j g_{ki} + X_k g_{ij}) = 3 h_{iik} X_k.
\end{equation*}
Thus Lemma~\ref{lemma:basic-tool} applies with $b = 9$ and $c = \frac{1}{3}$, so we deduce that $\cS^3 \cong \cS^3_0 \oplus \Omega^1_7$ where $\cS^3_0 = (\ker \rho) \cong \mb{77^*}$. Explicitly if $h \in \cS^3$ we write
\begin{equation} \label{eq:cS3-decomp}
h_{ijk} = \big( h_{ijk} - \tfrac{1}{9} (h_{ppi} g_{jk} + h_{ppj} g_{ki} + h_{ppk} g_{ij}) \big) + \tfrac{1}{9} (h_{ppi} g_{jk} + h_{ppj} g_{ki} + h_{ppk} g_{ij}),
\end{equation}
where the first term lies in $\cS^3_0 \cong \mb{77^*}$ and the second term lies in $\iota(\Omega^1_7) \cong \mb{7}$, which is the orthogonal complement of $\cS^3_0$ in $\cS^3$. We have shown that
\begin{equation} \label{eq:cS3}
\cS^3 \cong \mb{77^*} \oplus \mb{7}.
\end{equation}

\textbf{Step Two.} Let $V = \cS^3 \cong \mb{77^*} \oplus \mb{7}$, and let $W = \Omega^1_7 \otimes \cS^2_0$. An element $h \in W$ can be expressed as $h_{ijk} e_i \otimes e_j \otimes e_k$, where
\begin{equation*}
h_{ijk} = h_{ikj}, \quad \text{ and } \quad h_{ikk} = 0.
\end{equation*}
Define a map $\rho \colon W \to V$ by symmetrization. That is,
\begin{equation*}
(\rho h)_{ijk} = h_{ijk} + h_{jki} + h_{kij}. 
\end{equation*}
It is clear that $\rho h \in \cS^3$, since $h_{ijk}$ is already symmetric in $j,k$. Define a map $\iota \colon V \to W$ by
$$ (\iota f)_{ijk} = f_{ijk} - \tfrac{1}{7} f_{ipp} g_{jk}, $$
which is just the inclusion of $\cS^3$ into $\Omega^1_7 \otimes \cS^2$ followed by the orthogonal projection onto $\Omega^1_7 \otimes \cS^2_0$.

If $f \in \cS^3_0$, then $\iota f = f$, and hence $\rho \iota f = 3 f$. If $f$ lies in the orthogonal complement of $\cS^3_0$ in $\cS^3$, then by Step One we have $f_{ijk} = X_i g_{jk} + X_j g_{ki} + X_k g_{ij}$ for some unique $1$-form $X$. Hence $f_{ipp} = 9 X_i$, so
$$ (\iota f)_{ijk} = X_i g_{jk} + X_j g_{ki} + X_k g_{ij} - \tfrac{1}{7} (9 X_i) g_{jk} = - \tfrac{2}{7} X_i g_{jk} + X_j g_{ki} + X_k g_{ij}. $$
It follows in this case that
$$ (\rho \iota f)_{ijk} = (\iota f)_{ijk} + (\iota f)_{jki} + (\iota f)_{kij} = \tfrac{12}{7} (X_i g_{jk} + X_j g_{ki} + X_k g_{ij}) = \tfrac{12}{7} f_{ijk}. $$
Thus condition $\mathrm{(i)}$ of~\eqref{eq:basic-condition} is satisfied with $b_{77} = 3$ and $b_{7} = \frac{12}{7}$.

Moreover, for any $f \in V = \cS^3$ and $h \in W = \Omega^1_7 \otimes \cS^2_0$, we have
$$ \langle \rho h, f \rangle = (\rho h)_{ijk} f_{ijk} = (h_{ijk} + h_{jki} + h_{kij}) f_{ijk} = 3 h_{ijk} f_{ijk}, $$
and
$$ \langle h, \iota f \rangle = h_{ijk} f_{ijk} = h_{ijk} ( f_{ijk} - \tfrac{1}{7} f_{ipp} g_{jk} ) = h_{ijk} f_{ijk} - 0, $$
and hence $\langle \rho h, f \rangle = 3 \langle h, \iota f \rangle$. Thus condition $\mathrm{(ii)}$ of~\eqref{eq:basic-condition} is satisfied with $c_{77} = c_{7} = 3$.

We can therefore invoke Lemma~\ref{lemma:basic-tool} to deduce that
\begin{equation*}
\Omega^1_7 \otimes \cS^2_0 \cong (\ker \rho) \oplus \cS^3 \cong \underbrace{(\ker \rho)}_{105} \oplus (\mb{77^*} \oplus \mb{7}).
\end{equation*}

Explicitly if $h \in \Omega^1_7 \otimes \cS^2_0$ we can write $h = h_{105} + h_{77} + h_7$, where
\begin{equation} \label{eq:7-27-step2-decomp}
\begin{aligned}
h_{77} & = \tfrac{1}{3} \iota (\rho h)_{77}, \\
h_{7} & = \tfrac{7}{12} \iota (\rho h)_{7}, \\
h_{105} & = h - h_{77} - h_{7},
\end{aligned}
\end{equation}
where $(\rho h)_{77}$ and $(\rho h)_{7}$ are the trace-free and pure trace parts of $\rho h \in \cS^3 = \cS^3_0 \oplus \Omega^1_7$. The first term $h_{105}$ lies in $\ker \rho$, which is the orthogonal complement of $\cS^3$ in $\Omega^1_7 \otimes \cS^2_0$.

\textbf{Step Three.} The $105$-dimensional space from Step Two that is the orthogonal complement of $\cS^3$ in $\Omega^1_7 \otimes \cS^2_0$ can be decomposed further. Denote this space by $W$ and let $V = \cS^2_0 \oplus \Omega^2_{14} \cong \mb{27} \oplus \mb{14}$. Explicitly, $W$ is characterized by
\begin{equation} \label{eq:7-27-temp-h}
h_{ijk} \in W \, \iff \, \begin{cases} h_{ijk} = h_{ikj}, \\ h_{ikk} = 0, \\ h_{ijk} + h_{jki} + h_{kij} = 0. \end{cases} \end{equation}
Let $h \in W$. Observe from~\eqref{eq:7-27-temp-h} that
\begin{equation*}
0 = h_{iik} + h_{iki} + h_{kii} = 2 h_{iik} + 0,
\end{equation*}
and thus
\begin{equation} \label{eq:7-27-temp-h2}
h_{iik} = 0 \quad \text{ for all $h \in W$}.
\end{equation}

Define a map $\rho \colon V \to \cT^2$ by $(\rho h)_{ja} = h_{ijk} \ph_{iak}$. We claim that in fact $\rho h \in V = \cS^2_0 \oplus \Omega^2_{14}$. To establish this claim, we need to verify that the trace of $\rho h$ and the $\Omega^2_7$ part of $\rho h$ both vanish. We compute using~\eqref{eq:7-27-temp-h} and~\eqref{eq:7-27-temp-h2} that
\begin{equation*}
(\rho h)_{jj} = h_{ijk} \ph_{ijk} = 0,
\end{equation*}
and
\begin{align*}
(\rho h)_{ja} \ph_{jam} & = h_{ijk} \ph_{iak} \ph_{jam} \\
& = h_{ijk} (g_{ij} g_{km} - g_{im} g_{kj} - \ps_{ikjm} ) \\
& = h_{iim} - h_{mkk} - 0 = 0.
\end{align*}
Thus indeed $\rho$ maps $W$ into $V$.

Define a map $\iota \colon V \to \Omega^1_7 \otimes \cS^2$ by $(\iota A)_{ijk} = A_{jp} \ph_{pik} + A_{kp} \ph_{pij}$. We claim that $\iota A \in W$. To establish this claim, we need to verify that the last two conditions in~\eqref{eq:7-27-temp-h} are satisfied. Because $A = A_{27} + A_{14}$, by~\eqref{eq:omega214desc} we have
\begin{equation*}
(\iota A)_{ikk} = A_{kp} \ph_{pik} + A_{kp} \ph_{pik} = 0.
\end{equation*}
We also have
\begin{align*}
(\iota A)_{ijk} + (\iota A)_{jki} + (\iota A)_{kij} & = A_{jp} \ph_{pik} + A_{kp} \ph_{pij} + A_{kp} \ph_{pji} \\
& \qquad {} + A_{ip} \ph_{pjk} + A_{ip} \ph_{pkj} + A_{jp} \ph_{pki} \\
& = 0
\end{align*}
as the terms on the right-hand side cancel in pairs. Thus indeed $\iota$ maps $V$ into $W$.

Now we consider $\rho \iota \colon V \to V$. Using~\eqref{eq:Pop-action} and the fact that $A = A_{27} + A_{14}$ we compute
\begin{align*}
(\rho \iota A)_{ja} & = (\iota A)_{ijk} \ph_{iak} \\
& = (A_{jp} \ph_{pik} + A_{kp} \ph_{pij}) \ph_{iak} \\
& = A_{jp} (- 6 g_{pa}) + A_{kp} (g_{ja} g_{pk} - g_{jk} g_{pa} - \ps_{jpak} ) \\
& = - 6 A_{ja} + (\tr A) g_{ja} -A_{ja} - (\Pop A)_{ja} \\
& = - 7 (A_{27})_{ja} - 9 (A_{14})_{ja}.
\end{align*}
Thus condition $\mathrm{(i)}$ of~\eqref{eq:basic-condition} is satisfied with $b_{27} = -7$ and $b_{14} = -9$.

Similarly we compute
\begin{equation*}
\langle \rho h, A \rangle = (\rho h)_{ja} A_{ja} = h_{ijk} \ph_{iak} A_{ja},
\end{equation*}
and using~\eqref{eq:7-27-temp-h} and relabelling indices, we have
\begin{align*}
\langle h, \iota A \rangle & = h_{ijk} (\iota A)_{ijk} \\
& = h_{ijk} ( A_{jp} \ph_{pik} + A_{kp} \ph_{pij} ) \\
& = - h_{ijk} \ph_{ipk} A_{jp} - h_{ikj} \ph_{ipj} A_{kp} = -2 h_{ijk} \ph_{iak} A_{ja}.
\end{align*}
Thus we have $\langle \rho h, A \rangle = - \frac{1}{2} \langle h, \iota A \rangle$, so condition $\mathrm{(ii)}$ of~\eqref{eq:basic-condition} is satisfied with $c_{27} = c_{14} = - \frac{1}{2}$.

We can therefore invoke Lemma~\ref{lemma:basic-tool} to conclude that
\begin{equation*}
W \cong \mb{64} \oplus (\mb{27} \oplus \mb{14}),
\end{equation*}
where explicitly we have $h = h_{64} + h_{27} + h_{14}$ with
\begin{equation} \label{eq:7-21-decomp}
\begin{aligned}
h_{27} & = - \tfrac{1}{7} \iota (\rho h)_{27}, \\
h_{14} & = - \tfrac{1}{9} \iota (\rho h)_{14}, \\
h_{64} & = h - h^{27} - h^{14},
\end{aligned}
\end{equation}
where $(\rho h)_{27}$ and $(\rho h)_{14}$ are the components of $\rho h$ in $V = \cS^2_0 \oplus \Omega^2_{14}$. In particular, the $64$-dimensional representation corresponds to the kernel of $\rho \colon W \to V$, and thus by the definition of $\rho$ and~\eqref{eq:7-27-temp-h}, concretely we have
\begin{equation} \label{eq:64-desc-2}
h_{ijk} \in \mb{64} \, \iff \, \begin{cases} h_{ijk} = h_{ikj}, \\ h_{ikk} = 0, \\ h_{ijk} + h_{jki} + h_{kij} = 0, \\ h_{ijk} \ph_{iak} = 0. \end{cases}
\end{equation}

\textbf{Summary.} Combining steps one, two, and three above, we have described the decomposition
\begin{equation*}
\mb{7} \otimes \mb{27} \cong \underbrace{(\mb{77^*} \oplus \mb{7})}_{\cS^3} \oplus \underbrace{(\mb{64} \oplus \mb{27} \oplus \mb{14})}_{\text{$\perp$ of $\cS^3$ in $\Omega^1_7 \otimes \cS^2_0$}}
\end{equation*}

\subsection{Equivalence of two different descriptions of $\mb{64}$} \label{sec:equiv-64}

In the process of describing the splittings $\mb{7} \otimes \mb{14} \cong \mb{64} \oplus \mb{27} \oplus \mb{7}$ and $\mb{7} \otimes \mb{27} \cong (\mb{77^*} \oplus \mb{7}) \oplus (\mb{64} \oplus \mb{27} \oplus \mb{14})$, we determined two different explicit descriptions of $\mb{64}$, namely those given in~\eqref{eq:64-desc-1} and~\eqref{eq:64-desc-2}. In this section we construct an explicit isomorphism between these two descriptions.

\begin{prop} \label{prop:two-64-desc}
Consider the two explicit descriptions~\eqref{eq:64-desc-1} and~\eqref{eq:64-desc-2} of the $\mb{64}$ dimensional representation of $\G$, and denote them by $V$ and $W$, respectively. The linear maps $K \colon V \to W$ and $L \colon W \to V$ given by
$$ \beta_{ijk} \quad \overset{K}{\mapsto} \quad h_{ijk} = \frac{1}{\sqrt{3}} (\beta_{jik} + \beta_{kij}) $$
and
$$ h_{ijk} \quad \overset{L}{\mapsto} \quad \beta_{ijk} = \frac{1}{\sqrt{3}} (h_{jik} - h_{kij}) $$
are isometric isomorphisms of $V$ with $W$.
\end{prop}
\begin{proof}
First we show that the linear maps $K$ and $L$ actually do map $V \to W$ and $W \to V$, respectively.

Let $\beta \in V$, so the three conditions of~\eqref{eq:64-desc-1} are satisfied. Contracting the second condition with $\ph_{aim}$ and using the first and third conditions gives
\begin{align*}
0 & = \beta_{ijk} \ph_{ajk} \ph_{aim} = \beta_{ijk} (g_{ji} g_{km} - g_{jm} g_{ki} - \ps_{jkim}) \\
& = \beta_{iim} - \beta_{imi} - \beta_{ijk} \ps_{imjk} = 2 \beta_{iim} - \frac{1}{3} (\beta_{ijk} + \beta_{jki} + \beta_{kij}) \ps_{ijkm} \\
& = 2 \beta_{iim} + 0.
\end{align*}
Similarly, we have
\begin{align*}
\beta_{ijk} \ph_{ija} & = (- \beta_{jki} - \beta_{kij} ) \ph_{ija} = \beta_{jik} \ph_{ija} - \beta_{kij} \ph_{aij} \\
& = - \beta_{jik} \ph_{jia} - 0 = - \beta_{ijk} \ph_{ija}.
\end{align*}
Hence we have shown that for $\beta \in V$, we always have
\begin{equation} \label{eq:64-1-temp}
\beta_{iim} = 0 \qquad \text{and} \qquad \beta_{ijk} \ph_{ija} = 0.
\end{equation}
Define $h_{ijk}$ by $\sqrt{3} h_{ijk} = \beta_{jik} + \beta_{kij}$. We need to show that $h_{ijk}$ satisfies the four conditions of~\eqref{eq:64-desc-2}. The first condition, symmetry in $j,k$, is immediate by construction. For the second condition, we compute
$$ \sqrt{3} h_{ikk} = \beta_{kik} + \beta_{kik} = - 2 \beta_{kki} = 0 $$
by~\eqref{eq:64-1-temp}. For the third condition, we have
$$ \sqrt{3} (h_{ijk} + h_{jki} + h_{kij}) = (\beta_{jik} + \beta_{kij}) + (\beta_{kji} + \beta_{ijk}) + (\beta_{ikj} + \beta_{jki}) = 0 $$
using the skew-symmetry of $\beta_{ijk}$ in $j,k$. Finally, for the fourth condition, using~\eqref{eq:64-1-temp} and~\eqref{eq:64-desc-1} we have
$$ \sqrt{3} h_{ijk} \ph_{iak} = (\beta_{jik} + \beta_{kij}) \ph_{iak} = 0 + 0. $$
Hence, $K$ indeed maps $V$ into $W$.

Let $h \in W$, so the four conditions of~\eqref{eq:64-desc-2} are satisfied. Define $\beta_{ijk}$ by $\sqrt{3} \beta_{ijk} = h_{jik} - h_{kij}$. We need to show that $\beta_{ijk}$ satisfies the three conditions of~\eqref{eq:64-desc-1}. The first condition, skew-symmetry in $j,k$, is immediate by construction. For the second condition, we compute using the conditions in~\eqref{eq:64-desc-2} that
\begin{align*}
\sqrt{3} \beta_{ijk} \ph_{ajk} & = (h_{jik} - h_{kij}) \ph_{ajk} \\
& = - h_{jik} \ph_{jak} + (h_{ijk} + h_{jki}) \ph_{ajk} \\
& = 0 + 0 - h_{jik} \ph_{jak} = 0.
\end{align*}
Finally, for the third condition, using~\eqref{eq:64-desc-2} we have
\begin{align*}
\sqrt{3} (\beta_{ijk} + \beta_{jki} + \beta_{kij}) & = (h_{jik} - h_{kij}) + (h_{kji} - h_{ijk}) + (h_{ikj} - h_{jki}) = 0
\end{align*}
using the symmetry of $h_{ijk}$ in $j,k$. Hence, $L$ indeed maps $W$ into $V$.

To see that $L$ and $K$ are inverses of each other, we compute using~\eqref{eq:64-desc-1} that
\begin{align*}
(L K \beta)_{ijk} & = \frac{1}{\sqrt{3}} \big( (K \beta)_{jik} - (K \beta)_{kij} \big) \\
& = \frac{1}{3} \big( (\beta_{ijk} + \beta_{kji}) - (\beta_{ikj} + \beta_{jki}) \big) \\
& = \frac{1}{3} \big( 2 \beta_{ijk} + (- \beta_{kij} - \beta_{jki}) \big) \\
& = \beta_{ijk}.
\end{align*}
Since $L K = \Id_V$, we have $L = K^{-1}$, although one can similarly show directly that $K L = \Id_W$.

It remains to show that $K$, and thus $L$, is an isometry. Let $\beta, \gamma \in V$. First observe that
$$ \beta_{ijk} \gamma_{kji} = \beta_{ijk} (- \gamma_{jik} - \gamma_{ikj}) = - \beta_{ikj} \gamma_{jki} + \beta_{ijk} \gamma_{ijk}, $$
from which we deduce that
$$ \beta_{ijk} \gamma_{kji} = \tfrac{1}{2} \beta_{ijk} \gamma_{ijk}. $$
Using the above, we compute
\begin{align*}
\langle K \beta, K \gamma \rangle & = (K \beta)_{ijk} (K \gamma)_{ijk} = \frac{1}{3} (\beta_{jik} + \beta_{kij}) (\gamma_{jik} + \gamma_{kij}) \\
& = \frac{1}{3} ( \beta_{jik} \gamma_{jik} + \beta_{kij} \gamma_{jik} + \beta_{jik} \gamma_{kij} + \beta_{kij} \gamma_{kij} ) \\
& = \frac{1}{3} (2 \beta_{ijk} \gamma_{ijk} + 2 \beta_{ijk} \gamma_{kji}) = \beta_{ijk} \gamma_{ijk} = \langle \beta, \gamma \rangle,
\end{align*}
so $K$ is indeed an isometry.
\end{proof}

\subsection{Identities for elements of $\cS^2 (\Lambda^2)$} \label{sec:identities-Pop}

Before we can consider the remaining decomposition, namely that of $\Sym^2 (\mb{14})$ in Section~\ref{sec:S2-14}, we need to collect some important identities for elements of $\cS^2 (\Lambda^2) = \Gamma( \Sym^2 (\Lambda^2 T^* M) )$. These identities depend on the $\G$-structure, and involve the operator $\Pop$ on $\Omega^2$ introduced in~\eqref{eq:Pdefn}, as well as two linear maps $\iota_{\ph} \colon \cS \to \cS^2 (\Lambda^2)$ and $\rho_{\ph} \colon \cS^2 (\Lambda^2) \to \cS^2$ defined below in~\eqref{eq:rho-ph}. These identities are also used in Section~\ref{sec:curvature-torsion} to study the decomposition of the curvature tensor.

An element $U \in \cS^2 (\Lambda^2)$ satisfies
$$ U_{ijkl} = - U_{jikl} = - U_{ijlk} = U_{klij} $$
and corresponds to a self-adjoint operator on $\Omega^2$ via
\begin{equation} \label{eq:U-cS2Lambda2}
(U \beta)_{ij} = U_{ijkl} \beta_{kl} \qquad \text{for $\beta \in \Omega^2$}.
\end{equation}

\begin{rmk} \label{rmk:curv-operator-convention}
If $U = \tRm \in \cS^2 (\Lambda^2)$ is the Riemann curvature tensor of a Riemannian metric $g$, then the ``Riemann curvature operator'' $\wh{\sR}$ is the self-adjoint operator on $\Omega^2$ given by~\eqref{eq:U-cS2Lambda2} with an additional minus sign. That is, $(\wh{\sR} \beta)_{ij} = - R_{ijkl} \beta_{kl} = R_{ijlk} \beta_{kl}$. This is done so that $\wh{\sR}$ being a positive operator implies positive sectional curvature. Since we are not concerned with positivity of the operators $U \in \cS^2 (\Lambda^2)$, we use the definition~\eqref{eq:U-cS2Lambda2} which looks more natural. This issue would go away if we had defined the Riemann curvature tensor $R_{ijkl}$ in such a way that the Ricci tensor would be given by contraction on the first and third indices, rather than the first and fourth. This can be done by either defining $\tRm$ to be the negative of~\eqref{eq:Riemann-convention}, which some authors do but is nonstandard, or, what is better, by defining $R_{ijkl} = g_{km} R^m_{ijl}$. That is, by using the metric to identify a skew-symmetric bilinear form with a skew-adjoint operator by raising the \emph{first} index rather than the second. See also Remark~\ref{rmk:iota-g-sign}. 
\end{rmk}

Let $\sI \colon \Omega^2 \to \Omega^2$ denote the identity operator. Since $\Pop$ is self-adjoint by~\eqref{eq:Pop-coords}, it corresponds to an element of $\cS^2 (\Lambda^2)$. Indeed,~\eqref{eq:Pdefn} shows that $\Pop$ corresponds to the section $\ps \in \Omega^4 \subset \cS^2 (\Lambda^2)$. Moreover, from~\eqref{eq:Pop-squared} we have
\begin{equation} \label{eq:Pop-squared-v2}
\Pop^2 = 8 \sI - 2 \Pop.
\end{equation}
Recall from~\eqref{eq:curvature-tensors-decomp} that
\begin{equation} \label{eq:cS2Lambda2-decomp}
\cS^2 (\Lambda^2) = \Omega^4 \operp \K,
\end{equation}
where $\K$ is the space of curvature tensors. Further recall from Section~\ref{sec:forms} that any element of $\Omega^4$ can be written as $A \diamond \ps$ for some unique $A \in \cS^2 \oplus \Omega^2_7$, and thus each such $A \diamond \ps$ is an element of $\cS^2 (\Lambda^2)$. The particular case $\Pop = \ps$ corresponds to $A = \frac{1}{4} g$, so
\begin{equation} \label{eq:g-diamond-ps}
g \diamond \ps = 4 \Pop.
\end{equation}

Given an element $U \in \cS^2 (\Lambda^2)$, we can precompose or postcompose with $\Pop$ to obtain the linear operators $U \Pop$, $\Pop U$, and $\Pop U \Pop$ on $\Omega^2$. In terms of a local orthonormal frame, we have
\begin{equation} \label{eq:PU-relations}
(U \Pop)_{ijkl} = U_{ijpq} \ps_{pqkl}, \quad (\Pop U)_{ijkl} = \ps_{ijpq} U_{pqkl}, \quad (\Pop U \Pop)_{ijkl} = \ps_{ijpq} U_{pqab} \ps_{abkl}.
\end{equation}
Note that $\Pop U \Pop$ is again self-adjoint, so it corresponds to an element of $\cS^2 (\Lambda^2)$. However, $U \Pop$ and $\Pop U$ are not in general self-adjoint. In fact it is clear that
\begin{equation} \label{eq:PU-relations-b}
(U \Pop)_{ijkl} = (\Pop U)_{klij},
\end{equation}
and thus the sum $U \Pop + \Pop U$ lies in $\cS^2 (\Lambda^2)$. Moreover, although both $U \Pop + \Pop U$ and $\Pop U \Pop$ lie in $\cS^2 (\Lambda^2)$, they do not in general lie in the subspace $\K$ of curvature tensors. Rather, they also have components in the $\Omega^4$ factor of the decomposition~\eqref{eq:cS2Lambda2-decomp}. One of the goals of this section is to precisely describe the
$$ \cS^2 (\Lambda^2) = \Omega^4 \operp \iota_g (\cS^2) \operp \W $$
decompositions of both $U \Pop + \Pop U$ and $\Pop U \Pop$, especially for certain special types of $U \in \cS^2 (\Lambda^2)$.

The map $\iota_g \colon \cS^2 \to \K$ introduced in~\eqref{eq:iota-g} given by
\begin{equation} \label{eq:iota-g-v2}
(\iota_g h)_{ijkl} = g_{il} h_{jk} + g_{jk} h_{il} - g_{ik} h_{jl} - g_{jl} h_{ik}.
\end{equation}
may be regarded as a map $\iota_g \colon \cS^2 \to \cS^2 (\Lambda^2)$. Similarly, the map $\rho_g \colon \K \to \cS^2$ introduced in~\eqref{eq:rho-g} given by
\begin{equation} \label{eq:rho-g-v2}
(\rho_g U)_{jk} = U_{ljkl}.
\end{equation}
may be extended by zero to a map $\rho_g \colon \cS^2 (\Lambda^2) \to \cS^2$, since for $\eta \in \Omega^4$, we have $(\rho_g \eta)_{jk} = \eta_{ljkl} = 0$. Note that for $h = g$, equation~\eqref{eq:iota-g-v2} gives
$$ ((\iota_g g) \beta)_{ij} = ((\iota_g g) \beta)_{ijkl} \beta_{kl} = 2(g_{il} g_{jk} - g_{ik} g_{jl}) \beta_{kl} = - 4 \beta_{ij}. $$
We deduce that, with our convention for the action~\eqref{eq:U-cS2Lambda2} of $U \in \cS^2 (\Lambda^2)$ on $\Omega^2$, we have
\begin{equation} \label{eq:iota-g-identity}
\iota_g g = - 4 \sI.
\end{equation}

\begin{rmk} \label{rmk:iota-g-sign}
The factor of $4$ in~\eqref{eq:iota-g-identity} arises because we use the inner product on (skew-symmetric) \emph{tensors}, rather than the inner product on $2$-forms, which differs by a factor of $\frac{1}{2}$. The minus sign arises because of our conventions for the Riemann curvature tensor. (See Remark~\ref{rmk:curv-operator-convention}.) With the appropriate choices of inner product and curvature conventions, one can arrange that $\iota_g g = \sI$. Note that another way to ``fix'' the sign would be to define $\iota_g h = g \owedge h$ in~\eqref{eq:iota-g} to be the negative of what we chose, but that would introduce unpleasant minus signs in~\eqref{eq:curvature-n}, unless we also changed the curvature convention.
\end{rmk}

If $U, V \in \cS^2 (\Lambda^2)$, we have $\langle \Pop U, V \rangle = \ps_{ijab} U_{abkl} V_{ijkl}$. Using the symmetries of $U, V, \ps$, it is easy to see from this expression in indices that
\begin{equation} \label{eq:Pop-inner-products}
\langle \Pop U, V \rangle = \langle U \Pop, V \rangle = \langle U, \Pop V \rangle = \langle U, V \Pop \rangle.
\end{equation}

Let $A \in \cT^2$, so $A \diamond \ps \in \Omega^4$. For $U \in \cS^2 (\Lambda^2)$, we have
\begin{align*}
\langle U, A \diamond \ps \rangle & = U_{ijkl} (A \diamond \ps)_{ijkl} \\
& = U_{ijkl} (A_{ip} \ps_{pjkl} + A_{jp} \ps_{ipkl} + A_{kp} \ps_{ijpl} + A_{lp} \ps_{ijkp}).
\end{align*}
Using the symmetries of $U$, the four terms above are the same, and hence
\begin{equation} \label{eq:U-IP-A-diamond-ps-1}
\langle U, A \diamond \ps \rangle = 4 U_{ijkl} A_{ip} \ps_{pjkl}.
\end{equation}

\begin{cor} \label{cor:U-IP-diamond-ps}
Let $U \in \cS^2 (\Lambda^2)$. Then $U$ is a curvature tensor (that is, $U$ is orthogonal to $\Omega^4$) if and only if $U_{ijkl} \ps_{pjkl} = 0$. More generally, $U$ is orthogonal to just $\Omega^4_7$ if and only if $U_{ijkl} \ph_{jkl} = 0$.
\end{cor}
\begin{proof}
From~\eqref{eq:U-IP-A-diamond-ps-1}, we find that $U$ is orthogonal to $\Omega^4$ if and only if $U_{ijkl} A_{ip} \ps_{pjkl} = 0$ for all $A \in \cT^2$, which is clearly equivalent to $U_{ijkl} \ps_{pjkl} = 0$. More generally, if we only ask for orthogonality to $\Omega^4_7$, then we must have $U_{ijkl} A_{ip} \ps_{pskl} = 0$ for all $A \in \Omega^2_7$, since it is for such $A$ that we have $A \diamond \ps \in \Omega^4_7$. Hence we can take $A_{ip} = X_m \ph_{mip}$, and thus $U_{ijkl} X_m \ph_{mip} \ps_{pjkl} = 0$ for all $X \in \Omega^1$. This is equivalent to $U_{ijkl} \ph_{mip} \ps_{pjkl} = 0$, which using the symmetries of $U$ becomes
\begin{align*}
0 & = U_{ijkl} \ph_{mip} \ps_{jklp} \\
& = U_{ijkl} ( g_{mj} \ph_{ikl} + g_{mk} \ph_{jil} + g_{ml} \ph_{jki} - g_{ij} \ph_{mkl} - g_{ik} \ph_{jml} - g_{il} \ph_{jkm} ) \\
& = - U_{mikl} \ph_{ikl} - U_{mlij} \ph_{lij} - U_{mkij} \ph_{kij} - 0 + (\rho_g U)_{jl} \ph_{jml} - (\rho_g U)_{jk} \ph_{jkm}.
\end{align*}
Using the symmetry of $\rho_g U$, the last two terms vanish, and we are left with $- 3 U_{mjkl} \ph_{jkl} = 0$.
\end{proof}

Using the $\G$-structure $\ph$, we get another pair of linear maps $\iota_{\ph}$, $\rho_{\ph}$ as follows. Let $\iota_{\ph} \colon \cS^2 \to \cS^2 (\Lambda^2)$ be given by
\begin{equation} \label{eq:iota-ph}
(\iota_{\ph} h)_{ijkl} = h_{pq} \ph_{pij} \ph_{qkl}.
\end{equation}
It is clear that $\iota_{\ph} h \in \cS^2 (\Lambda^2)$. However, we show below in~\eqref{eq:iota-ph-of-g} that the image of $\iota_{\ph}$ is \emph{not} contained in the space $\K$ of curvature tensors. Let $\rho_{\ph} \colon \cS^2 (\Lambda^2) \to \cS^2$ be given by
\begin{equation} \label{eq:rho-ph}
(\rho_{\ph} U)_{pq} = U_{ijkl} \ph_{ijp} \ph_{klq}.
\end{equation}
It is easy to see that $\rho_{\ph} U$ is indeed symmetric. We also have
\begin{equation} \label{eq:rho-ph-adj}
\langle \rho_{\ph} U, h \rangle = U_{ijkl} \ph_{ijp} \ph_{klq} h_{pq} = \langle U, \iota_{\ph} h \rangle.
\end{equation}
Note that the symmetric $2$-tensor $F$ of Definition~\ref{defn:F} is precisely
\begin{equation} \label{eq:F-rho-ph}
F = \rho_{\ph} (\sR),
\end{equation}
where $\sR$ is the Riemann curvature tensor $\tRm$ thought of as a self-adjoint operator on $\Omega^2$. Note also that from
\begin{align*}
(\iota_{\ph} g)_{ijkl} & = g_{pq} \ph_{pij} \ph_{qkl} = \ph_{ijp} \ph_{klp} \\
& = g_{ik} g_{jl} - g_{il} g_{jk} - \ps_{ijkl} = - \tfrac{1}{2} (\iota_g g)_{ijkl} - \ps_{ijkl},
\end{align*}
and~\eqref{eq:iota-g-identity}, we obtain
\begin{equation} \label{eq:iota-ph-of-g}
\iota_{\ph} g = 2 \sI - \Pop.
\end{equation}

\begin{rmk} \label{rmk:rho-ph}
The maps $\iota_{\ph}, \rho_{\ph}$ were first discussed in Cleyton--Ivanov~\cite{CI}. Their map $c^{\phi}$ is the same as our map $\rho_{\ph}$, up to a constant. Their map $r_{\ph}$ is, again up to a constant, our map $\iota_{\ph}$ followed by the orthogonal projection $\cS^2 (\Lambda^2) \to \K$. A small number of the formulas we derive in this section are either explicit or at least implicit in~\cite{CI}.
\end{rmk}

It is easy to check that the maps $\iota_{\ph}, \rho_{\ph}$ satisfy the requirements of Lemma~\ref{lemma:basic-tool} to give an orthogonal decomposition
$$ \K \operp \Omega^4 = \cS^2 (\Lambda^2) = (\ker \rho_{\ph}) \operp \iota_{\ph} (\Omega^0 g) \operp \iota_{\ph} (\cS^2_0) $$
that is \emph{different} from the decomposition~\eqref{eq:curvature-classical-decomp}. We do not directly use this decomposition, although it is implicit in much of what follows. We do, however, use modifications of $\iota_{\ph}$ and $\rho_{\ph}$ to partially decompose the space $\W$ of Weyl tensors at the end of this section.

The next result gives the values of $\rho_g$ and $\rho_{\ph}$ on elements of the form $\iota_g h$, $\iota_{\ph} h$, and $h \diamond \ph$.
\begin{prop} \label{prop:ph-rep-maps}
Let $h \in \cS^2$. Then we have
\begin{equation} \label{eq:ph-rep-maps}
\begin{aligned}
\rho_g (\iota_g h) & = 5 h + (\tr h) g, & \rho_g (\iota_{\ph} h) & = h - (\tr h) g, & \rho_g (h \diamond \ps) & = 0, \\
\rho_{\ph} (\iota_g h) & = 4 h - 4 (\tr h) g, & \rho_{\ph} (\iota_{\ph} h) & = 36 h, & \rho_{\ph} (h \diamond \ps) & = 16 h - 16 (\tr h) g.
\end{aligned}
\end{equation}
\end{prop}
\begin{proof}
The first equation is from~\eqref{eq:iota-g-rho-g-7}. Using the definitions~\eqref{eq:iota-ph} of $\iota_{\ph}$ and~\eqref{eq:rho-g-v2} of $\rho_g$, and the symmetry of $h$, we compute
\begin{align*}
(\rho_g (\iota_{\ph} h))_{jk} & = (\iota_{\ph} h)_{ljkl} = h_{pq} \ph_{plj} \ph_{qkl}, \\
& = h_{pq} (g_{jq} g_{pk} - g_{jk} g_{pq} - \ps_{jpqk}) \\
& = h_{kj} - (\tr h) g_{jk} - 0,
\end{align*}
which is the second equation. The third equation is immediate since $\rho_g$ is zero on $\Omega^4$.

Using the definitions~\eqref{eq:iota-g-v2} of $\iota_g$ and~\eqref{eq:rho-ph} of $\rho_{\ph}$, and the skew-symmetry of $\ph$, we compute
\begin{align*}
(\rho_{\ph} (\iota_g h))_{pq} & = (\iota_g h)_{ijkl} \ph_{ijp} \ph_{klq} = (g_{il} h_{jk} + g_{jk} h_{il} - g_{ik} h_{jl} - g_{jl} h_{ik}) \ph_{ijp} \ph_{klq} \\
& = 4 g_{il} h_{jk} \ph_{ijp} \ph_{klq} = 4 h_{jk} \ph_{ljp} \ph_{klq} \\
& = 4 h_{jk} (g_{jq} g_{pk} - g_{jk} g_{pq} - \ps_{jpqk}) \\
& = 4 h_{qp} - 4 (\tr h) g_{pq} - 0,
\end{align*}
which gives the fourth equation. Similarly, we compute
$$ (\rho_{\ph} (\iota_{\ph} h))_{pq} = (\iota_{\ph} h)_{ijkl} \ph_{ijp} \ph_{klq} = h_{ab} \ph_{aij} \ph_{bkl} \ph_{ijp} \ph_{klq} = 36 h_{ab} g_{ap} g_{bq}, $$
yielding the fifth equation. Finally, we have
\begin{align*}
(\rho_{\ph} (h \diamond \ps))_{pq} & = (h \diamond \ps)_{ijkl} \ph_{ijp} \ph_{klq} \\
& = (h_{im} \ps_{mjkl} + h_{jm} \ps_{imkl} + h_{km} \ps_{ijml} + h_{lm} \ps_{ijkm}) \ph_{ijp} \ph_{klq} \\
& = (2 h_{im} \ps_{mjkl} + 2 h_{km} \ps_{ijml}) \ph_{ijp} \ph_{klq} \\
& = - 8 h_{im} \ph_{ijp} \ph_{mjq} - 8 h_{km} \ph_{klq} \ph_{mlp}.
\end{align*}
This becomes
\begin{align*}
(\rho_{\ph} (h \diamond \ps))_{pq} & = - 8 h_{im} (g_{im} g_{pq} - g_{iq} g_{pm} - \ps_{ipmq}) - 8 h_{km} (g_{km} g_{qp} - g_{kp} g_{qm} - \ps_{kqmp}) \\
& = - 8 (\tr h) g_{pq} + 8 h_{qp} - 0 - 8 (\tr h) g_{qp} + 8 h_{pq} + 0,
\end{align*}
which simplifies to the sixth equation.
\end{proof}

So far, we have identified \emph{three special classes} of elements in $\cS^2 (\Lambda^2)$, namely:
\begin{equation} \label{eq:special-cS2}
\begin{aligned}
A \diamond \ps & \in \Omega^4 \subset \cS^2 (\Lambda^2), \quad \text{for $A \in \cS^2 \oplus \Omega^2_7$}, \\
\iota_g h & \in \K \subset \cS^2 (\Lambda^2), \quad \text{for $h \in \cS^2$}, \\
\iota_{\ph} h & \in \cS^2 (\Lambda^2), \quad \text{for $h \in \cS^2$},
\end{aligned}
\end{equation}
Moreover, equations~\eqref{eq:g-diamond-ps},~\eqref{eq:iota-g-identity},~\eqref{eq:iota-ph-of-g}, and~\eqref{eq:Pop-squared-v2} show that
\begin{equation} \label{eq:special-g-maps}
g \diamond \ps = 4 \Pop, \qquad \iota_g g = - 4 \sI, \qquad \iota_{\ph} g = 2 \sI - \Pop, \qquad \Pop^2 = 8 \sI - 2 \Pop,
\end{equation}
so the subalgebra of $\cS^2 (\Lambda^2)$ generated by $\{ g \diamond \ps, \iota_g g, \iota_{\ph} g \}$ equals the subalgebra generated by $\{ \sI, \Pop \}$. We can thus restrict attention to the case where $h \in \cS^2_0$, so $\tr h = 0$. In particular, we can then decompose $\iota_{\ph} h$ into terms of the first two types in~\eqref{eq:special-cS2} plus a term in the space $\W$ of Weyl tensors.

\begin{prop} \label{prop:iotaph-decomp}
Let $h \in \cS^2_0$. Then
$$ \iota_{\ph} h \in \cS^2 (\Lambda^2) = \Omega^4 \operp \K = \Omega^4 \operp \iota_g (\cS^2) \operp \W $$
decomposes as
\begin{equation} \label{eq:iotaph-decomp}
\iota_{\ph} h = \tfrac{1}{3} h \diamond \ps + \tfrac{1}{5} \iota_g h + (\iota_{\ph} h)_{\W}.
\end{equation}
\end{prop}
\begin{proof}
We know that $\iota_{\ph} h$ decomposes orthogonally as
\begin{equation} \label{eq:iotaph-decomp-temp}
\iota_{\ph} h = B \diamond \ps + \iota_g \wt h + (\iota_{\ph} h)_{\W},
\end{equation}
for some unique $B \in \cS^2 \oplus \Omega^2_7$ and unique $\wt h \in \cS^2$, where $(\iota_{\ph} h)_{\W} \in \W$. Contracting both sides of~\eqref{eq:iotaph-decomp-temp} with $\ps$ on three indices, and using Corollary~\ref{cor:U-IP-diamond-ps}, we have
\begin{align*}
(B \diamond \ps)_{ijkl} \ps_{ajkl} & = (\iota_{\ph} h)_{ijkl} \ps_{ajkl} = (h_{pq} \ph_{pij} \ph_{qkl}) \ps_{ajkl} \\
& = - 4 h_{pq} \ph_{pij} \ph_{qaj} = - 4 h_{pq} (g_{pq} g_{ia} - g_{pa} g_{iq} - \ps_{piqa}) \\
& = 0 + 4 h_{ai} - 0,
\end{align*}
so $(B \diamond \ps)_{ijkl} \ps_{ajkl} = 4 h_{ia}$. But then, since $h \in \cS^2_0$, Corollary~\ref{cor:diamondinverse} gives
$$ B_{ia} = \tfrac{1}{12} (B \diamond \ps)_{ijkl} \ps_{ajkl} = \tfrac{1}{3} h_{ia}. $$
Applying the map $\rho_g$ to both sides of~\eqref{eq:iotaph-decomp-temp} and using Proposition~\ref{prop:ph-rep-maps} gives
$$
h = \rho_g (\iota_{\ph} h) = \rho_g ( B \diamond \ps + \iota_g \wt h + (\iota_{\ph} h)_{\W} ) = 0 + \big( 5 \wt h + (\tr \wt h) g \big) + 0. $$
Taking traces gives $0 = 12 \tr \wt h$, and thus $\wt h = \frac{1}{5} h$.
\end{proof}
The next three propositions and corollary give explicit formulas for $\Pop U + U \Pop$ and $\Pop U \Pop$ in the special cases where $U \in \cS^2 (\Lambda^2)$ is of the form $\iota_g h$, $h \diamond \ps$, $\iota_{\ph} h$, or $(\iota_{\ph} h)_{\W}$ for $h \in \cS^2_0$.

\begin{prop} \label{prop:iota-g-Pop}
Let $h \in \cS^2_0$, so $\iota_g h \in \K \subset \cS^2(\Lambda^2)$. Then we have
\begin{equation} \label{eq:iota-g-Pop}
\begin{aligned}
\Pop (\iota_g h) + (\iota_g h) \Pop & = - 2 h \diamond \ps, \\
\Pop (\iota_g h) \Pop & = - 4 h \diamond \ps - 4 \iota_g h + 4 \iota_{\ph} h.
\end{aligned}
\end{equation}
\end{prop}
\begin{proof}
We compute
\begin{align*}
[ \Pop (\iota_g h) ]_{ijkl} & = \ps_{ijab} (\iota_g h)_{abkl} \\
& = \ps_{ijab} (g_{al} h_{bk} + g_{bk} h_{al} - g_{ak} h_{bl} - g_{bl} h_{ak}) \\
& = \ps_{ijab} (2 g_{al} h_{bk} + 2 g_{bk} h_{al}) \\
& = 2 h_{bk} \ps_{ijlb} + 2 h_{al} \ps_{ijak},
\end{align*}
which we can rewrite as
\begin{equation} \label{eq:iota-g-Pop-temp}
[ \Pop (\iota_g h) ]_{ijkl} = - 2 h_{km} \ps_{ijml} - 2 h_{lm} \ps_{ijkm}.
\end{equation}
Using~\eqref{eq:PU-relations-b} gives
\begin{align} \nonumber
[ (\iota_g h) \Pop ]_{ijkl} = [\Pop (\iota_g h)]_{klij} & = - 2 h_{im} \ps_{klmj} - 2 h_{jm} \ps_{klim} \\ \label{eq:iota-g-Pop-temp2}
& = - 2 h_{im} \ps_{mjkl} - 2 h_{jm} \ps_{imkl}.
\end{align}
Adding the above two equations gives $\Pop (\iota_g h) + (\iota_g h) \Pop = - 2 h \diamond \ps$.

Using~\eqref{eq:iota-g-Pop-temp}, we compute
\begin{align*}
[\Pop (\iota_g h) \Pop]_{ijkl} & = [\Pop (\iota_g h)]_{ijab} \ps_{abkl} \\
& = (- 2 h_{am} \ps_{ijmb} - 2 h_{bm} \ps_{ijam}) \ps_{abkl} \\
& = - 4 h_{am} \ps_{ijmb} \ps_{aklb} \\
& = - 4 h_{am} \Big[ - \ph_{ajm} \ph_{ikl} - \ph_{iam} \ph_{jkl} - \ph_{ija} \ph_{mkl} \\
& \qquad \qquad \quad {} + g_{ia} g_{jk} g_{ml} + g_{ik} g_{jl} g_{ma} + g_{il} g_{ja} g_{mk} - g_{ia} g_{jl} g_{mk} - g_{ik} g_{ja} g_{ml} - g_{il} g_{jk} g_{ma} \\
& \qquad \qquad \quad {} - g_{ia} \ps_{jmkl} - g_{ja} \ps_{mikl} - g_{ma} \ps_{ijkl} + g_{ak} \ps_{ijml} - g_{al} \ps_{ijmk} \Big].
\end{align*}
Using the symmetry and tracelessness of $h$, the above simplifies to
\begin{align*}
[\Pop (\iota_g h) \Pop]_{ijkl} & = 0 + 0 + 4 h_{am} \ph_{aij} \ph_{mkl} - 4 h_{il} g_{jk} - 0 - 4 h_{jk} g_{il} \\
& \qquad {} + 4 h_{ik} g_{jl} + 4 h_{jl} g_{ik} + 0 + 4 h_{im} \ps_{jmkl} + 4 h_{jm} \ps_{mikl} + 0 \\
& \qquad {} - 4 h_{km} \ps_{ijml} + 4 h_{lm} \ps_{ijmk},
\end{align*}
which then becomes
$$ [\Pop (\iota_g h) \Pop]_{ijkl} = - 4 (h \diamond \ps)_{ijkl} - 4 (\iota_g h)_{ijkl} + 4 (\iota_{\ph} h)_{ijkl} $$
as claimed.
\end{proof}

\begin{prop} \label{prop:h-diamond-ps-Pop}
Let $h \in \cS^2_0$, so $h \diamond \ps \in \Omega^4_{27} \subset \cS^2(\Lambda^2)$. Then we have
\begin{equation} \label{eq:h-diamond-ps-Pop}
\begin{aligned}
\Pop (h \diamond \ps) + (h \diamond \ps) \Pop & = 2 h \diamond \ps - 4 \iota_g h - 4 \iota_{\ph} h, \\
\Pop (h \diamond \ps) \Pop & = - 8 \iota_g h + 8 \iota_{\ph} h.
\end{aligned}
\end{equation}
\end{prop}
\begin{proof}
We compute
\begin{align*}
[ \Pop (h \diamond \ps) ]_{ijkl} & = \ps_{ijab} (h \diamond \ps)_{abkl} \\
& = \ps_{ijab} (h_{am} \ps_{mbkl} + h_{bm} \ps_{amkl} + h_{km} \ps_{abml} + h_{lm} \ps_{abkm}) \\
& = 2 \ps_{ijab} h_{am} \ps_{mbkl} + h_{km} (4 g_{im} g_{jl} - 4 g_{il} g_{jm} - 2 \ps_{ijml}) \\
& \qquad {} + h_{lm} (4 g_{ik} g_{jm} - 4 g_{im} g_{jk} - 2 \ps_{ijkm}) \\
& = 2 h_{am} \ps_{ijab} \ps_{mklb} + 4 h_{ik} g_{jl} - 4 g_{il} h_{jk} - 2 h_{km} \ps_{ijml} \\
& \qquad {} + 4 g_{ik} h_{jl} - 4 h_{il} g_{jk} - 2 h_{lm} \ps_{ijkm},
\end{align*}
which simplifies to
\begin{equation} \label{eq:h-diamond-ps-temp}
[ \Pop (h \diamond \ps) ]_{ijkl} = 2 h_{am} \ps_{ijab} \ps_{mklb} - 4 (\iota_g h)_{ijkl} - 2 h_{km} \ps_{ijml} - 2 h_{lm} \ps_{ijkm}.
\end{equation}
Expanding the first term using the symmetry and tracelessness of $h$, we have
\begin{align*}
2 h_{am} \ps_{ijab} \ps_{mklb} & = 2 h_{am} \Big[ - \ph_{mja} \ph_{ikl} - \ph_{ima} \ph_{jkl} - \ph_{ijm} \ph_{akl} \\
& \qquad \qquad \quad {} + g_{im} g_{jk} g_{al} + g_{ik} g_{jl} g_{am} + g_{il} g_{jm} g_{ak} - g_{im} g_{jl} g_{ak} - g_{ik} g_{jm} g_{al} - g_{il} g_{jk} g_{am} \\
& \qquad \qquad \quad {} - g_{im} \ps_{jakl} - g_{jm} \ps_{aikl} - g_{am} \ps_{ijkl} + g_{mk} \ps_{ijal} - g_{ml} \ps_{ijak} \Big] \\
& = 0 + 0 - 2 h_{am} \ph_{mij} \ph_{akl} + 2 h_{il} g_{jk} + 0 + 2 h_{jk} g_{il} \\
& \qquad {} - 2 h_{ik} g_{jl} - 2 h_{jl} g_{ik} - 0 + 2 h_{ia} \ps_{ajkl} + 2 h_{ja} \ps_{iakl} - 0 \\
& \qquad {} + 2 h_{ka} \ps_{ijal} + 2 h_{la} \ps_{ijka}, 
\end{align*}
which simplifies to
$$ 2 h_{am} \ps_{ijab} \ps_{mklb} = - 2 (\iota_{\ph} h)_{ijkl} + 2 (\iota_g h)_{ijkl} + 2 (h \diamond \ps)_{ijkl}. $$
Substituting the above into~\eqref{eq:h-diamond-ps-temp} gives
\begin{equation} \label{eq:h-diamond-ps-Pop-temp}
[ \Pop (h \diamond \ps) ]_{ijkl} = - 2 (\iota_{\ph} h)_{ijkl} - 2 (\iota_g h)_{ijkl} + 2 h_{im} \ps_{mjkl} + 2 h_{jm} \ps_{imkl}.
\end{equation}
We use~\eqref{eq:PU-relations-b} to write
$$ [ \Pop (h \diamond \ps) + (h \diamond \ps) \Pop ]_{ijkl} = [ \Pop (h \diamond \ps) ]_{ijkl} + [ \Pop (h \diamond \ps) ]_{klij}, $$
which indeed yields
$$ \Pop (h \diamond \ps) + (h \diamond \ps) \Pop = - 4 \iota_{\ph} h - 4 \iota_g h + 2 h \diamond \ps. $$

For the second equation, first observe that
\begin{equation} \label{eq:h-diamond-ps-Pop-temp2}
(\iota_{\ph} h)_{ijpq} \ps_{pqkl} = h_{ab} \ph_{aij} \ph_{bpq} \ps_{pqkl} = - 4 h_{ab} \ph_{aij} \ph_{bkl} = - 4 (\iota_{\ph})_{ijkl}.
\end{equation}
Now using~\eqref{eq:h-diamond-ps-Pop-temp},~\eqref{eq:h-diamond-ps-Pop-temp2}, and~\eqref{eq:iota-g-Pop-temp2}, we have
\begin{align*}
[ \Pop (h \diamond \ps) \Pop ]_{ijkl} & = [ \Pop (h \diamond \ps) ]_{ijpq} \ps_{pqkl} \\
& = (- 2 (\iota_{\ph} h)_{ijpq} - 2 (\iota_g h)_{ijpq} + 2 h_{im} \ps_{mjpq} + 2 h_{jm} \ps_{impq}) \ps_{pqkl} \\
& = -2 (-4 (\iota_{\ph} h)_{ijkl}) - 2 (-2 h_{im} \ps_{mjkl} - 2 h_{jm} \ps_{imkl}) \\
& \qquad {} + 2 h_{im} (4 g_{mk} g_{jl} - 4 g_{ml} g_{jk} - 2 \ps_{mjkl}) + 2 h_{jm} (4 g_{ik} g_{ml} - 4 g_{il} g_{mk} - 2 \ps_{imkl}),
\end{align*}
which simplifies to $\Pop (h \diamond \ps) \Pop = 8 \iota_{\ph} h - 8 \iota_g h$.
\end{proof}

\begin{prop} \label{prop:iota-ph-Pop}
Let $h \in \cS^2_0$, so $\iota_{\ph} h \in \cS^2(\Lambda^2)$. Then we have
\begin{equation} \label{eq:iota-ph-Pop}
\begin{aligned}
\Pop (\iota_{\ph} h) + (\iota_{\ph} h) \Pop & = - 8 \iota_{\ph} h, \\
\Pop (\iota_{\ph} h) \Pop & = 16 \iota_{\ph} h.
\end{aligned}
\end{equation}
\end{prop}
\begin{proof}
We have
\begin{align*}
(\Pop (\iota_{\ph} h))_{ijkl} & = \ps_{ijab} (h_{pq} \ph_{pab} \ph_{qkl}) \\
& = - 4 h_{pq} \ph_{pij} \ph_{qkl} = - 4 (\iota_{\ph} h)_{ijkl}.
\end{align*}
Similarly one computes that $(\iota_{\ph} h) \Pop = - 4 \iota_{\ph} h$ and $\Pop (\iota_{\ph} h) \Pop = 16 \iota_{\ph} h$.
\end{proof}

\begin{cor} \label{cor:iota-phW-Pop}
Let $h \in \cS^2_0$, so $(\iota_{\ph} h)_{\W} \in \W \subset \cS^2(\Lambda^2)$. Then we have
\begin{equation} \label{eq:iota-phW-Pop}
\begin{aligned}
(\iota_{\ph} h)_{\W} & = - \tfrac{1}{3} h \diamond \ps - \tfrac{1}{5} \iota_g h + \iota_{\ph} h, \\
\Pop (\iota_{\ph} h)_{\W} + (\iota_{\ph} h)_{\W} \Pop & = - \tfrac{4}{15} h \diamond \ps + \tfrac{4}{3} \iota_g h - \tfrac{20}{3} \iota_{\ph} h, \\
\Pop (\iota_{\ph} h)_{\W} \Pop & = \tfrac{4}{5} h \diamond \ps + \tfrac{52}{15} \iota_g h + \tfrac{188}{15} \iota_{\ph} h.
\end{aligned}
\end{equation}
\end{cor}
\begin{proof}
The first equation is a rearrangement of~\eqref{eq:iotaph-decomp}. The others follow from this by straightforward computation using Propositions~\ref{prop:iota-g-Pop},~\ref{prop:h-diamond-ps-Pop},~\ref{prop:iota-ph-Pop}.
\end{proof}

The appearance of the $(\iota_{\ph} h)_{\W}$ term in $\W$ for $h \in \cS^2_0$ in Proposition~\ref{prop:iotaph-decomp} motivates us to consider the further decomposition of the space $\W$ using appropriate modifications of the maps $\iota_{\ph}$ and $\rho_{\ph}$. (Recall from Remark~\ref{rmk:Weyl-decomp} that a Weyl tensor $U$ should decompose into $U = U_{27} + U_{64} + U_{77}$.)

First note from~\eqref{eq:iota-ph-of-g} that $\iota_{\ph} g$ is orthogonal to $\W$. Also, if $U \in \W = \ker \rho_g$, then using Corollary~\ref{cor:U-IP-diamond-ps} we have
\begin{align*}
\tr (\rho_{\ph} U) & = U_{ijkl} \ph_{ijp} \ph_{klp} = U_{ijkl} (g_{ik} g_{jl} - g_{il} g_{jk} - \ps_{ijkl}) \\
& = 2 U_{klkl} - U_{ijkl} \ps_{ijkl} = - 2 \tr (\rho_g U) - 0 = 0,
\end{align*}
so $\rho_{\ph}$ maps $\W$ into $\cS^2_0$. We are therefore led to consider the linear maps
\begin{equation} \label{eq:iota-rho-W27}
\begin{aligned}
\ol{\iota}_{\ph} \colon \cS^2_0 & \to \W, & \qquad \ol{\iota}_{\ph} & = \pi_{\W} \circ \rest{\iota_{\ph}}{\cS^2_0}, \\
\ol{\rho}_{\ph} \colon \W & \to \cS^2_0, & \qquad \ol{\rho}_{\ph} & = \rest{\rho_{\ph}}{\W}.
\end{aligned}
\end{equation}
Let $h \in \cS^2_0$. Using Propositions~\ref{prop:iotaph-decomp} and~\ref{prop:ph-rep-maps}, we compute
\begin{align*}
\ol{\rho}_{\ph} (\ol{\iota}_{\ph} h) & = \rho_{\ph} \big( (\iota_{\ph} h)_{\W} \big) \\
& = \rho_{\ph} ( \iota_{\ph} h - \tfrac{1}{3} h \diamond \ps - \tfrac{1}{5} \iota_g h ) \\
& = 36 h - \tfrac{1}{3} (16 h) - \tfrac{1}{5} (4 h) = \tfrac{448}{15} h.
\end{align*}
Thus condition $\mathrm{(i)}$ of~\eqref{eq:basic-condition} is satisfied with $b = \frac{448}{15}$. Moreover, from~\eqref{eq:rho-ph-adj}, for $h \in \cS^2_0$ and $U \in \W$ we have
$$ \langle \ol{\rho}_{\ph} U, h \rangle = \langle \rho_{\ph} U, h \rangle = \langle U, \iota_{\ph} h \rangle = \langle U, \ol{\iota}_{\ph} h \rangle, $$
so condition $\mathrm{(ii)}$ of~\eqref{eq:basic-condition} is satisfied with $c = 1$. We can thus invoke Lemma~\ref{lemma:basic-tool} to conclude that
\begin{equation*}
\W \cong (\ker \ol{\rho}_{\ph}) \oplus \cS^2_0
\end{equation*}
where explicitly we have $U = U_{64 + 77} + U_{27}$ with
\begin{equation} \label{eq:W-decomp}
\begin{aligned}
U_{27} & = \tfrac{15}{448} \ol{\iota}_{\ph} (\ol{\rho}_{\ph} U) , \\
U_{64 + 77} & = U - U_{27}.
\end{aligned}
\end{equation}
We explain how to decompose $U_{64 + 77}$ into $U_{64} + U_{77}$ in Section~\ref{sec:curvature-decomp}.

\subsection{The decomposition $\cS^2 (\mb{14}) = \mb{77} \oplus \mb{1} \oplus \mb{27}$} \label{sec:S2-14}

To describe the decomposition $\cS^2 (\mb{14}) = \mb{77} \oplus \mb{1} \oplus \mb{27}$ we again use Lemma~\ref{lemma:basic-tool}. Let $V = \cS^2 = \mb{1} \oplus \mb{27}$ and let $W = \cS^2 (\mb{14})$. The map $\iota_g$ from~\eqref{eq:iota-g-v2} maps $\cS^2$ into the subspace $\K$ of $\cS^2 (\Lambda^2) = \cS^2 (\mb{7} \oplus \mb{14})$. Thus, given $h \in \cS^2$, we have $\iota_g h \in \cS^2 (\mb{7} \oplus \mb{14})$, so it can be regarded as a self-adjoint operator on $\mb{7} \oplus \mb{14}$. We can pre-compose and post-compose $\iota_g h$ by the orthogonal projection $\pi_{14}$ to get a self-adjoint operator on $\mb{14}$. That is, an element of $\cS^2 (\mb{14})$.

Explicitly, by~\eqref{eq:omega2-proj} we have $\pi_{14} = \frac{1}{6} (4 \sI + \Pop)$, and thus we obtain a linear map $\check{\iota} \colon \cS^2 \to \cS^2 (\mb{14})$
by
\begin{align*}
\check{\iota} \, h & = \pi_{14} (\iota_g h) \pi_{14} = \tfrac{1}{36} (4 \sI + \Pop) (\iota_g h) (4 \sI + \Pop) \\
& = \tfrac{1}{36} \big(16 \iota_g h + 4 (\Pop (\iota_g h) + (\iota_g h) \Pop) + (\Pop (\iota_g h) \Pop) \big).
\end{align*}
We also have a linear map $\check{\rho} \colon \cS^2 (\mb{14}) \to \cS^2$ by restricting $\rho_g$ from~\eqref{eq:rho-g-v2} to the subspace $\cS^2 (\mb{14})$, so
$$ \check{\rho} \, U = \rho_g U \qquad \text{for $U \in \cS^2(\mb{14})$.} $$
We want to apply Lemma~\ref{lemma:basic-tool} to the pair $\check{\iota}, \check{\rho}$.

If $h = g$, then from~\eqref{eq:special-g-maps} we have $\iota_g g = - 4 \sI$ and $\Pop^2 = 8 \sI - 2 \Pop$, and thus
\begin{align*}
\check{\iota} \, g & = \tfrac{1}{36} (16 (- 4 \sI) + 4( \Pop (-4 \sI) + (-4 \sI) \Pop ) + (\Pop (-4 \sI) \Pop) \\
& = \tfrac{1}{36} (- 64 \sI - 16 \Pop - 16 \Pop - 4 (8 \sI - 2 \Pop)) = \tfrac{1}{36} (- 96 \sI -24 \Pop) \\
& = \tfrac{1}{36} (24 (\iota_g g) - 6 g \diamond \ps).
\end{align*}
Applying $\check{\rho} = \rho_g$ to this, and using Proposition~\ref{prop:ph-rep-maps}, we get
$$ \check{\rho} (\check \iota \, g) = \tfrac{1}{36} (24 (5 g + 7 g) + 0) = 8 g. $$
If $h \in \cS^2_0$, so $\tr h = 0$, then Propositions~\ref{prop:iota-g-Pop} and~\ref{prop:ph-rep-maps} give
\begin{align*}
\check{\rho} (\check{\iota} \, h) & = \tfrac{1}{36} \rho_g \big(16 \iota_g h + 4 (\Pop (\iota_g h) + (\iota_g h) \Pop) + (\Pop (\iota_g h) \Pop) \big) \\
& = \tfrac{1}{36} \big( 16 \rho_g (\iota_g h) + 4 \rho_g (- 2 h \diamond \ps) + \rho_g (- \tfrac{8}{3} h \diamond \ps - \tfrac{16}{5} \iota_g h + 4 (\iota_{\ph} h)_{\W}) \big) \\
& = \tfrac{1}{36} (16 (5h) + 4 (0) + (0 - \tfrac{16}{5} 5h + 0)) = \tfrac{16}{9} h.
\end{align*}

We have therefore shown that
$$ \check{\rho} (\check{\iota} (\lambda g + h^0)) = 8 \lambda g + \tfrac{16}{9} h^0, $$
so condition $\mathrm{(i)}$ of~\eqref{eq:basic-condition} is satisfied with $b_1 = 8$ and $b_{27} = \frac{16}{9}$. 
Moreover, from the second equation in~\eqref{eq:iota-g-rho-g-7}, for $h \in \cS^2$ and $U \in \cS^2(\mb{14})$ we have
$$ \langle \check{\rho} \, U, h \rangle = \langle \rho_g U, h \rangle = \tfrac{1}{4} \langle U, \iota_g h \rangle = \tfrac{1}{4} \langle U, \check{\iota} \, h \rangle, $$
so condition $\mathrm{(ii)}$ of~\eqref{eq:basic-condition} is satisfied with $c_1 = c_{27} = \frac{1}{4}$. We can thus invoke Lemma~\ref{lemma:basic-tool} to conclude that
We can therefore invoke Lemma~\ref{lemma:basic-tool} to conclude that
\begin{equation*}
\cS^2 (\mb{14}) = \mb{77} \oplus \mb{1} \oplus \mb{27},
\end{equation*}
where explicitly we have $U = U_{77} + U_1 + U_{27}$ with
\begin{equation} \label{eq:S2-14-decomp}
\begin{aligned}
U_1 & = \tfrac{1}{8} \check{\iota} (\check{\rho} \, U)_1, \\
U_{27} & = \tfrac{9}{16} \check{\iota} (\check{\rho} U)_{27}, \\
U_{77} & = U - U_1 - U_{27}.
\end{aligned}
\end{equation}
In particular, the $77$-dimensional representation corresponds to $\ker \check{\rho}$, and thus concretely we have
\begin{equation} \label{eq:77-desc}
U_{ijkl} \in \mb{77} \, \iff \, \begin{cases} U_{ijkl} = - U_{jikl} = - U_{ijlk} = U_{klij}, \\ U_{ijkl} \ph_{klm} = 0, \\ U_{ljkl} = 0. \end{cases}
\end{equation}

\section{Curvature, torsion, and functionals} \label{sec:curvature-torsion}

In this section we apply the results of Section~\ref{sec:more-rep-theory}. We determine several independent relations between $\tRm$ and $\nab{} T$ obtained by decomposing the $\G$-Bianchi identity~\eqref{eq:g2bianchi} into components. We then use these results to simplify the evolution equations for torsion functionals that were derived in Section~\ref{sec:functionals}, and to determine their associated Euler--Lagrange equations. Next, we consider the decompositions of $R_{ijkl}$ and $\nab{m} T_{pq}$ into independent components corresponding to irreducible $\G$-representations, and identify those which are related by the $\G$-Bianchi identity, and those which can be made into $3$-forms, for the purpose of classifying all possible quasilinear second-order geometric flows of $\G$-structures.

\subsection{Decomposition of the $\G$-Bianchi identity into independent relations} \label{sec:g2bianchi-revisited}

In this section we use the representation-theoretic results that we established in Section~\ref{sec:more-rep-theory} to extract from the $\G$-Bianchi identity several independent relations between the Riemann curvature $\tRm$ and the covariant derivative $\nab{}T$ of the torsion.

Let us rewrite the $\G$-Bianchi identity~\eqref{eq:g2bianchi} in the form
\begin{equation} \label{eq:g2bianchi-decomp}
G_{pij} = \nab{i} T_{jp} - \nab{j} T_{ip} - T_{ia} T_{jb} \ph_{abp} - \tfrac{1}{2} R_{ijab} \ph_{abp} = 0.
\end{equation}
Here $G_{pij}$ are the components of a tensor $G \in \Gamma(T^* M \otimes \Lambda^2 (T^*M))$, because $G_{pij}$ is skew in $i,j$. This can therefore be decomposed into two components $G^7 + G^{14}$, where $G^k \in \Gamma(T^*M \otimes \Lambda^2_k (T^* M))$ for $k = 7, 14$. Using the decompositions
$$ \mb{7} \otimes \mb{7} = \mb{1} \oplus \mb{27} \oplus \mb{7} \oplus \mb{14} \qquad \text{and} \qquad \mb{7} \otimes \mb{14} = \mb{64} \oplus \mb{27} \oplus \mb{7}, $$
we can therefore extract from~\eqref{eq:g2bianchi-decomp} several independent relations. Specifically, we extract six relations: one in $\mb{1}$, one in $\mb{14}$, two in $\mb{7}$, and two in $\mb{27}$. In the present paper, we do not make explicit use of the $\mb{64}$ relation. (See also the discussion following Remark~\ref{rmk:Hodge-Lap-revisited} regarding explicit computation of the $\mb{64}$ component of the curvature in terms of torsion.)

\begin{lemma} \label{lemma:g2bianchi-contractions}
Let $G \in \Gamma(T^* M \otimes \Lambda^2 (T^*M))$ be as given in~\eqref{eq:g2bianchi-decomp}. The following contractions of $G$ with $\ph$, $\ps$, and the metric hold:
\begin{align}
G_{iim} & = (\Div T^t)_m - \nab{m} (\tr T) - (T(\Vop T))_m, \label{eq:Giim} \\
G_{ijp} \ph_{ijq} & = \KK{2}_{pq} + \nab{p} (\Vop T)_q - (\tr T) T_{pq} + T^2_{pq} + R_{pq}, \label{eq:Gijpphijq} \\
G_{pij} \ph_{ijq} & = 2 \, \KK{3}_{pq} - (T \oct T)_{qp} - \tfrac{1}{2} F_{pq}, \label{eq:Gpijphijq} \\
G_{ijk} \ph_{ijk} & = 2 \Div(\Vop T) - (\tr T)^2 + \langle T, T^t \rangle + \langle T, \Pop T \rangle + R, \label{eq:Gijkphijk} \\
G_{ijk} \ps_{ijkm} & = 2 \langle \nab{} T, \ps \rangle_m - 2 (\tr T) (\Vop T)_m + 2 (\Vop (T^2))_m + 2 (T^t (\Vop T))_m. \label{eq:Gijkpsijkm}
\end{align}
\end{lemma}
\begin{proof}
Contracting~\eqref{eq:g2bianchi-decomp}, we obtain
\begin{align*}
G_{iik} & = \nab{i} T_{ki} - \nab{k} T_{ii} - T_{ia} T_{kb} \ph_{abi} - \tfrac{1}{2} R_{ikab} \ph_{abi} \\
& = (\Div T^t)_k - \nab{k} (\tr T) - (T(\Vop T))_k - 0,
\end{align*}
which is~\eqref{eq:Giim}. Using~\eqref{eq:K-defn}, and~\eqref{eq:KK1-simp}, we compute
\begin{align*}
G_{ijp} \ph_{ijq} & = (\nab{j} T_{pi} - \nab{p} T_{ji} - T_{ja} T_{pb} \ph_{abi} - \tfrac{1}{2} R_{jpab} \ph_{abi}) \ph_{ijq} \\
& = \nab{j} T_{pi} \ph_{jqi} + \nab{p} T_{ji} \ph_{qji} - (T_{ja} T_{pb} + \tfrac{1}{2} R_{jpab}) (\ph_{abi} \ph_{jqi}) \\
& = \KK{2}_{pq} + \KK{1}_{pq} - (T_{ja} T_{pb} + \tfrac{1}{2} R_{jpab}) (g_{aj} g_{bq} - g_{aq} g_{bj} - \ps_{abjq}) \\
& = \KK{2}_{pq} + \KK{1}_{pq} - (\tr T) T_{pq} + T^2_{pq} + (T (\Pop T))_{pq} + \tfrac{1}{2} R_{pq} + \tfrac{1}{2} R_{pq} + 0 \\
& = \KK{2}_{pq} + \nab{p} (\Vop T)_q - (\tr T) T_{pq} + T^2_{pq} + R_{pq},
\end{align*}
which is~\eqref{eq:Gijpphijq}. Using~\eqref{eq:Aoct-defn},~\eqref{eq:F-defn}, and~\eqref{eq:K-defn}, we obtain
\begin{align*}
G_{pij} \ph_{ijq} & = ( \nab{i} T_{jp} - \nab{j} T_{ip} - T_{ia} T_{jb} \ph_{abp} - \tfrac{1}{2} R_{ijab} \ph_{abp} ) \ph_{ijq} \\
& = 2 \nab{i} T_{jp} \ph_{ijq} - T_{ia} T_{jb} \ph_{abp} \ph_{ijq} - \tfrac{1}{2} F_{pq} \\
& = 2 \, \KK{3}_{pq} - (T \oct T)_{qp} - \tfrac{1}{2} F_{pq},
\end{align*}
which is~\eqref{eq:Gpijphijq}. Contracting~\eqref{eq:Gpijphijq} on $p,q$ and using Definition~\ref{defn:K}, Lemma~\ref{lemma:trace-F},~\eqref{eq:Aoct-identities}, and~\eqref{eq:divVT} yields
\begin{align*}
G_{pij} \ph_{pij} & = 2 \nab{p} T_{ij} \ph_{pij} - \big( (\tr T)^2 - \langle T, T^t \rangle + \langle T, \Pop T \rangle \big) - \tfrac{1}{2} (-2R) \\
& = 2 \Div(\Vop T) - (\tr T)^2 + \langle T, T^t \rangle + \langle T, \Pop T \rangle + R,
\end{align*}
which is~\eqref{eq:Gijkphijk}. This can clearly also be obtained by contracting~\eqref{eq:Gijpphijq} and using~\eqref{eq:divVT}.

Finally, we have
\begin{align*}
G_{ijk} \ps_{ijkm} & = (\nab{j} T_{ki} - \nab{k} T_{ji} - T_{ia} T_{jb} \ph_{abk} - \tfrac{1}{2} R_{ijab} \ph_{abk}) \ps_{ijkm} \\
& = 2 \nab{j} T_{ki} \ps_{ijkm} + (T_{ia} T_{jb} + \tfrac{1}{2} R_{ijab}) \ph_{abk} \ps_{ijmk} \\
& = 2 \langle \nab{} T, \ps \rangle_m \\
& \qquad {} + (T_{ia} T_{jb} + \tfrac{1}{2} R_{ijab}) (g_{ai} \ph_{bjm} + g_{aj} \ph_{ibm} + g_{am} \ph_{ijb} - g_{bi} \ph_{ajm} - g_{bj} \ph_{iam} - g_{bm} \ph_{ija}) \\
& = 2 \langle \nab{} T, \ps \rangle_m - 2 (\tr T) (\Vop T)_m + 2 (\Vop (T^2))_m + 2 (T^t (\Vop T))_m + 0,
\end{align*}
where all the curvature terms above vanish either because $\tRc$ is symmetric and $\ph$ is skew, or by the Riemannian first Bianchi identity. We thus have~\eqref{eq:Gijkpsijkm}.
\end{proof}

\begin{thm} \label{thm:g2bianchi-decomp}
From the $\G$-Bianchi identity~\eqref{eq:g2bianchi-decomp} we can extract several independent relations between $\tRm$ and $\nabla T$. These are:
\begin{equation*}
\begin{aligned}
& (\mb{G1}) & & \quad R = (\tr T)^2 - \langle T, T^t \rangle - \langle T, \Pop T \rangle - 2 \Div (\Vop T), \\
& (\mb{G7_a}) & & \quad \Div T^t - \nab{} (\tr T) - T (\Vop T) = 0, \\
& (\mb{G7_b}) & & \quad \langle \nab{} T, \ps \rangle - (\tr T) \Vop T + \Vop (T^2) + T^t (\Vop T) = 0, \\
& (\mb{G14}) & & \quad \pi_{14}(\KK{3}) = - (\tr T) T_{14} + (T^2)_{14} + ( (\Pop T) T)_{14}, \\
& (\mb{G27_a}) & & \quad (\KK{3})_{27} = \tfrac{1}{2} (T \oct T)_{27} + \tfrac{1}{4} F_{27}, \\
& (\mb{G27_b}) & & \quad (\KK{2})_{27} = - \tfrac{1}{2} (\cL_{\Vop T} g)_{27} + (\tr T) T_{27} - T^2_{27} - \tRc_{27}.
\end{aligned}
\end{equation*}
\end{thm}
\begin{proof}
For convenience of notation throughout this proof, define $2$-tensors $H_{pq}$ and $\wt H_{pq}$ by
\begin{equation} \label{eq:H-temps}
H_{pq} := G_{pij} \ph_{ijq}, \qquad \wt H_{pq} := G_{ijp} \ph_{ijq}.
\end{equation}
We have the decomposition
$$ G_{pij} = G^7_{pij} + G^{14}_{pij}. $$
where $G^7 \in \Gamma(T^* M \otimes \Lambda^2_7 (T^* M))$ and $G^{14} \in \Gamma(T^* M \otimes \Lambda^2_{14} (T^* M))$. From~\eqref{eq:omega27vf}, we have that $G^7_{pij} = 0$ if and only if $H_{pq} = G_{pij} \ph_{ijq} = 0$. Thus, the $\mb{1} \oplus \mb{27} \oplus \mb{7} \oplus \mb{14}$ components of the identity $G^7_{pij} = 0$ correspond to $H_{pq} = 0$, which by~\eqref{eq:Gpijphijq} give
\begin{equation} \label{eq:H-full-temp}
H_{pq} = 2 \, \KK{3}_{pq} - (T \oct T)_{qp} - \tfrac{1}{2} F_{pq} = 0.
\end{equation}
The vanishing of the $\mb{1}$ part of $H_{pq}$ corresponds to $\tr H = H_{pp} = \ph_{pij} \ph_{ijp} = 0$, which by~\eqref{eq:Gijkphijk} yields the condition $(\mb{G1})$. This was of course expected because Corollary~\ref{cor:scalar-curvature} was obtained precisely by taking an appropriate trace of the $\G$-Bianchi identity. The vanishing of the $\mb{7}$ part of $H_{pq}$ corresponds to the vanishing of
\begin{align*}
(\Vop H)_m & = H_{pq} \ph_{pqm} = G_{pij} \ph_{ijq} \ph_{mpq} \\
& = G_{pij} (g_{im} g_{jp} - g_{ip} g_{jm} - \ps_{ijmp}) \\
& = G_{imi} - G_{iim} + G_{pij} \ps_{pijm},
\end{align*}
which simplifies to
\begin{equation} \label{eq:H7-temp7}
- 2 G_{iim} + G_{ijk} \ps_{ijkm} = 0.
\end{equation}
Below, after we find the additional information determined by the $\mb{7}$ part of $G^{14}$ in equation~\eqref{eq:H14-temp7}, we then deduce equations $(\mb{G7_a})$ and $(\mb{G7_b})$.

Applying $\pi_{14}$ to~\eqref{eq:H-full-temp} and using~\eqref{eq:Aoct-7-14} and the symmetry of $F_{pq}$ gives
\begin{align*}
0 & = (\pi_{14} H)_{qp} = 2 (\pi_{14} (\KK{3}))_{qp} - (\pi_{14} (T \oct T))_{pq} - 0 \\
& = - 2 \pi_{14} (\KK{3})_{pq} - \big(2 (\tr T) T_{14} - 2 (T^2)_{14} - 2 ( (\Pop T) T )_{14} \big){}_{pq}.
\end{align*}
Thus the vanishing of the $\mb{14}$ part of $H_{pq}$ is equivalent to the condition $(\mb{G14})$.

Taking the symmetric part of~\eqref{eq:H-full-temp} gives
\begin{equation} \label{eq:KK3symm}
(\KK{3})_{\symm} = \tfrac{1}{2} (T \oct T)_{\symm} + \tfrac{1}{4} F.
\end{equation}
Thus the vanishing of the $\mb{27}$ part of $H_{pq}$ is equivalent to the condition $(\mb{G27_a})$.

Now consider the vanishing of the component $G^{14}_{pij}$. Using~\eqref{eq:omega2-proj}, we get
\begin{align} \nonumber
6 G^{14}_{pij} & = 4 G_{pij} + G_{pab} \ps_{abij} \\ \nonumber
& = 4 G_{pij} + G_{pab} (- \ph_{abm} \ph_{ijm} + g_{ai} g_{bj} - g_{aj} g_{bi} ) \\ \nonumber
& = 4 G_{pij} - H_{pm} \ph_{mij} + G_{pij} - G_{pji} \\ \label{eq:G14-proj-temp}
& = 6 G_{pij} - H_{pm} \ph_{mij}.
\end{align}
We seek to extract the \emph{additional} information encoded in $G^{14}_{pij} = 0$ not already implied by $G^7_{pij} = 0$, which we saw was equivalent to $H_{pq} = 0$ in~\eqref{eq:H-full-temp}. Thus, the computation~\eqref{eq:G14-proj-temp} merely confirms that, if we already assume that $G^7_{pij} = 0$, then the new information given by the vanishing of $G^{14}_{pij}$ is equivalent to just setting $G_{pij} = 0$ and using $H_{pq} = 0$. That is, we can now assume that $G_{pij} \in \Gamma(T^* M \otimes \Lambda^2(T^* M))$.

Following Section~\ref{sec:7-14}, if we define
$$ \gamma_{pij} = (\rho G)_{pij} = G_{pij} + G_{ijp} + G_{jpi}, $$
then $\gamma = \rho G \in \Omega^3_7 \oplus \Omega^3_{27}$, encodes the $\mb{27} \oplus \mb{7}$ part of $G^{14} \in \mb{64} \oplus \mb{27} \oplus \mb{7}$. By Corollary~\ref{cor:diamondinverse}, the vanishing of $\gamma_7$ and $\gamma_{27}$ are equivalent, respectively, to the vanishing of the $\mb{7}$ and $\mb{27}$ parts of $\gamma^{\ph}_{pq} = \gamma_{pij} \ph_{qij}$. We have
\begin{align*}
\gamma^{\ph}_{pq} & = \gamma_{pij} \ph_{qij} = (G_{pij} + G_{ijp} + G_{jpi}) \ph_{qij} \\
& = G_{pij} \ph_{ijq} + 2 G_{ijp} \ph_{ijq} = H_{pq} + 2 \wt H_{pq}.
\end{align*}
Again, since we are assuming that $H_{pq} = 0$, the \emph{new} information given by the $\mb{27} \oplus \mb{7}$ parts of $G^{14} = 0$ are encoded in $\wt H_{pq} = 0$. Note that $\wt H_{pq}$ is a $2$-tensor which \emph{does} have $\mb{1} \oplus \mb{14}$ components, but these are just constant multiples of the $\mb{1} \oplus \mb{14}$ components of $H_{pq}$. That is, only the $\mb{7} \oplus \mb{27}$ components of $\wt H_{pq}$ contain any \emph{new} information.

By~\eqref{eq:Gijpphijq} we have
\begin{equation} \label{eq:wtH-full-temp}
\wt H_{pq} = \KK{2}_{pq} + \nab{p} (\Vop T)_q - (\tr T) T_{pq} + T^2_{pq} + R_{pq} = 0.
\end{equation}
The vanishing of the $\mb{7}$ part of $\wt H_{pq}$ corresponds to the vanishing of
\begin{align*}
(\Vop \wt H)_m & = \wt H_{pq} \ph_{pqm} = G_{ijp} \ph_{ijq} \ph_{mpq} \\
& = G_{ijp} (g_{im} g_{jp} - g_{ip} g_{jm} - \ps_{ijmp}) \\
& = 0 - G_{imi} + G_{ijp} \ps_{ijpm},
\end{align*}
which simplifies to
\begin{equation} \label{eq:H14-temp7}
G_{iim} + G_{ijk} \ps_{ijkm} = 0.
\end{equation}
Comparing~\eqref{eq:H7-temp7} and~\eqref{eq:H14-temp7}, we deduce that
\begin{equation} \label{eq:vecG-temp}
G_{iim} = 0, \qquad G_{ijk} \ps_{ijkm} = 0.
\end{equation}
Comparing with~\eqref{eq:Giim} and~\eqref{eq:Gijkpsijkm}, we conclude that the vanishing of the two $\mb{7}$ parts of the $\G$-Bianchi identity are equivalent to the conditions $(\mb{G7_a})$ and $(\mb{G7_b})$.

Taking the symmetric part of~\eqref{eq:wtH-full-temp} gives
\begin{equation} \label{eq:KK2symm}
\KK{2}_{\symm} = - \tfrac{1}{2} \cL_{\Vop T} g + (\tr T) T_{\symm} - T^2_{\symm} - \tRc.
\end{equation}
Thus the vanishing of the $\mb{27}$ part of $\wt H_{pq}$ is equivalent to the condition $(\mb{G27_b})$.
\end{proof}

For our purposes, it is most useful to repackage the independent relations in Theorem~\ref{thm:g2bianchi-decomp} as follows.

\begin{cor} \label{cor:g2bianchi-repackaged}
The following five tensors constructed from $\nab{}T$:
$$ \Div (\Vop T) \in \Omega^0, \qquad \nab{}(\tr T), \langle \nab{} T, \ps \rangle \in \Omega^1, \qquad (\KK{2})_{\symm}, (\KK{3})_{\symm} \in \cS^2, $$
can be expressed in terms of $\Div T^t$, $\cL_{\Vop T} g$, the curvature tensors $R$, $\tRc$, $F$, and lower order terms which are quadratic in the torsion. Explicitly, we have:
\begin{align*}
\Div (\Vop T) & = - \tfrac{1}{2} R + \tfrac{1}{2} (\tr T)^2 - \tfrac{1}{2} \langle T, T^t \rangle - \tfrac{1}{2} \langle T, \Pop T \rangle, \\
\nab{} (\tr T) & = \Div T^t - T (\Vop T), \\
\langle \nab{}T, \ps \rangle & = (\tr T) \Vop T - \Vop (T^2) - T^t (\Vop T), \\
(\KK{2})_{\symm} & = - \tfrac{1}{2} \cL_{\Vop T} g + (\tr T) T_{\symm} - T^2_{\symm} - \tRc, \\
(\KK{3})_{\symm} & = \tfrac{1}{2} (T \oct T)_{\symm} + \tfrac{1}{4} F.
\end{align*}
\end{cor}
\begin{proof}
The first three equations are rearrangements of $(\mb{G1})$, $(\mb{G7_a})$, and $(\mb{G7_b})$. The last two equations are~\eqref{eq:KK2symm} and~\eqref{eq:KK3symm}, respectively.
\end{proof}

\begin{rmk} \label{rmk:curvature-torsion-free}
Corollary~\ref{cor:g2bianchi-repackaged} yields the following useful expressions for the Ricci tensor $R_{pq}$ and the symmetric $2$-tensor $F_{pq} = R_{ijkl} \ph_{ijp} \ph_{klq}$ in terms of the torsion and its covariant derivative:
\begin{equation} \label{eq:Rc-F-in-terms-of-T}
\begin{aligned}
\tRc & = - (\KK{2})_{\symm} - \tfrac{1}{2} \cL_{\Vop T} g + (\tr T) T_{\symm} - T^2_{\symm}, \\
F & = 4 (\KK{3})_{\symm} - 2 (T \oct T)_{\symm}.
\end{aligned}
\end{equation}
If we expand the $\cL_{\Vop T} g$ term in the expression for $\tRc$ above, we obtain
\begin{equation} \label{eq:Ricci-from-T}
\begin{aligned}
R_{ij} & = - \tfrac{1}{2} (\nab{p} T_{iq} \ph_{pjq} + \nab{p} T_{jq} \ph_{piq}) - \tfrac{1}{2} (\nab{i} T_{pq} \ph_{jpq} + \nab{j} T_{pq} \ph_{ipq}) \\
& \qquad {} - \tfrac{1}{2} (T_{im} T_{pq} \ps_{pqmj} + T_{jm} T_{pq} \ps_{pqmi}) + \tfrac{1}{2} (\tr T) (T_{ij} + T_{ji}) - \tfrac{1}{2} (T_{im} T_{mj} + T_{jm} T_{mi}),
\end{aligned}
\end{equation}
which precisely agrees with~\cite[Equation (4.19), symmetrized]{K-flows}. In terms of a local orthonormal frame, the expression for $F$ above is
\begin{equation} \label{eq:F-from-T}
F_{ij} = 2 \nab{p} T_{qi} \ph_{pqj} + 2 \nab{p} T_{qj} \ph_{pqi} - 2 T_{pa} T_{qb} \ph_{pqi} \ph_{abj}.
\end{equation}
This can also be obtained by computing $F_{ij} = R_{pqab} \ph_{pqi} \ph_{abj}$ from the $\G$-Bianchi identity~\eqref{eq:g2bianchi} and symmetrizing. In particular, if $\ph$ is \emph{torsion-free}, then both $R_{ij}$ and $F_{ij}$ vanish. The fact that $\tRc$ and $F$ are expressible in terms of torsion is well-known. These formulas for example can be found, albeit in a very different form, in Cleyton--Ivanov~\cite[Lemma 4.4]{CI}. We also explain how to express two of the three independent components of the Weyl curvature in terms of torsion in Section~\ref{sec:curvature-decomp}.
\end{rmk}

\begin{rmk} \label{rmk:FandRic-closed}
Suppose $\dd \ph = 0$.  In this case, $T = T_{14}$ is skew-symmetric. Thus, using~\eqref{eq:omega214desc} for $T_{14}$ and~\eqref{eq:nablaph}, equation~\eqref{eq:Ricci-from-T} can be further simplified as
\begin{align*}
R_{ij} & = - \tfrac{1}{2} (\nab{p} T_{iq} \ph_{pjq} + \nab{p} T_{jq} \ph_{piq}) - \tfrac{1}{2} (\nab{i} T_{pq} \ph_{jpq} + \nab{j} T_{pq} \ph_{ipq}) \\
& \qquad {} - \tfrac{1}{2} (T_{im} T_{pq} \ps_{pqmj} + T_{jm} T_{pq} \ps_{pqmi}) + \tfrac{1}{2} (\tr T) (T_{ij} + T_{ji}) - \tfrac{1}{2} (T_{im} T_{mj} + T_{jm} T_{mi}) \\
&= - \tfrac{1}{2} (\nab{p} T_{qi} \ph_{pqj} + \nab{p} T_{qj} \ph_{pqi}) + \tfrac{1}{2} T_{pq} T_{im} \psi_{mjpq} + \tfrac{1}{2}T_{pq} T_{jm} \psi_{mipq} - T_{im} T_{mj} - T_{jm} T_{mi} \\
& \qquad {} - \tfrac{1}{2} (T_{im} T_{mj} + T_{jm} T_{mi}) \\
& = - \tfrac{1}{2} (\nab{p} T_{qi} \ph_{pqj} + \nab{p} T_{qj} \ph_{pqi}) - \tfrac{1}{2} (T_{im} T_{mj} + T_{jm} T_{mi}).
\end{align*} 
In particular, from \eqref{eq:F-from-T} we obtain
\begin{align} \label{eq:Ric-F-closed}
R_{ij} & = - \tfrac{1}{4} F_{ij} - \tfrac{1}{2}T_{pa} T_{qb} \ph_{pqi} \ph_{abj} -  T_{im} T_{mj} & & \text{when $\dd \ph = 0$},
\end{align}
which gives a relation between $\tRc$ and $F$ in terms of lower-order terms.
\end{rmk}

\begin{rmk} \label{rmk:FandRic-coclosed}
Suppose $\dd \psi = 0$. In this case, $T = T_1 + T_{27}$ is symmetric. The expression~\eqref{eq:Ricci-from-T} can be further simplified as
\begin{align*}
R_{ij} & = - \tfrac{1}{2} (\nab{p} T_{iq} \ph_{pjq} + \nab{p} T_{jq} \ph_{piq}) - \tfrac{1}{2} (\nab{i} T_{pq} \ph_{jpq} + \nab{j} T_{pq} \ph_{ipq}) \\
& \qquad {} - \tfrac{1}{2} (T_{im} T_{pq} \ps_{pqmj} + T_{jm} T_{pq} \ps_{pqmi}) + \tfrac{1}{2} (\tr T) (T_{ij} + T_{ji}) - \tfrac{1}{2} (T_{im} T_{mj} + T_{jm} T_{mi}) \\
& = \tfrac{1}{2} (\nab{p} T_{qi} \ph_{pqj} + \nab{p} T_{qj} \ph_{pqi})+ (\tr T) T_{ij} - T_{im} T_{mj}. 
\end{align*}
The above expression and~\eqref{eq:F-from-T} yield
\begin{align}
R_{ij} = \tfrac{1}{4} F_{ij} + \tfrac{1}{2} T_{pa} T_{qb} \ph_{pqi} \ph_{abj} + (\tr T) T_{ij} - T_{im} T_{mj} & & \text{when $\dd \ps = 0$},
\end{align}
which gives a relation between $\tRc$ and $F$ in terms of lower-order terms.
\end{rmk}

\begin{rmk} \label{rmk:d-squared}
From the fact that $\dd^2 = 0$, we get two identities for any $\G$-structures $\ph$, namely the $\Omega^6_7$ form $\dd^2 \ps = 0$, and the $\Omega^5_7 \oplus \Omega^5_{14}$ form $\dd^2 \ph = 0$. It is easy to check that the resulting conditions in $\mb{7} \oplus \mb{7} \oplus \mb{14}$ are equivalent to $(\mb{G7_a})$, $(\mb{G7_b})$, and $(\mb{G14})$ from Theorem~\ref{thm:g2bianchi-decomp}. This is of course expected, since $d$ can be written in terms of $\nab{}$.
\end{rmk}

\begin{cor} \label{cor:curlVT-revisited}
The vector field $\curl (\Vop T)$ is related to the vector fields $\Div T$ and $\Div T^t$ by
\begin{equation} \label{eq:curlVT-revisited}
\curl (\Vop T) = \Div T^t - \Div T + T^t (\Vop T) - T (\Vop T).
\end{equation}
Consequently, the $3$-form $\cL_{\Vop T} \ph$ can be expressed as
\begin{equation} \label{eq:LieVTph}
\cL_{\Vop T} \ph = \tfrac{1}{2} (\cL_{\Vop T} g) \diamond \ph + \tfrac{1}{2} ( \Div T - \Div T^t + T^t (\Vop T) + T (\Vop T) ) \hk \ps.
\end{equation}
\end{cor}
\begin{proof}
Equation~\eqref{eq:curlVT-revisited} follows from~\eqref{eq:curlVT} and the expression for $\langle \nab{}T, \ps \rangle$ in Corollary~\ref{cor:g2bianchi-repackaged}. Letting $W = \Vop T$ in~\eqref{eq:Lie-derivative-alternate}, we thus obtain
\begin{equation*}
\cL_{\Vop T} \ph = \tfrac{1}{2} (\cL_{\Vop T} g) \diamond \ph + \big( T^t (\Vop T) - \tfrac{1}{2} ( \Div T^t - \Div T + T^t (\Vop T) - T (\Vop T) ) \big) \hk \ps,
\end{equation*}
which simplifies to~\eqref{eq:LieVTph}.
\end{proof}

\subsection{Evolution of torsion functionals revisited} \label{sec:functionals-revisited}

In this section we revisit the evolution equations of Proposition~\ref{prop:evolution-torsion-quantities-2}, and simplify them using the results of Section~\ref{sec:g2bianchi-revisited}. As a result, this yields the Euler--Lagrange equations for these torsion functionals. Here we assume $M$ is compact, so that all integrals are defined.

Recall that $A = A_{1 + 27} + A_7 \in \cS \oplus \Omega^2_7$. Write $A_{1+27} = h$ and $A_7 = - \frac{1}{3} X \hk \ph$ as in Section~\ref{sec:basic-flow}, so that that flow~\eqref{eq:general-flow-A} is equivalent to~\eqref{eq:general-flow}. In this section it is convenient to use formulation~\eqref{eq:general-flow} of the flow, in terms of the pair $(h, X)$. We first rewrite Propostion~\ref{prop:evolution-torsion-quantities-2} in terms of this formulation of the flow.

If $B = B_{\symm} + B_7 \in \cS^2 \oplus \Omega^2_7$, then from~\eqref{eq:vec-transform} we have $B_7 = \frac{1}{6} (\Vop B) \hk \ph$, and thus using~\eqref{eq:1727metric-b} and the orthogonality of $\cS$ and $\Omega^2_7$, we have
\begin{equation} \label{eq:functionals-simp1}
\langle B, A \rangle = \langle B_{\symm} + \tfrac{1}{6} (\Vop B) \hk \ph, h - \tfrac{1}{3} X \hk \ph \rangle = \langle B_{\symm}, h \rangle - \tfrac{1}{3} \langle \Vop B, X \rangle
\end{equation}
and
\begin{equation} \label{eq:functionals-simp2}
\langle Y \hk \ph, A_7 \rangle = \langle Y \hk \ph, - \tfrac{1}{3} X \hk \ph \rangle = - 2 \langle Y, X \rangle.
\end{equation}

\begin{prop} \label{prop:evolution-torsion-quantities-3}
Let $\ph$ be a time-dependent family of $\G$-structures evolving by the flow~\eqref{eq:general-flow}. We have the following evolution equations for various quadratic integral quantities obtained from the torsion:
\begin{align*}
\delt \Big( \int_M (\tr T)^2 \vol \Big) & = \int_M \langle - 2 \nab{} (\tr T) - 2 (\tr T) \Vop T, X \rangle \vol \\
& \qquad {} + \int_M \langle (\tr T)^2 g - 2 (\tr T) T_{\symm}, h \rangle \vol, \\
\delt \Big( \int_M |T|^2 \vol \Big) & = \int_M \langle - 2 \Div T, X \rangle \vol \\
& \qquad {} + \int_M \langle - 2 (\KK{2})_{\symm} + |T|^2 g - 2 T T^t - 2(T (\Pop T))_{\symm}, h \rangle \vol, \\
\delt \Big( \int_M \langle T, T^t \rangle \vol \Big) & = \int_M \langle - 2 \Div T^t - 2 \Vop (T^2), X \rangle \vol \\
& \qquad {} + \int_M \langle 2 (\KK{3})_{\symm} + \langle T, T^t \rangle g - 2 (T^2)_{\symm} + 2 ((\Pop T) T)_{\symm}, h \rangle \vol, \\
\delt \Big( \int_M \langle T, \Pop T \rangle \vol \Big) & = \int_M \langle - 2 \langle \nab{} T, \ps \rangle + 2 T(\Vop T) - 2 T^t (\Vop T), X \rangle \vol \\
& \qquad {} + \int_M \Big[ \langle -2 (\KK{1})_{\symm} -2 (\KK{2})_{\symm} -2 (\KK{3})_{\symm} + 2 (\tr \KK{a}) g - \langle T, \Pop T \rangle g \\
& \qquad \qquad \qquad {} - 2 (T(\Pop T))_{\symm} - 2 ((\Pop T) T)_{\symm}, h \big\rangle \Big] \vol.
\end{align*}
\end{prop}
\begin{proof}
We rewrite the four equations from Proposition~\ref{prop:evolution-torsion-quantities-2} using~\eqref{eq:functionals-simp1} and~\eqref{eq:functionals-simp2}. For the first equation, we get
\begin{align*}
\delt \Big( \int_M (\tr T)^2 \vol \Big) & = \int_M \langle - 2 \nab{} (\tr T) + \tfrac{2}{3} (\tr T) \Vop T^t + \tfrac{1}{3} (\tr T) \Vop (\Pop T), X \rangle \vol \\
& \qquad {} + \int_M \langle (\tr T)^2 g - 2 (\tr T) T^t_{\symm}, h \rangle \vol.
\end{align*}
Using $\Vop T^t = - \Vop T$ and~\eqref{eq:14-part} yields the first equation. The second equation is immediate since
$$- T(\Pop T) + (\Pop T)T^t = - 2 (T(\Pop T))_{\symm}. $$
For the third equation, since $(\Pop T) T - T^t (\Pop T) = 2 ( (\Pop T) T)_{\symm}$ and $\Pop(T^2)$ is skew, we get
\begin{align*}
\delt \Big( \int_M \langle T, T^t \rangle \vol \Big) & = \int_M \langle - 2 \Div T^t + \tfrac{2}{3} \Vop ((T^t)^2) + \tfrac{1}{3} \Vop (\Pop (T^2)), X \rangle \\
& \qquad {} + \int_M \langle 2 (\KK{3})_{\symm} + \langle T, T^t \rangle g - 2 (T^t)^2_{\symm} + 2 ((\Pop T))_{\symm}, h \rangle \vol.
\end{align*}
Using $\Vop ((T^t)^2) = \Vop ( (T^2)^t) = - \Vop (T^2)$ and~\eqref{eq:14-part} yields the third equation.

The fourth equation requires more work. First, we get
\begin{align} \nonumber
\delt \Big( \int_M \langle T, \Pop T \rangle \vol \Big) & = \int_M \langle \tfrac{2}{3} \Vop(\KK{1}^t) + \tfrac{2}{3} \Vop(\KK{2}^t) + \tfrac{2}{3} \Vop(\KK{3}^t) - \tfrac{2}{3} \Vop \big( (\Pop T) T^t \big) + \tfrac{2}{3} \Vop \big( (\Pop T) T \big), X \rangle \vol \\ \nonumber
& \qquad {} + \int_M \Big[ \langle -2 (\KK{1}^t)_{\symm} -2 (\KK{2}^t)_{\symm} -2 (\KK{3}^t)_{\symm} + 2 (\tr \KK{a}) g - \langle T, \Pop T \rangle g \\ \label{eq:functional4-simp-temp}
& \qquad \qquad \qquad {} + 2 ((\Pop T) T^t)_{\symm} - 2 ((\Pop T) T)_{\symm}, h \big\rangle \Big] \vol.
\end{align}
Using $\Vop(\KK{a}^t) = - \Vop (\KK{a})$, Lemma~\ref{lemma:Vop-KK}, and the divergence theorem, the first integral in~\eqref{eq:functional4-simp-temp} becomes
$$ \int_M \langle - 2 \langle \nab{} T, \ps \rangle - \tfrac{2}{3} \Vop \big( (\Pop T) T^t \big) + \tfrac{2}{3} \Vop \big( (\Pop T) T \big), X \rangle \vol. $$
Since $\Pop T = - \Pop T^t$, we can apply~\eqref{eq:V-of-PAA} to each of the last two terms, and use $\Vop (B^t) = - \Vop B$, to obtain after some cancellation that
\begin{align*}
\tfrac{2}{3} \Vop \big( (\Pop T^t) T^t \big) + \tfrac{2}{3} \Vop \big( (\Pop T) T \big) & = \tfrac{2}{3} \big[ \Vop ((T^t)^2) - (\tr T^t) \Vop (T^t) + 2 T^t (\Vop (T^t)) - (T^t)^t (\Vop (T^t)) \big] \\
& \qquad {} + \tfrac{2}{3} \big[ \Vop (T^2) - (\tr T) \Vop T + 2 T (\Vop T) - T^t (\Vop T) \big] \\
& = \tfrac{2}{3} \big[ 2 T (\Vop T) - 2 T^t (\Vop T) + T(\Vop T) - T^t (\Vop T) \big] \\
& = 2 T(\Vop T) - 2 T^t (\Vop T).
\end{align*}
Thus the first integral in~\eqref{eq:functional4-simp-temp} finally becomes
$$ \int_M \langle - 2 \langle \nab{} T, \ps \rangle + 2 T(\Vop T) - 2 T^t (\Vop T), X \rangle \vol, $$
which along with $( (\Pop T) T^t)_{\symm} = - (T (\Pop T))_{\symm}$, yields the fourth equation.
\end{proof}

We can now incorporate the relations imposed by the $\G$-Bianchi identity.

\begin{cor} \label{cor:evolution-torsion-quantities-4}
Let $\ph$ be a time-dependent family of $\G$-structures evolving by the flow~\eqref{eq:general-flow}. We have the following evolution equations for various quadratic integral quantities obtained from the torsion:
\begin{align*}
\delt \Big( \int_M (\tr T)^2 \vol \Big) & = \int_M \langle - 2 \Div T^t + 2 T(\Vop T) - 2 (\tr T) \Vop T, X \rangle \vol \\
& \qquad + \int_M \langle (\tr T)^2 g - 2 (\tr T) T_{\symm}, h \rangle \vol,
\end{align*}
\begin{align*}
\delt \Big( \int_M |T|^2 \vol \Big) & = \int_M \langle - 2 \Div T, X \rangle \vol \\
& \qquad + \int_M \langle 2 \tRc + \cL_{\Vop T} g + |T|^2 g - 2 (\tr T) T_{\symm} + 2T^2_{\symm} - 2 T T^t - 2 (T(\Pop T))_{\symm}, h \rangle \vol,
\end{align*}
\begin{align*}
\delt \Big( \int_M \langle T, T^t \rangle \vol \Big) & = \int_M \langle - 2 \Div T^t - 2 \Vop (T^2), X \rangle \vol \\
& \qquad {} + \int_M \langle \tfrac{1}{2} F + \langle T, T^t \rangle g + (T \oct T)_{\symm} - 2 (T^2)_{\symm} + 2 ((\Pop T) T)_{\symm}, h \rangle \vol,
\end{align*}
\begin{align*}
\delt \Big( \int_M \langle T, \Pop T \rangle \vol \Big) & = \int_M \langle - 2 (\tr T) \Vop T + 2 \Vop (T^2) + 2 T(\Vop T), X \rangle \vol \\
& \qquad {} + \int_M \Big[ \langle 2 \tRc - \tfrac{1}{2} F - R g + (\tr T)^2 g - \langle T, T^t \rangle g \\
& \qquad \qquad \qquad {} - (T \oct T)_{\symm} - 2 (\tr T) T_{\symm} + 2 T^2_{\symm} - 2 ((\Pop T) T)_{\symm}, h \big\rangle \Big] \vol.
\end{align*}
\end{cor}
\begin{proof}
We use Corollary~\ref{cor:g2bianchi-repackaged} to replace $\nab{} (\tr T)$, $\langle \nab{}T , \ps \rangle$, $(\KK{2})_{\symm}$, and $(\KK{3})_{\symm}$ by expressions involving curvature, $\cL_{\Vop T} g$, $\Div T^t$, and lower order terms that are quadratic in the torsion. We also use~\eqref{eq:KK1-simp} to write
$$ (\KK{1})_{\symm} = ( \nab{} (\Vop T) )_{\symm} - (T (\Pop T))_{\symm} = \tfrac{1}{2} \cL_{\Vop T} g - (T (\Pop T))_{\symm}, $$
and~\eqref{eq:divVT} and Corollary~\ref{cor:g2bianchi-repackaged} to write
$$ \tr \KK{a} = \Div(\Vop T) + \langle T, \Pop T \rangle = - \tfrac{1}{2} R + \tfrac{1}{2} (\tr T)^2 - \tfrac{1}{2} \langle T, T^t \rangle + \tfrac{1}{2} \langle T, \Pop T \rangle. $$
Some cancellations occur. We omit the computational details.
\end{proof}
 
\begin{cor} \label{cor:evolution-torsion-functionals}
Let $\ph$ be a time-dependent family of $\G$-structures evolving by the flow~\eqref{eq:general-flow}. The expressions $\int_M |T_k|^2 \vol$, which are the squares of the $L^2$ norms of the independent components $T_k$ of the torsion for $k=1,27,7,14$, evolve as follows:
\begin{align*}
\delt \Big( \int_M |T_1|^2 \vol \Big) & = \int_M \langle - \tfrac{2}{7} \Div T^t + \tfrac{2}{7} T(\Vop T) - \tfrac{2}{7} (\tr T) \Vop T, X \rangle \vol \\
& \quad + \int_M \langle \tfrac{1}{7} (\tr T)^2 g - \tfrac{2}{7} (\tr T) T_{\symm}, h \rangle \vol,
\end{align*}
\begin{align*}
\delt \Big( \int_M |T_{27}|^2 \vol \Big) & = \int_M \langle - \Div T - \tfrac{5}{7} \Div T^t - \tfrac{2}{7} T (\Vop T) + \tfrac{2}{7} (\tr T) \Vop T - \Vop (T^2), X \rangle \vol \\
& \quad + \int_M \Big[ \langle \tRc + \tfrac{1}{4} F + \tfrac{1}{2} \cL_{\Vop T} g + \tfrac{1}{2} |T|^2 g + \tfrac{1}{2} \langle T, T^t \rangle g - \tfrac{1}{7} (\tr T)^2 g + \tfrac{1}{2} (T \oct T)_{\symm} \\
& \qquad \quad \quad {} - T T^t - \tfrac{5}{7} (\tr T) T_{\symm} - (T (\Pop T))_{\symm} + ((\Pop T) T)_{\symm}, h \rangle \Big] \vol,
\end{align*}
\begin{align*}
\delt \Big( \int_M |T_7|^2 \vol \Big) & = \int_M \langle - \tfrac{1}{3} \Div T + \tfrac{1}{3} \Div T^t + \tfrac{1}{3} (\tr T) \Vop T - \tfrac{1}{3} T (\Vop T), X \rangle \vol \\
& \quad + \int_M \Big[ \langle \tfrac{1}{6} R g + \tfrac{1}{6} \cL_{\Vop T} g + \tfrac{1}{6} |T|^2 g - \tfrac{1}{6} (\tr T)^2 g \\
& \qquad \quad \quad {} + \tfrac{1}{3} T^2_{\symm} - \tfrac{1}{3} T T^t - \tfrac{1}{3} ( T(\Pop T) )_{\symm}, h \rangle \Big] \vol,
\end{align*}
\begin{align*}
\delt \Big( \int_M |T_{14}|^2 \vol \Big) & = \int_M \langle - \tfrac{2}{3} \Div T + \tfrac{2}{3} \Div T^t - \tfrac{1}{3} (\tr T) \Vop T + \tfrac{1}{3} T (\Vop T) + \Vop (T^2), X \rangle \vol \\
& \quad + \int_M \Big[ \langle \tRc - \tfrac{1}{6} R g - \tfrac{1}{4} F + \tfrac{1}{3} \cL_{\Vop T} g + \tfrac{1}{3} |T|^2 g - \tfrac{1}{2} \langle T, T^t \rangle g + \tfrac{1}{6} (\tr T)^2 g - \tfrac{1}{2} (T \oct T)_{\symm} \\
& \qquad \quad \quad {} + \tfrac{5}{3} T^2_{\symm} - \tfrac{2}{3} T T^t - (\tr T) T_{\symm} - ((\Pop T) T)_{\symm} - \tfrac{2}{3} (T (\Pop T))_{\symm}, h \rangle \Big] \vol.
\end{align*}
\end{cor}
\begin{proof}
These follow from Corollary~\ref{cor:evolution-torsion-quantities-4}, using the relations in~\eqref{eq:torsion-formulas-b}. We omit the details.
\end{proof}

Corollary~\ref{cor:evolution-torsion-functionals} immediately yields the \emph{Euler--Lagrange equations} for the critical points of the torsion functionals $\int_M |T_k|^2 \vol$. Explicitly, suppose the $\G$-structure $\ph$ is a critical point of the functional $\ph \mapsto \int_M |T_k|^2 \vol$ with respect to all possible variations of $\ph$. Then by Corollary~\ref{cor:evolution-torsion-functionals} there exists a vector field $Y$ and a symmetric $2$-tensor $\ell$ such that
$$ 0 = \delt \Big( \int_M |T_k|^2 \vol \Big) = \int_M \big[ \langle Y, X \rangle + \langle \ell, h \rangle \big] \vol $$
for all vector fields $X$ and all symmetric $2$-tensors $h$. Hence $Y = 0$ and $\ell = 0$ are the Euler--Lagrange equations for this functional.

For example, $\ph$ is critical for $\ph \mapsto \int_M |T_1|^2 \vol$ if and only if
$$ \Div T^t - T (\Vop T) + (\tr T) \Vop T = 0 \quad \text{and} \quad (\tr T)^2 g - 2 (\tr T) T_{\symm} = 0. $$
Taking the trace of the second equation gives $7 (\tr T)^2 - 2 (\tr T)^2 = 0$, so $\tr T = 0$, which then automatically implies the second equation. The first equation then becomes $\Div T^t = T(\Vop T)$.

The Euler--Lagrange equations for the functionals $\ph \mapsto \int_M |T_k|^2 \vol$ for $k = 27, 7, 14$ are much more complicated. But Corollary~\ref{cor:evolution-torsion-functionals} shows that the second-order differential invariants of a $\G$-structure which arise in these Euler--Lagrange equations are:
$$ \tRc, \, R g, \, F, \, \cL_{\Vop T} g, \, \Div T, \, \Div T^t. $$
There are exactly four symmetric $2$-tensors, and two vector fields. These are precisely the independent second-order differential invariants of a $\G$-structure which are $3$-forms. (See Theorem~\ref{thm:invariants}.)

\subsection{Decomposition of $\tRm$ into independent components} \label{sec:curvature-decomp}

In this section we investigate the decomposition of the Riemann curvature tensor $\tRm$ into irreducible $\G$-representations. More precisely, we completely determine the structure of the self-adjoint curvature operator $\sR \colon \Omega^2 \to \Omega^2$ induced from the Riemann curvature tensor, in terms of the orthogonal splitting $\Omega^2 = \Omega^2_7 \oplus \Omega^2_{14}$. We also obtain geometric characterizations on the structure of this operator. This is the $\G$-analogue of the classical decomposition of the curvature operator on a $4$-dimensional oriented Riemannian manifold, as described, for example, in Besse~\cite[1.122--1.129]{Besse}.

Recall that the Riemann curvature tensor $\tRm$ is an element of $\K \subset \cS^2 (\Lambda^2)$, and thus can be regarded as a self-adjoint operator on the space $\Omega^2$ of $2$-forms. We use the notation $\sR$ to denote $\tRm$ when we want to think of its as a self-adoint operator on $\Omega^2$. We also have a decomposition $\Omega^2 = \Omega^2_7 \oplus \Omega^2_{14}$, with projection operators $\pi_7$ and $\pi_{14}$, which by~\eqref{eq:omega2-proj} satisfy
\begin{equation} \label{eq:curv-projections}
6 \pi_7 = 2 \sI - \Pop, \qquad 6 \pi_{14} = 4 \sI + \Pop.
\end{equation}
Given the projection operators $\pi_a$, for $a = 7, 14$, we can thus decompose $\sR$ as
$$ \sR = (\pi_7 + \pi_{14}) \sR (\pi_7 + \pi_{14}) = \sR^7_7 + \sR^7_{14} + \sR^{14}_7 + \sR^{14}_{14}, $$
where $\sR^a_b = \pi_b \sR \pi_a \colon \Omega^2_a \to \Omega^2_b$, for $a, b \in \{7, 14\}$. From~\eqref{eq:curv-projections}, we obtain
\begin{equation} \label{eq:curvature-components-0}
\begin{aligned}
36 \, \sR^7_7 & = 4 \sR - 2 \Pop \sR - 2 \sR \Pop + \Pop \sR \Pop, \\
36 \, \sR^7_{14} & = 8 \sR + 2 \Pop \sR - 4 \sR \Pop - \Pop \sR \Pop, \\
36 \, \sR^{14}_7 & = 8 \sR - 4 \Pop \sR + 2 \sR \Pop - \Pop \sR \Pop, \\
36 \, \sR^{14}_{14} & = 16 \sR + 4 \Pop \sR + 4 \sR \Pop + \Pop \sR \Pop.
\end{aligned}
\end{equation}
Observe that the operators $\sR^7_{7}$ and $\sR^{14}_{14}$ are self-dual, since they can be written as linear combinations of self-dual operators:
\begin{equation} \label{eq:curvature-components-1}
\begin{aligned}
36 \, \sR^7_7 & = 4 \sR - 2 (\Pop \sR + \sR \Pop) + \Pop \sR \Pop, \\
36 \, \sR^{14}_{14} & = 16 \sR + 4 (\Pop \sR + \sR \Pop) + \Pop \sR \Pop.
\end{aligned}
\end{equation}
By contrast, the adjoint of $\sR^7_{14}$ is $\sR^{14}_7$, so the operator $\sR^7_{14} + \sR^{14}_7$ is self adjoint. Explicitly, we have:
\begin{equation} \label{eq:curvature-components-2}
36 (\sR^7_{14} + \sR^{14}_7) = 16 \sR - 2 (\Pop \sR + \sR \Pop) - 2 \Pop \sR \Pop.
\end{equation}

In terms of the splitting $\Omega^2 = \Omega^2_7 \oplus \Omega^2_{14}$, the self-adjoint operator $\sR$ corresponds to the block matrix
$$ \sR = \begin{pmatrix} \sR^7_7 & \sR^{14}_7 \\[0.2em] \sR^7_{14} & \sR^{14}_{14} \end{pmatrix} = \begin{pmatrix} \sR^7_7 & 0 \\ 0 & \sR^{14}_{14} \end{pmatrix} + \begin{pmatrix} 0 & \sR^{14}_7 \\ \sR^7_{14} & 0 \end{pmatrix}, $$
which is a sum of two self-adjoint operators, one purely diagonal (and thus preserving the splitting), and one purely off-diagonal (and thus reversing the splitting).

Recall from~\eqref{eq:curvature-7} that we have
\begin{equation} \label{eq:sR-sum}
\sR = \tfrac{1}{84} R \, \iota_g g + \tfrac{1}{5} \iota_g (\tRc^0) + \sW,
\end{equation}
where $\sW$ denotes the Weyl curvature tensor $\tW$, thought of as a self-adjoint operator on $\Omega^2$. The expressions~\eqref{eq:curvature-components-1} and~\eqref{eq:curvature-components-2} for the three self-dual operators $\sR^7_7$, $\sR^{14}_{14}$, and $\sR^7_{14} + \sR^{14}_7$ are linear in $\sR$, so we can compute them for $\iota_g g$, $\iota_g \tRc^0$, and $\sW$ separately.

For simplicity, we temporarily rewrite~\eqref{eq:sR-sum} as
$$ \sR = \sA + \sB + \sW, \qquad \text{where $\sA = \tfrac{1}{84} R \, \iota_g g$ and $\sB = \tfrac{1}{5} \iota_g (\tRc^0)$}. $$
From $\iota_g g = - 4 \sI$ in~\eqref{eq:iota-g-identity}, and from $\Pop^2 = 8 \sI - 2 \Pop$ in~\eqref{eq:Pop-squared-v2}, we have
$$ \Pop (\iota_g g) + (\iota_g g) \Pop = - 8 \Pop \quad \text{and} \quad \Pop (\iota_g g) \Pop = - 4 \Pop^2 = - 32 \sI + 8 \Pop. $$
Thus, replacing $\sR$ by $\sA = \frac{1}{84} R \, \iota_g g$ in~\eqref{eq:curvature-components-1} and~\eqref{eq:curvature-components-2} and using the above, we obtain
\begin{equation*}
\begin{aligned}
36 \, \sA^7_7 & = \tfrac{1}{84} R ( - 48 \sI + 24 \Pop ), \\
36 \, \sA^{14}_{14} & = \tfrac{1}{84} R ( - 96 \sI - 24 \Pop ), \\
36 (\sA^7_{14} + \sA^{14}_7) & = 0,
\end{aligned}
\end{equation*}
which can be rewritten as
\begin{equation} \label{eq:curv-A-decomp}
\begin{aligned}
36 \, \sA^7_7 & = \tfrac{1}{14} R g \diamond \ps + \tfrac{1}{7} R \iota_g g, \\
36 \, \sA^{14}_{14} & = - \tfrac{1}{14} R g \diamond \ps + \tfrac{2}{7} R \iota_g g, \\
36 (\sA^7_{14} + \sA^{14}_7) & = 0,
\end{aligned}
\end{equation}
Equations~\eqref{eq:curv-A-decomp} show that the scalar curvature contribution $\sA = \frac{1}{84} R \, \iota_g g$ of the Riemann curvature operator $\sR$ is purely diagonal. Note that the two operators $\sA^7_7$ and $\sA^{14}_{14}$ are not curvature type operators in $\K$, as they both have $4$-form components which are multiples of $\Pop = \ps$. However, these two $4$-form components cancel each other in the sum as expected.

From~\eqref{eq:iota-g-Pop} with $h = \frac{1}{5} \tRc^0 \in \cS^2_0$, we have
$$ \Pop \sB + \sB \Pop = - \tfrac{2}{5} \tRc^0 \diamond \ps \quad \text{and} \quad \Pop \sB \Pop = - \tfrac{4}{5} \tRc^0 \diamond \ps - \tfrac{4}{5} \iota_g (\tRc^0) + \tfrac{4}{5} \iota_{\ph} (\tRc^0). $$
Thus, replacing $\sR$ by $\sB = \tfrac{1}{5} \iota_g (\tRc^0)$ in~\eqref{eq:curvature-components-1} and~\eqref{eq:curvature-components-2}, we obtain
\begin{equation*}
\begin{aligned}
36 \, \sB^7_7 & = \tfrac{4}{5} \iota_{\ph} (\tRc^0), \\
36 \, \sB^{14}_{14} & = - \tfrac{12}{5} \tRc^0 \diamond \ps + \tfrac{12}{5} \iota_g (\tRc^0) + \tfrac{4}{5} \iota_{\ph} (\tRc^0), \\
36 (\sB^7_{14} + \sB^{14}_7) & = \tfrac{12}{5} \tRc^0 \diamond \ps + \tfrac{24}{5} \iota_g (\tRc^0) - \tfrac{8}{5} \iota_{\ph} (\tRc^0).
\end{aligned}
\end{equation*}
Substituting $\iota_{\ph} (\tRc^0) = \frac{1}{3} \tRc^0 \diamond \ps + \frac{1}{5} \iota_g (\tRc^0) + (\iota_{\ph} \tRc^0)_{\W}$ from Proposition~\ref{prop:iotaph-decomp}, the above become
\begin{equation} \label{eq:curv-B-decomp}
\begin{aligned}
36 \, \sB^7_7 & = \tfrac{4}{15} \tRc^0 \diamond \ps + \tfrac{4}{25} \iota_g (\tRc^0) + \tfrac{4}{5} (\iota_{\ph} (\tRc^0))_{\W}, \\
36 \, \sB^{14}_{14} & = - \tfrac{32}{15} \tRc^0 \diamond \ps + \tfrac{64}{25} \iota_g (\tRc^0) + \tfrac{4}{5} (\iota_{\ph} (\tRc^0))_{\W}, \\
36 (\sB^7_{14} + \sB^{14}_7) & = \tfrac{28}{15} \tRc^0 \diamond \ps + \tfrac{112}{25} \iota_g (\tRc^0) - \tfrac{8}{5} (\iota_{\ph} (\tRc^0))_{\W}.
\end{aligned}
\end{equation}
Equations~\eqref{eq:curv-B-decomp} show that the traceless Ricci curvature contribution $\sB = \frac{1}{5} \iota_g (\tRc^0)$ of the Riemann curvature operator $\sR$ has both diagonal and off-diagonal components. Note that each of the three self-adjoint operators $\sB^7_7$, $\sB^{14}_{14}$, and $\sB^7_{14} + \sB^{14}_7$ are not curvature type operators in $\K$, as they all have $4$-form components which are multiples of $\tRc^0 \diamond \ph$. However, these three $4$-form components cancel each other in the sum as expected. Moreover, these three operators also have a component in the space $\W$ of Weyl tensors, which are multiples of the projection onto $\W$ of $\iota_{\ph} (\tRc^0)$. Again, the three $\W$ components cancel each other in the sum as expected, since $\sB = \tfrac{1}{5} \iota_g (\tRc^0)$ is a curvature tensor orthogonal to $\W$.

Now consider the Weyl curvature operator $\sW$. Then equations~\eqref{eq:curvature-components-1} and~\eqref{eq:curvature-components-2} give
\begin{equation} \label{eq:curv-W-decomp}
\begin{aligned}
36 \, \sW^7_7 & = 4 \sW - 2 (\Pop \sW + \sW \Pop) + \Pop \sW \Pop, \\
36 \, \sW^{14}_{14} & = 16 \sW + 4 (\Pop \sW + \sW \Pop) + \Pop \sW \Pop, \\
36 (\sW^7_{14} + \sW^{14}_7) & = 16 \sW - 2 (\Pop \sW + \sW \Pop) - 2 \Pop \sW \Pop.
\end{aligned}
\end{equation}
From Remark~\ref{rmk:Weyl-decomp} we have $\sW = \sW_{27} + \sW_{64} + \sW_{77}$. From the decompositions
$$ \Sym^2 (\mb{7}) = \mb{1} \oplus \mb{27}, \qquad \mb{7} \otimes \mb{14} = \mb{64} \oplus \mb{7} \oplus \mb{27}, \qquad \Sym^2 (\mb{14}) = \mb{77} \oplus \mb{1} \oplus \mb{27}, $$
we see that
\begin{equation} \label{eq:W64-77}
\text{$\sW_{64}$ only contributes to $\sW^7_{14} + \sW^{14}_7$, \, \, and $\sW_{77}$ only contributes to $\sW^{14}_{14}$.}
\end{equation}
However, the $\sW_{27}$ component could in principle contribute to all three expressions in~\eqref{eq:curv-W-decomp}.

In the notation of~\eqref{eq:iota-rho-W27}, let
\begin{equation} \label{eq:varpi}
\varpi = \ol{\rho}_{\ph} (\sW) = \rho_{\ph} (\sW),
\end{equation}
so that by~\eqref{eq:W-decomp} we have
\begin{equation} \label{eq:W27-varpi}
\sW_{27} = \tfrac{15}{448} \ol{\iota}_{\ph} \varpi = \tfrac{15}{448} (\iota_{\ph} \varpi)_{\W}.
\end{equation}
Thus to obtain the ${}^7_7$, ${}^{14}_{14}$, and ${}^7_{14} + {}^{14}_7$ contributions from $\sW_{27}$, we replace $\sW$ on the right-hand sides of the equations in~\eqref{eq:curv-W-decomp} with $\frac{15}{448} (\iota_{\ph} \varpi)_{\W}$. Using Corollary~\ref{cor:iota-phW-Pop}, some arithmetic gives
\begin{equation*}
\begin{aligned}
36 (\sW_{27})^7_7 & = \iota_{\ph} \varpi, \\
36 (\sW_{27})^{14}_{14} & = - \tfrac{3}{16} \varpi \diamond \ps + \tfrac{3}{16} \iota_g \varpi + \tfrac{1}{16} \iota_{\ph} \varpi, \\
36 ((\sW_{27})^7_{14} + (\sW_{27})^{14}_7) & = - \tfrac{3}{14} \varpi \diamond \ps - \tfrac{3}{7} \iota_g \varpi + \tfrac{1}{7} \iota_{\ph} \varpi.
\end{aligned}
\end{equation*}
Substituting $\iota_{\ph} \varpi = (\iota_{\ph} \varpi)_{\W} + \frac{1}{3} \varpi \diamond \ps + \frac{1}{5} \iota_g \varpi$ from~\eqref{eq:iota-phW-Pop}, some more arithmetic yields
\begin{equation} \label{eq:curv-W27-decomp}
\begin{aligned}
36 (\sW_{27})^7_7 & = \tfrac{1}{3} \varpi \diamond \ps + \tfrac{1}{5} \iota_g \varpi + (\iota_{\ph} \varpi)_{\W}, \\
36 (\sW_{27})^{14}_{14} & = - \tfrac{1}{6} \varpi \diamond \ps + \tfrac{1}{5} \iota_g \varpi + \tfrac{1}{16} (\iota_{\ph} \varpi)_{\W}, \\
36 ((\sW_{27})^7_{14} + (\sW_{27})^{14}_7) & = - \tfrac{1}{6} \varpi \diamond \ps - \tfrac{2}{5} \iota_g \varpi + \tfrac{1}{7} (\iota_{\ph} \varpi)_{\W}.
\end{aligned}
\end{equation}
Equations~\eqref{eq:curv-W27-decomp} show that the $\sW_{27}$ contribution of the Riemann curvature operator $\sR$ has both diagonal and off-diagonal components. Note that each of the three self-adjoint operators $(\sW_{27})^7_7$, $(\sW_{27})^{14}_{14}$, and $(\sW_{27})^7_{14} + (\sW_{27})^{14}_7$ are not curvature type operators in $\W$, as they all have $4$-form components which are multiples of $\varpi \diamond \ph$ and $\iota_g (\cS^2_0)$ components which are multiples of $\iota_g (\varpi)$. However, these all cancel out in the sum as expected, leaving an element of $\W$.

Combining~\eqref{eq:curv-A-decomp},~\eqref{eq:curv-B-decomp},~\eqref{eq:W64-77}, and~\eqref{eq:curv-W27-decomp}, we have finally shown that
\begin{align} \nonumber
36 \, \sR^7_7 & = (\tfrac{1}{14} R g + \tfrac{4}{15} \tRc^0 + \tfrac{1}{3} \varpi) \diamond \ps + \iota_g (\tfrac{1}{7} R g + \tfrac{4}{25} \tRc^0 + \tfrac{1}{5} \varpi) + ( \iota_{\ph} (\tfrac{4}{5} \tRc^0 + \varpi) )_{\W}, \\ \nonumber
36 \, \sR^{14}_{14} & = ( - \tfrac{1}{14} R g - \tfrac{32}{15} \tRc^0 - \tfrac{1}{6} \varpi) \diamond \ps + \iota_g (\tfrac{2}{7} R g + \tfrac{64}{25} \tRc^0 + \tfrac{1}{5} \varpi) + ( \iota_{\ph} (\tfrac{4}{5} \tRc^0 + \tfrac{1}{16} \varpi) )_{\W} + 36 \, \sW_{77}, \\ \label{eq:curv-R-decomp-final}
36 (\sR^7_{14} + \sR^{14}_7) & = ( \tfrac{28}{15} \tRc^0 - \tfrac{1}{6} \varpi) \diamond \ps + \iota_g ( \tfrac{112}{25} \tRc^0 - \tfrac{2}{5} \varpi) + ( \iota_{\ph} ( - \tfrac{8}{5} \tRc^0 + \tfrac{1}{7} \varpi) )_{\W} + 36 \, \sW_{64}.
\end{align}
We can symbolically represent the three equations in~\eqref{eq:curv-R-decomp-final} by
\begin{equation} \label{eq:curv-R-decomp-symbolic}
\sR = \begin{pmatrix} \{ Rg, \, \tRc^0, \, \varpi \} & \{ \tRc^0, \, \varpi, \, \sW_{64} \} \\[0.2em] \{ \tRc^0, \, \varpi, \, \sW_{64} \} & \{ Rg, \, \tRc^0, \, \varpi, \, \sW_{77} \} \end{pmatrix}
\end{equation}
which shows exactly which components of the curvature contribute to each block with respect to the splitting $\Omega^2 = \Omega^2_7 \oplus \Omega^2_{14}$.

Equation~\eqref{eq:curv-R-decomp-symbolic} should be compared to the classical case of the decomposition of the Riemann curvature operator $\sR$ on an oriented Riemannian $4$-manifold, with respect to the splitting $\Omega^2 = \Omega^2_+ \oplus \Omega^2_-$, which symbolically is
$$ \sR = \begin{pmatrix} \{ Rg, \, \sW_+ \} & \{ \tRc^0 \} \\[0.2em] \{ \tRc^0 \} & \{ Rg, \, \sW_- \} \end{pmatrix}. $$
(See Besse~\cite[1.122--1.129]{Besse}.) A consequence of this decomposition is that $\sR$ preserves the splitting (acts diagonally) if and only if $g$ is Einstein, and that $\sR$ reverses the splitting (acts anti-diagonally) if and only if $g$ is scalar-flat and Weyl-flat. We can obtain a similar result for $\G$-structures as follows.

First, a closer inspection of~\eqref{eq:curv-R-decomp-final} reveals that these equations can be rewritten as
\begin{align} \nonumber
36 \, \sR^7_7 & = (\tfrac{1}{14} R g + \tfrac{1}{15} \sigma_{7,7}) \diamond \ps + \iota_g (\tfrac{1}{7} R g + \tfrac{1}{25} \sigma_{7,7}) + ( \iota_{\ph} (\tfrac{1}{5} \sigma_{7,7}) )_{\W}, \\ \nonumber
36 \, \sR^{14}_{14} & = ( - \tfrac{1}{14} R g - \tfrac{1}{30} \sigma_{14,14}) \diamond \ps + \iota_g (\tfrac{2}{7} R g + \tfrac{1}{25} \sigma_{14,14}) + ( \iota_{\ph} (\tfrac{1}{80} \sigma_{14,14} )_{\W} + 36 \, \sW_{77}, \\ \label{eq:curv-R-decomp-better}
36 (\sR^7_{14} + \sR^{14}_7) & = ( \tfrac{1}{30} \sigma_{7,14}) \diamond \ps + \iota_g ( \tfrac{2}{25} \sigma_{7,14}) + ( \iota_{\ph} ( - \tfrac{1}{35} \sigma_{7,14} )_{\W} + 36 \, \sW_{64},
\end{align}
where we have defined the three tensors $\sigma_{7,7}$, $\sigma_{14,14}$, and $\sigma_{7,14}$ in $\cS^2_0$ by
\begin{equation} \label{eq:sigmas}
\sigma_{7,7} = 4 \tRc^0 + 5 \varpi, \qquad \sigma_{14,14} = 64 \tRc^0 + 5 \varpi, \qquad \sigma_{7,14} = 56 \tRc^0 - 5 \varpi.
\end{equation}

\begin{thm} \label{thm:main-curv-decomp}
Let $\sR$ be the Riemann curvature of the metric $g$ induced from a $\G$-structure $\ph$, thought of as a self-adjoint operator on $\Omega^2 = \Omega^2_7 \oplus \Omega^2_{14}$. Let $\sR^a_b = \pi_b \sR \pi_a \colon \Omega^2_a \to \Omega^2_b$, for $a, b \in \{7, 14\}$. Then we have
\begin{enumerate}[(i)]
\item $\sR^7_7 = 0$ \, $\iff$ \, $( \, R = 0$ \, and \, $\sigma_{7,7} = 4 \tRc^0 + 5 \varpi = 0 \,)$,
\item $\sR^{14}_{14} = 0$ \, $\iff$ \, $( \, R = 0$ \, and \, $\sigma_{14,14} = 64 \tRc^0 + 5 \varpi = 0$ \, and \, $\sW_{77} = 0 \, )$,
\item $\sR^7_{14} + \sR^{14}_7 = 0$ \, $\iff$ \, $( \, \sigma_{7,14} = 56 \tRc^0 - 5 \varpi = 0$ \, and \, $\sW_{64} = 0 \, )$.
\end{enumerate}
Consequently:
\begin{itemize} \setlength\itemsep{-1mm}
\item $\sR$ preserves the splitting (acts diagonally) if and only if $\sW_{64} = 0$ and $\varpi = \frac{56}{5} \tRc^0$,
\item $\sR$ reverses the splitting (acts anti-diagonally) if and only if $\tRc = 0$, $\varpi = 0$, and $\sW_{77} = 0$.
\end{itemize}
\end{thm}
\begin{proof}
Statements \emph{(i)}, \emph{(ii)}, and \emph{(iii)} are immediate from~\eqref{eq:curv-R-decomp-better} and~\eqref{eq:sigmas}, and the fact that the space $\cS^2 (\Lambda^2)$ splits as a direct sum $\Omega^4 \oplus \iota_g (\cS^2) \oplus (\W_{27} \oplus \W_{64} \oplus \W_{77})$. The consequences then follow.
\end{proof}

It is convenient to rewrite these formulas by expressing $\varpi$ in terms of $Rg$, $\tRc$, and the symmetric tensor $F$ of~\eqref{eq:F-defn}. From~\eqref{eq:F-rho-ph},~\eqref{eq:sR-sum},~\eqref{eq:varpi}, and Proposition~\ref{prop:ph-rep-maps} we obtain
\begin{align} \nonumber
F & = \rho_{\ph} \big( \tfrac{1}{84} R \, \iota_g g + \tfrac{1}{5} \iota_g (\tRc^0) + \sW \big) \\ \nonumber
& = \tfrac{1}{84} R \rho_{\ph} (\iota_g g) + \tfrac{1}{5} \rho_{\ph} (\iota_g (\tRc^0)) + \varpi \\ \nonumber
& = \tfrac{1}{84} R (- 24 g) + \tfrac{1}{5} (4 \tRc^0) + \varpi \\ \label{eq:F-varpi0}
& = - \tfrac{2}{7} R g + \tfrac{4}{5} \tRc^0 + \varpi, 
\end{align}
which, upon writing $\tRc^0 = \tRc - \frac{1}{7} Rg$, becomes
\begin{equation} \label{eq:F-varpi}
F = - \tfrac{2}{5} R g + \tfrac{4}{5} \tRc + \varpi.
\end{equation}
Using~\eqref{eq:F-varpi0}, the traceless symmetric tensors $\sigma_{a,b}$ of~\eqref{eq:sigmas} become
\begin{equation} \label{eq:sigmas-alt}
\sigma_{7,7} = 5 F^0, \qquad \sigma_{14,14} = 60 \tRc^0 + 5 F^0, \qquad \sigma_{7,14} = 60 \tRc^0 - 5 F^0.
\end{equation}
\begin{cor} \label{cor:main-curv-decomp}
Consider the hypotheses of Theorem~\ref{thm:main-curv-decomp}. Then we have
\begin{enumerate}[(i)]
\item $\sR^7_7 = 0$ \, $\iff$ \, $( \, F = 0 \,)$,
\item $\sR^{14}_{14} = 0$ \, $\iff$ \, $( \, R = 0$ \, and \, $F^0 = - 12 \tRc^0$ \, and \, $\sW_{77} = 0 \, )$,
\item $\sR^7_{14} + \sR^{14}_7 = 0$ \, $\iff$ \, $( \, F^0 = 12 \tRc^0$ \, and \, $\sW_{64} = 0 \, )$.
\end{enumerate}
Consequently:
\begin{itemize} \setlength\itemsep{-1mm}
\item $\sR$ preserves the splitting (acts diagonally) if and only if $\sW_{64} = 0$ and $F^0 = 12 \tRc^0$,
\item $\sR$ reverses the splitting (acts anti-diagonally) if and only if $\tRc = 0$, $F = 0$, and $\sW_{77} = 0$.
\end{itemize}
\end{cor}
\begin{proof}
These all follow from Theorem~\ref{thm:main-curv-decomp}, equations~\eqref{eq:sigmas-alt}, and $\tr F = - 2 R$ from Lemma~\ref{lemma:trace-F}.
\end{proof}

\begin{rmk} \label{rmk:CI-Einstein}
In Cleyton--Ivanov~\cite[Equation (4.23)]{CI} the authors define a manifold $(M, \ph)$ with $\G$-structure to be \emph{generalized Einstein} if $\lambda_1 \tRc^0 + \lambda_2 F^0 = 0$ for $(\lambda_1, \lambda_2) \in \R^2 \setminus \{ (0,0) \}$. Our Corollary~\ref{cor:main-curv-decomp} gives geometric meaning to generalized Einstein structures with $(\lambda_1, \lambda_2) = (0, 1)$, $(12, 1)$, $(-12, 1)$.
\end{rmk}

The discussion in this section has shown that there are three independent second-order differential invariants of a $\G$-structure coming from the Riemann curvature tensor $\tRm$ \emph{which are $3$-forms}, namely $Rg$, $\tRc^0$, $\varpi$, and these lie in $\mb{1}$, $\mb{27}$, $\mb{27}$, respectively. From equation~\eqref{eq:F-varpi} we see that a more convenient basis for this space is $\{ Rg, \tRc, F \}$, as their definitions are computationally simpler. In Section~\ref{sec:nabla-T-decomp} we find three more independent second-order differential invariants \emph{which are $3$-forms} coming from $\nab{} T$.

\begin{rmk} \label{rmk:Hodge-Lap-revisited}
Using~\eqref{eq:F-varpi0} or~\eqref{eq:F-varpi}, we can rewrite~\eqref{eq:Hodge-Lap-ph} for the Hodge Laplacian $\Delta_d \ph$ as
\begin{align*}
\Delta_{\dd} \ph & = \big( \tfrac{1}{3} (\Div T) \hk \ph + \tfrac{1}{3} |T|^2 g - T^t T + \tfrac{2}{21} R g + \tfrac{1}{5} \tRc^0 + \tfrac{1}{4} \varpi \big) \diamond \ph \\
& = \big( \tfrac{1}{3} (\Div T) \hk \ph + \tfrac{1}{3} |T|^2 g - T^t T + \tfrac{1}{15} R g + \tfrac{1}{5} \tRc + \tfrac{1}{4} \varpi \big) \diamond \ph.
\end{align*}
This shows explicitly that $\Delta_d \ph$ does indeed depend on the Ricci curvature, since $F$ does, even though it was not clear from our original expression~\eqref{eq:Hodge-Lap-ph}.
\end{rmk}

We close this section by explaining how our work allows one to compute an explicit formula for $\sW_{64}$ purely in terms of the torsion. First, from~\eqref{eq:curvature-components-0} we have
$$ 36 \, \sR^7_{14} = 8 \sR + 2 \Pop \sR - 4 \sR \Pop - \Pop \sR \Pop, $$
which says that
$$ 36 (\sR^7_{14})_{ijkl} = 8 R_{ijkl} + 2 \ps_{ijab} R_{abkl} - 4 R_{ijpq} \ps_{pqkl} - \ps_{ijab} R_{abpq} \ps_{pqkl}. $$
Contracting both sides with $\ph_{klm}$ gives
\begin{align*}
36 (\sR^7_{14})_{ijkl} \ph_{klm} & = 8 R_{ijkl} \ph_{klm} + 2 \ps_{ijab} R_{abkl} \ph_{klm} - 4 R_{ijpq} \ps_{pqkl} \ph_{klm} - \ps_{ijab} R_{abpq} \ps_{pqkl} \ph_{klm} \\
& = 8 R_{ijkl} \ph_{klm} + 2 \ps_{ijab} R_{abkl} \ph_{klm} + 16 R_{ijpq} \ph_{pqm} + 4 \ps_{ijab} R_{abpq} \ph_{pqm} \\
& = 24 R_{ijkl} \ph_{klm} + 6 \ps_{ijab} R_{abkl} \ph_{klm}.
\end{align*}
By the $\G$-Bianchi identity~\eqref{eq:g2bianchi}, $R_{ijkl} \ph_{klm}$ can be expressed purely in terms of the torsion. Since $(\sR^7_{14})_{ijkl}$ is of type $\Omega^2_7$ in the $k,l$ indices, from~\eqref{eq:omega27vf} we have
$$ 36 (\sR^7_{14})_{ijkl} \ph_{klm} \ph_{mpq} = 6 \cdot 36 (\sR^7_{14})_{ijpq}. $$
Thus $(\sR^7_{14})_{ijpq}$ can be expressed purely in terms of torsion, and hence so can $(\sR^{14}_7)_{ijpq} = (\sR^7_{14})_{pqij}$. We also know from Remark~\ref{rmk:curvature-torsion-free} that both $\tRc$ and $F$ can be expressed purely in terms of torsion, and thus by~\eqref{eq:F-varpi} and the third equation in~\eqref{eq:curv-R-decomp-final} we conclude that $\sW_{64}$ can be expressed purely in terms of torsion. In particular, if $T = 0$, then $\tRc = 0$, $\varpi = 0$ (so $\sW_{27} = 0$), and $\sW_{64} = 0$. Thus only $\sW_{77}$ can be nonzero for a torsion-free $\G$-structure. (This is a classical result of Alekseevski\u{i}~\cite{Al}.) Note also that the second equation in~\eqref{eq:curv-R-decomp-final} can be used to obtain a general formula for $\sW_{77}$ in terms of Riemann curvature and torsion.

\subsection{Determination of the components of $\nab{} T$ that are $3$-forms} \label{sec:nabla-T-decomp}

In this section we consider the decomposition of $\nab{} T$ into irreducible $\G$-representations and identify all those components which correspond to $3$-forms. (That is, those which lie in $\mb{1} \oplus \mb{27} \oplus \mb{7}$.) We do not give explicit formulas for all the independent components of $\nab{} T$, as we do not require them, but this can be done using the results of Section~\ref{sec:more-rep-theory}. Moreover, as we are only concerned with the leading order behaviour of the various components of $\nab{} T$ which can be $3$-forms, for the purposes of analyzing the short time existence behaviour of flows of $\G$-structures constructed from such flows, we need only determine the leading (second) order terms of the components of $\nab{}T$ which lie in $\mb{1}$, $\mb{27}$, or $\mb{7}$, and those only up to constants.

We begin by recalling that the torsion decomposes as $T = T_1 + T_{27} + T_7 + T_{14}$, so that at every point it lies in the representation $\mb{1} \oplus \mb{27} \oplus \mb{7} \oplus \mb{14}$. Thus, since $\nab{} T \in \Gamma(T^* M \otimes T^* M \otimes T^* M)$, using equations~\eqref{eq:7-7},~\eqref{eq:7-14}, and~\eqref{eq:7-27} we deduce that at every point we have
\begin{equation} \label{eq:nabT-prelim}
\begin{aligned}
\nab{} T & = \nab{} T_1 + \nab{} T_{27} + \nab{} T_7 + \nab{} T_{14} \\
& \in (\mb{7} \otimes \mb{1}) \oplus (\mb{7} \otimes \mb{27}) \oplus (\mb{7} \otimes \mb{7}) \oplus (\mb{7} \otimes \mb{14}) \\
& = \mb{7} \oplus ( \mb{77^*} \oplus \mb{7} \oplus \mb{64} \oplus \mb{27} \oplus \mb{14} ) \oplus (\mb{1} \oplus \mb{27} \oplus \mb{7} \oplus \mb{14}) \oplus (\mb{64} \oplus \mb{27} \oplus \mb{7}).
\end{aligned}
\end{equation}
In particular, we infer from~\eqref{eq:nabT-prelim} that $\nab{} T$ contains the following components in $\mb{1} \oplus \mb{27} \oplus \mb{7}$:
\begin{itemize} \setlength\itemsep{-1mm}
\item One component in $\mb{1}$, coming from $\nab{} T_7$.
\item Three components in $\mb{27}$, coming from $\nab{} T_{27}$, $\nab{} T_7$, and $\nab{} T_{14}$.
\item Four components in $\mb{7}$, coming from $\nab{} T_1$, $\nab{} T_{27}$, $\nab{} T_7$, and $\nab{} T_{14}$.
\end{itemize}
We proceed to identify the above eight components, only up to lower order terms and an overall constant, by analyzing each $\nab{} T_k$ for $k \in \{ 1, 27, 7, 14 \}$. We write $\lot$ to denote ``\emph{$\ell$\!ower order terms}''.

\begin{prop} \label{prop:nabT1}
To leading order and up to an overall constant, the $\mb{7}$ component of $\nab{} T$ determined by $\nab{} T_1$ is $\Div T^t$.
\end{prop}
\begin{proof}
Since $T_1 = \frac{1}{7} (\tr T) g$, we get $\nab{} T_1 = \frac{1}{7} [\nab{} (\tr T)] \otimes g$, which corresponds (up to a constant) to $\nab{} (\tr T)$ under the isomorphism $f \leftrightarrow f g$ for $f \in \Omega^0$. The result now follows from Corollary~\ref{cor:g2bianchi-repackaged}.
\end{proof}

\begin{prop} \label{prop:nabT7}
To leading order and up to overall constants, a basis for the independent $\mb{1} \oplus \mb{27} \oplus \mb{7}$ components of $\nab{} T$ determined by $\nab{} T_7$ is given by $R$, $\cL_{\Vop T} g$, and $\Div T^t - \Div T$.
\end{prop}
\begin{proof}
Recall from~\eqref{eq:VT} that $T_7$ is equivalent to $(\Vop T)_q = T_{ij} \ph_{ijq}$, so $\nab{} T_7$ is equivalent to $\nab{} (\Vop T)$. Thus the $\mb{1}$ component is given by $\nab{k} (\Vop T)_k = \Div (\Vop T)$, which by Corollary~\ref{cor:g2bianchi-repackaged} is equivalent to the scalar curvature $R$ up to lower order terms. Up to a constant, the $\mb{1} \oplus \mb{27}$ component is $\nab{p} (\Vop T)_q + \nab{q} (\Vop T)_p = (\cL_{\Vop T} g)_{pq}$. Finally, the $\mb{7}$ component corresponds to $(\nab{p} (\Vop T)_q \ph_{pqk}) = (\curl (\Vop T))_k$, which by~\eqref{eq:curlVT} and Corollary~\ref{cor:g2bianchi-repackaged} corresponds to $\Div T^t - \Div T$ up to lower order terms.
\end{proof}

\begin{prop} \label{prop:nabT14}
To leading order and up to overall constants, a basis for the independent $\mb{27} \oplus \mb{7}$ components of $\nab{} T$ determined by $\nab{} T_{14}$ is given by $-4 \cL_{\Vop T} g + 2 R g - 12 \tRc + 3 F$ and $\Div T - \Div T^t$.
\end{prop}
\begin{proof}
Write $T_{\skew} = T_7 + T_{14}$. From~\eqref{eq:omega2-proj}, we have
\begin{align*}
6 \nab{i} (T_{14})_{jk} & = \nab{i} (4 (T_{\skew})_{jk} + 2 \ps_{jkpq} (T_{\skew})_{pq}) \\
& = 2 \nab{i} T_{jk} - 2 \nab{i} T_{kj} + \nab{i} (\ps_{jkpq} T_{pq}) \\
& = 2 \nab{i} T_{jk} - 2 \nab{i} T_{kj} + \nab{i} T_{pq} \ps_{jkpq} + \lot.
\end{align*}
Let $\beta_{ijk} = 6 \nab{i} (T_{14})_{jk} \in \mb{7} \otimes \mb{14}$. From Section~\ref{sec:7-14}, we have $\mb{7} \otimes \mb{14} = \mb{64} \oplus \mb{27} \oplus \mb{7}$, where the $\mb{27} \oplus \mb{7}$ part is contained in the $3$-form $\gamma_{ijk} = \beta_{ijk} + \beta_{jki} + \beta_{kij}$, and explicitly in the symmetric and skew-symmetric parts of the $2$-tensor $\gamma^{\ph}_{ia} = \gamma_{ijk} \ph_{ajk}$. Using the conditions~\eqref{eq:7-14-conditions}, we have
\begin{align*}
\gamma^{\ph}_{ia} & = \gamma_{ijk} \ph_{ajk} = (\beta_{ijk} + \beta_{jki} + \beta_{kij}) \ph_{ajk} \\
& = 0 + (\beta_{jki} - \beta_{kji}) \ph_{ajk} \\
& = 2 \beta_{jki} \ph_{jka}.
\end{align*}
Thus we have
\begin{align*}
\gamma^{\ph}_{ia} = 2 \beta_{jki} \ph_{jka} & = 2 (2 \nab{j} T_{ki} - 2 \nab{j} T_{ik} + \nab{j} T_{pq} \ps_{kipq}) \ph_{jka} + \lot \\
& = 2 \nab{j} T_{pq} (\ph_{jak} \ps_{ipqk}) + 4 \nab{j} T_{ki} \ph_{jka} + 2 \nab{j} T_{ik} \ph_{jak} + \lot \\
& = 2 \nab{j} T_{pq} (g_{ji} \ph_{apq} + g_{jp} \ph_{iaq} + g_{jq} \ph_{ipa} - g_{ai} \ph_{jpq} - g_{ap} \ph_{ijq} - g_{aq} \ph_{ipj} ) \\
& \qquad{} + 4 \, \KK{3}_{ia} + 4 \, \KK{2}_{ia} + \lot.
\end{align*}
The above simplifies further to
\begin{align} \nonumber
\tfrac{1}{2} \gamma^{\ph}_{ia} & = \nab{i} T_{pq} \ph_{apq} + (\Div T)_q \ph_{qia} - (\Div T^t)_p \ph_{pia} - \langle \nab{} T, \ph \rangle g_{ia} + \nab{j} T_{aq} \ph_{jiq} + \nab{j} T_{pa} \ph_{jpi} \\ \nonumber
& \qquad{} + 2 \, \KK{3}_{ia} + 2 \, \KK{2}_{ia} + \lot \\ \label{eq:nabT14temp}
& = \KK{1}_{ia} + ((\Div T - \Div T^t) \hk \ph)_{ia} - \langle \nab{} T, \ph \rangle g_{ia} + \KK{2}_{ai} + \KK{3}_{ai} + 2 \, \KK{3}_{ia} + 2 \, \KK{2}_{ia} + \lot.
\end{align}
Up to an overall constant and lower order terms, the symmetric part of $\gamma^{\ph}_{ia}$ is thus
$$ (\KK{1}_{\symm})_{ia} - \langle \nab{} T, \ph \rangle g_{ia} + 3 (\KK{2}_{\symm})_{ia} + 3 (\KK{3}_{\symm})_{ia}. $$
(One can check directly using the results of Section~\ref{sec:nabT} that, to leading order, the above symmetric $2$-tensor has no $\mb{1}$ component, as expected.) Using Remark~\ref{rmk:KK-curv}, Corollary~\ref{cor:scalar-curvature}, and equations~\eqref{eq:KK2symm} and~\eqref{eq:KK3symm}, up to lower order terms this is
$$ \tfrac{1}{2} \cL_{\Vop T} g + \tfrac{1}{2} R g + 3 (- \tfrac{1}{2} \cL_{\Vop T} g - \tRc) + 3 ( \tfrac{1}{4} F ) = - \cL_{\Vop T} g + \tfrac{1}{2} R g - 3 \tRc + \tfrac{3}{4} F, $$
yielding the claimed result.

To extract the $\mb{7}$ part of $\nab{} T_{14}$, which corresponds to the skew-symmetric part of $\gamma^{\ph}_{ia}$, we applying the $\Vop$ operator to~\eqref{eq:nabT14temp}, obtaining
$$ \tfrac{1}{2} (\Vop \gamma^{\ph}) = \Vop (\KK{1}) + \Vop ( (\Div T - \Div T^t) \hk \ph) + \Vop(\KK{3}) + \Vop(\KK{2}) + \lot. $$
Using Lemma~\ref{lemma:Vop-KK}, Corollary~\ref{cor:g2bianchi-repackaged}, and~\eqref{eq:vec-transform}, up to lower order terms this is $6 (\Div T - \Div T^t)$, yielding the claimed result.
\end{proof}

\begin{prop} \label{prop:nabT27}
To leading order and up to overall constants, a basis for the independent $\mb{27} \oplus \mb{7}$ components of $\nab{} T$ determined by $\nab{} T_{27}$ is given by $4 \cL_{\Vop T} g + 4 \tR + F$ and $\Div T + \frac{5}{7} \Div T^t$.
\end{prop}
\begin{proof}
Write $T_{\symm} = T_1 + T_{27}$. Then we have
\begin{align*}
\nab{i} (T_{27})_{jk} & = \nab{i} ( (T_{\symm})_{jk} - \tfrac{1}{7} (\tr T) g_{jk} ) \\
& = \tfrac{1}{2} \nab{i} T_{jk} + \tfrac{1}{2} \nab{i} T_{kj} - \tfrac{1}{7} \nab{i} (\tr T) g_{jk}.
\end{align*}
Let $h_{ijk} = \nab{i} (T_{27})_{jk} \in \mb{7} \otimes \mb{27}$. From Section~\ref{sec:7-27}, we have $\mb{7} \otimes \mb{27} = (\mb{77^*} \oplus \mb{7}) \oplus (\mb{64} \oplus \mb{27} \oplus \mb{14})$, where the $\mb{7}$ part is obtained by taking the trace of the symmetrization of $h_{ijk}$, and the $\mb{27}$ part is the symmetric part of the $2$-tensor $h_{ijk} \ph_{iak}$. (In fact, as described in Section~\ref{sec:7-27} we actually use $(h_{105})_{ijk} \ph_{iak}$, but from~\eqref{eq:7-27-step2-decomp}, the difference between $h$ and $h_{105}$ is fully symmetric, and thus vanishes when contracted with $\ph$ on two indices.)

Therefore, to obtain the $\mb{7}$ part, we compute the trace of $h_{ijk} + h_{jki} + h_{kij}$, which is $h_{ikk} + h_{kki} + h_{kik} = h_{ikk} + 2 h_{kki}$. This is
\begin{align*}
h_{ikk} + 2 h_{kki} & = \tfrac{1}{2} \nab{i} T_{kk} + \tfrac{1}{2} \nab{i} T_{kk} - \tfrac{1}{7} \nab{i} (\tr T) g_{kk} \\
& \qquad {} + 2 (\tfrac{1}{2} \nab{k} T_{ki} + \tfrac{1}{2} \nab{k} T_{ik} - \tfrac{1}{7} \nab{k} (\tr T) g_{ki} ) \\
& = 0 + (\Div T)_i + (\Div T^t)_i - \tfrac{2}{7} \nab{i} (\tr T).
\end{align*}
Using Corollary~\ref{cor:g2bianchi-repackaged}, up to lower order terms this is $\Div T + \frac{5}{7} \Div T^t$, yielding the claimed result.

To obtain the $\mb{27}$ part, we need the symmetric part of $h_{ijk} \ph_{iak}$. We compute
\begin{align*}
h_{ijk} \ph_{iak} & = (\tfrac{1}{2} \nab{i} T_{jk} + \tfrac{1}{2} \nab{i} T_{kj} - \tfrac{1}{7} \nab{i} (\tr T) g_{jk}) \ph_{iak} \\
& = \tfrac{1}{2} \KK{2}_{ja} - \tfrac{1}{2} \KK{3}_{ja} + \tfrac{1}{7} \nab{i} (\tr T) \ph_{ija}.
\end{align*}
Up to an overall constant, the symmetric part of the above is $(\KK{2})_{\symm} - (\KK{3})_{\symm}$. (By Definition~\ref{defn:K}, this has no $\mb{1}$ component, as expected.) Using~\eqref{eq:KK2symm} and~\eqref{eq:KK3symm}, up to lower order terms this is $- \frac{1}{2} \cL_{\Vop T} g - \tRc - \frac{1}{4} F$, yielding the claimed result.
\end{proof}

From the above four propositions, we immediately conclude the following result.
\begin{thm} \label{thm:nabT-possibilities}
The only independent second-order differential invariants of a $\G$-structure coming from $\nab{} T$ which are $3$-forms are the vector fields $\Div T, \, \Div T^t$ and the symmetric $2$-tensors $\cL_{\Vop T} g, \, R g, \, \tRc, \, F$. Observe that this list includes as a proper subset the independent second-order differential invariants of a $\G$-structure coming from $\tRm$ which are $3$-forms, namely the symmetric $2$-tensors $Rg, \, \tRc, \, F$.
\end{thm}

\section{Symbols and short-time existence of flows of $\G$-structures} \label{sec:symbols-short-time}

In this section we establish short-time existence and uniqueness for a large class of flows of $\G$-structures, using DeTurck's trick and the explicit computation of the symbols of the various independent second-order linear differential operators in $\G$-geometry.

As discussed in~\cite{K-flows}, a general flow of $\G$-structures can be written in the form
\begin{equation} \label{eq:general-flow-2}
\delt \ph = h \diamond \ph + X \hk \ps
\end{equation}
for some time-dependent symmetric $2$-tensor $h$ and vector field $X$. (Note that the $\diamond$ operation defined in~\eqref{eq:diamond-coords} depends on the metric and hence on the $\G$-structure $\ph$.) In order to obtain a parabolic flow, we need $h$ and $X$ to be \emph{second-order differential invariants} of $\ph$, which are linear in the second derivatives. We classified the independent second-order differential invariants of $\ph$ in Section~\ref{sec:curvature-torsion}. Indeed, the following result follows directly from the discussions in Sections~\ref{sec:g2bianchi-revisited},~\ref{sec:curvature-decomp}, and~\ref{sec:nabla-T-decomp}.

\begin{thm} \label{thm:invariants}
There are six \emph{independent} second-order differential invariants of a $\G$-structure that can be used to define a flow~\eqref{eq:general-flow-2} of $\G$-structures.

There are four independent possibilities for the symmetric $2$-tensor $h$, namely:
\begin{itemize} \setlength\itemsep{-1mm}
\item $R g$, where $R$ is the scalar curvature;
\item $\tRc$, the Ricci curvature;
\item $F$, the $\ph$-Ricci curvature;
\item $\cL_{\Vop T} g$, where $(\Vop T)_k = T_{ij} \ph_{ijk}$ is the vector torsion.
\end{itemize}
There are also two independent possibilities for the vector field $X$, namely:
\begin{itemize} \setlength\itemsep{-1mm}
\item $\Div T$, the divergence of $T$;
\item $\Div T^t$, the divergence of the transpose of $T$.
\end{itemize}
\end{thm}

\subsection{Differential operators, ellipticity, and parabolicity} \label{sec:diff-ops}

We begin by reviewing the notion of a parabolic PDE and the existence and uniqueness of solutions of such equations. Other sources for the discussion below are~\cite[\textsection 3.2]{Chow-Knopf},~\cite[\textsection 5.1]{AH}, and~\cite[\textsection 4]{Topping}.

Let $(M,g)$ be a Riemannian manifold with Levi-Civita connection $\nabla$. Given two vector bundles $E$, $F$ over $M$, a linear differential operator $L \colon \Gamma(E) \to \Gamma(F)$ of order $m$ is a linear map such that, for every $x \in M$, in terms of local frames for $E$ and $B$, we can write
\begin{equation}
L(\sigma)^b (x) = \sum_{l=0}^m [\hat L_l(x)]_a^{b,i_1,\ldots,i_l} [\nabla^{l}_{i_1,\ldots,i_l} \sigma(x)]^a = \sum_{l=0}^m [\hat L_l(x)]^b (\nabla^l \sigma(x))
\end{equation}
where for each $l =0, 1, \ldots, m$, we write $\nabla^l \sigma \in \Gamma((T^*M)^{\otimes l} \otimes E)$ to denote the $l$-th covariant derivative of $\sigma$, and $\hat L_l \in \Gamma( (TM)^{\otimes l} \otimes \mathrm{Hom}(E,F))$. Here the index $a$ corresponds to a local frame for $E$ and the index $b$ corresponds to a local frame for $F$.

For any such linear differential operator, we define its principal symbol so that for each $x \in M$ and $\xi \in T^*_x M$, the map
$$ \sigma_{\xi} (L) \colon E_x \to F_x $$
is the linear homomorphism
\begin{equation}
\begin{aligned}
[\sigma_{\xi}(L) (\sigma)]^b & = [\hat L_m (x)]^b (\xi, \ldots, \xi, \sigma), \\
& = [\hat L_m(x)]^{b, i_1, \ldots, i_m}_a \xi_{i_1} \cdots \xi_{i_m} \sigma^a.
\end{aligned}
\end{equation}

The principal symbol satisfies the fundamental properties
$$ \sigma_{\xi} (P + Q) = \sigma_{\xi} (P) + \sigma_{\xi} (Q), \qquad \sigma_{\xi} (P \circ Q) = \sigma_{\xi} (P) \circ \sigma_{\xi} (Q), $$
whenever $P$, $Q$ are linear differential operators so that either $P + Q$ or $P \circ Q$ is well defined.

\begin{defn} \label{def:ellipticity}
A linear differential operator $L \colon \Gamma(E) \to \Gamma(F)$ is called \emph{elliptic} if for any $x \in M$, $\xi \in T^*_x M$, $\xi \neq 0$, the principal symbol $\sigma_{\xi} (L) \colon E_x \to F_x$ is a linear isomorphism.

Let $E$ be a vector bundle over $M$ with a fibre metric $\langle \cdot, \cdot \rangle$. Consider a second-order linear differential operator $L \colon \Gamma (E) \to \Gamma(E)$. If there is a constant $c > 0$ such that for any $\xi \in T^*_x M$, $\xi \neq 0$ and $v \in E_x$, we have
$$ \langle \sigma_{\xi}(L) (v), v \rangle \geq c |\xi|^2 |v|^2, $$
then $L$ is called \emph{strongly elliptic}.
\end{defn}

We can extend the definition of ellipticity and strong ellipticity to the setting of nonlinear differential operators as follows.
\begin{defn} \label{defn:linearization}
Let $E$, $F$ be vector bundles over $M$, let $\mathcal U \subseteq \Gamma(E)$ be open, and let $P \colon \mathcal U \to \Gamma(F)$ be a nonlinear differential operator. The operator $P$ is called elliptic at $v \in \mathcal U$ if the linearization 
\begin{align*}
D_v P & \colon \Gamma(E) \to \Gamma(F), \\
(D_v P) (w) & := \rest{\frac{d}{ds}}{s=0} P(v + sw),
\end{align*}
is an elliptic linear differential operator. 

Similarly, if $P \colon \mathcal U \to \Gamma(E)$ is a second-order differential operator and $E$ is endowed with a bundle metric $\langle \cdot, \cdot \rangle$, we say that $P$ is strongly elliptic at $\sigma \in \mathcal U$ if its linearization $D_{\sigma} P \colon \Gamma(E) \to \Gamma(E)$ is a strongly elliptic linear differential operator.

A nonlinear evolution equation of the form $\frac{\partial}{\partial t} \sigma = P(\sigma)$, where $\sigma \in \mathcal U$, is called \emph{parabolic} at $\sigma$ if $P$ is strongly elliptic at $\sigma$.
\end{defn}

The importance of the above definition is due to the following standard result.

\begin{thm}
Let $M$ be a Riemannian manifold, let $E$ be a vector bundle over $M$ endowed with a fibre metric $\langle \cdot, \cdot \rangle$, and let $\mathcal U \subseteq \Gamma(E)$ be open. Let $P \colon \mathcal U \to \Gamma(E)$ be a second-order quasilinear differential operator, which is strongly elliptic at $\sigma_0 \in \mathcal U$. Then there exists $\eps > 0$ and for any $ t\in [0,\eps)$ a unique $\sigma(t) \in \mathcal U$, such that
\begin{equation} \label{eq:paraboli_IVP}
\frac{\partial \sigma(t)}{\partial t} = P(\sigma(t)), \qquad \sigma(0) = \sigma_0.
\end{equation}
That is, a nonlinear evolution equation $\frac{\partial}{\partial t} \sigma = P(\sigma)$ which is parabolic at $\sigma_0$ has a unique short time smooth solution with initial condition $\sigma(0) = \sigma_0$.
\end{thm}

\subsection{DeTurck's trick for flows of $\G$-structures} \label{sec:DeTurck}

DeTurck's trick was originally used to establish the existence and uniqueness of solutions to the Ricci flow in~\cite{DeTurck}. We now discuss DeTurck's trick for a flow of $\G$-structures. Consider again the flow of $\G$-structures
\begin{equation} \label{gflow}
\frac{\partial \ph}{\partial t} = h \diamond_{\ph} \ph + X \hk \ps 
\end{equation}
where $h$ is a family of time-dependent symmetric $2$-tensors and $X$ is a time-dependent vector field on $M$. If~\eqref{gflow} is not parabolic, and the failure of parabolicity is due solely to the diffeomorphism invariance of the system, then we can use DeTurck's trick.

Suppose that $W$ is a time-dependent vector field such that the \emph{modified} flow
\begin{equation} \label{modifiedgflow}
\frac{\partial \ph}{\partial t} = h \diamond_{\ph} \ph + X \hk \ps + \mathcal{L}_W \ph
\end{equation}
is parabolic, so it has a unique solution $\ol{\ph}(t)$ for short time.

Let $\Theta_t \colon M \to M$ be the $1$-parameter family of diffeomorphisms of $M$ whose flow is $-W$. That is,
$$ \begin{cases} \dfrac{\partial \Theta_t(p)}{\partial t} = - W(\Theta_t (p), t) \\[0.3em]
\qquad \Theta_0 = \Id_M. \end{cases} $$
Since $M$ is compact, the family of diffeomorphisms $\Theta_t$ exists by~\cite[Lemma 3.15]{Chow-Knopf} as long as the solution $\ol{\ph}(t)$ exists. Define
$$ \ph(t) = \Theta^*_t(\ol{\ph}(t)). $$
Then, using the fact that both $h$ and $X$ are taken to be diffeomorphism invariant quantities depending on $\ol{\ph}(t)$, we have
\begin{align*}
\frac{\partial \ph(t)}{\partial t} = \delt(\Theta^*_t(\ol{\ph}(t))) & = \Theta^*_t (\cL_{-W(t)} \ol{\ph}(t) + \partial_t \ol{\ph}(t) ) \\
& = \Theta^*_t ( \cL_{-W(t)} \ol{\ph}(t) + h \diamond_{\ol{\ph}(t)} \ol{\ph}(t) + X \hk \ol{\ps}(t) + \cL_{W(t)} \ol{\ph}(t) ) \\
& = h \diamond_{\Theta^*_t \ol{\ph} (t)} ( \Theta^*_t (\ol{\ph}(t)) ) + X \hk ( \Theta^*_t (\ol{\ps}(t)) ) \\
& = h \diamond_{\ph(t)} \ph(t) + X \hk \ps(t).
\end{align*}
Thus $\ph(t) = \Theta^*_t (\ol{\ph}(t))$ is a solution of~\eqref{gflow} with a given initial condition. Uniqueness follows from the uniqueness of solutions~\eqref{modifiedgflow}, which we are assuming is parabolic.

\subsection{Nonlinear differential operators on $(M, \ph)$ and principal symbols} \label{sec:symbols}

Let $M$ be a $7$-manifold with a $\G$-structure $\ph$ and induced Riemannian metric $g$, and consider a linear differential operator $L \colon \Omega^3 (M) \to \Gamma(F)$ of order $m$, where $F$ is a vector bundle over $M$.

Expressing any $3$-form $\gamma$ as $\gamma = h \diamond \ph + X \hk \ps$ we define
$$ L_{\ph} \colon \Gamma (\sym{2} \oplus T^*M) \to \Gamma(F) $$
to be the linear differential operator
$$L_{\ph} (h,X) = L(h \diamond \ph + X \hk \ps). $$
Since the operator $(h,X) \mapsto h \diamond \ph + X \hk \ps$ is a zero order linear differential operator it follows that
$$ \sigma_{\xi} (L_{\ph}) (h,X) = \sigma_{\xi} (L) \circ \sigma_{\xi} (\cdot \diamond \ph + \cdot \hk \ps) (h,X) = \sigma_{\xi}(L) (h \diamond \ph + X \hk \ps). $$

Moreover, if $\Gamma(F) = \Omega^3 = \Omega^3_1 \oplus \Omega^3_{27} \oplus \Omega^3_7$, and we denote by $\pi_1$, $\pi_{27}$, $\pi_7$ the associated projections with $\pi_{1+27} = \pi_1 \oplus \pi_{27}$, then we denote
$$ L_{\symm} = \pi_{1+27} \circ L_{\ph}, \qquad L_7 = \pi_7 \circ L_{\ph} $$
and write $L_{\ph} (h,X) = L_{\symm} (h,X) \oplus L_7 (h,X)$. It is clear that the principal symbols of $L_{\symm}$ and $L_7$ are
$$ \sigma_{\xi} (L_{\symm}) = \pi_{1+27} \circ \sigma_{\xi} (L_{\ph}), \qquad \sigma_{\xi} (L_7) = \pi_7 \circ \sigma_{\xi} (L_{\ph}). $$
In the following, in order to simplify our notation, we do not distinguish between a linear operator $L \colon \Omega^3 \to \Gamma(F)$ and its description $L_{\ph} \colon \Gamma(\sym{2} \oplus T^*M ) \to \Gamma(F)$ as a linear differential operator acting on pairs $(h,X) \in \Gamma(\sym{2} \oplus T^*M )$, once a particular $\G$-structure $\ph$ is specified. Similarly, it is more convenient to use the isomorphism $(h,X) \mapsto h \diamond_{\ph} \ph + X \hk \ps$ to express a linear differential operator $L \colon \Omega^3 \to \Omega^3$ as an operator $L \colon \Gamma(\sym{2} \oplus T^*M) \to \Gamma(\sym{2} \oplus T^*M)$, with $L = L_{\symm} \oplus L_7$.

Moreover, by~\eqref{eq:ABph}, the bundle metric $\langle (h_1, X_1), (h_2, X_2) \rangle = \langle h_1, h_2 \rangle + \langle X_1, X_2 \rangle$ on $\sym{2} \oplus T^* M$ is uniformly equivalent to the natural inner product on $\Lambda^3( T^*M)$. Hence, $L$ is strongly elliptic if and only if there is a constant $c>0$ such that for any $x \in M$, $\xi \in T^*_x M$, $\xi \neq 0 $, and any $(h,X) \in \mathrm{S}^2(T^*_x M) \oplus T^*_x M$, we have
$$ \langle \sigma_{\xi}(L) (h, X), (h, X) \rangle \geq c | (h, X) |^2 =c ( |h|^2 + |X|^2 ). $$

We now consider various first and second-order nonlinear differential operators acting on a $\G$-structure $\ph$, which are differential invariants of the $\G$-structure. We compute their linearizations and associated principal symbols. We write $\lot$ to denote ``\emph{$\ell$\!ower order terms}'', and only compute the linearizations up to such $\lot$, because only the $m^{\text{th}}$ derivative terms contribute to the principal symbol of an $m^{\text{th}}$ order differential operator. In this section, because we differentiate contractions, we need to be careful about our subscript/superscript abuse of notation.

\begin{prop} \label{prop:linearizations}
Consider a variation $\frac{\partial}{\partial t} \ph = h \diamond \ph + X \hk \ps$ of a $\G$-structure $\ph$. Then the induced variations of the first-order differential invariants $T$, $T_{\symm}$, $T_{\skew}$, and $\Vop T$ are given by
\begin{align*}
[(D_{\ph} T) (h, X)]_{pq} & = \nab{a} h_{bp} \ph_{abq} + \nab{p} X_q + \lot, \\
[(D_{\ph} T_{\symm}) (h,X)]_{pq} & = \tfrac{1}{2} \nab{a} h_{bp} \ph_{abq} + \tfrac{1}{2} \nab{a} h_{bq} \ph_{abp} + \tfrac{1}{2} \nab{p} X_q + \tfrac{1}{2} \nab{q} X_p + \lot, \\
[(D_{\ph} T_{\skew}) (h,X)]_{pq} & = \tfrac{1}{2} \nab{a} h_{bp} \ph_{abq} - \tfrac{1}{2} \nab{a} h_{bq} \ph_{abp} + \tfrac{1}{2} \nab{p} X_q - \tfrac{1}{2} \nab{q} X_p + \lot, \\
[(D_{\ph} \Vop T) (h, X)]_k & = \nab{k} (\tr h) - (\Div h)_k + (\curl X)_k + \lot.
\end{align*}
Moreover, the induced variations of the second-order differential invariants $\cL_{\Vop T} g$, $F$, $\Div T$, and $\Div T^t$ are given by
\begin{align*}
[(D_{\ph} \cL_{\Vop T} g) (h, X)]_{jk} & = \nab{j} (\nab{k} (\tr h) - (\Div h)_k + (\curl X)_k) \\
& \qquad {} + \nab{k} (\nab{j} (\tr h) - (\Div h)_j + (\curl X)_j) + \lot, \\
[(D_{\ph} F) (h, X)]_{jk} & = 2 (\nab{p} \nab{a} h_{bq} + \nab{a} \nab{p} h_{bq}) \ph_{abj} \ph_{pqk} + \lot, \\
[(D_{\ph} \Div T) (h, X)]_k & = \nab{a} (\Div h)_b \ph_{abk} + \Delta X_k + \lot, \\
[(D_{\ph} \Div T^t) (h, X)]_k & = \nab{k} (\Div X) + \lot.
\end{align*}
\end{prop}
\begin{proof}
By Definition~\ref{defn:linearization}, if $Q := Q(\ph)$ is a differential invariant of a $\G$-structure $\ph$, then
$$ (D_{\ph} Q) (h, X) = \frac{\partial}{\partial t} Q(\ph(t)). $$
From Remark~\ref{rmk:flows1-evolution-torsion} we know that given a variation $\frac{\partial \ph}{\partial t} = h \diamond \ph + X \hk \ps$ of a $\G$-structure, the variation of the torsion is given by
$$ \frac{\partial T_{pq}}{\partial t} = \nab{a} h_{bp} \ph_{abq} + \nab{p} X_q + T_{pa} h_{aq} + T_{pa} X_b \ph_{baq}, $$
yielding the expression for $D_{\ph} T$, and the expressions for $D_{\ph} T_{\symm}$ and $D_{\ph} T_{\skew}$ then follow.

Using $\frac{\partial}{\partial t} g^{ij} = - 2 h^{ij}$ which follows from Lemma~\ref{lemma:evolution-1}, we compute the variation of $(\Vop T)_k = g^{pi} g^{qj}T_{pq} \ph_{ijk}$ as
\begin{align*}
\frac{\partial (\Vop T)_k}{\partial t} & = (\nab{a} h_{bp} \ph_{abq} + \nab{p} X_q + T_{pa} h_{aq} + T_{pa} X_b \ph_{baq}) \ph_{pqk} - 2 h_{ai} T_{aj} \ph_{ijk} - 2h_{bj} T_{ib} \ph_{ijk} \\
& \qquad {} + T_{ij} (h \diamond \ph)_{ijk} + T_{ij} (X \hk \ps)_{ijk}.
\end{align*}
Thus, ignoring lower order terms, we get
\begin{align*}
\frac{\partial (\Vop T)_k}{\partial t} & = \nab{a} h_{bp} \ph_{abq} \ph_{kpq} + \nab{p} X_q \ph_{pqk} + \lot \\
& = \nab{a} h_{bp} (g_{ak} g_{bp} - g_{ap} g_{bk} - \ps_{abkp}) + (\curl X)_k + \lot \\
& = \nab{k} (\tr h) - (\Div h)_k - 0 + (\curl X)_k + \lot,
\end{align*}
yielding the expression for $D_{\ph} \Vop T$. 

By~\eqref{eq:evolution-nabla}, the variation $\frac{\partial}{\partial t} \nabla$ of the connection introduces terms which are first derivatives of $h$ or $X$. Thus, in the computations of the variations of the second-order differential invariants, such terms contribute to $\lot$. With this understood, the expression for the variation of $(\cL_{\Vop T} g)_{jk} = \nab{j} (\Vop T)_k + \nab{k} (\Vop T)_j$ then follows immediately from that of $\Vop T$.

Recall that from~\eqref{eq:F-from-T}, we have
$$ F_{jk} = 2 \nab{p} T_{qj} \ph_{pqk} + 2 \nab{p} T_{qk} \ph_{pqj} - 2 T_{pa} T_{qb} \ph_{pqj} \ph_{abk}, $$
from which we compute
\begin{align} \nonumber
\frac{\partial}{\partial t} F_{jk} & = 2 \nab{p} (\nab{a} h_{bq} \ph_{abj} + \nab{q} X_j + \lot) \ph_{pqk} \\ \nonumber
& \qquad {} + 2 \nab{p} (\nab{a} h_{bq} \ph_{abk} + \nab{q} X_k + \lot) \ph_{pqj} + \lot \\ \nonumber 
& = 2 \nab{p} \nab{a} h_{bq} \ph_{abj} \ph_{pqk} + 2 \nab{p} \nab{a} h_{bq} \ph_{abk} \ph_{pqj} \\ \label{eq:F-symbol-temp}
& \qquad{} + 2 \nab{p} \nab{q} X_j \ph_{pqk} + 2 \nab{p} \nab{q} X_k \ph_{pqj} + \lot.
\end{align}
Note that
\begin{align} \nonumber
2 \nab{p} \nab{q} X_j \ph_{pqk} & = \nab{p} \nab{q} X_j \ph_{pqk} + \nab{q} \nab{p} X_j \ph_{qpk} \\ \label{eq:F-symbol-temp2}
& = (\nab{p} \nab{q} X_j - \nab{q} \nab{p} X_j) \ph_{pqk} = - R_{pqjm} X_m \ph_{pqk}.
\end{align}
In the second term of~\eqref{eq:F-symbol-temp}, we swap the roles of $p,a$ and $q,b$ and use the symmetry of $h$. In the third and fourth terms above, we use~\eqref{eq:F-symbol-temp2}. The result is the expression for $D_{\ph} F$.

Finally, we compute the variations of $\Div T$ and $\Div T^t$. Proceeding as before, we have
\begin{align*}
\frac{\partial}{\partial t} \Div T_k & = \frac{\partial}{\partial t} (g^{ij} \nab{i} T_{jk}) = g^{ij} \nab{i} \Big( \frac{\partial T_{jk}}{\partial t} \Big) + \lot \\
& = g^{ij} \nab{i} (\nab{a} h_{bj} \ph_{abk} + \nab{j} X_k + \lot) + \lot \\
& = \nab{i} \nab{a} h_{bi} \ph_{abk} + \nab{i} \nab{i} X_k + \lot.
\end{align*}
The second term above is $\Delta X_k$, and by the Ricci identity the first term is $(\nab{a} \nab{i} h_{bi} + \lot) \ph_{abk} = \nab{a} (\Div h)_b \ph_{abk} + \lot$, yielding the expression for $D_{\ph} \Div T$.

Similarly we have
\begin{align*}
\frac{\partial}{\partial t} \Div T^t_k & = \frac{\partial}{\partial t}(g^{ij} \nab{i} T_{kj}) = g^{ij} \nab{i} \Big( \frac{\partial T_{kj}}{\partial t} \Big) + \lot \\
& = g^{ij} \nab{i} (\nab{a} h_{bk} \ph_{abj} + \nab{k} X_j + \lot) + \lot \\
& = \nab{i} \nab{a} h_{bk} \ph_{iab} + \nab{i} \nab{k} X_i + \lot.
\end{align*}
The first term above is purely lower order, because
\begin{align*}
\nab{i} \nab{a} h_{bk} \ph_{iab} & = \tfrac{1}{2} \nab{i} \nab{a} h_{bk} \ph_{iab} + \tfrac{1}{2} \nab{a} \nab{i} h_{bk} \ph_{aib} \\
& = \tfrac{1}{2} (\nab{i} \nab{a} h_{bk} - \nab{a} \nab{i} h_{bk}) \ph_{iab} \\
& = - \tfrac{1}{2} (R_{iabm} h_{mk} + R_{iakm} h_{bm}) \ph_{iab} = \lot.
\end{align*}
Applying the Ricci identity to the second term gives $\nab{i} \nab{k} X_i = \nab{k} \nab{i} X_i + \lot = \nab{k} (\Div X) + \lot$, yielding the expression for $D_{\ph} \Div T^t$.
\end{proof}

Using Proposition~\ref{prop:linearizations}, we easily compute the principal symbol of a differential operator associated to a $\G$-structure, using the usual procedure of replacing $\nab{i}$ by $\xi_i$ in the linearization.
 
\begin{prop} \label{prop:symbols}
Consider a variation $\frac{\partial}{\partial t} \ph = h \diamond \ph + X \hk \ps$ of a $\G$-structure $\ph$. For any nonzero $\xi \in T^*_x M$, we have the following principal symbols of first-order nonlinear differential operators:
\begin{align}
\sigma_{\xi} (D_{\ph} T) (h,X)_{jk} & = \xi_a h_{bj} \ph_{abk} + \xi_j X_k, \label{eq:psymbT} \\
\sigma_{\xi} (D_{\ph} T_{\symm}) (h,X)_{jk} & = \tfrac{1}{2} \big[ (\xi_a h_{bj} \ph_{abk} + \xi_a h_{bk} \ph_{abj}) + (\xi_j X_k + \xi_k X_j) \big], \label{eq:psymbTsym} \\
\sigma_{\xi} (D_{\ph} T_{\skew}) (h,X)_{jk} & = \tfrac{1}{2} \big[ (\xi_a h_{bj} \ph_{abk} - \xi_a h_{bk} \ph_{abj}) + (\xi_j X_k - \xi_k X_j) \big], \label{eq:psymbTskew} \\
\sigma_{\xi} (D_{\ph} \Vop T) (h,X)_k & = \xi_k \tr h - \xi_a h_{ak} + \xi_a X_b \ph_{abk}. \label{eq:psymbVT}
\end{align}
Moreover, we have the following principal symbols of second-order nonlinear differential operators:
\begin{align}
\sigma_{\xi} (D_{\ph} \cL_{\Vop T} g)_{jk} & = 2 \xi_j \xi_k \tr h - \xi_j \xi_a h_{ak} - \xi_k \xi_a h_{aj} + \xi_j \xi_b X_c \ph_{bck} + \xi_k \xi_b X_c \ph_{bcj}, \label{eq:psymbLieVTg} \\
\sigma_{\xi} (D_{\ph} F) (h,X)_{jk} & = 4 \xi_a \ph_{abj} h_{bq} \xi_p \ph_{pqk}, \label{eq:psymbF} \\
\sigma_{\xi} (D_{\ph} \Div T) (h,X)_k & = \xi_a \xi_m h_{mb} \ph_{abk} +|\xi|^2 X_k, \label{eq:psymbdivT} \\
\sigma_{\xi} (D_{\ph} \Div T^t) (h,X)_k & = \xi_k \langle \xi, X \rangle, \label{eq:psymbdivTt} \\
\sigma_{\xi} (D_{\ph} \tRc) (h,X)_{jk} & = - |\xi|^2 h_{jk} + (\xi_j \xi_a h_{ak} + \xi_k \xi_a h_{aj}) - \xi_j \xi_k \tr h, \label{eq:psymbricci} \\
\sigma_{\xi} (D_{\ph} R) (h,X) & = -2 |\xi|^2 \tr h + 2 h(\xi, \xi). \label{eq:psymbR}
\end{align}
\end{prop}
\begin{proof}
The first eight symbols are immediate from Proposition~\ref{prop:linearizations}. The symbols for the Ricci curvature and scalar curvature are standard, and can be found, for example, in Chow--Knopf~\cite[Section 2.1]{Chow-Knopf}. Note that we have an extra factor of $2$ in these because $\frac{\partial}{\partial t} g_{ij} = 2 h_{ij}$ by~\eqref{eq:evolution-1}.
\end{proof}

We also need the following related result. Define the map
$$ \delta^* \colon \vf \to \Omega^3, \qquad \delta^* W = \cL_W \ph. $$
From~\eqref{eq:Lie-derivative-alternate} we have
\begin{equation} \label{eq:delta-star}
\delta^* W = \cL_W \ph = \tfrac{1}{2} (\cL_W g) \diamond \ph + (- \tfrac{1}{2} \curl W + T^t W) \hk \ps.
\end{equation}

\begin{prop} \label{prop:symbol-Lie-derivative}
Let $\delta^* \colon \vf \to \Omega^3$ be as in~\eqref{eq:delta-star}. For any nonzero $\xi \in T^*_x M$, we have
\begin{equation} \label{eq:LieSymbol}
\begin{aligned}
[\pi_{1+27} \circ \sigma_{\xi} (\delta^*) (W)]_{jk} & = \tfrac{1}{2} (\xi_j W_k + \xi_k W_j), \\
[\pi_7 \circ \sigma_{\xi} (\delta^*) (W)]_k& = - \tfrac{1}{2} \xi_p W_q \ph_{pqk}, 
\end{aligned}
\end{equation}
and $\sigma_{\xi}(\delta^*) \colon T^*_x M \to \Lambda^3 (T^*_x M)$ is injective.
\end{prop}
\begin{proof}
The expressions in~\eqref{eq:LieSymbol} follow from~\eqref{eq:delta-star}, because $(\cL_W g)_{jk} = \nab{j} W_k + \nab{k} W_j$ and $(\curl W)_k = \nab{p} W_q \ph_{pqk}$. Suppose $W \in \ker \ker \sigma_{\xi}(\delta^*)$. In particular we get $\xi_j W_k + x_k W_j = 0$. Multiplying by $\xi_j W_k$ and summing, we obtain
$$ 0 = |\xi|^2 |W|^2 + \langle W, \xi \rangle^2, $$
which implies that $W = 0$, so $\sigma_{\xi}(\delta^*)$ is injective.
\end{proof}

\subsection{Ellipticity modulo diffeomorphisms} \label{sec:elliptic-mod-diffeos}

Recall from Theorem~\ref{thm:invariants} the classification of independent second-order differential invariants of a $\G$-structure which are $3$-forms. In this section we consider differential operators on $\G$-structures of the general form
\begin{equation} \label{eq:general-operator}
P (\ph) = (a_1 \tRc +a_2 \cL_{\Vop T} g + a_3 F + a_4 R g) \diamond \ph + ( b_1 \Div T + b_2 \Div T^t) \hk \ps,
\end{equation}
for \emph{constants} $a_1, a_2, a_3, a_4, b_1, b_2$, and analyze their principal symbols. Note that $P (\ph)$ is invariant under diffeomorphisms. That is, 
\begin{equation*}
P (\Theta^* \ph) = \Theta^* (P(\ph)),
\end{equation*}
for any diffeomorphism $\Theta \colon M \to M$. It follows that for any vector field $W \in \vf$, we have
\begin{equation} \label{eq:diffeo_invariance}
\cL_W (P(\ph)) = D_{\ph} P(\cL_W \ph).
\end{equation}
Since $W \mapsto \cL_W (P(\ph))$ is a first-order linear differential operator on $W$, whereas 
$$ W \mapsto D_{\ph} P(\cL_W \ph) = (D_{\ph} P \circ \delta^*) (W)$$ is \emph{a priori} a third-order differential operator, it follows that
\begin{equation*}
\sigma_{\xi} (D_{\ph} P \circ \delta^*) = \sigma_{\xi} (D_{\ph} P) \circ \sigma_{\xi} (\delta^*) = 0.
\end{equation*}
We therefore deduce that
\begin{equation} \label{eq:im-delta-star}
\im ( \sigma_{\xi} (\delta^*)) = \{ (\tfrac{1}{2} (\xi \otimes V + V \otimes \xi), - \tfrac{1}{2} \xi \times V) : V \in T^*_x M \} \subseteq \ker \sigma_{\xi} (D_{\ph} P).
\end{equation}

Hence, by the injectivity of $\sigma_{\xi} (\delta^*)$, the principal symbol of $D_{\ph} P$ always has a kernel of dimension \emph{at least} $7$ that is due to diffeomorphism invariance, so $P$ is never an elliptic differential operator. This is a quite typical phenomenon when one considers nonlinear differential operators of a geometric nature.

In this section, we distinguish several cases in which the failure of ellipticity is \emph{only} due to diffeomorphism invariance, in the sense that the kernel of the principal symbol of $D_{\ph} P$ is precisely equal to $\im (\sigma_{\xi} (\delta^*))$. In a similar spirit, we also study the principal symbol of the linearization of $F$ and show that 
$$ \ker \sigma_{\xi} (D_{\ph} F) = \im (\sigma_{\xi} (\delta^*)) + \{(0,X) : X \in T^*_x M \}. $$
That is, the kernel of $\sigma_{\xi} (D_{\ph} F)$ is \emph{due only to diffeomorphism invariance and isometric variations}.

In order to study this failure of ellipticity, we define the following two linear maps:
\begin{equation} \label{eq:B-maps}
\begin{aligned}
B_1 & \colon S^2 (T^*_x M) \to T^*_x M, & B_1 (h)_k & = \xi_a h_{ak} - \tfrac{1}{2} \xi_k \tr h, \\
B_2 & \colon T^*_x M \to T^*_x M, & B_2 (X)_k & = \xi_a X_b \ph_{abk}.
\end{aligned}
\end{equation}
The map $B_1$ is the symbol of the \emph{Bianchi map} $\cS^2 \to \Omega^1$ given by $h \mapsto \nab{i} h_{ik} - \frac{1}{2} \nab{k} (\tr h)$. The twice contracted Riemannian second Bianchi identity~\eqref{eq:riem2B3} implies that the Ricci curvature tensor $\tRc$ is in the kernel of the Bianchi map. The map $B_2$ is the symbol of the curl operator from Definition~\ref{defn:curl}. The reason we need the map $B_2$ is because, as explained in Section~\ref{sec:DeTurck}, when we apply the DeTurck trick in Section~\ref{sec:short-time} we need to add a term of the form $\cL_W \ph$ to the right-hand side of our flow, and the $\curl$ operator shows up in the $\Omega^3_7$ part of $\cL_W \ph$, by equation~\eqref{eq:Lie-derivative-framed}.

Consider an operator of the form~\eqref{eq:general-operator}. We say that $P$ is \emph{Ricci-like} if
\begin{equation} \label{eq:ricci_like}
P(\ph) = (- \tRc + a \cL_{\Vop T} g) \diamond \ph + (b_1 \Div T +b_2 \Div T^t) \hk \ps.
\end{equation}
That is, if $a_3 = a_4 = 0$, $a_1 = -1$, and $a_2 = a \in \R$.

In the case of a Ricci-like operator, it is convenient to define the operator $\tilde B \colon \cS^2 \oplus \Omega^1 \to \Omega^1$ by 
\begin{equation} \label{eq:tildeB_def}
\tilde B(h,X)_k = B_1(h)_k - a \sigma_{\xi} (D_{\ph} \Vop T) (h,X)_k.
\end{equation}

\begin{prop} \label{prop:ricci_like}
Let $P$ be a Ricci-like operator as in~\eqref{eq:ricci_like}. In terms of the maps $B_1, B_2$ of~\eqref{eq:B-maps} and $\tilde B$ of~\eqref{eq:tildeB_def}, the $1+27$ and $7$ parts of the principal symbol of linearization $D_{\ph} P$ can be expressed as
\begin{equation} \label{eq:psymbP2}
\begin{aligned}
\pi_{1+27} \circ \sigma_{\xi} (DP) (h, X)_{jk} & = |\xi|^2 h_{jk} - \xi_j (\tilde B(h,X))_k - \xi_k (\tilde B(h,X))_j, \\
\pi_7 \circ \sigma_{\xi} (DP) (h, X)_l & = (b_1 + b_2) \big[ \, |\xi|^2 X_l + B_2 ( \tilde B(h,X))_l \, \big] \\
& \quad + (ab_1 + (1+a)b_2) (- \xi_a \xi_i h_{ib} \ph_{abl} + \xi_l \langle\xi,X\rangle - |\xi|^2 X_l ).
\end{aligned}
\end{equation}
\end{prop}
\begin{proof} 
Using~\eqref{eq:psymbLieVTg},~\eqref{eq:psymbdivT},~\eqref{eq:psymbdivTt}, and~\eqref{eq:psymbricci}, we get that the principal symbol of $D_{\ph} P$ for a Ricci-like operator $P$ satisfies 
\begin{equation*}
\begin{aligned}
\pi_{1+27} \circ \sigma_{\xi} (D_{\ph} P) (h,X) & = |\xi|^2 h_{jk} - (\xi_j \xi_a h_{ak} + \xi_k \xi_a h_{aj}) + \xi_j \xi_k \tr h + 2 a \xi_j \xi_k \tr h \\
& \qquad {} - a \xi_j \xi_a h_{ak} - a \xi_k \xi_a h_{aj} + a \xi_j\xi_b X_c \ph_{bck} + a \xi_k \xi_b X_c \ph_{bcj}, \\
\pi_{7} \circ \sigma_{\xi} (D_{\ph} P) (h,X)_l & = b_1 \xi_a \xi_i h_{iq} \ph_{aql} + b_1 |\xi|^2 X_l + b_2 \xi_l \langle \xi, X \rangle.
\end{aligned}
\end{equation*}
We can use $B_1$ and~\eqref{eq:psymbVT} to express the $1+27$ part of $\sigma_{\xi} (D_{\ph} P)$ as follows:
\begin{align*}
\pi_{1+27} \circ \sigma_{\xi} (D_{\ph} P) (h,X)_{jk} & = |\xi|^2 h_{jk} - (\xi_j \xi_a h_{ak} + \xi_k \xi_a h_{aj}) + \xi_j \xi_k \tr h +2 a \xi_j \xi_k \tr h \\
& \qquad {} - a \xi_j \xi_a h_{ak} - a \xi_k \xi_a h_{aj} + a\xi_j \xi_b X_c \ph_{bck} + a\xi_k \xi_b X_c \ph_{bcj} \\
& = |\xi|^2 h_{jk} - \xi_j B_1(h)_k - \xi_k B_1(h)_j + a \xi_j (\xi_k \tr h - \xi_a h_{ak} + \xi_b X_c \ph_{bck}) \\
& \qquad {} + a \xi_k (\xi_j \tr h - \xi_a h_{aj} + \xi_b X_c \ph_{bcj}) \\
& = |\xi|^2 h_{jk} - \xi_j \big( B_1(h)_k - a \sigma_{\xi} (D_{\ph} \Vop T)_k \big) - \xi_k \big( B_1(h)_j - a \sigma_{\xi} (D_{\ph} \Vop T)_j \big).
\end{align*}
Using the definition of $\tilde B$, we obtain the expression for $\pi_{1+27} \circ \sigma_{\xi} (D_{\ph} P) (h,X)_{jk}$.

From~\eqref{eq:psymbVT} we obtain
\begin{equation} \label{eq:B2VT}
\begin{aligned}
B_2 (\sigma_{\xi} (D_{\ph} \Vop T))_k & = \xi_a (\xi_b \tr h - \xi_i h_{ib} + \xi_i X_j \ph_{ijb}) \ph_{kab} \\
& = - \xi_a \xi_i h_{ib} \ph_{abk} + \xi_a \xi_i X_j (g_{ik} g_{ja} - g_{ia} g_{jk} - \ps_{ijka}) \\
& = - \xi_a \xi_i h_{ib} \ph_{abk} + \xi_k \langle \xi, X \rangle - |\xi|^2 X_k.
\end{aligned}
\end{equation}
We also have
\begin{equation} \label{eq:B2B1}
B_2 (B_1(h))_k = \xi_a (\xi_i h_{ib} - \tfrac{1}{2} \xi_b \tr h) \ph_{abk} = \xi_a \xi_i h_{ib} \ph_{abk}.
\end{equation}

Using~\eqref{eq:B2VT} and~\eqref{eq:B2B1}, the $7$ part of the symbol of $D_{\ph} P$ becomes
\begin{align*}
\pi_7 \circ \sigma_{\xi} (D_{\ph} P) (h,X)_l & = b_1 \xi_a \xi_i h_{iq} \ph_{aql} + b_1 |\xi|^2 X_l + b_2 \xi_l \langle \xi, X \rangle \\
& = b_1 \xi_a \xi_i h_{iq} \ph_{aql} + b_1 |\xi|^2 X_l + b_2 (B_2 (\sigma_{\xi} (D_{\ph} \Vop T))_l + \xi_a \xi_i h_{ib} \ph_{abl} + |\xi|^2 X_l ) \\
& = (b_1 + b_2) |\xi|^2 X_l + (b_1 + b_2) B_2 ( B_1(h) )_l + b_2 B_2 (\sigma_{\xi} (D_{\ph} \Vop T))_l.
\end{align*}
Substituting~\eqref{eq:tildeB_def} and~\eqref{eq:B2VT}, we obtain the expression for $\pi_7 \circ \sigma_{\xi} (D_{\ph} P) (h,X)_{jk}$.
\end{proof}

\begin{rmk} \label{rmk:ricci_like-interesting}
The most interesting case of Proposition~\ref{prop:ricci_like} occurs when $a b_1 + (1+a) b_2 = 0$, because in this case, the principal symbol of the operator $D_{\ph} P$ takes the simple form
\begin{equation}
\begin{aligned}
\pi_{1+27} \circ \sigma_{\xi} (D_{\ph} P) (h,X)_{jk} & = |\xi|^2 h_{jk} - \xi_j \tilde B(h,X)_k - \xi_k \tilde B(h,X)_j, \\
\pi_7 \circ \sigma_{\xi} (D_{\ph} P) (h,X)_l & = (b_1 + b_2) \big[ \, |\xi|^2 X_l +B_2 ( \tilde B(h,X) )_l \, \big],
\end{aligned}
\end{equation}
where the operator $\tilde B$ plays a role similar to the role of the Bianchi operator $B_1$ in the analysis of the principal symbol of the Ricci tensor, for instance in the Ricci flow.
\end{rmk}

We now consider the operator $\tilde B$ and its adjoint in detail. From the definition~\eqref{eq:tildeB_def} of the map $\tilde B$, using equations~\eqref{eq:B-maps} and~\eqref{eq:psymbVT} we obtain
\begin{align*}
\tilde B(h,X)_k & = B_1(h)_k - a \sigma_{\xi} (D_{\ph} \Vop T) (h,X)_k \\
& = \xi_a h_{ak} - \tfrac{1}{2} \xi_k \tr h -a(\xi_k \tr h - \xi_i h_{ik} + \xi_i X_j \ph_{ijk}) \\
& = (1+a) \xi_a h_{ak} - (a+ \tfrac{1}{2}) \xi_k \tr h -a \xi_a X_b \ph_{abk}.
\end{align*}
To determine the adjoint $\tilde B^*$ of $\tilde B$, we use the above to compute
\begin{align*}
\langle \tilde B(h,X), Y \rangle & = [(1+a) \xi_a h_{ak} - (a + \tfrac{1}{2}) \xi_k \tr h -a \xi_a X_b \ph_{abk}] Y_k \\
& = (1+a) h_{ak} \tfrac{1}{2} (\xi_a Y_k + \xi_k Y_a) - (a+ \tfrac{1}{2}) \langle h,g\rangle \langle \xi,Y\rangle +a \xi_a Y_k \ph_{akb} X_b \\
& =h_{ak} \big( \tfrac{1}{2} (1+a) (\xi_a Y_k + \xi_k Y_a) - (a + \tfrac{1}{2}) \langle \xi,Y\rangle g_{ak} \big) + a\xi_a Y_k \ph_{akb} X_b \\
& = \langle (h,X), \tilde B^*(Y) \rangle.
\end{align*}
Thus $\tilde B^*(Y) = (\tilde B^*_1(Y),\tilde B^*_2(Y))$ where
\begin{align*}
\tilde B^*_1 (Y)_{ak} & = \tfrac{1}{2} (1+a) (\xi_a Y_k + \xi_k Y_a) - (a + \tfrac{1}{2}) \langle \xi, Y \rangle g_{ak}, \\
\tilde B^*_2 (Y)_b & = a \xi_a Y_k \ph_{akb}.
\end{align*}

\begin{lemma} \label{lemma:adjoint-tildeB-injective}
The map $\tilde B^* \colon \Omega^1 \to \cS^2 \oplus \Omega^1$ is injective. Consequently, $\dim (\ker \tilde B) = 28$.
\end{lemma}
\begin{proof}
Let $Y\in \ker \tilde B^*$ so that $\tilde B^*_2 (Y) =0$. This says $\xi \times Y = 0$, so $Y = \lambda \xi$ for some $\lambda \in \R$. Substituting this into $B^*_1 (Y) = 0$ yields
\begin{align*}
0 = \tilde B_1^*(Y)_{ak} = \lambda(1+a) \xi_a \xi_k - (a+ \tfrac{1}{2}) \lambda |\xi|^2 g_{ak}.
\end{align*}
Taking the norm of both sides above, we get
\begin{align*}
0 & = | \tilde B_1^*(Y) |^2 \\
& = ( \lambda (1+a) \xi_a \xi_k - (a + \tfrac{1}{2}) \lambda |\xi|^2 g_{ak}) (\lambda (1+a) \xi_a \xi_k - (a + \tfrac{1}{2}) \lambda |\xi|^2 g_{ak} ) \\
& = \lambda^2 (1+a)^2 |\xi|^4 +7 (a + \tfrac{1}{2})^2 \lambda^2 |\xi|^4 - 2 (a + \tfrac{1}{2}) (1+a) \lambda^2 |\xi|^4 \\
& = ( (1+a)^2 + 7 (a + \tfrac{1}{2})^2 - 2 (a + \tfrac{1}{2}) (1+a)) \lambda^2 |\xi|^4 \\
& = ( ( 1+a - a - \tfrac{1}{2})^2 + 6 (a + \tfrac{1}{2})^2 ) \lambda^2 |\xi|^4 \\
& = ( \tfrac{1}{4} + 6 (a + \tfrac{1}{2})^2 ) \lambda^2 |\xi|^4.
\end{align*}
Since $\xi \neq 0$, we get $\lambda = 0$ and thus $Y = 0$.

The injectivity of $\tilde B^*$ implies that $\dim (\im \tilde B^*) =7$, and from the decomposition
$$ \Lambda^3(T^*_x M) \cong S^2 (T^*_x M) \oplus T^*_x M = \ker \tilde B \oplus \im \tilde B^* $$
we deduce that
\begin{equation} \label{eq:dim-ker-tildeB}
\dim (\ker \tilde B) = 35 - 7 = 28
\end{equation}
as claimed.
\end{proof}

The following is our main result on Ricci-like operators.
\begin{prop} \label{prop:ricci_like-weak-ellipticity}
Consider a Ricci-like differential operator $P$, as in~\eqref{eq:ricci_like}, which satisfies
\begin{equation} \label{eq:riccilike_weak_ell}
b_1 + b_2 \neq 0, \qquad b_1 \neq a (b_1 + b_2).
\end{equation}
Then, for any $\G$-structure $\ph$, we have $\ker (\sigma_{\xi} (D_{\ph} P)) = \im (\sigma_{\xi} (\delta^*))$.

Moreover, if $b_1 + b_2=1$ and $a= -b_2$, then for every $(h,X) \in \ker (\tilde B)$, we have 
$$ \sigma_{\xi} (D_{\ph} P) (h,X) = |\xi|^2 (h,X). $$
In particular, in this special case, $\sigma_{\xi} (D_{\ph} P)$ preserves $\ker (\tilde B)$. This choice for $a, b_1, b_2$ corresponds to differential operators of the form
$$ P (\ph) = (- \tRc + a \cL_{\Vop T} g) \diamond \ph + ( (1+a) \Div T - a \Div T^t ) \hk \ps. $$
\end{prop}

\begin{proof}
To prove these assertions we first observe that~\eqref{eq:psymbP2} implies that
\begin{equation} \label{eq:symbDP_ker_1}
\begin{aligned}
& \qquad \rest{\sigma_{\xi} (DP)}{\ker\tilde B} (h,X)
\\ & = (|\xi|^2 h_{jk}, (b_1 + b_2) |\xi|^2 X_l + (a b_1 + (1+a) b_2) ( - \xi_a \xi_i h_{ib} \ph_{abl} + \xi_l \langle \xi, X \rangle - |\xi|^2 X_l) ).
\end{aligned}
\end{equation}
We claim that under the assumptions~\eqref{eq:riccilike_weak_ell}, we have
\begin{equation} \label{eq:kerBkersigma}
\ker \tilde B\cap \ker \sigma_{\xi} (D_{\ph} P) = \{0\}.
\end{equation}
To see this, note that by~\eqref{eq:symbDP_ker_1}, if $(h,X) \in \ker \tilde B \cap \ker \sigma_{\xi} (D_{\ph} P)$, then $h=0$ and
\begin{equation}
\begin{aligned}
0 & = (b_1 + b_2) |\xi|^2 X_l + ( ab_1 + (1+a) b_2) (\xi_l \langle \xi, X \rangle - |\xi|^2 X_l) \\
& = (b_1 + b_2 - a b_1 - (1+a) b_2) |\xi|^2 X_l + ( ab_1 + (1+a) b_2) \xi_l \langle \xi, X \rangle \\
& = ((1-a) b_1 - a b_2 ) |\xi|^2 X_l + (a b_1 + (1+a) b_2) \xi_l \langle \xi, X \rangle.
\end{aligned}
\end{equation}
Therefore, decomposing $X = X^{\perp} + \frac{\langle \xi, X\rangle}{|\xi|^2} \xi$, we have
\begin{align*}
0& = ((1-a)b_1 - a b_2 ) |\xi|^2 X_l^{\perp} + ((1-a)b_1 - a b_2+ab_1 + (1+a)b_2 ) \xi_l \langle\xi, X\rangle\\
& = (b_1 - a (b_1 + b_2 )) |\xi|^2 X_l^{\perp} + (b_1 + b_2) \xi_l \langle\xi,X\rangle,
\end{align*}
which implies that $X=0$, provided that~\eqref{eq:riccilike_weak_ell} holds. 

Since $\dim (\ker \tilde B)= 2 8$ by~\eqref{eq:dim-ker-tildeB}, and $\dim\Lambda^3(T^*_x M) =35$, equation~\eqref{eq:kerBkersigma} implies that 
$$ \dim (\ker \sigma_{\xi} (D_{\ph} P) ) \leq 35 - 28 =7. $$

On the other hand, by~\eqref{eq:im-delta-star}, we already know that $\ker (\sigma_{\xi} (D_{\ph} P))$ contains the $7$-dimensional subspace
$$\im (\sigma_{\xi} (\delta^*)) = \{( \xi \otimes V + V \otimes \xi, - \xi \times V) : V\in T^*_x M \}, $$
hence $\dim (\ker \sigma_{\xi} (D_{\ph} P)) \geq 7$, which proves that
$$ \ker (\sigma_{\xi} (D_{\ph} P)) = \im (\sigma_{\xi} (\delta^*)). $$

Finally, if $b_1 + b_2 = 1$ and $a = -b_2$, then $b_1 = 1 + a$ and $a b_1 + (1+a) b_2 = 0$. In particular,~\eqref{eq:riccilike_weak_ell} holds, and from~\eqref{eq:symbDP_ker_1} we find that for every $(h,X) \in\ker\tilde B$, we have
$$ \sigma_{\xi} (D_{\ph} P) (h, X) = (|\xi|^2 h, |\xi|^2 X). $$
Therefore, in this case, $\sigma_{\xi} (D_{\ph} P)$ preserves the subspace $\ker\tilde B$.
\end{proof}

We now analyze the principal symbol of the linearization $D_{\ph} F$ of the symmetric $2$-tensor $F$ of~\eqref{eq:F-defn}.

\begin{prop} \label{prop:kernel_symb_F}
For any $x \in M$ and nonzero $\xi \in T^*_x M$, we have
\begin{equation} \label{eq:kernel_symb_F}
\ker (\sigma_{\xi} (D_{\ph} F)) = \{( \xi \otimes V + V \otimes \xi, X) : V, X \in T^*_xM \}.
\end{equation}
Moreover, the operator $\sigma_{\xi} (D_{\ph} F \oplus 0_{T^*_x M})$ preserves the subspace $\{ (h,0) : h(\xi) = 0 \}$, and we have an orthogonal decomposition
\begin{equation} \label{eq:symbDF_decomp}
S^2 (T^*_x M) \oplus T^*_x M = \ker (\sigma_{\xi} (D_{\ph} F)) \oplus E^+ \oplus E^-
\end{equation}
where $E_{\pm} \subset \{(h,0) : h(\xi) =0\}$ and
\begin{align*}
E^+ & = \ker( 4| \xi|^2 - \sigma_{\xi} (D_{\ph} F \oplus 0_{T^*_x M}) ), \\
E^- & = \ker ( 4 |\xi|^2 + \sigma_{\xi} (D_{\ph} F \oplus 0_{T^*_x M}) ),
\end{align*}
so that
$$ \rest{\sigma_{\xi} (D_{\ph} F)}{E_{\pm}} (h,0) = \pm 4 |\xi|^2 h. $$
In particular, for any $(h,X) \in S^ 2(T^*_x M) \oplus T^*_x M$, we have
\begin{equation} \label{eq:symbDFineq}
| \langle \sigma_{\xi} (D_{\ph} F \oplus 0_{T^*_x M}) (h,X), (h,X) \rangle | \leq 4 |\xi|^2 | (h,X) |^2.
\end{equation}
\end{prop} 
\begin{proof}
Recall from~\eqref{eq:psymbF} that the symbol of $F$ is given by 
\begin{equation*}
\sigma_{\xi} (D_{\ph} F) (h,X)_{jk}=4\xi_p\xi_q h_{ab} \ph_{paj} \ph_{qbk}.
\end{equation*}
From the computation
\begin{align*}
\langle \sigma_{\xi} (D_{\ph} F \oplus 0_{T^*_x M}) (h,X), (f, Y) \rangle & = 4\xi_p \xi_q h_{ab} \ph_{paj} \ph_{qbk} f_{jk} \\
& = 4\xi_p\xi_q f_{jk} \ph_{pja} \ph_{qkb} h_{ab} \\
& = \langle (h, X), \sigma_{\xi} (D_{\ph} F \oplus 0_{T^*_x M}) (f, Y) \rangle
\end{align*}
we see that $\sigma_{\xi} (D_{\ph} F \oplus 0_{T^*_x M})$ is a self-adjoint linear operator on $S^ 2(T^*_x M) \oplus T^*_x M$. Hence this map is diagonalizable with real eigenvalues, and $S^2 (T^*_x M) \oplus T^*_x M$ decomposes into the orthogonal eigenspaces of $\sigma_{\xi} (D_{\ph} F \oplus 0_{T^*_x M})$. 

Define the operator $B^* \colon T_x^* M \oplus T_x^* M \to S^2 (T^*_x M) \oplus T^*_x M$ by
$$ B^* (V,X) = (\xi \otimes V + V \otimes \xi, X), \quad \text{for $V, X \in T_x^* M$}. $$
The first component of $B^*$ is the principal symbol of the operator $V \mapsto \cL_V g$, so as is expected by the diffeomorphism invariance of the curvature tensor, we should have
\begin{equation} \label{eq:imBkerDF}
\im B^* \subseteq \ker (\sigma_{\xi} (D_{\ph} F) ) = \ker (\sigma_{\xi} (D_{\ph} F \oplus 0_{T^*_x M}) ).
\end{equation}
Indeed, this holds since
$$ 4 \xi_p \xi_q (V_a \xi_b + \xi_a V_b) \ph_{paj} \ph_{qbk} = 0. $$
Observe that the kernel here is \emph{larger} than the space $\im (\sigma_{\xi} (\delta^*))$ of~\eqref{eq:im-delta-star}, as it also contains all the isometric variations of the $\G$-structure. This is expected, because to highest order
$$ F_{jk} = R_{abcd} \ph_{abj} \ph_{cdk} + \lot $$
depends only on the induced Riemannian metric.

From the computation
$$ \langle B^* (V,X), (h,Y) \rangle = (\xi_i V_j + \xi_j V_i) h_{ij} + X_k Y_k = 2 \langle h(\xi), V \rangle + \langle Y, X \rangle $$
we see that the adjoint $B$ of $B^*$ is the map
\begin{equation} \label{eq:adjoint-Bstar}
B (h,Y) = (2 h(\xi), Y).
\end{equation}
Since we expect to show that $\ker (\sigma_{\xi} (D_{\ph} F \oplus 0_{T^*_x M}) ) = \im B^*$, the remaining eigenspaces should span the orthogonal complement of $\im B^*$, namely the kernel of its adjoint $B$, which by~\eqref{eq:adjoint-Bstar} would be
\begin{equation} \label{eq:kerB}
\ker B = \{ (h,0), h(\xi) =0 \}.
\end{equation}
To see that indeed $\ker (\sigma_{\xi} (D_{\ph} F \oplus 0_{T^*_x M})) \cap \ker B = \{ 0 \}$, observe that
\begin{align} \nonumber
|\sigma_{\xi} (D_{\ph} F) (h,X) |^2 & = \langle \sigma_{\xi} (D_{\ph} F) (h, X), \sigma_{\xi} (D_{\ph} F) (h, X) \rangle \\ \nonumber
& =16 \xi_p\xi_q h_{ab} \ph_{paj} \ph_{qbk} \xi_m \xi_n h_{st} \ph_{msj} \ph_{ntk} \\ \nonumber
& =16 \xi_p \xi_q \xi_m \xi_n h_{ab} h_{st} \ph_{paj} \ph_{msj} \ph_{qbk} \ph_{ntk} \\ \nonumber
& =16 \xi_p \xi_q \xi_m \xi_n h_{ab} h_{st} (g_{pm} g_{as} - g_{ps} g_{am} - \ps_{pams}) \ph_{qbk} \ph_{ntk} \\ \nonumber
& =16 (|\xi|^2 \xi_q \xi_n h_{ab} h_{at} - \xi_p h_{pt} \xi_m h_{mb} \xi_q \xi_n) (g_{qn} g_{bt} - g_{qt} g_{bn} - \ps_{qbnt}) \\ \label{eq:norm_symbol_DF}
& =16 |\xi|^4 |h|^2 - 16 |\xi|^2 |h(\xi) |^2 - 16 |\xi|^2 |h(\xi) |^2 + 16 (h(\xi,\xi))^2.
\end{align}
Hence, if $(h,X) \in \ker (\sigma_{\xi} (D_{\ph} F \oplus 0_{T^*_x M})) \cap \ker B$, we have $h(\xi) =0$ and thus $|\xi|^4 |h|^2 =0$. Since $\xi \neq 0$, we get $h = 0$. This, together with~\eqref{eq:imBkerDF}, proves that 
\begin{equation} \label{eq:kersymb_imB}
\ker \sigma_{\xi} (D_{\ph} F \oplus 0_{T^*_x M}) = \im B^*,
\end{equation}
and thus~\eqref{eq:kernel_symb_F}.

On the other hand, let $(h, X)$ be an eigenvector of $\sigma_{\xi} (D_{\ph} F \oplus 0_{T^*_x M})$ with eigenvalue $\lambda \neq 0$. Then $(h,X) \perp \ker (\sigma_{\xi} (D_{\ph} F \oplus 0_{T^*_x M}))$, and thus, by~\eqref{eq:kersymb_imB}, we have $(h, X) \in \ker B$. Then~\eqref{eq:kerB} implies that $X=0$, $h(\xi) = 0,$ and hence the eigenvalue equation becomes
\begin{equation} \label{eq:F_eigen_equation}
\sigma_{\xi} (D_{\ph} F\oplus 0_{T^*_x M}) (h,0) = \lambda (h,0).
\end{equation}
Using~\eqref{eq:norm_symbol_DF}, we get
\begin{equation*}
| \sigma_{\xi} (D_{\ph} F \oplus 0_{T^*_x M}) (h,0) |^2 = \lambda^2 |h|^2 \quad \iff \quad 16 |\xi|^4 |h|^2 = \lambda^2 |h|^2,
\end{equation*}
which implies that $\lambda = \pm 4 |\xi|^2$. Denoting the corresponding eigenspaces as
\begin{align*}
E^+ & = \ker (4 |\xi|^2 - \sigma_{\xi} (D_{\ph} F \oplus 0_{T^*_x M}) ) \\
E^- & = \ker (4 |\xi|^2 + \sigma_{\xi} (D_{\ph} F \oplus 0_{T^*_x M}) ),
\end{align*}
we obtain the orthogonal decomposition~\eqref{eq:symbDF_decomp}.
\end{proof}

We can understand the eigenspaces $E^+$, $E^-$ more geometrically as follows. Given a nonzero $\xi \in T^*_x M$, we have an orthogonal decomposition
\begin{equation} \label{eq:temp-Fdecomp}
S^2 (T^*_x M) = \Span \{ \xi \otimes V + V \otimes \xi : V \in T_x^* M \} \oplus S^2 (\xi^{\perp}),
\end{equation}
where $\xi^{\perp}$ is the orthogonal complement of $\Span \{ \xi \}$ in $T^*_x M$. It is easy to see that $h \in S^2 (\xi^{\perp})$ if and only if $h(\xi) = 0$. (Note that by~\eqref{eq:kerB} we have $S^2(\xi^{\perp})$ is precisely $\ker B$.)

Let $\hat{\xi} = \frac{1}{|\xi|} \xi$, so $|\hat{\xi}| = 1$. Consider the operator $J \colon T^*_x M \to T^*_x M$ given by
$$ J_{ij} = -\hat{\xi}_p \ph_{pij}. $$
This is the skew-adjoint operator $J(Y) = \hat{\xi} \times Y$, corresponding to the skew-symmetric bilinear form $J = - \hat{\xi} \hk \ph \in \Lambda^2_7 (T^*_x M)$. Note that $J$ preserves $\xi^{\perp}$. Moreover, from
\begin{align*}
J^2_{ij} & = J_{ik} J_{kj} = (-\hat{\xi}_p \ph_{pik})(-\hat{\xi}_q \ph_{qkj}) \\
& = \hat{\xi}_p \hat{\xi}_q (g_{pj} g_{iq} - g_{pq} g_{ji} - \ps_{pijq}) = \hat{\xi}_i \hat{\xi}_j - g_{ij},
\end{align*}
we have
\begin{equation} \label{eq:J-sq}
J^2 = \hat{\xi} \otimes \hat{\xi} - \sI.
\end{equation}
That is, $J$ is a complex structure on the $6$-dimensional space $\xi^{\perp}$.

Given $h \in S^2 (\xi^{\perp})$, define
$$ h^{\pm} = \tfrac{1}{2} (h \pm J h J). $$
Thus we can write
$$ h = h^+ + h^- = \tfrac{1}{2} (h + J h J) + \tfrac{1}{2} (h - J h J). $$
Observe using~\eqref{eq:psymbF} that
\begin{equation} \label{eq:JhJ-F}
\begin{aligned}
(J h J)_{ab} & = J_{ai} h_{ij} J_{jb} = (- \hat{\xi}_p \ph_{pai}) h_{ij} (- \hat{\xi}_q \ph_{qjb}) \\
& = - |\xi|^{-2} \xi_p \ph_{pia} h_{ij} \xi_q \ph_{qjb} = - \tfrac{1}{4 |\xi|^2} \sigma_{\xi} (D_{\ph} F) (h, X). 
\end{aligned}
\end{equation}
Using~\eqref{eq:JhJ-F}, we conclude that if $h \in S^2(\xi^{\perp})$, then
$$ h^{\pm} = 0 \iff h = \mp JhJ \iff \sigma_{\xi}(D_{\ph} F \oplus 0_{T^*_x M}) = \pm 4 |\xi|^2 h \iff h \in E^{\pm}. $$
In particular, if $h \in E^-$, then $J h J = h$, so $\tr h = \tr(JhJ) = \tr(h J^2) = - \tr h$, and thus $\tr h = 0$. It is also easy to see that the element $k = g - \hat{\xi} \otimes \hat{\xi}$ of $S^2 (\xi^{\perp})$ lies in $E^+$.

Observe also that since $h \in S^2(\xi^{\perp})$, we have $h (\hat{\xi} \otimes \hat{\xi}) = (\hat{\xi} \otimes \hat{\xi}) h = 0$. Using this and~\eqref{eq:J-sq}, we have
\begin{align*}
h = \mp J h J & \implies J h = \mp J^2 h J = \mp (-\sI + \hat{\xi} \otimes \hat{\xi}) h J = \pm h J, \\
Jh = \pm hJ & \implies JhJ = \pm hJ^2 = \pm h(-\sI + \hat{\xi} \otimes \hat{\xi}) = \mp h.
\end{align*}
Thus we can equivalently describe the eigenspaces $E^{\pm}$ by
$$ E^{\pm} = \{ h \in S^2 (\xi^{\perp}) : h J = \pm J h \}. $$

\subsection{Breaking the diffeomorphism invariance} \label{sec:break_diffeo_invariance}

In this section we prove that, given a background $\G$-structure $\tilde \ph$, it is always possible to modify a Ricci-like operator $P$ to an operator which is strongly elliptic at $\tilde \ph$, that is an operator $Q$ whose symbol satisfies
$$ \langle [\sigma_{\xi} (D_{\tilde \ph} Q)] (h, X), (h, X) \rangle \geq c |\xi|^2 | (h,X) |^2 = c |\xi|^2(|h|^2+|X|^2) $$
for some constant $c > 0$. (In fact we consider a slightly more general situation than just Ricci-like operators, as we also generalize it to allow an $F \diamond \ph$ term.)

As in~\eqref{eq:ricci_like}, consider a Ricci-like operator
\begin{equation} \label{eq:Pdef}
P(\ph) = (- \tRc + a \cL_{\Vop T} g) \diamond \ph + (b_1 \Div T + b_2 \Div T^t) \hk \ps.
\end{equation} 
Denoting by $\tilde g$ the Riemannian metric induced by $\tilde \ph$ and by $\tilde \Gamma$ its Christoffel symbols, define the vector field $W(\ph, \tilde \ph)$ on $M$ by
\begin{equation} \label{eq:Wdef}
W^k =g^{ij} (\Gamma_{ij}^k - \tilde\Gamma_{ij}^k) - 2 a (\Vop T)^k = \tilde W^k - 2 a (\Vop T)^k,
\end{equation}
and the operator $Q$ by
\begin{equation} \label{eq:Qdef}
Q(\ph) = P(\ph) + \cL_{W(\ph,\tilde \ph)} 
\ph.
\end{equation}

We begin with the following lemma.
\begin{lemma} \label{lemma:Q}
Let $M$ be a $7$-manifold admitting a $\G$-structure $\tilde \ph$, and define the differential operators $P$ and $Q$ acting on $\G$-structures, as in~\eqref{eq:Pdef} and~\eqref{eq:Qdef}. Then we have
$$ Q(\ph) = (- \tRc + \tfrac{1}{2} \cL_{\tilde W} g) \diamond \ph + (- \tfrac{1}{2} \curl \tilde W + (b_1 - a) \Div T + (b_2 + a) \Div T^t - a q_7 (\ph) + W \hk T) \hk \ps, $$
where 
$$ q_7 (\ph) =T (\Vop T) -T^t (\Vop T). $$
\end{lemma}
\begin{proof}
Using~\eqref{eq:Lie-derivative-alternate} and the definition of $W$ we have
\begin{equation} \label{eq:Q-temp}
\begin{aligned}
Q(\ph) & = (- \tRc + a \cL_{\Vop T} g + \tfrac{1}{2} \cL_W g) \diamond \ph + ( b_1 \Div T + b_2 \Div T^t - \tfrac{1}{2} \curl W + W \hk T) \hk \ps \\
& = ( - \tRc + \tfrac{1}{2} \cL_{\tilde W} g) \diamond \ph + ( b_1 \Div T + b_2 \Div T^t - \tfrac{1}{2} \curl W + W \hk T) \hk \ps.
\end{aligned}
\end{equation}
Now, using Corollary~\ref{cor:curlVT-revisited}, we compute
\begin{align*}
\curl W & = \curl \tilde W - 2 a \curl (\Vop T) \\
& = \curl \tilde W + 2 a ( \Div T - \Div T^t ) - 2 a T^t (\Vop T) + 2 a T (\Vop T).
\end{align*}
Substituting the above into~\eqref{eq:Q-temp} yields the result.
\end{proof}

The main result of this section is Proposition~\ref{prop:strong_ellipticity_Q_withF}, which demonstrates that several modifications of Ricci-like operators are strongly elliptic in the sense of Definition~\ref{def:ellipticity}.

We first need to understand the linearization of the operator $Q$ defined in~\eqref{eq:Qdef}. Observe that, ignoring lower order terms, we have
\begin{equation} \label{eq:DQ-temp}
\begin{aligned}
D_{\tilde \ph} Q & = (D_{\tilde \ph} ( - \tRc + \tfrac{1}{2} \cL_{\tilde W} g) + \lot) \diamond \tilde \ph \\
& \qquad {} + (D_{\tilde \ph} (- \tfrac{1}{2} \curl \tilde W + (b_1- a) \Div T + (b_2 + a) \Div T^t) + \lot) \hk \tilde \ps.
\end{aligned}
\end{equation}
From equations~\eqref{eq:psymbricci},~\eqref{eq:psymbdivT}, and~\eqref{eq:psymbdivTt}, we have
\begin{align*}
\sigma_{\xi} ( - D_{\tilde \ph} \tRc) (h,X)_{jk} & = |\xi|^2 h_{jk} - \xi_j( \xi_a h_{ak} - \tfrac{1}{2} \xi_k \tr h ) - \xi_k ( \xi_a h_{aj} - \tfrac{1}{2} \xi_j \tr h), \\
\sigma_{\xi} (D_{\tilde \ph} \Div T) (h,X)_k & = \xi_a\xi_p h_{pb} \ph_{abk} + |\xi|^2 X_k, \\
\sigma_{\xi} (D_{\tilde \ph} \Div T^t) (h,X)_k & = \xi_k \langle \xi,X\rangle.
\end{align*}
On the other hand, it is well known (see~\cite[Chapter 3, \textsection 3.2]{Chow-Knopf} for example) that the linearization of $\tilde W$ is, up to lower order terms, given by the 
Bianchi operator, namely 
$$ (D_{\tilde \ph} \tilde W) (h,X) = 2 (\Div_{\tilde g} h - \tfrac{1}{2} \nabla \tr_{\tilde g} h) + \lot. $$
Note that the factor of $2$ here is because for any variation $h \diamond \ph + X \hk \ps$ of $\G$-structures, the corresponding variation of the metric is $2h$.

From this we obtain
\begin{align*}
\sigma_{\xi} (\tfrac{1}{2}D_{\tilde \ph} \cL_{\tilde W}g ) (h,X)_{jk} & = \xi_j (\xi_p h_{pk} - \tfrac{1}{2} \xi_k \tr h) + \xi_k(\xi_p h_{pj} - \tfrac{1}{2} \xi_j \tr h ), \\
\sigma_{\xi} (- \tfrac{1}{2}D_{\tilde \ph} \curl \tilde W) (h,X)_l& = - \xi_a (\xi_p h_{pb} - \tfrac{1}{2} \xi_b \tr h ) \ph_{abl}.
\end{align*}
Combining the above computations, we deduce that
$$ \sigma_{\xi} [ D_{\tilde \ph} ( - \tRc + \tfrac{1}{2} \cL_{\tilde W} g ) ] (h,X)_{jk} = |\xi|^2 h_{jk}, $$
and also that
\begin{align*}
& \qquad \sigma_{\xi} (D_{\tilde \ph} (- \tfrac{1}{2} \curl \tilde W + (b_1 - a) \Div T + (b_2 + a) \Div T^t )) (h,X)_l \\
& = - \xi_a (\xi_p h_{pb} - \tfrac{1}{2} \xi_b \tr h ) \ph_{abl} + (b_1 - a) (\xi_a \xi_p h_{pb} \ph_{abl} + |\xi|^2 X_l) + (b_2 + a) \xi_l \langle \xi, X \rangle \\
& = (b_1 - a -1) \xi_a \xi_p h_{pb} \ph_{abl} + (b_1 - a) |\xi|^2 X_l + (b_2+a) \xi_l \langle \xi, X \rangle.
\end{align*}
From the above expressions and~\eqref{eq:DQ-temp}, we finally conclude that for any $(h, X) \in S^2 (T^*_x M) \oplus T^*_x M$, we have
\begin{equation} \label{eq:symbDQ}
\begin{aligned}
\langle \sigma_{\xi} (D_{\tilde \ph} Q) (h, X), (h, X) \rangle & = |\xi|^2 |h|^2 + (b_1 - a -1) \xi_a \xi_p h_{pb} \ph_{abl} X_l \\
& \qquad {} + (b_1 - a) |\xi|^2 |X|^2 + (b_2 + a) \langle \xi, X \rangle^2.
\end{aligned}
\end{equation}
Moreover, noting from~\eqref{eq:B-maps} that $|B_2 (X) |^2 = \xi_a X_b \ph_{abk} \xi_i X_j \ph_{ijk} = |\xi|^2 |X|^2 - \langle \xi, X \rangle^2$, by completing the square and using $|h(\xi)|^2 \leq |h|^2 |\xi|^2$, we obtain
\begin{align*}
\xi_a \xi_p h_{pb} \ph_{abl} X_l & = - \langle h(\xi), B_2 (X) \rangle \\
& = \tfrac{1}{4} |h(\xi) |^2 - 2 \left \langle \tfrac{1}{2} h(\xi), B_2 (X) \right \rangle + | B_2 (X) |^2 - \tfrac{1}{4} | h(\xi) |^2 - | B_2 (X) |^2 \\
& = \left| \tfrac{1}{2} h(\xi) - B_2 (X) \right|^2 - \tfrac{1}{4} | h(\xi) |^2 - | B_2 (X) |^2 \\
& \geq - \tfrac{1}{4} |\xi|^2 |h|^2 - |\xi|^2 |X|^2 + \langle \xi, X \rangle^2,
\end{align*}
and similarly that
$$ \xi_a \xi_p h_{pb} \ph_{abl} X_l \leq \tfrac{1}{4} |\xi|^2 |h|^2 + |\xi|^2 |X|^2 - \langle \xi, X \rangle^2, $$
which combine to give
$$ |\xi_a \xi_p h_{pb} \ph_{abl} X_l| \leq \tfrac{1}{4} |\xi|^2 |h|^2 + |\xi|^2 |X|^2 - \langle \xi, X \rangle^2. $$
It then follows from~\eqref{eq:symbDQ} that
\begin{align} \nonumber
\langle \sigma_{\xi} (D_{\tilde \ph} Q) (h, X), (h, X) \rangle & \geq |\xi|^2 |h|^2 - |b_1 - a - 1| (\tfrac{1}{4} |\xi|^2 |h|^2 + |\xi|^2 |X|^2 - \langle \xi, X \rangle^2) \\ \nonumber
& \qquad {} + (b_1 - a) |\xi|^2 |X|^2 + (b_2 + a) \langle \xi, X \rangle^2 \\ \nonumber
& = (1- \tfrac{1}{4} |b_1 - a - 1|) |\xi|^2 |h|^2 + (b_1 - a - |b_1 - a - 1|) |\xi|^2 |X|^2 \\ \label{eq:symbDQ2}
& \qquad {} + (|b_1 - a - 1| + b_2 + a) \langle \xi, X \rangle^2.
\end{align}
In Proposition~\ref{prop:strong_ellipticity_Q_withF} below we prove our most general strong ellipticity result. However, for the sake of clarity, we first demonstrate that $D_{\tilde \ph} Q$ is strongly elliptic in the following two special cases.

\textbf{Special case I:} $b_1 = 1 + a$, $b_2 = - a$. By~\eqref{eq:Pdef} this choice corresponds to operators $P$ of the form
$$ P (\ph) = ( - \tRc + a \cL_{\Vop T} g) \diamond \ph + ( (1+a) \Div T - a \Div T^t) \hk \ps. $$
For this type of operator,~\eqref{eq:symbDQ} becomes
\begin{equation} \label{eq:riccilike_good}
\langle\sigma_{\xi} (D_{\tilde \ph}Q) (h, X), (h, X) \rangle = |\xi|^2 (|h|^2+|X|^2) = |\xi|^2 | (h,X) |^2,
\end{equation}
thus $D_{\tilde \ph}Q$ is strongly elliptic, according to Definition~\ref{def:ellipticity}.

More generally, one can easily check that if $b_1 = 1 + a$ and $b_2 \geq -a$ we still obtain that
\begin{equation}
\langle \sigma_{\xi} (D_{\tilde \ph} Q) (h, X), (h, X) \rangle \geq | (h,X) |^2,
\end{equation}
showing that $D_{\tilde \ph} Q$ is strongly elliptic. $\blacktriangle$

\textbf{Special case II:}  $a = - \frac{1}{2}$, $b_1 = 1$, $b_2 = 0$. By~\eqref{eq:Pdef} this choice corresponds to operators $P$ of the form
$$ P (\ph) = (- \tRc - \tfrac{1}{2} \cL_{\Vop T} g) \diamond \ph + \Div T \hk \ps. $$
Note from Corollary~\ref{cor:evolution-torsion-quantities-4} that the higher order terms in the negative gradient flow of the functional 
$$ \ph \mapsto \tfrac{1}{2} \int_M |T|^2 \vol $$
form exactly this operator $P$. In this case~\eqref{eq:symbDQ2} implies that
\begin{equation} \label{eq:symbDQ3}
\langle \sigma_{\xi} (DQ) (h, X), (h,X ) \rangle \geq \tfrac{7}{8} (|\xi|^2 |h|^2 + |\xi|^2 |X|^2) = \tfrac{7}{8} | (h,X) |^2.
\end{equation}
Thus $D_{\tilde \ph}Q$ is strongly elliptic in this case as well.

More generally, one can check using~\eqref{eq:symbDQ2} that if $0 \leq b_1 - a - 1 < 4$ and $b_1 + b_2 \geq 1$, then
\begin{equation} \label{eq:Q_strong_ellipticity_general}
\begin{aligned}
\langle \sigma_{\xi} (D_{\tilde \ph} Q) (h, X), (h, X) \rangle & = c |\xi|^2 |h|^2 + |\xi|^2 |X|^2 + (b_1 + b_2 - 1) \langle \xi, X \rangle^2 \\
& \geq c |\xi|^2 | (h,X) |^2,
\end{aligned}
\end{equation}
where $c =1- \frac{1}{4} (b_1 - a - 1) > 0$, showing that $D_{\tilde \ph} Q$ is strongly elliptic. $\blacktriangle$

We can now state and prove our main result for strong ellipticity of second-order quasilinear differential operators on $\G$-structures. 
\begin{prop} \label{prop:strong_ellipticity_Q_withF}
Let $M$ be a $7$-manifold with a $\G$-structure $\tilde \ph$, and let $\hat P \colon \Omega^3_+ (M) \to \Omega^3 (M)$ be the quasilinear differential operator
\begin{equation}
\hat P (\ph) = (- \tRc + a \cL_{\Vop T} g + \lambda F) \diamond \ph + (b_1 \Div T + b_2 \Div T^t) \hk \ps.
\end{equation}
Let $W(\ph, \tilde \ph)^k = g^{ij} (\Gamma_{ij}^k - \tilde \Gamma_{ij}^k) -2 a \Vop T^k$ as in~\eqref{eq:Wdef}, where $\tilde \Gamma$ denotes the Christoffel symbols of the Riemannian metric $\tilde g$ induced by $\tilde \ph$, and suppose that
\begin{align} \label{eq:strong_ellipt_conditions}
& \quad b_1 + b_2 \geq 1, & 0 & \leq b_1 - a-1 < 4, & |\lambda| & < \tfrac{1}{4} (1 - \tfrac{1}{4} (b_1 - a - 1)).
\end{align}
Then the differential operator $\hat Q (\ph) = \hat P(\ph) + \cL_{W(\ph,\tilde \ph)} g$ is strongly elliptic at $\tilde \ph$.
\end{prop}
\begin{proof}
Let $Q$ be the operator defined in~\eqref{eq:Qdef} and let $\tilde Q = Q + \lambda F$. Combining~\eqref{eq:symbDFineq} with~\eqref{eq:Q_strong_ellipticity_general}, we obtain
\begin{equation}
\begin{aligned}
\langle \sigma_{\xi} (D_{\tilde \ph} \hat Q) (h, X), (h, X) \rangle & = \langle \sigma_{\xi} (D_{\tilde \ph} Q) (h, X), (h, X) \rangle + \lambda \langle \sigma_{\xi} (D_{\tilde \ph} F) (h, X), (h,X ) \rangle \\
& \geq c |\xi|^2 | (h,X) |^2 - 4 |\lambda| \, |\xi|^2 \, | (h,X) |^2 \\
& = (c - 4 |\lambda|) |\xi|^2 | (h,X) |^2,
\end{aligned}
\end{equation}
where $c = 1- \frac{1}{4} (b_1 - a - 1) >0$. Thus $D_{\tilde \ph} \hat Q$ is strongly elliptic since $ |\lambda| < \frac{1}{4} c$.
\end{proof}

\subsection{Short-time existence and uniqueness of flows of $\G$-structures} \label{sec:short-time}

In this section we prove our main short-time existence and uniqueness theorem for flows of $\G$-structures. The argument is a slight modification of the DeTurck argument for Ricci flow, as described for example in Chow--Knopf~\cite[Chapter 3]{Chow-Knopf}. See also Remark~\ref{rmk:STE}.

\begin{thm} \label{thm:main-short-time}
Let $(M,\ph_0)$ be a compact $7$-manifold with a $\G$-structure $\ph_0$. Consider the flow
\begin{equation} \label{eq:g2flow_general}
\begin{aligned}
\frac{\partial}{\partial t} \ph(t) & = (- \tRc + a \cL_{\Vop T} g + \lambda F) \diamond \ph + (b_1 \Div T+ b_2 \Div T^t) \hk \ps, \\
\ph(0) & = \ph_0,
\end{aligned}
\end{equation}
and suppose that $0 \leq b_1 - a-1 < 4$, $b_1 + b_2 \geq 1$ and $|\lambda| < \frac{1}{4} c$, where $c = 1 - \frac{1}{4} (b_1 - a-1) > 0$.

Then there exists $\eps > 0$ and a unique smooth one-parameter family of $\G$-structures $\ph(t)$ for $t \in [0,\eps)$, solving~\eqref{eq:g2flow_general}.
\end{thm}
\begin{proof}
Let $\tilde \ph= \ph_0$ and define $W(\ph, \tilde \ph)$ as in~\eqref{eq:Wdef}. From Proposition~\ref{prop:strong_ellipticity_Q_withF} we know that the linearization of the quasilinear operator
$$ \hat Q(\ph) = (- \tRc + a \cL_{\Vop T} g + \lambda F) \diamond \ph + (b_1 \Div T + b_2 \Div T^t) \hk \ps + \cL_W \ph $$
is a strongly elliptic, under the assumptions of the theorem. Thus, from standard parabolic theory, there is a unique smooth solution $\hat \ph(t)$ for $t \in [0,\eps)$ of the evolution equation
\begin{equation*}
\begin{aligned}
\frac{\partial}{\partial t} \hat \ph(t) & = \hat Q (\hat \ph(t)), \\
\hat \ph(0) & = \ph_0.
\end{aligned}
\end{equation*}
Now, let $\Theta_t \colon M \to M$, $t\in [0, \eps)$ be the one-parameter family of diffeomorphisms defined by
\begin{equation*}
\begin{aligned}
\frac{\partial}{\partial t} \Theta_t & = - W(\hat\ph(t), \tilde \ph) \circ \Theta_t, \\
\Theta_0 & = \Id_M.
\end{aligned}
\end{equation*}
It is then easy to see that $\ph(t) = \Theta_t^* \hat \ph(t)$ satisfies~\eqref{eq:g2flow_general}.

To prove uniqueness, suppose that $\ph_i(t)$, for $i=1,2$, both satisfy~\eqref{eq:g2flow_general} and let $(G_i)_t \colon M \to M$ be a one-parameter family of diffeomorphisms given by the flow of $-2 a (\Vop T)_{\ph_i(t)}$, so that
\begin{equation} \label{eq:Git-flow}
\begin{aligned}
\frac{\partial}{\partial t} (G_i)_t & = -2 a (\Vop T)_{\ph_i(t)} \circ (G_i)_t, \\
(G_i)_0 & = \Id_M.
\end{aligned}
\end{equation}
Using~\eqref{eq:Lie-derivative-alternate} and~\eqref{eq:curlVT}, one computes that $\bar\ph_i(t) = (G_i)_t^*\ph_i(t)$ solves an evolution equation of the form
\begin{equation} \label{eq:barph-flow}
\begin{aligned}
\frac{\partial}{\partial t} \bar \ph_i (t) & = (- \tRc_{\bar g_i (t)} + \lambda F_{\bar \ph_i (t)} + \lot) \diamond \bar \ph_i (t) \\
& \qquad {} + ( (b_1 - a) \Div T_{\bar \ph_i (t)} + (b_2 + a) \Div T_{\bar \ph_i (t)}^t + \lot) \hk \bar \ps_i (t), \\
\bar \ph(0) & = \ph_0.
\end{aligned}
\end{equation}

Define $(\Theta_i)_t \colon M \to M$ for $t \in [0, \eps_i')$ to be the solution to
\begin{align*}
\frac{\partial}{\partial t} (\Theta_i)_t & = \Delta_{g_i(t), g_0} (\Theta_i)_t, \\
(\Theta_i)_0 & = \Id_M,
\end{align*}
which is the harmonic map heat flow from $(M, g_0)$ to $(M, g_i (t))$ with initial value the identity map.

Setting $\hat \ph_i(t) = ((\Theta_i)_t^{-1})^* \bar \ph_i (t)$, and using the fact that $\Delta_{g_i(t), g_0} (\Theta_i)_t = - \tilde W (\hat g_i (t), g_0)$, we have
\begin{equation} \label{eq:Fidiffeos}
\frac{\partial}{\partial t}  (\Theta_i)_t = - \tilde W (\hat g_i (t),\tilde g).
\end{equation}
Using~\eqref{eq:Lie-derivative-alternate} again and~\eqref{eq:barph-flow}, it follows that $\hat \ph_i(t)$, for $i=1,2$ satisfy
\begin{equation} \label{eq:evolbarphi}
\begin{aligned}
\frac{\partial}{\partial t} \hat \ph_i (t) & = ( - \tRc_{\hat g_i (t)} + \lambda F_{\hat \ph_i(t)} + \tfrac{1}{2} \cL_{\tilde W (\hat g_i (t), \tilde g)} \hat g + \lot) \diamond \hat \ph_i (t), \\
& \qquad {} + ((b_1 - a) \Div T_{\hat \ph_i (t)} + (b_2 + a) \Div T_{\hat \ph_i (t)}^t - \tfrac{1}{2} \curl \tilde W + \lot ) \hk \hat \ps_i (t), \\
\hat \ph_i (0) & = \ph_0,
\end{aligned}
\end{equation}
for any $t \in [0,\eps_i')$. Under the assumptions of the theorem, the operator on the right-hand side of~\eqref{eq:evolbarphi} is strongly elliptic, by Lemma~\ref{lemma:Q} and Proposition~\ref{prop:strong_ellipticity_Q_withF}. Hence the uniqueness of standard parabolic theory gives $\hat \ph_1(t) = \hat \ph_2 (t)$ for all $t \in [0, \eps')$, where $\eps' = \min \{ \eps'_1, \eps'_2 \}$. Therefore, $\hat g_1 (t) = \hat g_2 (t)$ and thus by~\eqref{eq:Fidiffeos}, we have $(\Theta_1)_t = (\Theta_2)_t$ and consequently
$$ \bar \ph_1(t) = (\Theta_1)_t^* \hat \ph_1 (t) = (\Theta_2)_t^* \hat \ph_2 (t) = \bar \ph_2 (t), $$
for all $t \in [0,\eps')$.

Since $(G_i)_t \circ (G_i)_t^{-1} = \Id_M$, using~\eqref{eq:Git-flow} we have
\begin{align*}
0 & = \frac{\partial}{\partial t} ( (G_i)_t \circ (G_i)_t^{-1}) \\
& = \frac{\partial (G_i)_t}{\partial t} \circ (G_i)_t^{-1} + ((G_i)_t)_* \Big( \frac{\partial (G_i)_t^{-1}}{\partial t} \Big) \\
& = -2 a (\Vop T)_{\ph_i(t)} + ((G_i)_t)_* \Big( \frac{\partial (G_i)_t^{-1}}{\partial t} \Big),
\end{align*}
and thus
$$ \frac{\partial (G_i)_t^{-1}}{\partial t} = 2 a ((G_i)_t^{-1})_* (\Vop T)_{\ph_i(t)} = 2 a (\Vop T)_{(G_i)_t^* \ph_i (t)} \circ (G_i)_t^{-1} = 2 a (\Vop T)_{\bar \ph_i(t)} \circ (G_i)_t^{-1}. $$
Hence, since $\bar \ph_1 (t) = \bar \ph_2 (t)$, we have $(G_1^{-1})_t = (\G^{-1})_t$ and thus $\ph_1 (t) = \ph_2 (t)$ for all $t \in [0,\eps')$.

Since we can do the above on an open neighbourhood of any $t \in [0, \eps)$, it follows that the subset of $[0, \eps)$ on which $\ph_1 (t) = \ph_ 2(t)$ is both open and closed, and thus $\ph_1 (t) = \ph_2 (t)$ for all $t \in [0,\eps)$.
\end{proof}

\begin{rmk} \label{rmk:STE}
The argument for uniqueness in the proof of Theorem~\ref{thm:main-short-time} is slightly more involved than the usual argument for Ricci flow, in the sense that we need two steps to pass from $\ph$ to $\bar \ph$ and then to $\hat \ph$, rather than just one step, because we still want to use the harmonic map heat flow, which has good local existence. Since $\tilde W = W - 2 a \Vop T$, if we try the usual approach that works in Ricci flow, we would need to establish a good local existence theory for another flow of maps. Our approach avoids this by introducing the additional step mentioned above.
\end{rmk}

\begin{rmk} \label{rmk:Ricci-isometric-coupled}
One particularly interesting case of Theorem~\ref{thm:main-short-time} occurs when we take $a=\lambda=b_2=0$ and $b_1=1$, which satisfies the needed inequalities. This corresponds to the flow
$$ \frac{\partial}{\partial t} \ph(t) = - \tRc \diamond \ph + (\Div T) \hk \ps. $$
This flow induces \emph{precisely} the Ricci flow $\frac{\partial}{\partial t} g = - 2 \tRc$ on the metric, and the only other thing it does to the $\G$-structure $\ph$ is to deform it by the isometric flow~\eqref{eq:isometric-flow}. This ``coupling'' of Ricci flow with the isometric flow thus has good short-time existence and uniqueness. Note that if we did not add the isometric flow to the Ricci flow (that is, if we also took $b_1 = 0$), then the ``pure Ricci flow'' for $\G$-structures does not satisfy the hypotheses of Theorem~\ref{thm:main-short-time}.
\end{rmk}

\begin{rmk} \label{rmk:Weiss-Witt-special-case}
Consider the negative gradient flow of the torsion energy functional $\ph \mapsto \frac{1}{2} \int_M |T|^2 \vol$. By the second equation in Corollary~\ref{cor:evolution-torsion-quantities-4}, this flow is
$$ \frac{\partial}{\partial t} \ph(t) = ( -\tRc - \tfrac{1}{2} \cL_{\Vop T} g ) \diamond \ph + (\Div T) \hk \ps + \lot. $$
That is, we have $a = -\frac{1}{2}$, $b_1 = 1$, and $\lambda = b_2 = 0$, which satisfies the needed inequalities of Theorem~\ref{thm:main-short-time}. We thus recover as a special case the result of Weiss--Witt~\cite{WW1} that this flow has short-time existence and uniqueness.
\end{rmk}

\subsection{Future questions} \label{sec:future}

The analysis in this section raises many interesting questions for future exploration. In Theorem~\ref{thm:main-short-time} we have determined a large class of geometric flows of $\G$-structures which admit a DeTurck trick to establish short-time existence and uniqueness, with no condition on the initial torsion. These are properties of the flow which depend \emph{only} on the second-order terms. Other properties of the flow (such as the characterization of fixed points) are \emph{very sensitive} to the lower order terms.

For example, suppose we take our flow (up to lower order terms) to be the ``coupling'' of Ricci flow with isometric flow described in Remark~\ref{rmk:Ricci-isometric-coupled}. Then can we choose the lower order terms so that the fixed points of the flow are precisely the torsion-free $\G$-structures? That is, are there specific combinations of terms $Q(T) = Q_7 (T) + Q_{1+27} (T)$ which are homogeneous quadratic in the torsion $T$ such that
$$ \big[ - \tRc + Q_{1+27} (T) = 0 \quad \text{and} \quad \Div T + Q_7 (T) = 0 \big] \, \iff \, T = 0. $$
Such a result would be analogous to the result in Bryant~\cite{Bryant} and Cleyton--Ivanov~\cite{CI-closed} which shows that a closed $\G$-structure on a compact manifold inducing an Einstein metric must necessarily be torsion-free and thus Ricci-flat. If such a $Q(T)$ exists, that would give a \emph{preferred coupling} of the Ricci flow to the isometric flow.

Another nice property to ask for a geometric flow of $\G$-structures is that the induced flow of $|T|^2$ is of the form
$$ \frac{\partial}{\partial t} |T|^2 = \Delta |T|^2 + \lot, $$
where the lower order terms are such that the above equation is amenable to a maximum principle.

Of course, once a particular geometric flow of $\G$-structures is chosen, all of the usual questions arise: characterization of the singular time; derivative estimates; long-time existence and convergence; stability; singularity types; solitons; monotonicity of certain quantities; and so on.

\begin{rmk} \label{rmk:Gao-Chen}
Chen~\cite{Chen} considers a class of flows of $\G$-structures which he calls ``reasonable''. His definition of a \emph{reasonable flow} is that it admits short-time existence and uniqueness, and that (up to lower order terms), it has $h = - \tRc$ and $X$ is any vector field depending linearly on $\tRm$ and $\nab{} T$. We have shown that there are no $\mb{7}$ components arising from $\tRm$, and the only independent $\mb{7}$ components arising from $\nab{} T$ are $\Div T$ and $\Div T^t$. Thus, a reasonable flow for Chen \emph{assumes short-time existence and uniqueness}, and is of the form $\frac{\partial}{\partial t} \ph = P(\ph)$, where $P(\ph)$ is of the form~\eqref{eq:general-operator} with $a_1 = -1$, $a_2=a_3=a_4=0$, and $b_1, b_2$ are arbitrary. Chen proves general Shi-type derivative estimates for this class of flows. It would be interesting to see if the flows of Theorem~\ref{thm:main-short-time}, which \emph{all do have} short-time existence and uniqueness, admit Shi-type estimates as in~\cite{Chen}.
\end{rmk}

Another natural direction is to study the role of the optimal $\ph$-connection $\hnab{}$ described in Definition~\ref{defn:optimal-connection} for flows of $\G$-structures. Since $\hnab{}$ has torsion, its ``Ricci tensor'' is not symmetric. The skew part of the Ricci tensor of $\hnab{}$ depends on the torsion $\sT$ of the connection, which is essentially the torsion $T$ of the $\G$-structure, in a repackaged form. Of course, the decomposition into irreducible $\G$-representations of the Ricci tensor of $\hnab{}$ is expressible in terms of the independent second-order differential invariants of a $\G$-structure, and thus we do not get any new flows this way. But given the naturality of $\hnab{}$, its Ricci tensor may be a natural direction in which to flow. Another possibility is to consider the negative gradient flow of the Yang-Mills energy of $\hnab{}$. These questions are being investigated by the authors.

The first author has obtained similar results~\cite{D-Spin7} for geometric flows of $\Spin{7}$-structures. More precisely, he considers the negative gradient flow of the functional which is the $L^2$-norm of the torsion, but over all $\Spin{7}$-structures (not necessarily isometric) and proves short-time existence and uniqueness. It turns out that in the $\Spin{7}$ case the terms appearing in the negative gradient flow are all the terms one could get that are second-order differential invariants which are $4$-forms.

\end{document}